\documentclass[leqno,english]{amsart}

\usepackage{hyperref}   
\usepackage[square,numbers]{natbib}     
\usepackage{amsthm}     
\usepackage{amsmath}    
\usepackage{amssymb}    
\usepackage{amsxtra}    
\usepackage{eucal}      
\usepackage[bbgreekl]{mathbbol}   
\usepackage[all]{xy}    

\usepackage{tikz}
\usetikzlibrary{matrix}
\usetikzlibrary{shapes}
\usetikzlibrary{arrows}
\usetikzlibrary{calc,3d}
\usetikzlibrary{decorations,decorations.pathmorphing,decorations.pathreplacing,decorations.markings,snakes}
\usetikzlibrary{through}

\usepackage{tikz-cd}

\tikzset{snakeit/.style={decorate, decoration={snake, amplitude=.2mm,segment length=1mm}}}

\tikzset{ext/.style={circle, draw,inner sep=1pt},int/.style={circle,draw,fill,inner sep=1.4pt},nil/.style={inner sep=1pt}}
\tikzset{cy/.style={circle,draw,fill,inner sep=2pt},scy/.style={circle,draw,inner sep=2pt},scyx/.style={draw,cross out,inner sep=2pt},scyt/.style={draw,regular polygon,regular polygon sides=3,inner sep=0.95pt}}
\tikzset{exte/.style={circle, draw,inner sep=3pt},inte/.style={circle,draw,fill,inner sep=3pt}}
\tikzset{diagram/.style={matrix of math nodes, row sep=3em, column sep=2.5em, text height=1.5ex, text depth=0.25ex}}
\tikzset{diagram2/.style={matrix of math nodes, row sep=0.5em, column sep=0.5em, text height=1.5ex, text depth=0.25ex}}

\theoremstyle{plain}

\newtheorem{introthm}{Theorem}      
\renewcommand{\theintrothm}{\Alph{introthm}}

\swapnumbers                        
\newtheorem{thm}[subsubsection]{Theorem}
\newtheorem{prop}[subsubsection]{Proposition}
\newtheorem{lemm}[subsubsection]{Lemma}

\theoremstyle{definition}
\newtheorem{defn}[subsubsection]{Definition}
\newtheorem{constr}[subsubsection]{Construction}
\newtheorem{rem}[subsubsection]{Recollections}
\newtheorem{remark}[subsubsection]{Remark}

\numberwithin{subsubsection}{section}
\renewcommand{\thesubsubsection}{\thesection.\arabic{subsubsection}}

\numberwithin{equation}{subsubsection}

\title[The Intrinsic Formality of $\mathrm{E_n}$-operads]{The Intrinsic Formality of $\mathbf{E_n}$-operads}

\date{April 4, 2017 (Revised on February 17, 2018)}

\author{Benoit Fresse}
\address{Laboratoire Paul Painlev\'e\\
Universit\'e de Lille\\
Cit\'e Scientifique - B\^ atiment M2\\
F-59655 Villeneuve d'Ascq Cedex, France}
\email{Benoit.Fresse@univ-lille.fr}

\author{Thomas Willwacher}
\address{ETH Z\"urich\\
Department of Mathematics\\
R\"amistrasse 101\\
CH-8092 Zürich Z\"urich, Switzerland}
\email{thomas.willwacher@math.ethz.ch}

\thanks{B.F. acknowledges support by grant ANR-11-BS01-002 ``HOGT'' and by Labex ANR-11-LABX-0007-01 ``CEMPI''.
T.W. has been partially supported by the Swiss National Science foundation, grant 200021-150012,
the SwissMAP NCCR funded by the Swiss National Science foundation,
and the European Research Council, ERC StG 678156-GRAPHCPX.
The authors thank Victor Turchin for fruitful and motivating exchanges which are at the origin of this work. We
are also grateful to the referee for her/his thorough reading of this paper and for her/his accurate
comments on our work.}


\DeclareMathOperator{\kk}{\mathbb{k}}   
\DeclareMathOperator{\NN}{\mathbb{N}}   
\DeclareMathOperator{\ZZ}{\mathbb{Z}}   
\DeclareMathOperator{\QQ}{\mathbb{Q}}   
\DeclareMathOperator{\RR}{\mathbb{R}}   
\DeclareMathOperator{\II}{\mathbb{I}}   
\DeclareMathOperator{\mymod}{mod}       
\DeclareMathOperator{\coker}{coker}     
\DeclareMathOperator{\colim}{colim}     

\DeclareMathOperator{\ICat}{\mathcal{I}}            
\DeclareMathOperator{\MCat}{\mathcal{M}}            
\DeclareMathOperator{\Mod}{\mathcal{M}\mathit{od}}  
\DeclareMathOperator{\Simp}{\mathit{s}\mathcal{S}\mathit{et}}   

\DeclareMathOperator{\Hopf}{\mathcal{H}\mathit{opf}}    
\DeclareMathOperator{\Seq}{\mathcal{S}\mathit{eq}}  
\DeclareMathOperator{\Op}{\mathcal{O}\mathit{p}}    

\DeclareMathOperator{\ComCat}{\mathcal{C}\mathit{om}}   

\DeclareMathOperator{\f}{\mathit{f}}        
\DeclareMathOperator{\dg}{\mathit{dg}}      
\DeclareMathOperator{\gr}{\mathit{gr}}      
\DeclareMathOperator{\cosimp}{\mathit{c}}   
\DeclareMathOperator{\simp}{\mathit{s}}     

\DeclareMathOperator{\Map}{\mathtt{Map}}    
\DeclareMathOperator{\Mor}{\mathtt{Mor}}    
\DeclareMathOperator{\Hom}{\mathtt{Hom}}    
\DeclareMathOperator{\Der}{\mathtt{Der}}    
\DeclareMathOperator{\CoDer}{\mathtt{CoDer}}    
\DeclareMathOperator{\BiDer}{\mathtt{BiDer}}    
\DeclareMathOperator{\Def}{\mathtt{Def}}    
\DeclareMathOperator{\CoDef}{\mathtt{CoDef}}    
\DeclareMathOperator{\BiDef}{\mathtt{BiDef}}    

\DeclareMathOperator{\Id}{\mathit{Id}}      
\DeclareMathOperator{\id}{\mathit{id}}      
\DeclareMathOperator{\pt}{\mathit{pt}}      
\DeclareMathOperator{\unit}{\mathbb{1}}     

\DeclareMathOperator{\FreeOp}{\mathbb{F}}   
\DeclareMathOperator{\Sym}{\mathbb{S}}      
\DeclareMathOperator{\LLie}{\mathbb{L}}     
\DeclareMathOperator{\BB}{\mathbb{B}}       
\DeclareMathOperator{\KK}{\mathbb{K}}       

\DeclareMathOperator{\DGOmega}{\mathtt{\Omega}}     

\DeclareMathOperator{\DGSigma}{\mathtt{\Sigma}} 
\DeclareMathOperator{\Res}{\mathtt{Res}}        

\DeclareMathOperator{\ecell}{\mathbb{e}}        

\DeclareMathOperator{\DGB}{\mathtt{B}}          
\DeclareMathOperator{\DGC}{\mathtt{C}}          
\DeclareMathOperator{\DGD}{\mathtt{D}}          
\DeclareMathOperator{\DGE}{\mathtt{E}}          
\DeclareMathOperator{\DGF}{\mathtt{F}}          
\DeclareMathOperator{\DGG}{\mathtt{G}}          
\DeclareMathOperator{\DGH}{\mathtt{H}}          
\DeclareMathOperator{\DGI}{\mathtt{I}}          
\DeclareMathOperator{\DGK}{\mathtt{K}}          
\DeclareMathOperator{\DGL}{\mathtt{L}}          
\DeclareMathOperator{\DGN}{\mathtt{N}}          
\DeclareMathOperator{\DGS}{\mathtt{S}}          

\DeclareMathOperator{\DGMC}{\mathtt{MC}}        

\DeclareMathOperator{\Diag}{\mathtt{Diag}}      
\DeclareMathOperator{\Tot}{\mathtt{Tot}}        

\DeclareMathOperator{\palg}{\mathfrak{p}}       

\DeclareMathAlphabet{\mathsfit}{OT1}{cmss}{m}{sl}   
\DeclareMathOperator{\AOp}{\mathsfit{A}}        
\DeclareMathOperator{\BOp}{\mathsfit{B}}        
\DeclareMathOperator{\COp}{\mathsfit{C}}        
\DeclareMathOperator{\DOp}{\mathsfit{D}}        
\DeclareMathOperator{\EOp}{\mathsfit{E}}        
\DeclareMathOperator{\FMOp}{\mathsfit{FM}}      
\DeclareMathOperator{\PiOp}{\mathsfit{\Pi}}        
\DeclareMathOperator{\IOp}{\mathsfit{I}}        
\DeclareMathOperator{\KOp}{\mathsfit{K}}        
\DeclareMathOperator{\HOp}{\mathsfit{H}}        
\DeclareMathOperator{\MOp}{\mathsfit{M}}        
\DeclareMathOperator{\NOp}{\mathsfit{N}}        
\DeclareMathOperator{\POp}{\mathsfit{P}}        
\DeclareMathOperator{\QOp}{\mathsfit{Q}}        
\DeclareMathOperator{\ROp}{\mathsfit{R}}        
\DeclareMathOperator{\SOp}{\mathsfit{S}}        
\DeclareMathOperator{\SuspOp}{\mathsf{\Lambda}}     

\DeclareMathOperator{\ComOp}{\mathsfit{Com}}    
\DeclareMathOperator{\AsOp}{\mathsfit{As}}      
\DeclareMathOperator{\LieOp}{\mathsfit{Lie}}    
\DeclareMathOperator{\PoisOp}{\mathsfit{Pois}}      

\DeclareMathOperator{\gra}{gra}                     
\DeclareMathOperator{\GraOp}{\mathsfit{Gra}}        
\DeclareMathOperator{\FGraphOp}{\mathsfit{fGraphs}}   
\DeclareMathOperator{\GraphOp}{\mathsfit{Graphs}}   
\DeclareMathOperator{\hoe}{\mathsfit{hoe}}      
\DeclareMathOperator{\GCOp}{\mathit{GC}}      
\DeclareMathOperator{\FGCOp}{\mathit{fGC}}      
\DeclareMathOperator{\HGCOp}{\mathit{HGC}}    
\DeclareMathOperator{\FHGCOp}{\mathit{fHGC}}    

\DeclareMathOperator{\ICGOp}{\mathsfit{ICG}}    

\DeclareMathOperator{\kset}{\underline{\mathsf{k}}}     
\DeclareMathOperator{\lset}{\underline{\mathsf{l}}}     
\DeclareMathOperator{\rset}{\underline{\mathsf{r}}}     

\DeclareMathOperator{\ttree}{\underline{\mathsf{T}}}    
\DeclareMathOperator{\gammatree}{\underline{\mathsf{\Gamma}}}    

\begin{document}

\begin{abstract}
We establish that $E_n$-operads satisfy a rational intrinsic formality theorem for $n\geq 3$.
We gain our results in the category of Hopf cooperads in cochain graded dg-modules
which defines a model for the rational homotopy of operads
in spaces.
We consider, in this context, the dual cooperad of the $n$-Poisson operad $\PoisOp_n^c$,
which represents the cohomology of the operad of little $n$-discs $\DOp_n$.
We assume $n\geq 3$.
We explicitly prove that a Hopf cooperad in cochain graded dg-modules $\KOp$
is weakly-equivalent (quasi-isomorphic) to $\PoisOp_n^c$
as a Hopf cooperad
as soon as we have an isomorphism at the cohomology level $\DGH^*(\KOp)\simeq\PoisOp_n^c$
when $4\nmid n$.
We just need the extra assumption that $\KOp$ is equipped with an involutive isomorphism
mimicking the action of a hyperplane reflection on the little $n$-discs operad
in order to extend this formality statement
in the case $4\mid n$.
We deduce from these results that any operad in simplicial sets $\POp$
which satisfies the relation $\DGH^*(\POp,\QQ)\simeq\PoisOp_n^c$
in rational cohomology (and an analogue of our extra involution requirement in the case $4\mid n$)
is rationally weakly equivalent to an operad in simplicial sets $\DGL\DGG_{\bullet}(\PoisOp_n^c)$
which we determine from the $n$-Poisson cooperad $\PoisOp_n^c$.
We also prove that the morphisms $\iota: \DOp_m\rightarrow\DOp_n$, which link the little discs operads together,
are rationally formal as soon as $n-m\geq 2$.

These results enable us to retrieve the (real) formality theorems of Kontsevich by a new approach,
and to sort out the question of the existence of formality quasi-isomorphisms
defined over the rationals (and not only over the reals)
in the case of the little discs operads
of dimension $n\geq 3$.
\end{abstract}

\maketitle

\tableofcontents

\section*{Introduction}
The applications of $E_n$-operads, as objects governing homotopy commutative structures, have recently multiplied.
Let us mention: the second generation of proofs of the existence of deformation
quantizations of Poisson manifolds by Tamar\-kin \cite{Tamarkin}
and Kontsevich~\cite{KontsevichMotives}
which has hinted for the existence of an action of the Grothendieck-Teichm\"uller group
on moduli spaces of deformation
quantizations (see~\cite{Dolgushev,KontsevichMotives,WillwacherGraphs});
the interpretation of the Goodwillie-Weiss tower of embedding spaces
in terms of the homotopy of $E_n$-operads (see notably the results of~\cite{AroneLambrechtsVolic,AroneTurchin});
and the factorization homology of manifolds (see notably~\cite{Francis,LurieHigherAlgebra}
for an introduction to this relationship).

To be more precise, the just cited results on deformation quantization of Poisson manifolds concern the case of $E_2$-operads.
However higher dimensional generalizations of the deformation quantization problem,
which involve applications of $E_n$-operads for all $n\geq 2$ in this subject,
have recently emerged in the works of Calaque-Pantev-To\"en-Vaqui\'e-Vezzosi~\cite{CalaqueAl}
(see also the introduction of~\cite{PantevAl} for an outline of this higher dimensional deformation quantization program).
These applications in deformation quantization (as well as some of the already cited results of~\cite{AroneLambrechtsVolic})
rely on the important observation that $E_n$-operads are formal as operads.
The formality of $E_2$-operads was actually proved by Tamarkin, by using the existence of Drinfeld's associators (see~\cite{TamarkinFormality}),
while Kontsevich established a real version of the formality result in the case $n\geq 2$ (see~\cite{KontsevichMotives}).
Let us mention that the full development of the program of~\cite{CalaqueAl},
which takes place in the setting of (derived) algebraic geometry,
requires to work over arbitrary $\QQ$-algebras.
The understanding of the rational homotopy type of embedding spaces through the Goodwillie-Weiss calculus in~\cite{AroneTurchin}
also requires an understanding of the rational homotopy type of a unitary version of $E_n$-operads (where a term of arity zero is considered)
though the real case of the formality theorem of $E_n$-operads is sufficient for the homological results
of~\cite{AroneLambrechtsVolic}.

This paper fits these new developments of the theory of $E_n$-operads.
Our goal is to prove a rational intrinsic formality theorem, which implies the formality of $E_n$-operads (over the rationals),
and which asserts that $E_n$-operads are characterized by their homology when $n\geq 3$
and when we work up to rational homotopy equivalence of operads.
We establish this result in the category of operads in topological spaces (equivalently, in simplicial sets),
and in the category of Hopf cooperads in cochain differential graded $\QQ$-modules (the category of Hopf cochain dg-cooperads for short).
The class of $E_n$-operads is naturally defined in the category of topological spaces
and consists of objects which are weakly-equivalent
to the little $n$-discs operads $\DOp_n$ (see~\cite{BoardmanVogt,May}),
while the category of Hopf cochain dg-cooperads contains algebraic models
for the rational homotopy of these operads.




We give more details about our statements in the next paragraphs.
We get our results by new obstruction theory methods, which differ from Kontsevich's approach,
and which we can moreover use to prove the homotopy uniqueness
of our formality weak-equivalences
for all $n\geq 3$,
as well as the rational formality of the morphisms $\iota: \DOp_m\rightarrow\DOp_n$
which link the little discs operads together
for $n-m\geq 2$.
We already mentioned that the formality of $E_2$-operads follows from the existence of Drinfeld's associators.
The intrinsic formality of $E_2$-operads is still an open problem however and is related to a question
raised by Vladimir Drinfeld (see~\cite[\S 5, Remarks]{Drinfeld})
about the vanishing of certain obstructions to the existence
of associators.

\subsection*{Background}
Recall that a Hopf operad in a symmetric monoidal category $\MCat$ consists of an operad
in the category of (counitary cocommutative) coalgebras in $\MCat$ (see for instance~\cite[\S I.3.2]{FresseBook}).
The name `Hopf cooperad' refers to the structure, dual to a Hopf operad in the categorical sense,
which consists of a cooperad in the category of (unitary commutative) algebras
in $\MCat$.
We use Hopf cooperads (rather than Hopf operads) in order to handle convergence problems which generally arise
when we deal with coalgebra structures,
and because the definition of our models for the rational homotopy of operads
naturally relies on contravariant functors.

In what follows, we generally use the prefix dg to refer to objects formed in a base category of differential graded modules.
We use the expression `cochain graded', or just the prefix `cochain', to refer to dg-objects equipped with a non-negative upper grading,
while the expression `chain graded' symmetrically refers to dg-objects equipped with a non-negative lower grading.
In fact, the constructions of this paper make sense as soon as we work in a category of dg-modules
over a field of characteristic zero,
but we prefer to take the field of rational numbers as ground ring in the account of this introduction,
since, as we just explained, one motivating feature of our methods is the proof of formality results
that hold over this primary field.
Similarly, we only consider the homology (respectively, the cohomology) with rational coefficients for the moment,
and we also set $\DGH_*(-) = \DGH_*(-,\QQ)$ (respectively, $\DGH^*(-) = \DGH^*(-,\QQ)$) for short.

The homology of the little $n$-discs operads $\DGH_*(\DOp_n)$
naturally forms a Hopf operad in the symmetric monoidal category of chain graded modules $\MCat = \gr_*\Mod$.
For $n\geq 2$, we have an identity of Hopf operads
$\DGH_*(\DOp_n) = \PoisOp_n$,
where $\PoisOp_n$ is the $n$-Poisson operad (also called the $n$-Gerstenhaber operad).
This operad $\PoisOp_n$ is generated by a commutative product operation of degree zero $\mu = \mu(x_1,x_2)\in\PoisOp_n(2)$
together with a Poisson bracket operation $\lambda = \lambda(x_1,x_2)\in\PoisOp_n(2)$,
of degree $\deg(\lambda) = n-1$,
and which is symmetric when $n$ is even, antisymmetric when $n$ is odd.
To get our formality statement, we also use that the little $n$-discs operad is equipped with an orientation reversing
involution $J: \DOp_n\xrightarrow{\simeq}\DOp_n$.
The morphism $J_*: \PoisOp_n\rightarrow\PoisOp_n$ induced by this involution at the homology level
can be determined by the formulas $J_*(\mu) = \mu$ and $J_*(\lambda) = -\lambda$
on our generating operations $\mu,\lambda\in\PoisOp_n(2)$.

In what follows, we generally consider a non-unitary version of the little $n$-discs operad $\DOp_n$, with nothing in arity zero $\DOp_n(0) = \varnothing$,
but we use an extension of the standard symmetric structure of an operad in order to encode the structure
of the standard unitary little $n$-discs operad $\DOp_n^+$
inside this non-unitary operad $\DOp_n$.
In short, the unitary operad $\DOp_n^+$ differs from $\DOp_n$ by the arity zero term $\DOp_n^+(0) = *$.
The idea is that the operadic composition operations associated to this extra arity zero element $*$
can be determined by giving restriction operators $u^*: \DOp_n^+(l)\rightarrow\DOp_n^+(k)$,
associated to the injective maps $u: \{1,\dots,k\}\rightarrow\{1,\dots,l\}$, for $k,l>0$,
together with augmentation morphisms $\epsilon: \DOp_n^+(r)\rightarrow\DOp_n^+(0)$ (which are just trivial in our case
since we have $\DOp_n^+(0) = *$).
These operations are clearly defined within the non-unitary operad $\DOp_n$ underlying $\DOp_n^+$.
Furthermore, by considering this extra structure on $\DOp_n$,
we get an object which is equivalent to the unitary little discs operad $\DOp_n^+$.

We similarly deal with a non-unitary version of the $n$-Poisson operad $\PoisOp_n$
equipped with extra structures which we associate to a unitary version
of this operad $\PoisOp_n^+$.
We explicitly consider restriction operators $u^*: \PoisOp_n(l)\rightarrow\PoisOp_n(k)$ and augmentation morphisms $\epsilon: \PoisOp_n(r)\rightarrow\QQ$
which model composition operations with an extra arity-zero element $e\in\PoisOp_n^+(0)$
such that $\mu(e,x_1) = x_1 = \mu(x_1,e)$ and $\lambda(e,x_1) = 0 = \lambda(x_1,e)$
in the unitary $n$-Poisson operad $\PoisOp_n^+$.

We generally say that an operad $\POp$ in a symmetric monoidal category $\MCat$
is non-unitary when we have $\POp(0) = \varnothing$ (the initial object of the category or just nothing)
and unitary when we have in contrast $\POp(0) = \unit$ (the unit object of our symmetric monoidal category).
We use the phrase `augmented $\Lambda$-operad' to refer to the enrichment of the standard structure
of an operad defined by our restriction operators.
We just omit to mention the augmentation in the context of simplicial sets,
because the augmentations are trivially given by terminal maps
with values in the one-point set $*$
in this case. We simplify our terminology similarly
in the case of Hopf operads.
The main observation is that the category of unitary operads in any symmetric monoidal category $\MCat$
is isomorphic to the category of augmented non-unitary $\Lambda$-operads in $\MCat$ (see~\cite[\S\S I.2.2-2.4]{FresseBook}).

\subsection*{The topological interpretation of the formality results}
We already mentioned that the category of Hopf cochain dg-cooperads defines a model for the rational homotopy theory of operads.
We rely on the classical Sullivan rational homotopy theory of spaces to get such a result.
We start with the Sullivan realization functor $\DGG_{\bullet}: \dg^*\ComCat_+\rightarrow\Simp^{op}$,
which goes from the category of unitary commutative cochain dg-algebras $\dg^*\ComCat_+$
to the category of simplicial sets $\Simp$,
and which is classically used to retrieve the rational homotopy of a space
from a model in $\dg^*\ComCat_+$.
We explicitly have $\DGG_{\bullet}(A) = \Mor_{\dg^*\ComCat_+}(A,\DGOmega^*(\Delta^{\bullet}))$, for any $A\in\dg^*\ComCat_+$,
where $\DGOmega^*(\Delta^n)$ denotes the algebra of piece-wise linear forms
on the simplex $\Delta^n$, for any $n\in\NN$.
We just have to consider a derived version $\DGL\DGG_{\bullet}(-)$ of this functor $\DGG_{\bullet}(-)$
in order to get a homotopy meaningful result.
The functor $\DGG_{\bullet}: \dg^*\ComCat_+\rightarrow\Simp^{op}$ is (contravariant) symmetric monoidal,
and as a byproduct, carries a cooperad in unitary commutative cochain dg-algebras (thus, a Hopf cochain dg-cooperad)
to an operad in simplicial sets.

We apply this construction to the Hopf cooperad in cochain graded modules $\AOp = \DGH^*(\DOp_n)$ which we identify with a Hopf cochain dg-cooperad
equipped with a trivial differential $\delta = 0$.
We now have
$\DGH^*(\DOp_n) = \PoisOp_n^c$,
where $\PoisOp_n^c$ denotes the dual Hopf cooperad in graded modules
of the $n$-Poisson operad $\PoisOp_n$.
This object $\DGH^*(\DOp_n) = \PoisOp_n^c$ admits a simple cofibrant resolution in the category of Hopf cochain dg-cooperads
which is defined by the Chevalley-Eilenberg cochain complex $\DGC_{CE}^*(\palg_n)$
of an operad in Lie algebras $\palg_n$ whose terms $\palg_n(r)$, $r>0$,
are graded versions of the classical Lie algebras
of infinitesimal braids (the graded Drinfeld-Kohno Lie algebras).
We then set
$\DGL\DGG_{\bullet}(\PoisOp_n^c) := \DGG_{\bullet}(\DGC_{CE}^*(\palg_n))$
to get an operad in simplicial sets $\DGG_{\bullet}(\DGC_{CE}^*(\palg_n))$
associated to the cohomology Hopf cooperad $\DGH^*(\DOp_n) = \PoisOp_n^c$.
For this choice of cofibrant resolution, we actually have an identity $\DGG_{\bullet}(\DGC_{CE}^*(\palg_n)) = \DGMC_{\bullet}(\palg_n)$,
where $\DGMC_{\bullet}(\palg_n)$ denotes the operad in simplicial sets whose components $\DGMC_{\bullet}(\palg_n(r))$
are simplicial sets of Maurer-Cartan elements
associated to the Lie dg-algebras $\palg_n(r)\hat{\otimes}\DGOmega^*(\Delta^{\bullet})$, $r>0$.
Hence, we eventually have
$\DGL\DGG_{\bullet}(\PoisOp_n^c) = \DGMC_{\bullet}(\palg_n)$.
By construction, we also have the relation $\DGH^*(\DGMC_{\bullet}(\palg_n))\simeq\PoisOp_n^c$
which is equivalent to the identity $\DGH_*(\DGMC_{\bullet}(\palg_n),\QQ)\simeq\PoisOp_n$
at the homology level.

We still consider Hopf cooperads equipped with $\Lambda$-structures in order to get a counterpart
of the category of $\Lambda$-operads
in our model (and a model of unitary operads therefore).
We can see that the Hopf cooperad $\DGC_{CE}^*(\palg_n)$ actually forms a cofibrant resolution
of the object $\DGH^*(\DOp_n)$
in the category of Hopf $\Lambda$-cooperads
in cochain graded dg-modules.
The operad $\DGL\DGG_{\bullet}(\PoisOp_n^c) = \DGG_{\bullet}(\DGC_{CE}^*(\palg_n))$ accordingly inherits a $\Lambda$-structure,
and is therefore associated to a unitary operad in simplicial sets
in the classical sense.

The simplicial set version of our formality result reads as follows:

\begin{introthm}\label{Result:SimplicialSetIntrinsicFormality}
Let $\POp$ be any $\Lambda$-operad in simplicial sets.
We assume that the spaces $\POp(r)$ underlying this operad are good with respect to the Bousfield-Kan $\QQ$-localization.
We also assume that $\POp$ is connected as an operad in the sense that we have the identity $\POp(1) = *$
in arity one.

If we have an isomorphism of Hopf $\Lambda$-operads at the rational homology level $\DGH_*(\POp,\QQ)\simeq\PoisOp_n$, for some $n\geq 3$,
and if we moreover assume that $\POp$ is equipped with an involution $J: \POp\xrightarrow{\simeq}\POp$
reflecting the involution of the $n$-Poisson operad $J_*: \PoisOp_n\xrightarrow{\simeq}\PoisOp_n$
in the case $4\mid n$,
then $\POp$ is rationally weakly-equivalent to $\DGL\DGG_{\bullet}(\PoisOp_n^c) = \DGMC_{\bullet}(\palg_n)$
as an operad in simplicial sets. We more explicitly have a chain of morphisms of $\Lambda$-operads
\begin{equation*}
\xymatrix{ \POp & \ROp\ar[l]_-{\sim}\ar@{.>}[r]^-{\sim^{\QQ}}_-{\exists} & \DGMC_{\bullet}(\palg_n) },
\end{equation*}
where $\ROp$ represents a cofibrant resolution of the object $\POp$ in the category of $\Lambda$-operads in simplicial sets,
and the right-hand side map $\ROp\rightarrow\DGMC_{\bullet}(\palg_n)$
defines a realization, at the simplicial set level,
of our rational homology isomorphism $\DGH_*(\ROp,\QQ)\simeq\DGH_*(\POp,\QQ)\simeq\PoisOp_n$.
\end{introthm}

We can actually entirely perform our constructions in the equivariant setting so that the intermediate object $\ROp$
is also equipped with an involution $J: \ROp\xrightarrow{\simeq}\ROp$
whenever $\POp$ does so,
and the morphisms occurring in our theorem
are involution preserving. We can also check that the morphism $\ROp\rightarrow\DGMC_{\bullet}(\palg_n)$
realizing our rational cohomology isomorphism is homotopically unique. (We just need an extra equivariance requirement
to get this uniqueness result in the case $4\mid n-3$.)

The model structure which we use to construct cofibrant resolutions of operads in simplicial sets is defined in~\cite[\S II.8.4]{FresseBook}.
Let us simply mention for the moment that we take the morphisms of operads that define a weak-equivalence
of simplicial sets arity-wise $\phi: \POp(r)\xrightarrow{\sim}\QOp(r)$
as class of weak-equivalences $\phi: \POp\xrightarrow{\sim}\QOp$
in our model category of $\Lambda$-operads.

We refer to Bousfield-Kan's monograph~\cite{BousfieldKan} for the notion of a good space
and to Bousfield-Guggenheim's memoir~\cite{BousfieldGugenheim}
for the relationship between the Bousfield-Kan $\QQ$-localization
and the Sullivan rationalization
of spaces. Recall simply that any simply connected (or, more generally, any nilpotent connected) space is good,
and that the Bousfield-Kan $\QQ$-localization agrees with the Sullivan rationalization
for simply connected (or, more generally, for nilpotent connected) spaces
whose rational homology is degree-wise finitely generated
as a $\QQ$-module.
In our theorem, we implicitly require that the components of our operad $\POp$ fulfill this homological finiteness condition
when we assume $\DGH_*(\POp)\simeq\PoisOp_n$.
In practice, we can also restrict ourselves to the case of operads whose components
form simply connected spaces in our theorem and forget about the difficulties
of the Bousfield-Kan $\QQ$-localization for non-nilpotent spaces.

We have an operadic enhancement of the Sullivan rationalization functor (see~\cite[\S II.10, \S II.12]{FresseBook})
which, to any cofibrant operad in simplicial sets $\ROp$, associates another operad $\ROp\sphat$
whose components are weakly-equivalent to the Sullivan rationalization
of the spaces $\ROp(r)$, $r>0$. We moreover have an operad morphism $\ROp\rightarrow\ROp\sphat$
which realizes the universal morphisms of the Sullivan rationalization
arity-wise.
The dotted map of our statement is equivalent to a weak-equivalence $\ROp\sphat\xrightarrow{\sim}\DGMC_{\bullet}(\palg_n)$
connecting this rationalization $\ROp\sphat$ to the operad $\DGL\DGG_{\bullet}(\PoisOp_n^c) := \DGMC_{\bullet}(\palg_n)$
which we deduce from our cohomology cooperad $\PoisOp_n^c$.

We can apply our result to a simplicial model $\POp = \EOp_n$ of the operad of little $n$-discs $\DOp_n$
and use the geometric realization construction to go back to topological operads. The operad of little $n$-discs $\DOp_n$
does not fulfill the operadic connectedness condition of our theorem $\POp(1) = *$,
but we have models of this operad (e.g. the Fulton-MacPherson operad~\cite{GetzlerJones})
which do so.
We then have $|\EOp_n|\sim\DOp_n$, and our topological formality statement implies the case $n\geq 3$
of the following claim (while the case $n=2$ follows from the existence
of rational Drinfeld's associators):

\begin{introthm}\label{Result:TopologicalFormality}
Let $n\geq 2$.
We have a chain of morphisms of $\Lambda$-operads in topological spaces
\begin{equation*}
\xymatrix{ \DOp_n & \ROp_n\ar[l]_-{\sim}\ar@{.>}[r]^-{\sim^{\QQ}}_-{\exists} & |\DGMC_{\bullet}(\palg_n)| },
\end{equation*}
where $\ROp_n$ represents a cofibrant resolution of the little $n$-discs operad $\DOp_n$,
and the right-hand side map $\ROp_n\rightarrow|\DGMC_{\bullet}(\palg_n)|$
defines a realization, at the topological space level,
of the rational homology isomorphism $\DGH_*(\ROp_n,\QQ)\simeq\DGH_*(\DOp_n,\QQ)\simeq\PoisOp_n$.
\end{introthm}

We can take the Fulton-MacPherson operad $\ROp_n = \FMOp_n$, which is a classical instance of a cofibrant resolution
of the little $n$-discs operad $\DOp_n$,
to fit the chain of this theorem.
Let us mention that Kontsevich's proof of the formality of $E_n$-operads (see~\cite{KontsevichMotives,LambrechtsVolic})
implies the existence of formality weak-equivalences
such that
\begin{equation*}
\FMOp_n\xrightarrow{\sim^{\RR}}|\DGMC_{\bullet}(\palg_n\otimes\RR)|
\end{equation*}
when we pass to real coefficients. This approach left open the question of the existence
of formality weak-equi\-va\-len\-ces defined over $\QQ$
in the case $n\geq 3$. This problem is just solved by our intrinsic formality theorem.

By our methods, we can also prove a rational formality statement for the morphisms $\iota: \DOp_m\hookrightarrow\DOp_n$
which link the little discs operads together.
We only have null morphisms $\palg_m\rightarrow 0\rightarrow\palg_n$
between our Lie algebra operads.
When we pass to simplicial sets, we have $\DGMC_{\bullet}(0) = *$. The operad which has these one-point sets $*$
as components is identified with the operad of commutative monoids.
We exactly get this morphism, which goes through the one-point set operad $*$,
in our relative formality statement:

\begin{introthm}\label{Result:RelativeTopologicalFormality}
The morphisms of Theorem~\ref{Result:TopologicalFormality} fit in homotopy commutative diagrams
in the category of $\Lambda$-operads:
\begin{equation*}
\xymatrix{ \DOp_m\ar[d]_{\iota} & \ROp_m\ar[l]_-{\sim}\ar@{.>}[d]\ar@{.>}[r]^-{\sim^{\QQ}} &
|\DGMC_{\bullet}(\palg_m)|\ar[d]^{\equiv *} \\
\DOp_n & \ROp_n\ar[l]_-{\sim}\ar@{.>}[r]^-{\sim^{\QQ}} &
|\DGMC_{\bullet}(\palg_n)| },
\end{equation*}
for any pair $n>m\geq 2$ such that $n-m\geq 2$.
\end{introthm}

We mostly study the little discs operads of dimension $n\geq 2$ in this paper, but we may also consider the extension
of the above theorem to the case where we start
with the operad of little intervals $\DOp_1$ (little discs of dimension $1$)
in our morphism $\iota: \DOp_1\rightarrow\DOp_n$.
Recall that this particular operad $\DOp_1$ has contractible connected components
and satisfies $\pi_0\DOp_1 = \AsOp$,
where $\AsOp$ is the operad that governs the category of associative monoids (the associative operad).
We consequently have a chain of weak-equivalences $\DOp_1\xleftarrow{\sim}\ROp_1\xrightarrow{\sim}\AsOp$,
where we regard the operad in sets $\AsOp$ as a discrete operad in topology spaces.
We then have the following relative formality result,
where we again consider a trivial morphism (which goes through the one-point set operad)
in order to connect $\AsOp$ to the operad $|\DGMC_{\bullet}(\palg_n)|$:

\begin{introthm}\label{Result:InitialRelativeTopologicalFormality}
The morphisms of Theorem~\ref{Result:TopologicalFormality} fit in homotopy commutative diagrams
in the category of $\Lambda$-operads:
\begin{equation*}
\xymatrix{ \DOp_1\ar[d]_{\iota} & \ROp_1\ar[l]_-{\sim}\ar@{.>}[d]\ar[r]^-{\sim} &
\AsOp\ar[d]^{\equiv *} \\
\DOp_n & \ROp_n\ar[l]_-{\sim}\ar@{.>}[r]^-{\sim^{\QQ}} &
|\DGMC_{\bullet}(\palg_n)| },
\end{equation*}
for all $n\geq 3$.
\end{introthm}

Let us insist that the diagrams of the above statements commute up to homotopies of $\Lambda$-operads,
and not only up to arity-wise homotopies of spaces.
We refer to~\cite{LambrechtsVolic} and~\cite{TurchinWillwacher} for the previously known results (which involve
the Kontsevich formality weak-equivalences
and hold over the reals) about the relative formality of the little discs operads.
Let us mention that Victor Turchin and the second author of this paper have established
the existence of obstructions
to the formality of the morphisms $\iota: \DOp_{n-1}\rightarrow\DOp_n$
in~\cite{TurchinWillwacher}.
The range of our theorems is therefore optimal.

\setcounter{introthm}{0}
\renewcommand{\theintrothm}{\Alph{introthm}'}

\subsection*{The algebraic statements}
We go back to the functor $\DGG_{\bullet}: \AOp\mapsto\DGG_{\bullet}(\AOp)$
which we obtain by applying the Sullivan realization functor arity-wise
to any Hopf cochain dg-cooperad $\AOp$.
We can see that this functor has a right adjoint
$\DGOmega^*_{\sharp}: \POp\mapsto\DGOmega^*_{\sharp}(\POp)$,
which maps an operad in simplicial sets $\POp$
to a Hopf cochain dg-cooperad $\DGOmega^*_{\sharp}(\POp)$,
and which we can use to define an operadic upgrading of the Sullivan functor of piece-wise
linear forms $\DGOmega^*: X\mapsto\DGOmega^*(X)$
from simplicial sets $X\in\Simp$
to unitary commutative cochain dg-algebras $\DGOmega^*(X)\in\dg^*\ComCat_+$ (see~\cite[\S II.10.1, \S II.12.1]{FresseBook}).
We just need to restrict ourselves to operads $\POp$ that satisfy the connectedness condition $\POp(1) = *$
in order to give a sense to this construction.
We can still extend our functor $\DGOmega^*_{\sharp}: \POp\mapsto\DGOmega^*_{\sharp}(\POp)$
to operads equipped with a $\Lambda$-structure.
We get that $\DGOmega^*_{\sharp}(\POp)$ inherits a Hopf $\Lambda$-cooperad structure
when $\POp$ is a $\Lambda$-operad.

We use this correspondence and the rational homotopy theory of~\cite{FresseBook}
to reduce the claim of Theorem~\ref{Result:SimplicialSetIntrinsicFormality}
to the following algebraic statement:

\begin{introthm}\label{Result:HopfDGOperadIntrinsicFormality}
Let $\KOp$ be any Hopf $\Lambda$-cooperad in cochain graded dg-modules.
If we have an isomorphism of Hopf $\Lambda$-cooperads at the cohomology level $\DGH^*(\KOp)\simeq\PoisOp_n^c$, for some $n\geq 3$,
and if we moreover assume that $\KOp$ is equipped with an involution $J: \KOp\xrightarrow{\simeq}\KOp$
reflecting the involution of the $n$-Poisson operad $J_*: \PoisOp_n\xrightarrow{\simeq}\PoisOp_n$
in the case $4\mid n$,
then we have a chain of morphisms of Hopf $\Lambda$-cooperads
\begin{equation*}
\xymatrix{ \KOp & \cdot\ar[l]_-{\sim}\ar[r]^-{\sim} & \cdot & \PoisOp_n^c\ar[l]_-{\sim} },
\end{equation*}
which induce an isomorphism in homology arity-wise, and make $\KOp$ an object equivalent to $\PoisOp_n^c = \DGH^*(\DOp_n)$
in the homotopy category of Hopf $\Lambda$-cooperads
in cochain graded dg-modules.
\end{introthm}

We can naively define the homotopy category of Hopf $\Lambda$-cooperads in cochain graded dg-modules
as the category obtained by formally inverting the morphisms
of Hopf $\Lambda$-cooperads
which induce an isomorphism in homology. We still use the terminology of `weak-equivalence', which we borrow from general homotopical algebra,
to refer to this class of morphisms.

We may use the above algebraic form of our intrinsic formality theorem in order to retrieve (and yet improve) the result of Kontsevich
about the formality of the chain operad of little $n$-discs $\DOp_n$ for $n\geq 3$ (while the case $n=2$ still follows
from the existence of rational Drinfeld's associators). We then consider the object $\DGOmega^*_{\sharp}(\EOp_n)$
associated to any simplicial model $\EOp_n$
of the operad $\DOp_n$
that satisfies our connectedness condition $\EOp_n(1) = *$.
We have $\DGH^*\DGOmega^*_{\sharp}(\EOp_n) = \DGH^*(\DOp_n) = \PoisOp_n^c$, and the result
of Theorem~\ref{Result:HopfDGOperadIntrinsicFormality}
implies that $\DGOmega^*_{\sharp}(\EOp_n)$
is connected to $\PoisOp_n^c$
by a chain of weak-equivalences of Hopf $\Lambda$-cooperads
in cochain graded dg-modules $\DGOmega^*_{\sharp}(\EOp_n)\xleftarrow{\sim}\cdot\xrightarrow{\sim}\PoisOp_n^c$.

We can still forget about algebra structures and form the dual object, in the category of dg-modules,
of the Hopf $\Lambda$-cooperad $\DGOmega^*_{\sharp}(\EOp_n)$.
We then get an augmented $\Lambda$-operad in dg-modules $\DGOmega^*_{\sharp}(\EOp_n)^{\vee}$.
We have a chain of weak-equivalences connecting this object $\DGOmega^*_{\sharp}(\EOp_n)^{\vee}$
to the usual chain operad $\DGC_*(\EOp_n)$ associated to $\EOp_n$,
and hence, to the usual chain operad of little $n$-discs $\DGC_*(\DOp_n)$,
where we again consider a chain complex with rational coefficients $\DGC_*(\DOp_n) = \DGC_*(\DOp_n,\QQ)$.

We therefore have the following result, which is a consequence of our intrinsic formality statement in the case $n\geq 3$,
and of the existence of (rational) Drinfeld's associators in the case $n=2$:

\begin{introthm}\label{Result:ChainOperadFormality}
Let $n\geq 2$. We have a chain of morphisms of augmented $\Lambda$-operads
\begin{equation*}
\xymatrix{ \DGC_*(\DOp_n,\QQ) & \cdot\ar[l]_-{\sim}\ar[r]^-{\sim} & \PoisOp_n }
\end{equation*}
which induce an isomorphism in homology arity-wise and make the rational chain operad of little $n$-discs $\DGC_*(\DOp_n) = \DGC_*(\DOp_n,\QQ)$
an object equivalent to $\PoisOp_n = \DGH_*(\DOp_n,\QQ)$
in the homotopy category of augmented $\Lambda$-operads in dg-modules.
\end{introthm}

Recall again that we use augmented $\Lambda$-operads as models of unitary operads,
which in the context of dg-modules
are operads $\POp^+$ satisfying $\POp^+(0) = \QQ$.
The result of this theorem
is therefore equivalent
to the existence of a chain of morphisms of operads
in dg-modules $\DGC_*(\DOp_n^+,\QQ)\xleftarrow{\sim}\cdot\xrightarrow{\sim}\PoisOp_n^+$
which induce an isomorphism in homology arity-wise and connect the unitary version of the chain operad
of little $n$-discs $\DGC_*(\DOp_n^+) = \DGC_*(\DOp_n^+,\QQ)$
to the unitary $n$-Poisson operad $\PoisOp_n^+ = \DGH_*(\DOp_n^+,\QQ)$.

The construction of Kontsevich only returns formality weak-equivalences defined over $\RR$ (see~\cite{KontsevichMotives,LambrechtsVolic}).
The statements of~\cite{GuillenAl} imply that this formality statement descends to $\QQ$
if we drop the arity zero term and focus on the non-unitary operads
underlying our objects,
but the extension of such rational formality weak-equivalences to unitary operads
was an open question yet. Indeed, the methods of~\cite{GuillenAl} are based on the definition of minimal models for operads
in dg-modules,
but we have to forget about unitary structures in order to guarantee the existence
of such minimal models (see also~\cite{DolgushevSchneider} for another effective approach of this descent theorem
for the formality of non-unitary operads).

We also have a formality statement for the morphisms of chain operads $\iota_*: \DGC_*(\DOp_m)\rightarrow\DGC_*(\DOp_n)$
induced by the morphism of topological operads $\iota: \DOp_m\rightarrow\DOp_n$.
We then consider the morphism of operads in graded modules $\iota_*: \PoisOp_m\rightarrow\PoisOp_n$
which makes the commutative product operations correspond in our operads
and which sends the Lie bracket operation $\lambda\in\PoisOp_m(2)$
to zero. We equivalently deal with a composite $\PoisOp_m\rightarrow\ComOp\rightarrow\PoisOp_n$,
where $\ComOp$ denotes the operad (in plain $\QQ$-modules)
that governs the category of commutative algebras (the commutative operad for short).
We can still identify this morphism with the morphism
induced by the operad embedding $\iota: \DOp_m\hookrightarrow\DOp_n$
in homology, and our result reads as follows:

\begin{introthm}\label{Result:RelativeChainOperadFormality}
The morphisms of Theorem~\ref{Result:ChainOperadFormality} fit in homotopy commutative diagrams
in the category of augmented $\Lambda$-operads
in dg-modules:
\begin{equation*}
\xymatrix{ \DGC_*(\DOp_m,\QQ)\ar[d]_{\iota_*} & \cdot\ar[l]_-{\sim}\ar@{.>}[d]\ar[r]^-{\sim} &
\PoisOp_m\ar[d]^{\iota_*} \\
\DGC_*(\DOp_n,\QQ) & \cdot\ar[l]_-{\sim}\ar[r]^-{\sim} &
\PoisOp_n },
\end{equation*}
for any pair $n>m\geq 2$ such that $n-m\geq 2$.
\end{introthm}

We can still extend this result to the case of the morphisms $\iota_*: \DGC_*(\DOp_1)\rightarrow\DGC_*(\DOp_n)$
with the chain operad of little intervals $\DGC_*(\DOp_1) = \DGC_*(\DOp_1,\QQ)$
as source object.
We then have a chain of weak-equivalences $\DGC_*(\DOp_1)\xleftarrow{\sim}\cdot\xrightarrow{\sim}\AsOp$,
where $\AsOp$ now denotes the associative operad in the category of (plain) $\QQ$-modules.
In this case, we have the following result, where we again consider a natural morphism $\iota_*: \AsOp\rightarrow\PoisOp_n$
which goes through the commutative operad $\ComOp$
and which represents the morphism induced by the operad embedding $\iota: \DOp_1\hookrightarrow\DOp_n$
in homology:

\begin{introthm}\label{Result:InitialRelativeChainOperadFormality}
The morphisms of Theorem~\ref{Result:ChainOperadFormality} fit in homotopy commutative diagrams
in the category of augmented $\Lambda$-operads
in dg-modules:
\begin{equation*}
\xymatrix{ \DGC_*(\DOp_1,\QQ)\ar[d]_{\iota_*} & \cdot\ar[l]_-{\sim}\ar@{.>}[d]\ar[r]^-{\sim} &
\AsOp\ar[d]^{\iota_*} \\
\DGC_*(\DOp_n,\QQ) & \cdot\ar[l]_-{\sim}\ar[r]^-{\sim} &
\PoisOp_n },
\end{equation*}
for all $n\geq 3$.
\end{introthm}

If we pass to the reals, then we exactly retrieve the relative formality statements of~\cite{LambrechtsVolic}
with the improved range of~\cite{TurchinWillwacher}.
Let us mention that the relative formality of the little discs operads is used in~\cite{AroneTurchin}
in order to define small chain complexes
computing the layers of the Goodwillie-Weiss tower
of chains on embedding spaces (see notably~\cite[Proposition 7.1, Corollary 8.1]{AroneTurchin}).

\subsection*{Plan and contents of the paper}
We devote most of our efforts to the algebraic version of our intrinsic formality theorems.
We use obstruction methods to establish the existence of a morphism $\phi: \ROp\xrightarrow{\sim}\PoisOp_n^c$
fitting the claim of Theorem~\ref{Result:HopfDGOperadIntrinsicFormality},
and we use the rational homotopy theory of~\cite[\S\S II.8-12]{FresseBook}
to derive the claim of Theorem~\ref{Result:SimplicialSetIntrinsicFormality} (the topological version of our intrinsic formality theorems)
from this result.
We use a similar approach when we establish the formality of the morphisms that link the little discs operads together.
We review the background of our constructions and we explain our obstruction problem
in a preliminary section of this paper~(\S\ref{Background}).

We check afterwards that the obstruction to the existence of our morphism lies in the cohomology of a bicosimplicial biderivation complex
which we associate to the cohomology cooperad $\DGH^*(\DOp_n) = \PoisOp_n^c$.
We prove that the cohomology of this bicosimplicial biderivation complex can be computed by using a deformation bicomplex
which combines a cooperadic variant of the Harrison complex of commutative algebras
with a cooperadic cobar complex.
We address these topics in~\S\ref{DeformationComplexes}.

We establish that this deformation bicomplex reduces to a variant of the graph complexes considered in earlier works
of the second author of this article.
We elaborate on the computations carried out in these previous studies in order to establish
that the cohomology of our deformation bicomplex vanishes in negative degree (up to odd classes with respect to the action
of involution operators).
We therefore have no obstruction to the existence of our map. We address this part of our proof in~\S\ref{GraphHomology}.
We recap our constructions and complete the proof of our main theorems in the concluding section of the paper~\S\ref{MainResultProofs}.
We devote two appendices to the definition of simplicial (cotriple) resolutions
and of cosimplicial (triple) coresolutions
in the context of Hopf cooperads.

In what follows, we mostly forget about the topological interpretation of our constructions.
We only tackle the applications of our results in the context of topological spaces (actually, simplicial sets)
in the concluding section of the paper.
We also refer to~\cite[\S II.14.1.9]{FresseBook} for the definition of the operads $\DGMC_{\bullet}(\palg_n)$
which we consider in the simplicial set context.
We do not use these operads further in this paper and we do not use the graded Drinfeld-Kohno Lie algebra operads
either. Besides, we mostly study the case of $E_n$-operads such that $n\geq 2$.
We only consider the case of $E_1$-operads (and the operad of little intervals $\DOp_1$)
in the concluding section of the article.

\setcounter{section}{-1}

\section{Background and the statement of the obstruction problem}\label{Background}
The first purpose of this section is to fix our conventions on the model category of Hopf $\Lambda$-cooperads
which we use in our constructions. We essentially borrow our definitions
from the book~\cite{FresseBook}
to which we refer for more details.
Then we explain the obstruction methods which we use to prove
our intrinsic formality result.

\subsection*{The base symmetric monoidal category of dg-modules}

\subsubsection{The category of dg-modules}\label{Background:DGModules}
We assume that $\kk$ is a characteristic zero field which we take as ground ring for our categories of modules.
We deal with objects defined in a base category of differential graded modules over $\kk$.

In the first instance, we consider the category $\dg\Mod$ whose objects are the $\kk$-modules $K$
equipped with a lower grading $K = \oplus_n K_n$ (which runs over $\ZZ$ in general)
together with a differential, usually denoted by $\delta: K\rightarrow K$,
which decreases degrees by one.
Recall that we also use the expression `dg-module' to refer to this generic differential
graded module structure.

When we form our model for the rational homotopy of operads, we rather deal with dg-modules
equipped with an upper grading $K = \oplus_n K^n$
such that $K^n = 0$ for $n<0$
together with a differential $\delta: K\rightarrow K$
that increases this upper grading by one.
We adopt the notation $\dg^*\Mod$ for this category of dg-modules, to which we also refer as the category of cochain graded dg-modules.
We can use the standard equivalence between lower and upper grading $K_n = K^{-n}$
to identify any cochain graded dg-module $K\in\dg^*\Mod$ with a lower graded dg-module such that $K_n = 0$ for $n>0$.
We accordingly identify the category of cochain graded dg-modules with a subcategory of the category of dg-modules $\dg^*\Mod\subset\dg\Mod$.
We generally use the notation $\deg(\xi)$ for the (lower) degree of any homogeneous element $\xi$ in a dg-module $K\in\dg\Mod$,
while we use the notation $\deg^*(\xi)$, with an extra upper-script $*$, to specify upper degrees.
We accordingly have $\deg(\xi) = - \deg^*(\xi)$ for any element $\xi$ in a cochain graded dg-module $K\in\dg^*\Mod$.

We also have an identity $\DGH_n(K) = \DGH^{-n}(K)$ when we take the cohomology of a cochain graded dg-module $K\in\dg^*\Mod$.
In what follows, we rather use the homology modules $\DGH_*(K)$ when we deal with general dg-modules $K\in\dg\Mod$,
and the equivalent cohomology modules $\DGH^*(K)$ when our object belongs to the category of cochain graded dg-modules $K\in\dg^*\Mod$.

We still consider a category of graded modules $\gr\Mod$ which we can identify with dg-modules
equipped with a trivial differential $\delta = 0$.
We accordingly have $\gr\Mod\subset\dg\Mod$ and we also consider a subcategory of cochain graded modules
such that $\gr^*\Mod = \gr\Mod\cap\dg^*\Mod$.

We equip the category of dg-modules $\dg\Mod$ with its standard model structure, where a morphism is a weak-equivalence
when this morphism induces an isomorphism in homology, a fibration when it is degree-wise surjective,
and a cofibration when it has the left lifting property with respect to any acyclic fibration.
In fact, we just get the class of all degree-wise injective morphisms as cofibrations in $\dg\Mod$
since we assume that our ground ring is a field.
We use the same definition to provide the category of cochain graded dg-modules $\dg^*\Mod$
with a model structure,
with the morphisms that induce an isomorphism in homology as class of weak-equivalences,
the morphisms that are degree-wise surjective as class of fibrations,
and the morphisms that have the left lifting property with respect to the acyclic fibrations
as class of cofibrations.
Recall simply that a morphism of cochain graded dg-modules is a cofibration in this model category $\dg^*\Mod$
if and only if this morphism is injective in positive degrees (under our assumption
that we take a field as ground ring again, see~\cite[Proposition II.5.1.11]{FresseBook}).

\subsubsection{Symmetric monoidal structures and structured objects in the category of dg-modules}\label{Background:DGModulesTensorStructure}
We equip the base category of dg-modules with its usual symmetric monoidal structure,
where we have a symmetry isomorphism, involving a sign,
which we determine by the usual rules
of homological algebra.
We also deal with an obvious restriction of this symmetric monoidal structure
to the category of cochain graded dg-modules $\dg^*\Mod\subset\dg\Mod$.

We use this symmetric monoidal structure to define our categories of structured objects, such as the category of (co)operads,
the category of (co)unitary (co)commutative (co)algebras, {\dots}
Recall that the category of unitary commutative algebras in a base symmetric monoidal category inherits a symmetric monoidal structure again,
with a tensor product operation formed in the base category. We have a similar statement when we deal with counitary cocommutative coalgebras.

\subsection*{Hopf $\Lambda$-cooperads}

\subsubsection{The notion of a coaugmented $\Lambda$-cooperad}\label{Background:LambdaCooperads}
We mainly work in the category of cochain graded dg-modules when we deal with cooperads.
We basically define a cooperad in cochain graded dg-modules as a collection of cochain graded dg-modules $\COp(r)\in\dg^*\Mod$, $r>0$,
where each object $\COp(r)$ is equipped with an action of the symmetric group on $r$ letters $\Sigma_r$,
together with composition coproducts
$\circ_i^*: \COp(k+l-1)\rightarrow\COp(k)\otimes\COp(l)$,
defined for every $k,l>0$, and $i = 1,\dots,k$, and which fulfill an obvious dual of the standard equivariance, unit,
and associativity relations of operads.
We also generally assume $\COp(1) = \kk$
when we deal with cooperads
and the counit morphism $\epsilon: \COp(1)\rightarrow\kk$,
which we associate to the composition structure of our object,
is given by the identity of the ground field.
We use this technical requirement to sort out convergence issues which usually occur with comultiplicative structures.

We say that $\COp$ forms a coaugmented $\Lambda$-cooperad when we have coaugmentation morphisms $\epsilon_*: \kk\rightarrow\COp(r)$,
defined for all arities $r>0$, together with corestriction operators $u_*: \COp(k)\rightarrow\COp(l)$,
which we associate to the injective maps $u: \{1<\dots<k\}\rightarrow\{1<\dots<l\}$, $k,l>0$,
and which satisfy natural compatibility relations with respect to the other structure
operations of our cooperad (we refer to~\cite[\S II.11.1]{FresseBook}
for details on this definition).
The letter $\Lambda$ in our name of this category of cooperads refers to the category which has the finite ordinals
as objects $\rset = \{1<\dots<r\}$, $r\in\NN$,
and the injective maps between finite ordinals as morphisms $u: \{1<\dots<k\}\rightarrow\{1<\dots<l\}$.
Let us insist that we consider all injective maps (and not only the monotonous ones)
as morphisms in $\Lambda$.
In what follows, we also use the full subcategory $\Lambda_{>1}\subset\Lambda$
spanned by the finite ordinals
such that $r>1$.
We just get that the underlying collection of a $\Lambda$-cooperad forms a covariant diagram over the category $\Lambda$.
We still set $\Sigma$ for the category which has the same objects $\rset = \{1<\dots<r\}$, $r\in\NN$, as the category $\Lambda$
but where we only retain the bijective maps $u: \{1<\dots<r\}\xrightarrow{\simeq}\{1<\dots<r\}$
as morphisms. We accordingly have the identity $\Sigma = \coprod_r\Sigma_r$, where we identify the symmetric group $\Sigma_r$
with the full subcategory of $\Sigma$ generated by the object $\rset = \{1<\dots<r\}$, for any $r\geq 0$.

We use the notation $\ComOp^c/\dg^*\Lambda\Op^c$ for the category of coaugmented cooperads in cochain graded dg-modules.
The notation $\ComOp^c$ refers to the dual cooperad (in the category of $\kk$-modules) of the operad of commutative algebras $\ComOp$.
In what follows, we also call this object $\ComOp^c$ the `commutative cooperad' for short.
The expression $\ComOp^c/-$ in the notation of our category of coaugmented cooperads $\ComOp^c/\dg^*\Lambda\Op^c$
refers to the observation that the coaugmentation morphisms of a coaugmented $\Lambda$-cooperad $\epsilon: \kk\rightarrow\COp(r)$, $r>0$,
define a morphism over this cooperad $\epsilon_*: \ComOp^c\rightarrow\COp$.

\subsubsection{The $n$-Poisson cooperad}\label{Background:PoissonCooperad}
Recall that the $n$-Poisson operad $\PoisOp_n$ is defined for any $n\geq 1$ (if we want to stay within
the category of chain graded modules, otherwise we may consider the case of an arbitrary $n\in\ZZ$).
But $\PoisOp_n$ represents the homology of the little $n$-discs operad
only when $n\geq 2$,
and therefore, we only consider this case $n\geq 2$
in what follows.

Let $\PoisOp_n^c$ be the collection formed by the dual modules $\PoisOp_n^c(r) = \PoisOp_n(r)^{\vee}$
of the components of the $n$-Poisson operad $\PoisOp_n$,
for any $n\geq 2$.
This object inherits a natural cooperad structure, because each module $\PoisOp_n(r)$, $r>0$,
has a finite dimension over the ground field
in each degree.
Indeed, in this situation, we can dualize the composition products
of the $n$-Poisson operad $\PoisOp_n$
to retrieve cooperad coproducts
on the objects $\PoisOp_n^c(r) = \PoisOp_n(r)^{\vee}$.

In the introduction of this paper, we also mentioned that the $n$-Poisson operad $\PoisOp_n$
is equipped with restriction operators $u^*: \PoisOp_n(l)\rightarrow\PoisOp_n(k)$,
associated to the injective maps $u: \{1<\dots<k\}\rightarrow\{1<\dots<l\}$, $k,l>0$, which determine composition operations involving an extra arity zero element $e\in\PoisOp_n^+(0)$
in a unitary extension of the $n$-Poisson operad $\PoisOp_n^+$.
In short, we can represent these restriction operators by the substitution formula $(u^* p)(x_1,\dots,x_k) = p(y_1,\dots,y_l)$,
for any operation $p\in\PoisOp_n(l)$,
where we set $y_j = x_{u^{-1}(j)}$ if $j\in\{u(1),\dots,u(k)\}$ and $y_j = e$
otherwise.
Then we just assume that we have the relations $\mu(e,x_1) = x_1 = \mu(x_1,e)$ and $\lambda(e,x_1) = 0 = \lambda(x_1,e)$
for the generating operations of our operad $\mu,\lambda\in\PoisOp_n(2)$,
and we use the usual equivariance and associativity relations
of operads
to determine the image of any operation $p\in\PoisOp_n(l)$ under our restriction operator $u^*: \PoisOp_n(l)\rightarrow\PoisOp_n(k)$.
We similarly have augmentation morphisms $\epsilon: \PoisOp_n(r)\rightarrow\kk$,
defined for all arities $r>0$,
and which intuitively model the full composites $\epsilon(p) = p(e,\dots,e)$
with our extra unit element $e$
in the unitary Poisson operad $\PoisOp_n^+$.
By duality, we get corestriction operators $u_*: \PoisOp_n^c(k)\rightarrow\PoisOp_n^c(l)$ and coaugmentations $\epsilon_*: \kk\rightarrow\PoisOp_n^c(r)$
when we pass to the $n$-Poisson cooperad $\PoisOp_n^c$, which therefore inherits a coaugmented $\Lambda$-cooperad structure
in our sense.
We back to the definition of these operations in~\S\ref{Background:PoissonFreeStructure}.

\subsubsection{The notion of a Hopf $\Lambda$-cooperad}\label{Background:HopfLambdaCooperads}
We already recalled in~\S\ref{Background:DGModulesTensorStructure} that the category of unitary commutative algebras
in a base symmetric monoidal category inherits
a symmetric monoidal structure.
We deal with unitary commutative algebras in cochain graded dg-modules, and we use the notation $\dg^*\ComCat_+$
for this category of unitary commutative algebras. (For short, we also use the phrase `unitary commutative cochain dg-algebra'
to refer to the objects of this category.)

The notion of a coaugmented $\Lambda$-cooperad actually makes sense in any base symmetric monoidal category $\MCat$.
We generally use the notation $\ComOp^c/\MCat\Lambda\Op^c$ for the category of coaugmented $\Lambda$-cooperads in such a category $\MCat$.
We define our category of Hopf $\Lambda$-cooperads $\dg^*\Hopf\Lambda\Op^c$
as the category of coaugmented $\Lambda$-cooperads $\dg^*\Hopf\Lambda\Op^c = \ComOp^c/\MCat\Lambda\Op^c$
in the category of unitary commutative cochain dg-algebras $\MCat = \dg^*\ComCat_+$
equipped with the symmetric monoidal structure inherited from the base category of dg-modules (see~\S\ref{Background:DGModulesTensorStructure}).
We therefore get that a Hopf $\Lambda$-cooperad consists of a collection of unitary commutative cochain dg-algebras $\AOp(r)$, $r>0$,
equipped with the structure operations of a coaugmented $\Lambda$-cooperad,
all formed in the category of unitary commutative cochain dg-algebras.
We may equivalently assume that the unit morphism $\eta: \kk\rightarrow\AOp(r)$
and the product operation $\mu: \AOp(r)\otimes\AOp(r)\rightarrow\AOp(r)$,
which determine the commutative algebra structure of each object $\AOp(r)$,
define morphisms of coaugmented $\Lambda$-cooperads,
where we regard the ground field $\kk$ as the components of the commutative cooperad $\ComOp^c(r) = \kk$, $r>0$,
and we equip the tensor products of dg-modules $\AOp(r)\otimes\AOp(r)$, $r>0$,
with the obvious diagonal cooperad structure.

Let us observe that the coaugmentation morphisms of a Hopf $\Lambda$-cooperad $\AOp\in\dg^*\Hopf\Lambda\Op^c$
are necessarily identified with the natural unit morphisms $\eta: \kk\rightarrow\AOp(r)$
of the objects $\AOp(r)\in\dg^*\ComCat_+$.
We therefore generally omit to specify coaugmentations
when we deal with Hopf $\Lambda$-cooperads.

\subsubsection{The Hopf structure on the $n$-Poisson cooperad}\label{Background:PoissonHopfCooperad}
The $n$-Poisson operad $\PoisOp_n$ inherits a Hopf operad structure because this operad
is identified with the homology of an operad
in topological spaces $\PoisOp_n = \DGH_*(\DOp_n)$.
We have the coproduct formulas $\Delta(\mu) = \mu\otimes\mu$ and $\Delta(\lambda) = \lambda\otimes\mu + \mu\otimes\lambda$
for the generating operations of this operad $\mu,\lambda\in\PoisOp_n(2)$.
We get a Hopf $\Lambda$-cooperad structure when we pass to the dual object $\PoisOp_n^c$.

We also have an explicit presentation of the graded algebras $\PoisOp_n^c(r) = \DGH^*(\DOp_n(r))$
by generators and relations
which is given by a graded version of the classical Arnold presentation
of the cohomology of configuration spaces (see for instance~\cite[\S I.4.2]{FresseBook}
for a survey and references on this statement).
We only use consequences of this observation in what follows.
Namely, we will explain in~\S\ref{GraphHomology} that the $n$-Poisson operad
is weakly-equivalent to an operad of graphs. The proof of this claim relies on the Arnold presentation (see the bibliographical
references cited in this subsequent section), but we do not need more details
on this proof.

\subsubsection{The adjunctions between cooperads and collections}\label{Background:Collections}
Besides cooperads, we consider the category, denoted by $\dg^*\Sigma\Seq_{>1}^c$, formed by collections $\MOp = \{\MOp(r),r>1\}$
whose terms are cochain graded dg-modules $\MOp(r)\in\dg^*\Mod$
equipped with an action of the symmetric group $\Sigma_r$,
for all $r>1$.
We use the name `symmetric collection' when we want to specifically refer to an object of this category $\MOp\in\dg^*\Sigma\Seq_{>1}^c$.
We just use the word collection otherwise, with a more general meaning which we may use for any category of collections
shaped on the sequence of the non-negative integers (like the category of $\Lambda$-collections
which we introduce next).
In our reference~\cite{FresseBook}, the expression 'symmetric sequence' is generally used for this category of collections.
The notation $\Sigma\Seq$ is motivated by this terminology. The superscript $c$ is added to the notation in order to indicate
that we regard our category as the category of collections underlying cooperads,
whereas the subscript $>1$ refers to the fact that our collections
are only defined in arity $r>1$.

To a cooperad $\COp$, we associate the object $\overline{\COp}\in\dg^*\Sigma\Seq_{>1}^c$
such that:
\begin{equation*}
\overline{\COp}(r) = \begin{cases} 0, & \text{if $r=0,1$}, \\
\COp(r), & \text{otherwise}.
\end{cases}
\end{equation*}
We refer to this collection $\overline{\COp}$ as the coaugmentation coideal of our cooperad $\COp$.
The mapping $\overline{\omega}: \COp\mapsto\overline{\COp}$ obviously gives a functor $\overline{\omega}: \dg^*\Op^c\rightarrow\dg^*\Sigma\Seq_{>1}^c$
from the category of cooperads $\dg^*\Op^c$
to the category of symmetric collections $\dg^*\Sigma\Seq_{>1}^c$.
This functor admits a right adjoint
\begin{equation*}
\FreeOp^c: \dg^*\Sigma\Seq_{>1}^c\rightarrow\dg^*\Op^c
\end{equation*}
which associates a cofree cooperad $\FreeOp^c(\MOp)\in\dg^*\Op^c$
to any symmetric collection $\MOp\in\dg^*\Sigma\Seq_{>1}^c$ (see~\cite[\S C.1]{FresseBook} for a detailed survey of this construction).

The mapping $\overline{\omega}: \COp\mapsto\overline{\COp}$
also induces a functor $\overline{\omega}: \ComOp^c/\dg^*\Lambda\Op^c\rightarrow\overline{\ComOp}{}^c/\dg^*\Lambda\Seq_{>1}^c$
from the category of coaugmented cooperads $\ComOp^c/\dg^*\Lambda\Op^c$
towards the category, denoted by $\overline{\ComOp}{}^c/\dg^*\Lambda\Seq_{>1}^c$,
whose objects are covariant $\Lambda_{>1}$-diag\-rams
equipped with a coaugmentation over the coaugmentation coideal of the commutative cooperad $\overline{\ComOp}{}^c$.
Indeed, we immediately see that the coaugmentation coideal $\overline{\COp}$
of a coaugmented $\Lambda$-cooperad $\COp$
inherits such a diagram
structure.
In what follows, we call `coaugmented $\Lambda$-collections' the objects of this category $\MOp\in\overline{\ComOp}{}^c/\dg^*\Lambda\Seq_{>1}^c$.
(The terminology of~\cite{FresseBook} for this category is the category of `coaugmented covariant $\Lambda$-sequences'.)
Recall nonetheless that we may use the simple word `collection' to refer to a generic structure defined by a collection
shaped on the sequence of the non-negative integers,
and this form of structure includes the category of coaugmented $\Lambda$-collections
as a particular case.
The cofree cooperad functor $\FreeOp^c: \MOp\mapsto\FreeOp^c(\MOp)$
lifts to a functor
\begin{equation*}
\FreeOp^c: \overline{\ComOp}{}^c/\dg^*\Lambda\Seq_{>1}^c\rightarrow\ComOp/\dg^*\Lambda\Op^c
\end{equation*}
from the category of coaugmented $\Lambda$-collections $\overline{\ComOp}{}^c/\dg^*\Lambda\Seq_{>1}^c$
to the category of coaugmented $\Lambda$-cooperads $\ComOp/\dg^*\Lambda\Op^c$
and this extended cofree cooperad functor defines a right adjoint
of the extended coaugmentation
coideal functor $\overline{\omega}: \ComOp^c/\dg^*\Lambda\Op^c\rightarrow\overline{\ComOp}{}^c/\dg^*\Lambda\Seq_{>1}^c$.
We refer to~\cite[\S C.1]{FresseBook} for further details on this observation.

We can also lift our cofree cooperad functor to Hopf cooperads.
We then consider the category $\dg^*\Hopf\Sigma\Seq_{>1}^c$, whose objects $\AOp = \{\AOp(r),r>1\}$
are collections of unitary commutative cochain dg-algebras $\AOp(r)\in\dg^+\ComCat_+$
equipped with an action of the symmetric groups,
together with the category $\dg^*\Hopf\Lambda\Seq_{>1}^c$, whose objects $\AOp = \{\AOp(r),r>0\}$ are $\Lambda_{>1}$-diagrams
in the category of unitary commutative cochain dg-algebras.
We use the expression `Hopf symmetric collection' for the first considered category $\dg^*\Hopf\Sigma\Seq_{>1}^c$,
and the expression `Hopf $\Lambda$-collection' for the category $\dg^*\Hopf\Sigma\Seq_{>1}^c$
considered in second. We may also use the expression `Hopf collection' as a generic name for both categories,
or when the context makes clear which category of Hopf collections we consider.
Let us observe that any Hopf $\Lambda$-collection $\AOp\in\dg^*\Hopf\Lambda\Seq_{>1}^c$
trivially inherits a coaugmentation morphism $\epsilon_*: \overline{\ComOp}{}^c\rightarrow\AOp$
which is given by the unit morphism $\eta: \kk\rightarrow\AOp(r)$ of the algebra $\AOp(r)$
in each arity $r>1$.
We just get that the plain cofree cooperad functor $\FreeOp^c: \MOp\mapsto\FreeOp^c(\MOp)$
lifts as functors $\FreeOp^c: \dg^*\Hopf\Sigma\Seq_{>1}^c\rightarrow\dg^*\Hopf\Op^c$
and $\FreeOp^c: \dg^*\Hopf\Lambda\Seq_{>1}^c\rightarrow\dg^*\Hopf\Lambda\Op^c$
which are right adjoint to the obvious lifting
of the coaugmentation coideal functors $\overline{\omega}: \dg^*\Hopf\Op^c\rightarrow\dg^*\Hopf\Sigma\Seq_{>1}^c$
and $\overline{\omega}: \dg^*\Hopf\Lambda\Op^c\rightarrow\dg^*\Hopf\Lambda\Seq_{>1}^c$
(see~\cite[Proposition II.9.3.4 and Proposition II.11.4.2]{FresseBook}).

\subsubsection{The algebraic adjunction relations}\label{Background:AlgebraicAdjunctions}
The categories of cooperads and collections which we consider in this paper can be arranged on two parallel squares,
which we depict in Figure~\ref{Fig:AlgebraicAdjunctions}.
\begin{figure}[t]
\begin{equation*}
\xymatrix@!R=5em@!C=5em{ & \overline{\ComOp}{}^c/\dg^*\Sigma\Seq_{>1}^c\ar@<+3pt>@{.>}[rr]^-{\overline{\ComOp}{}^c/\Lambda\otimes_{\Sigma}-}
\ar@<+3pt>@{.>}[dd]|(0.5){\vbox to 4em{}}^(0.3){\overline{\ComOp}{}^c/\Sym(-)} &&
\overline{\ComOp}{}^c/\dg^*\Lambda\Seq_{>1}^c\ar@<+3pt>@{.>}[dd]^(0.3){\overline{\ComOp}{}^c/\Sym(-)}\ar@<+3pt>[ll]^-{(1')} \\
\ComOp^c/\dg^*\Op^c\ar@<+3pt>@{.>}[rr]^-{\ComOp^c/\Lambda\otimes_{\Sigma}-}
\ar@<+3pt>@{.>}[dd]^(0.3){\ComOp^c/\Sym(-)}\ar[ur]^(0.5){\overline{\omega}} &&
\ComOp^c/\dg^*\Lambda\Op^c\ar@<+3pt>@{.>}[dd]^(0.3){\ComOp^c/\Sym(-)}\ar@<+3pt>[ll]^-{(1)}\ar[ur]^(0.5){\overline{\omega}} & \\
& \dg^*\Hopf\Sigma\Seq_{>1}^c\ar@<+3pt>@{.>}[rr]|(0.5){\qquad}\ar@<+3pt>[uu]|(0.5){\vbox to 4em{}}^(0.7){(3')} &&
\dg^*\Hopf\Lambda\Seq_{>1}^c\ar@<+3pt>[uu]^(0.7){(2')}\ar@<+3pt>[ll]|(0.5){\qquad}^(0.6){(4')} \\
\dg^*\Hopf\Op^c\ar@<+3pt>@{.>}[rr]\ar@<+3pt>[uu]^(0.7){(3)}\ar[ur]^(0.5){\overline{\omega}} &&
\dg^*\Hopf\Lambda\Op^c\ar@<+3pt>[uu]^(0.7){(2)}\ar@<+3pt>[ll]^-{(4)}\ar[ur]^(0.5){\overline{\omega}} & }
\end{equation*}
\caption{}\label{Fig:AlgebraicAdjunctions}
\end{figure}
The diagonal arrows materialize the coaugmentation coideal functors that link the cooperad categories of the foreground
to the collection categories of the background.
The vertical and horizontal solid arrows in the foreground and background squares
materialize the obvious forgetful functors
that link these categories of cooperads
and collections.
The dotted arrows represent the left adjoint functors of these forgetful functors.
This cubical diagram entirely commutes in both the forgetful functor and the adjoint functor directions.

The horizontal adjoint functors of the figure are given by a coend construction.
To be explicit, for an object $\COp\in\ComOp^c/\dg^*\Op^c$,
we first set:
\begin{equation*}
(\Lambda\otimes_{\Sigma}\COp)(r) := \int^{\kset\in\Sigma}\Mor_{\Lambda}(\kset,\rset)\otimes\COp(k),
\end{equation*}
for each $r>0$, where we use the notation $S\otimes\COp(r)$ to refer to a coproduct of copies of the object $\COp(r)$
over the set $S = \Mor_{\Lambda}(\kset,\rset)$.
Then we perform the relative coproducts
\begin{equation*}
(\ComOp^c/\Lambda\otimes_{\Sigma}\COp)(r) := \ComOp^c(r)\bigoplus_{(\Lambda\otimes_{\Sigma}\ComOp^c)(r)}(\Lambda\otimes_{\Sigma}\COp)(r)
\end{equation*}
in order to collapse the morphism $(\Lambda\otimes_{\Sigma}\epsilon_*): (\Lambda\otimes_{\Sigma}\ComOp^c)(r)\rightarrow(\Lambda\otimes_{\Sigma}\COp)(r)$
induced by the coaugmentation of our object $\epsilon_*: \ComOp^c\rightarrow\COp$
into a single coaugmentation map $\epsilon_*: \ComOp^c(r)\rightarrow(\ComOp^c/\Lambda\otimes_{\Sigma}\COp)(r)$,
for each $r>0$. We clearly have $\COp(1) = \kk$ and we can just check that the composition coproducts
of the cooperad $\COp$ extend to this object $\ComOp^c/\Lambda\otimes_{\Sigma}\COp$,
which therefore forms a coaugmented $\Lambda$-cooperad in our sense (we refer to~\cite[\S II.11.2]{FresseBook}
for more details on this construction).
The adjunction relation between this mapping $\ComOp^c/\Lambda\otimes_{\Sigma}-$
and the forgetful functor from coaugmented $\Lambda$-cooperads
to coaugmented cooperads
follows from the abstract definition of coends.
We perform a similar coend construction when we start with an object of the category of Hopf cooperads.
We just form our coend in the category of unitary commutative cochain dg-algebras (instead of the category of cochain graded dg-modules)
as well as our relative coproducts in the second step of the construction (thus, we replace
the relative direct sum in the above formula by a relative tensor product).

We use the same constructions in the collection setting.
We just forget about components of arity $r=1$ (and about composition coproducts as well)
in this case.
We accordingly deal with the coaugmentation coideal $\overline{\ComOp}{}^c$ rather than with the full commutative cooperad $\ComOp^c$
when we perform our constructions for collections, and we simply mark this change
in our notation.

The vertical adjoint functors of our diagram are given by a relative symmetric algebra
construction.
To be explicit, in both cases cooperads and collections, we forget about $\Lambda$-structures in a first step, and we form our objects arity-wise,
by the relative tensor products $\kk/\Sym(M) = \kk\otimes_{\Sym(\kk)}\Sym(M)$
in the category of plain unitary commutative algebras,
where $M$ denotes any cochain graded dg-module equipped with a coaugmentation $\epsilon_*: \kk\rightarrow M$, and $\Sym(-)$
refers to the standard symmetric algebra functor on the category of dg-modules.
Thus, for an object $\COp$ of the category of coaugmented cooperads $\ComOp^c/\dg^*\Op^c$, we explicitly set:
\begin{equation*}
\ComOp^c/\Sym(\COp)(r) = \ComOp^c(r)\otimes_{\Sym(\ComOp^c(r))}\Sym(\COp(r)),
\end{equation*}
for each $r>0$. The relative tensor product has the effect of identifying the image
of the morphism $\Sym(\kk) = \Sym(\ComOp^c(r))\rightarrow\Sym(\COp(r))$
induced by the coaugmentation $\epsilon_*: \ComOp^c(r)\rightarrow\COp(r)$
of our object $\COp\in\ComOp^c/\dg^*\Op^c$
with the natural unit morphism of the symmetric algebra $\ComOp^c(r) = \kk\rightarrow\Sym(\COp(r))$.
Then we check that these objects $\ComOp^c/\Sym(\COp) = \{\ComOp^c/\Sym(\COp)(r),r>0\}$
inherit a natural Hopf cooperad structure (see~\cite[\S II.11.4.4 and Proposition II.11.4.5]{FresseBook} for details).
In the case where $\COp\in\ComOp^c/\dg^*\Lambda\Op^c$, we check that this Hopf cooperad $\ComOp^c/\Sym(\COp)$
is also provided with corestriction operators
which we define by taking the obvious extension of the corestriction operators
of our object $u_*: \COp(k)\rightarrow\COp(l)$
to the symmetric algebra $\Sym(\COp(k))$ (together with the constant map on the factor $\ComOp^c(k) = \kk$
in our relative tensor product),
and we accordingly get that the object $\ComOp^c/\Sym(\COp)$ forms a Hopf $\Lambda$-cooperad
(see again~\cite[\S II.11.4.4 and Proposition II.11.4.5]{FresseBook} for more details).
In both cases, coaugmented cooperads and coaugmented $\Lambda$-cooperads, the adjunction relation with the obvious
forgetful functor follows from the interpretation of the symmetric algebra as a free object
in the category of unitary commutative algebras and from the analogous categorical interpretation
of the relative tensor product of our formula.

In the collection setting, we use similar constructions.
In this case, we simply forget about the (composition coproducts and the) components of arity $r=1$ of objects.
For this reason, we use the coaugmentation coideal $\overline{\ComOp}{}^c$ (yet again)
rather than the full commutative cooperad $\ComOp^c$
in the version of our relative symmetric algebra construction
for collections (and we adapt our notation accordingly).

\subsubsection{Model structures on cooperads and collections}\label{Background:CooperadModelCategories}
We use the adjunction relations of the previous paragraph to provide our categories of cooperads
and collections with a model structure.
We define a model structure on the category of plain cooperads in cochain graded dg-modules first.
We just assume that a morphism $\phi: \COp\rightarrow\DOp$ in $\dg^*\Op^c$
is a weak-equivalence if this morphism defines a weak-equivalence of cochain graded dg-modules arity-wise $\phi: \COp(r)\xrightarrow{\sim}\DOp(r)$,
a cofibration if this morphism defines a cofibration of cochain graded dg-modules arity-wise $\phi: \COp(r)\rightarrowtail\DOp(r)$
(thus, if $\phi: \COp(r)\rightarrow\DOp(r)$ is injective in positive degrees, for every arity $r\geq 1$),
and we characterize the class of fibrations by the right lifting property with respect to the class of acyclic cofibrations.
We refer to~\cite[\S II.9.2]{FresseBook} for the proof that these classes of morphisms
fulfill the axioms of model categories. Recall simply that $\dg^*\Op^c$ has a set of generating (acyclic) cofibrations
which consists of the (acyclic) cofibrations $\phi: \COp\rightarrow\DOp$
whose domains and codomains $\COp,\DOp\in\dg^*\Op^c$ vanish in arity $r\gg 0$
and form bounded dg-modules of finite dimension over the ground field
in each arity $r>0$.

We provide the category of under objects $\ComOp^c/\dg^*\Op^c$ with the canonical model structure induced by our model structure on $\dg^*\Op^c$
so that a morphism in $\ComOp^c/\dg^*\Op^c$
forms a weak-equivalence (respectively, a cofibration, a fibration)
if and only if this morphism defines a weak-equivalence (respectively, a cofibration, a fibration)
in $\dg^*\Op^c$.
We then use our square of adjunction relations (1-4) in~\S\ref{Background:AlgebraicAdjunctions}
to transport this model structure to our other categories of cooperads.
We basically assume that the forgetful functors create the class of weak-equivalences and fibrations
in each case.
We also use our left adjoint functors to transport our sets of generating (acyclic) cofibrations
to each of our model categories,
which are all cofibrantly generated therefore.
We refer to~\cite[\S\S II.9.2-9.3, \S\S II.11.3-11.4]{FresseBook} for the proof of the validity of these constructions.

We follow the same procedure in the collection context. We start with the category of plain symmetric collections $\dg^*\Sigma\Seq_{>1}^c$,
for which we use the same definitions as in the case of plain cooperads.
We actually retrieve the injective model structure of a category of diagrams in this case (see for instance~\cite[Proposition A.2.8.2]{Lurie}).
We can also identify the fibrations of the model category of symmetric collections $\dg^*\Sigma\Seq_{>1}^c$
with the morphisms of symmetric collections which are surjective in all degrees
because our ground field has characteristic zero by assumption.
We again transport this model structure on symmetric collections to the category of under objects $\overline{\ComOp}{}^c/\dg^*\Sigma\Seq_{>1}^c$,
and to our other categories of collections
afterwards,
by assuming that the forgetful functors in the square of adjunctions
of~\S\ref{Background:AlgebraicAdjunctions}
create weak-equivalences
and fibrations.

We readily check that the diagonal coaugmentation coideal functors in the diagram of Figure~\ref{Fig:AlgebraicAdjunctions}
fit in Quillen adjunctions with the cofree cooperad functor as right adjoint.
We may actually see that the coaugmentation coideal functors create cofibrations
in our model categories of cooperads (and not only preserve cofibrations):
a morphism $\phi: \COp\rightarrow\DOp$ forms a cofibration in any of our model categories of cooperads
if and only if this morphism induces, on coaugmentation coideals,
a morphism which forms a cofibration $\overline{\phi}: \overline{\COp}\rightarrowtail\overline{\DOp}$
in the corresponding model category of collections (see~\cite[Proposition II.11.3.7]{FresseBook}
for a particular case of this statement).

\subsubsection{The $n$-Poisson cooperad as a cofibrant $\Lambda$-cooperad}\label{Background:PoissonFreeStructure}
We can easily analyze the $\Lambda$-diagram structure of the $n$-Poisson cooperad $\PoisOp_n^c$
in order to check that $\PoisOp_n^c$ is cofibrant as a coaugmented $\Lambda$-collection (and hence, as a coaugmented $\Lambda$-cooperad
though $\PoisOp_n^c$
is certainly not cofibrant as a Hopf $\Lambda$-cooperad).

For convenience, we prefer to examine the structure of the $n$-Poisson operad $\PoisOp_n$ first.
We dualize our constructions afterwards.
We use that the graded module $\PoisOp_n(r)$, which defines the component of arity $r$ of the Poisson operad $\PoisOp_n$,
is identified with the module freely spanned by monomials
\begin{equation*}
\pi(x_1,\dots,x_r) = \pi_1(x_{1 j_1},\dots,x_{j_{1 n_1}})\cdot\ldots\cdot\pi_s(x_{s j_1},\dots,x_{s j_{n_s}}),
\end{equation*}
whose factors $\pi_i = \pi_i(x_{i j_1},\dots,x_{i j_{n_i}})$, $i = 1,\dots,s$, represent Lie monomials
on sets of variables $\{x_{i j_1},\dots,x_{i j_{n_i}}\}$
such that $\{x_1,\dots,x_r\} = \{x_{1 j_1},\dots,x_{1 j_{n_1}}\}\amalg\cdots\amalg\{x_{s j_1},\dots,x_{s j_{n_s}}\}$
and which have degree one in each variable $x_{i j_k}$.
To give an example, the expression $\pi(x_1,\dots,x_6) = [[x_1,x_6],x_2]\cdot x_3\cdot [x_5,x_4]$
represents an element of $\PoisOp_n(6)$. We then use standard algebraic notation for the product $x_1\cdot x_2 = \mu(x_1,x_2)$
and the Lie bracket $[x_1,x_2] = \lambda(x_1,x_2)$.
We now consider the submodule $\SOp\PoisOp_n(r)\subset\PoisOp_n(r)$ spanned by the monomials $\pi = \pi(x_1,\dots,x_r)$
whose factors $\pi_i = \pi_i(x_{i j_1},\dots,x_{i j_{n_i}})$, $i = 1,\dots,s$,
are Lie monomials of weight $n_i>1$,
for each $r>1$.
We have for instance $[[x_1,x_5],x_2]\cdot [x_4,x_3]\in\SOp\PoisOp(5)$, but $[[x_1,x_6],x_2]\cdot x_3\cdot [x_5,x_4]\not\in\SOp\PoisOp(6)$.
We actually regard these graded modules $\SOp\PoisOp_n(r)$ as quotient objects of the components
of the Poisson operad $\PoisOp_n(r)$.
When we dualize, we get a symmetric collection of graded modules $\SOp\PoisOp_n^c = \{\SOp\PoisOp_n^c(r),r>1\}$
which forms a subobject of the coaugmentation coideal $\overline{\PoisOp}{}_n^c$
in $\gr^*\Sigma\Seq_{>1}^c$.

Let $\pi(x_1,\dots,x_k)^{\vee}\in\PoisOp_n^c(k)$ denote the dual basis of the basis of the Poisson monomials $\pi(x_1,\dots,x_k)\in\PoisOp_n(k)$
in the $n$-Poisson cooperad $\PoisOp_n^c(k)$.
Recall that the restriction operators of the $n$-Poisson operad $u^*: \PoisOp_n(l)\rightarrow\PoisOp_n(k)$
model operations involving a variable permutation together with composition operations
with an extra arity zero element $e$
such that $\mu(e,x_1) = x_1 = \mu(x_1,e)$ and $\lambda(e,x_1) = 0 = \lambda(x_1,e)$,
where we go back to operadic notation for the product $x_1\cdot x_2 = \mu(x_1,x_2)$
and the Lie bracket $[x_1,x_2] = \lambda(x_1,x_2)$ (see~\S\ref{Background:PoissonCooperad}).
To give a simple example, for the element $\pi(x_1,\dots,x_6) = [[x_1,x_6],x_2]\cdot x_3\cdot [x_5,x_4]$
and the injective map $u: \{1<\dots<5\}\rightarrow\{1<\dots<6\}$
such that $u(1) = 6$, $u(2) = 4$, $u(3) = 2$, $u(4) = 1$, $u(5) = 5$,
we get $u^*\pi(x_1,\dots,x_6) = \pi(x_4,x_3,e,x_2,x_5,x_1) = [[x_4,x_1],x_3]\cdot e\cdot [x_5,x_2] = [[x_4,x_1],x_3]\cdot [x_5,x_2]$.
But for the injective map $u: \{1<\dots<5\}\rightarrow\{1<\dots<6\}$
such that $u(1) = 5$, $u(2) = 4$, $u(3) = 1$, $u(4) = 6$, $u(5) = 3$,
we get $u^*\pi(x_1,\dots,x_6) = \pi(x_3,e,x_5,x_2,x_1,x_4) = [[x_3,x_4],e]\cdot x_5\cdot [x_1,x_2] = 0$.

From this description of the restriction operators on $\PoisOp_n$, we readily get that the dual corestriction operators
of our cooperad $u_*: \PoisOp_n^c(k)\rightarrow\PoisOp_n^c(l)$
are defined on our dual basis of the Poisson monomials $\pi(x_1,\dots,x_k)^{\vee}\in\PoisOp_n^c(k)$
by the mapping such that:
\begin{equation*}
u_*(\pi(x_1,\dots,x_k)^{\vee}) = (\pi(x_{u(1)},\dots,x_{u(k)})\cdot x_{j_{k+1}}\cdot\ldots\cdot x_{j_{l}})^{\vee}
\end{equation*}
where $\{j_{k+1},\dots,j_{l}\}$ represents the complement of the set $\{u(1),\dots,u(k)\}$
inside $\lset = \{1,\dots,l\}$.
We similarly get the formula:
$\epsilon_*(1) = (x_1\cdot\ldots\cdot x_r)^{\vee}$
for the coaugmentations $\epsilon_*: \ComOp^c(r)\rightarrow\PoisOp_n^c(r)$, $r>0$.
We easily deduce from this description of the coaugmentations and of the corestriction operators on $\PoisOp_n^c$
that the inclusion of symmetric collections $\SOp\PoisOp_n^c\hookrightarrow\overline{\PoisOp}{}_n^c$
induces an isomorphism of coaugmented $\Lambda$-collections in graded modules:
\begin{equation*}
\overline{\ComOp}{}^c/\Lambda\otimes_{\Sigma}(\overline{\ComOp}{}^c\oplus\SOp\PoisOp_n^c)\xrightarrow{\simeq}\overline{\PoisOp}{}_n^c,
\end{equation*}
where we regard $\overline{\ComOp}{}^c\oplus\SOp\PoisOp_n^c$ as an object of the category of coaugmented symmetric collections $\overline{\ComOp}{}^c/\Sigma\Seq_{>1}^c$
and we use the functor $\overline{\ComOp}{}^c/\Lambda\otimes_{\Sigma}-$
defined in~\S\ref{Background:AlgebraicAdjunctions}.

The observation that the object $\overline{\PoisOp}{}_n^c$ is cofibrant as a coaugmented $\Lambda$-collection
immediately follows from the existence of this decomposition
in $\overline{\ComOp}{}^c/\dg^*\Lambda\Seq_{>1}^c$ (by definition of our model structure).

\subsubsection{The (algebraic) augmentation ideal of the $n$-Poisson cooperad}\label{Background:PoissonAugmentationIdeal}
The commutative monomials $\mu(x_1,\dots,x_r) = x_1\cdot\ldots\cdot x_r$ also define canonical group-like elements
in the coalgebras $\PoisOp_n(r)$, $r>0$, and the collection of coalgebra maps $\eta: \kk\rightarrow\PoisOp_n(r)$
such that $\eta(1) = \mu(x_1,\dots,x_r)$
actually defines an operad morphism from the operad of commutative algebras $\ComOp$
towards the $n$-Poisson operad $\PoisOp_n$.
We moreover have a decomposition $\PoisOp_n(r) = \ComOp(r)\oplus\IOp\PoisOp_n(r)$, where we identify $\ComOp(r)$
with the summand spanned by these monomials $\mu(x_1,\dots,x_r) = x_1\cdot\ldots\cdot x_r$
inside $\PoisOp_n(r)$,
while $\IOp\PoisOp_n(r)$ is the graded module spanned by the basis elements of the Poisson operad $\pi = \pi(x_1,\dots,x_r)$
which have at least one Lie monomial of weight $n_i>1$
as factor $\pi_i = \pi_i(x_{i j_1},\dots,x_{i j_{n_i}})$.

When we dualize this structure, we get morphisms of unitary commutative algebras $\eta_*: \PoisOp_n^c(r)\rightarrow\kk$,
right inverse to the natural coaugmentation map $\epsilon_*: \kk\rightarrow\PoisOp_n^c(r)$,
and which define a morphism of Hopf $\Lambda$-cooperads $\eta_*: \PoisOp_n^c\rightarrow\ComOp^c$.
We can also identify the graded module such that $\IOp\PoisOp_n^c(r) = \IOp\PoisOp_n(r)^{\vee}$
with the augmentation ideal
of this augmented algebra structure on $\PoisOp_n^c(r)$.
We accordingly have a splitting formula:
\begin{equation*}
\PoisOp_n^c(r) = \ComOp^c(r)\oplus\IOp\PoisOp_n^c(r),
\end{equation*}
for each arity $r>0$, where we identify $\ComOp^c(r)$ with the summand of the module $\PoisOp_n^c(r)$
spanned by the basis element $\epsilon_*(1) = (x_1\cdot\ldots\cdot x_r)^{\vee}$.
We should note that the collection $\IOp\PoisOp_n^c$ is concentrated in arity $r>1$
since we have $\PoisOp_n^c(1) = \ComOp^c(1) = \kk$
by definition of our cooperads.
We also have an identity:
\begin{equation*}
\IOp\PoisOp_n^c(r) = \Lambda\otimes_{\Sigma}\SOp\PoisOp_n^c(r),
\end{equation*}
for each $r>1$, because in the description of~\S\ref{Background:PoissonFreeStructure}, we can identify the module $\IOp\PoisOp_n^c(r)$
with the submodule of $\PoisOp_n^c(r)$ spanned by the dual basis elements $\pi^{\vee} = \pi(x_1,\dots,x_r)^{\vee}$
of the monomials $\pi = \pi(x_1,\dots,x_r)$
which have at least one Lie monomial of weight $n_i>1$ as factor $\pi_i = \pi_i(x_{i j_1},\dots,x_{i j_{n_i}})$.
We then use the basic definition, for general symmetric collections, of the coend construction of~\S\ref{Background:AlgebraicAdjunctions}.
We also consider the obvious restriction of the corestriction operators of the $n$-Poisson cooperad to the augmentation ideals $\IOp\PoisOp_n^c(r)$, $r>1$,
to regard the collection of these graded modules $\IOp\PoisOp_n^c$
as an object of the category of graded $\Lambda$-collections $\gr^*\Lambda\Seq_{>1}^c$.
We readily see that we actually have an identity $\IOp\PoisOp_n^c = \Lambda\otimes_{\Sigma}\SOp\PoisOp_n^c$ in the category of $\Lambda$-collections.

In what follows, we say that $\IOp\PoisOp_n^c$ represents the free $\Lambda$-collection generated by the symmetric collection $\SOp\PoisOp_n^c$
to depict this relation $\IOp\PoisOp_n^c = \Lambda\otimes_{\Sigma}\SOp\PoisOp_n^c$
in the category of $\Lambda$-collections.
We use this structure result in our study of the deformation complex of the $n$-Poisson cooperad.
In~\S\ref{GraphHomology}, we also deal with a dual expression of the object $\IOp\PoisOp_n$
underlying the $n$-Poisson operad $\PoisOp_n$. (We explain this dual construction with more details in this subsequent section.)

\subsection*{The definition of resolutions and the obstruction problem}

\subsubsection{The algebraic cotriple resolution}\label{Background:CotripleResolution}
We apply the standard cotriple construction to the adjunction
$\ComOp^c/\Sym(-): \ComOp^c/\dg^*\Lambda\Op^c\rightleftarrows\dg^*\Hopf\Lambda\Op^c :\omega$
in order to get a simplicial resolution
$\ROp_{\bullet} = \Res^{com}_{\bullet}(\PoisOp_n^c)$
of the object~$\PoisOp_n^c$
in the category of Hopf $\Lambda$-cooperads. We give full details on this construction
in the appendices (more specifically, in~\S\ref{CotripleResolution}).

Briefly say for the moment that this cotriple resolution $\ROp_{\bullet} = \Res^{com}_{\bullet}(\PoisOp_n^c)$
forms a Reedy cofibrant simplicial object
in the category of Hopf $\Lambda$-cooperads in cochain graded dg-modules (because $\AOp = \PoisOp_n^c$
is cofibrant as a coaugmented $\Lambda$-cooperad).
Recall also that the geometric realization of this simplicial object
\begin{equation*}
\ROp = |\Res^{com}_{\bullet}(\PoisOp_n^c)|
\end{equation*}
(in the sense of model categories) forms a cofibrant resolution of the object~$\AOp = \PoisOp_n^c$
in~$dg^*\Hopf\Lambda\Op^c$.

In~\S\ref{Background:AlgebraicAdjunctions}, we explain that our relative symmetric algebra functor on cooperads $\ComOp^c/\Sym(-)$
is given, in each arity, by a relative symmetric algebra construction $\kk/\Sym(-)$
in the category of plain unitary commutative algebras.
We can also identify the components of the cotriple resolution $\Res^{com}_{\bullet}(\PoisOp_n^c)(r)$, $r>0$,
with simplicial objects $\Res^{com}_{\bullet}(\PoisOp_n^c(r))\in\simp\dg^*\ComCat_+$
which we form within the category of unitary commutative algebras $\dg^*\ComCat_+$
by applying the cotriple resolution construction
to the functor $\kk/\Sym(-): \kk/\dg^*\Mod\rightarrow\dg^*\ComCat_+$
and to the objects $\PoisOp_n^c(r)\in\dg^*\ComCat_+$.
We similarly have an identity:
\begin{equation*}
|\Res^{com}_{\bullet}(\PoisOp_n^c)|(r) = |\Res^{com}_{\bullet}(\PoisOp_n^c(r))|
\end{equation*}
when we pass to geometric realizations.
We use this correspondence in our verification that the geometric realization $\ROp = |\Res^{com}_{\bullet}(\PoisOp_n^c)|$
forms an object weakly-equivalent to $\PoisOp_n^c$ (see~\S\ref{CotripleResolution})
and in our study of deformation complexes (when we prove that the bicosimplicial deformation complex of the $n$-Poisson cooperad
reduces to a Harrison cohomology in the algebraic direction).

\subsubsection{The cooperadic triple coresolution}\label{Background:TripleCoresolution}
We also need coresolutions in the cooperad direction.

Let $\KOp$ be any object in the category of Hopf $\Lambda$-cooperads $\dg^*\Hopf\Lambda\Op^c$.
In a first step, we apply the triple coresolution construction to the adjunction
$\overline{\omega}: \dg^*\Hopf\Lambda\Op^c\rightleftarrows\dg^*\Hopf\Lambda\Seq_{>1}^c :\FreeOp^c$
in order to get a cosimplicial coresolution
$\QOp^{\bullet} = \Res_{op}^{\bullet}(\KOp)$
of this object $\KOp$ in the category $\dg^*\Hopf\Lambda\Op^c$. We give full details on this construction
in the appendices (more specifically, in~\S\ref{TripleCoresolution}).
Briefly say for the moment that this triple coresolution $\QOp^{\bullet} = \Res_{op}^{\bullet}(\KOp)$
forms a Reedy fibrant cosimplicial object in the category of Hopf $\Lambda$-cooperads
in cochain graded dg-modules (without any further assumption on $\KOp$).

In a second step, we perform the totalization of this cosimplicial object
\begin{equation*}
\QOp = \Tot\Res_{op}^{\bullet}(\KOp)
\end{equation*}
(in the sense of model categories) in order to obtain a fibrant coresolution of~$\KOp$
in~$dg^*\Hopf\Lambda\Op^c$.

\subsubsection{The application of Bousfield's obstruction theory}\label{Background:BousfieldObstructionTheory}
We now consider the function space:
\begin{equation*}
T = \Map_{\dg^*\Hopf\Lambda\Op^c}(|\Res^{com}_{\bullet}(\PoisOp_n^c)|,\Tot\Res_{op}^{\bullet}(\KOp)),
\end{equation*}
where we take our cofibrant resolution of the $n$-Poisson cooperad on the source $\ROp = |\Res^{com}_{\bullet}(\PoisOp_n^c)|$
and the just defined fibrant resolution of our Hopf $\Lambda$-cooperad $\KOp\in\dg^*\Hopf\Lambda\Op^c$
on the target $\QOp = \Tot\Res_{op}^{\bullet}(\KOp)$.
We then have
\begin{equation*}
T = \Tot^h\Tot^v\underbrace{\Map_{\dg^*\Hopf\Lambda\Op^c}(\Res^{com}_{\bullet}(\PoisOp_n^c),\Res_{op}^{\bullet}(\KOp))}_{= X^{\bullet\,\bullet}}
\end{equation*}
by end interchange, where we consider the bicosimplicial space
such that
\begin{equation*}
X^{k l} = \Map_{\dg^*\Hopf\Lambda\Op^c}(\Res^{com}_k(\PoisOp_n^c),\Res_{op}^l(\KOp)),
\end{equation*}
for any $(k,l)\in\NN^2$, and $\Tot^h$ refers to the totalization of this space in the horizontal direction $k\in\NN$
of our bicosimplicial structure
while $\Tot^v$ refers to the totalization in the vertical direction $l\in\NN$.

We already mentioned that $\Res^{com}_{\bullet}(\PoisOp_n^c)$ is Reedy cofibrant as a simplicial object in $\dg^*\Hopf\Lambda\Op^c$
and that $\Res_{op}^{\bullet}(\KOp)$ is Reedy fibrant as a cosimplicial object in $\dg^*\Hopf\Lambda\Op^c$.
We easily deduce from these assertions and general properties of function spaces that our bicosimplicial space $X = X^{\bullet\,\bullet}$
is Reedy fibrant as a bicosimplicial object of the category of simplicial sets.

Let $\Tot^{\Delta}$ denote the totalization space of the diagonal object of our cosimplicial space $\Diag(X)^k = X^{k k}$.
We still have:
\begin{equation*}
T = \Tot^{\Delta}\Map_{\dg^*\Hopf\Lambda\Op^c}(\Res^{com}_{\bullet}(\PoisOp_n^c),\Res_{op}^{\bullet}(\KOp)),
\end{equation*}
because we have a general identity $\Tot^{\Delta}(X) = \Tot^h\Tot^v(X) = \Tot^v\Tot^h(X)$
for any bicosimplicial space $X = X^{\bullet\,\bullet}\in\cosimp\cosimp\Simp$ (see for instance~\cite[Proposition 8.1]{Shipley}).
We often omit to mark the diagonalization operation $\Diag(-)$ in our formulas.
We notably write $\pi^s\pi_t(X)$ for the degree $s$ cohomotopy of the diagonal cosimplicial sets $\pi_t(\Diag(X)^{\bullet})$,
where we consider the homotopy of the object $X = X^{\bullet\,\bullet}\in\cosimp\cosimp\Simp$
in degree $t\in\NN$.

We now assume that we have a degree zero cocycle $z\in\pi^0\pi_0(X)$ in the cohomotopy of the cosimplicial set $\pi_0(X) = \pi_0(\Diag(X)^{\bullet})$.
We check in the next section that each space $\Diag(X)^k = X^{k k}$, $k\in\NN$ (actually, each space $X^{k l}$, $k,l\in\NN$),
is isomorphic to a simplicial abelian group (actually, to a simplicial module over the ground field).
In this situation, the Bousfield obstruction theory~\cite{BousfieldObstructions} (see also the textbook \cite[\S VIII.4]{GoerssJardine})
implies that the obstructions to (homotopically) lifting a representative $\phi\in\Diag(X)^0$
of our cocycle $z = [\phi]\in\pi^0\pi_0(X)$
to the whole towers of totalization spaces $\Tot^{\Delta}(X) = \lim_s\Tot_s^{\Delta}(X)$
lie in the cohomotopy groups
$\DGE^2_{s+1 s} = \pi^{s+1}\pi_s(X,\phi^0)$
taken at the base points $\phi^0 = d^k\cdots d^1(\phi)\in\Diag(X)^k$
and where $s>0$.
We also get that the obstructions to the (homotopy) uniqueness of our liftings lie in the cohomotopy groups $\DGE^2_{s s} = \pi^s\pi_s(X,\phi)$,
for $s>0$.

\subsubsection{The obstruction problem}\label{Background:ObstructionProblem}
We now consider the case where we have an isomorphism
at the cohomology level:
\begin{equation*}
\chi: \PoisOp_n^c\xrightarrow{\simeq}\DGH^*(\KOp),
\end{equation*}
where we still assume that $\KOp$ is an object in the category of Hopf $\Lambda$-cooperads $\dg^*\Hopf\Lambda\Op^c$.
We check in the next section that we have a natural isomorphism:
\begin{equation*}
\pi^0\pi_0(X)\xrightarrow{\simeq}\Mor_{\gr^*\Hopf\Lambda\Op^c}(\PoisOp_n^c,\DGH^*(\KOp))
\end{equation*}
from the set of degree zero cocycles $z\in\pi^0\pi_0(X)$ of our bicosimplicial function space in~\S\ref{Background:BousfieldObstructionTheory}
to the set of morphisms of Hopf $\Lambda$-cooperads $f: \PoisOp_n^c\rightarrow\DGH^*(\KOp)$
(see Theorem~\ref{DeformationComplexes:BicosimplicialComplex:DegreeZeroCocycles}).
We therefore pick a degree zero cocycle $z = [\phi]\in\pi^0\pi_0(X)$
that corresponds to our isomorphism $\chi: \PoisOp_n^c\xrightarrow{\simeq}\DGH^*(\KOp)$.
In the first instance, we aim to check that the cohomotopy groups $\DGE^0_{s+1 s} = \pi^{s+1}\pi_s(X,\phi^0)$
vanish for all $s>0$. We are going to prove that this statement holds when $4\nmid n$.
If we have such a vanishing relation, then we can lift a representative of our cocycle to the whole tower
of totalization spaces $\Tot^{\Delta}(X) = \lim_s\Tot_s^{\Delta}(X)$
and, as a by-product, we get a morphism
\begin{equation*}
\phi: |\Res^{com}_{\bullet}(\PoisOp_n^c)|\rightarrow\Tot\Res_{op}^{\bullet}(\KOp)
\end{equation*}
which realizes our cohomology isomorphism.

We may also assume that $\KOp$ is equipped with an involution $J: \KOp\xrightarrow{\simeq}\KOp$,
and in parallel, we consider the involution inherited by the $n$-Poisson cooperad $J_*: \PoisOp_n^c\rightarrow\PoisOp_n^c$
from the little $n$-discs operad $\DOp_n$ (as we explained in the introduction of the paper).
We require, in this context, that our isomorphism $\chi: \PoisOp_n^c\xrightarrow{\simeq}\DGH^*(\KOp)$
preserves the involution operations at the cohomology level.
We then have isomorphisms induced by the involution operations on the source and target objects of our function space.
We consider the conjugate of these operations.
We accordingly get a map on our bicosimplicial space $J: X\xrightarrow{\simeq} X$, which satisfies $J^2=\id$,
and whose fixed points represent involution preserving functions.

In this context, we can also use the Bousfield obstruction theory equivariantly with respect to the action of involutions.
We then use that the objects $X^{k l}$ are simplicial modules over our characteristic zero ground field (and not only abelian groups) for all $k,l\in\NN$,
and that the codegeneracy operators of the bicosimplicial object $X$ preserve this module structure
as well as our involution operator $J: X^{k l}\rightarrow X^{k l}$.
We essentially deduce from this requirement that the subspaces of fixed points $(X^{k k})^J$ inside $X^{k k}$, $k\in\NN$,
still form a Reedy fibrant cosimplicial object in the category of simplicial sets,
while we have $\Tot_s^{\Delta}(X)^J = \Tot_s^{\Delta}(X^J)$ and $\Tot^{\Delta}(X)^J = \Tot^{\Delta}(X^J) = \lim_s\Tot_s^{\Delta}(X^J)$
by interchange of limits.
In this setting, we can assume that we take a representative of our degree zero cocycle $[\phi]\in\pi^0\pi_0(X)$
which is invariant under the action of the involution $J: X\xrightarrow{\simeq} X$
on the bicosimplicial function space $X$, because we have the interchange formula $\pi^0\pi_0(X)^{J_*} = \pi^0(\pi_0(X)^J) = \pi^0(\pi_0(X^J))$.
We also have the identity $\pi^s\pi_t(X^J) = \pi^s\pi_t(X)^{J_*}$
for each $t>0$ and for every $s\in\NN$,
because these cohomotopy groups are determined by the homology of a conormalized complex of cosimplicial modules
over the ground field $\kk$.

Then we aim to check that the submodules of fixed points $\pi^{s+1}\pi_s(X)^{J_*}$
vanish for all $s>0$
inside our cohomotopy groups $\DGE^0_{s+1 s} = \pi^{s+1}\pi_s(X,\phi^0)$.
We will prove that this result holds without any condition on the dimension parameter $n$.
If we have such a vanishing relation, then we can lift a representative of our cocycle equivariantly to the tower of totalization spaces
$\Tot^{\Delta}(X) = \lim_s\Tot_s^{\Delta}(X)$ and we accordingly get that our cohomology isomorphism
is realized by a $J$-equivariant morphism.
%
We are similarly going to check the vanishing of the modules $\pi^s\pi_s(X)^{J_*}$
in order to prove the homotopy uniqueness of our morphisms.

\subsubsection{The obstruction problem for morphisms}\label{Background:RelativeObstructionProblem}
To check our formality theorem for morphisms, we consider the objects $\KOp_n = \DGOmega^*_{\sharp}(\EOp_n)$,
associated to cofibrant models of $E_n$-operads in the category of simplicial sets $\EOp_n$,
and which we take to define the homotopy type of $E_n$-operads
in the category of Hopf $\Lambda$-cooperads, for any $n\geq 2$.
We may still assume that these Hopf $\Lambda$-cooperads are equipped with involutive isomorphisms $J: \KOp_n\xrightarrow{\simeq}\KOp_n$
mimicking the action of hyperplane reflections on the little discs spaces.
We also assume that we have morphisms $\iota^*: \KOp_n\rightarrow\KOp_m$, preserving the involution operations,
and which model the embeddings $\iota: \DOp_m\rightarrow\DOp_n$
on the little discs operads $\DOp_n$.
We form the diagram:
\begin{equation*}
\xymatrix{ 
|\Res^{com}_{\bullet}(\PoisOp_n^c)|
\ar[d]_{\iota^*}\ar@{.>}[r]^-{\sim} &
\Tot\Res_{op}^{\bullet}(\KOp_n)\ar[d]^{\iota^*}
\\
|\Res^{com}_{\bullet}(\PoisOp_m^c)|
\ar@{.>}[r]^-{\sim} &
\Tot\Res_{op}^{\bullet}(\KOp_m)
},
\end{equation*}
for any $n>m>1$, where the dotted arrows represent formality weak-equivalences which we produce by working out the obstruction
problem of the previous paragraph.

We can apply the constructions of the previous paragraph to the Hopf $\Lambda$-cooperad $\KOp = \KOp_m$.
We then replace the isomorphism $\chi$, which we consider in this previous construction,
by the morphism
\begin{equation*}
\PoisOp_n^c\xrightarrow{\iota^*}\PoisOp_m^c = \DGH^*(\KOp_m),
\end{equation*}
which we deduce from the relation $\DGH^*(\KOp_m) = \DGH^*\DGOmega^*_{\sharp}(\EOp_m) = \PoisOp_m^c$.
We aim to check in this case that the cohomotopy groups $\DGE^0_{s s} = \pi^s\pi_s(X,\phi^0)$
vanish for all $s>0$. We are going to prove that this statement holds as soon $n-m\geq 2$.
If we have such a vanishing relation, then we can conclude that the morphism $\iota^*$
has a unique realization up to homotopy in the category of Hopf $\Lambda$-cooperads.
This result implies that the above diagram commutes (up to homotopy yet),
which is the claim of our formality statement
for morphisms.

We use a similar method to address the case $m=1$, where our operad in simplicial sets $\EOp_1$ is weakly-equivalent
to the operad of associative algebras $\AsOp$.
We then have $\DGH^*(\KOp_1) = \DGH^*(\EOp_1) = \AsOp^c$, where $\AsOp^c$ denotes the dual cooperad
of the associative operad in the category $\kk$-modules.
We just replace the $m$-Poisson cooperad $\PoisOp_m^c$
by this cooperad $\AsOp^c$
in our construction. We will see that our cohomotopy groups $\DGE^0_{s s} = \pi^s\pi_s(X,\phi^0)$
also vanish in this case, for all $s>0$, as soon as we consider a Poisson cooperad $\PoisOp_n^c$
such that $n\geq 3$ as target object.
We tackle this particular case in the concluding section of this paper only. We focus on the case $m\geq 2$,
where we deal with a Poisson cooperad $\PoisOp_m^c$,
otherwise.

\renewcommand{\thesubsection}{\thesection.\arabic{subsection}}
\renewcommand{\thesubsubsection}{\thesubsection.\arabic{subsubsection}}
\numberwithin{subsubsection}{subsection}

\section{From biderivations to deformation bicomplexes of Hopf cooperads}\label{DeformationComplexes}
In this section, we give an effective description, in terms of the homology of a deformation complex, of the homotopy of the cosimplicial object
which captures the obstruction to the existence of our formality map in~\S\ref{Background:ObstructionProblem}.

We use that our cosimplicial object is defined by the diagonal of a bicosimplicial function space.
In a first step, we prove that the homotopy of this bicosimplicial space
is isomorphic to the homology of a bicosimplicial complex of biderivations
which we determine from the algebraic cotriple resolution
and from the cooperadic triple coresolution
of our Hopf cooperads.
In a second step, we establish that this bicosimplicial biderivation complex is weakly-equivalent
to a deformation bicomplex which we define by combining the classical Harrison complex
of commutative algebras in one direction and the cobar complex of cooperads in the other direction.
Eventually, we prove that, in the case of the cooperads $\PoisOp_n^c$, $n\geq 2$,
we can replace the cooperadic cobar complex by a small complex
which we define by using the Koszul duality of operads.

We just explain the definition of our notion of biderivation before tackling the definition of the biderivation complex.
We devote a preliminary section to this subject.

\setcounter{subsection}{0}

\subsection{Preliminaries: modules, bicomodules and biderivations}\label{DeformationComplexes:Biderivations}
In the definition of the bicosimplicial complex we use biderivations with respect to morphisms of Hopf $\Lambda$-cooperads $\phi: \ROp\rightarrow\QOp$.
When we address the reduction of this bicosimplicial biderivation complex to the deformation bicomplex,
we also deal with comodule and module structures
that underlie our Hopf $\Lambda$-cooperads.
We can actually give a sense to the notion of a coderivation as soon as we have a comodule over a cooperad,
and we can give a sense to the notion of a derivation as soon as we have a module over a Hopf collection.
We therefore explain the definition of these concepts first and we address the definition of a biderivation
afterwards.

In the first instance, we just review the definition of the dg-modules of homomorphisms
associated to the category of $\Lambda$-collections
and which contain our dg-modules of biderivations as submodules.

\begin{rem}[Homomorphisms]\label{DeformationComplexes:Biderivations:Homomorphisms}
We consider the internal hom bifunctor of the category of dg-modules
\begin{equation}
\Hom_{\dg\Mod}(-,-): \dg\Mod^{op}\times\dg\Mod\rightarrow\dg\Mod
\end{equation}
which represents the right adjoint of the tensor product on $\dg\Mod$.
We say that a map $f: C\rightarrow D$ is a homomorphism of dg-modules (as opposed to a morphism of dg-modules)
when this map is an element of this hom-object $f\in\Hom_{\dg\Mod}(C,D)$.
Recall simply that $f$ defines an element of (lower) degree $d$ in this hom-object when $f$ raises (lower) degrees by $d$,
and that the differential of $f$ inside $\Hom_{\dg\Mod}(C,D)$
is given by the formula $\delta(f) = f\delta - \pm\delta f$,
where we take the commutator of $f$ with the internal differential of the dg-modules $C$ and $D$.

We then use the end-formula:
\begin{equation}\label{DeformationComplexes:Biderivations:Homomorphisms:Definition}
\Hom_{\dg\ICat\Seq^c}(\MOp,\NOp) = \int_{\rset\in\ICat}\Hom_{\dg\Mod}(\MOp(r),\NOp(r)),
\end{equation}
where we set $\ICat = \Sigma$ (respectively, $\ICat = \Lambda$) to provide the category of symmetric collections (respectively, of $\Lambda$-collections)
with a hom bifunctor with values in the category of dg-modules.
We also call `homomorphisms of symmetric collections (respectively, of $\Lambda$-collections)'
the elements of this dg-hom.
We have an obvious counterpart of these dg-modules of homomorphisms for symmetric collections (respectively, for $\Lambda$-collections)
in graded modules $\Hom_{\gr\ICat\Seq^c}(\MOp,\NOp)$
which follows from the identity between graded modules and dg-modules equipped with a trivial differential.

Recall that we assume that our objects are concentrated in arity $r>1$ in our definition of the category of symmetric collections
(see~\S\ref{Background:TripleCoresolution}), and similarly when we deal with $\Lambda$-collections.
In fact, we may equivalently assume that a symmetric collection (respectively, a $\Lambda$-collection) is a diagram $\MOp$
defined over the entire category $\ICat = \Sigma$ (respectively, $\ICat = \Lambda$),
but for which we have $\MOp(r) = 0$ for $r = 0,1$.
We implicitly use this correspondence when we perform our end~(\ref{DeformationComplexes:Biderivations:Homomorphisms:Definition})
over the whole category $\ICat = \Sigma$ (respectively, $\ICat = \Lambda$).

In what follows, we also form hom-objects with (Hopf) cooperads as source $\MOp = \COp$ or as target object $\NOp = \COp$.
In principle, we have to take the coaugmentation coideal $\overline{\COp}$ in to order to fulfill our connectedness
requirements for collections $\overline{\COp}(0) = \overline{\COp}(1) = 0$,
but this does not change the result of our hom-object construction
as soon as one of our objects does fulfill the connectedness relations. We therefore often perform this abuse of notation
when we deal with homomorphisms.
\end{rem}

\begin{defn}[Bicomodules over cooperads]\label{DeformationComplexes:Biderivations:Bicomodules}
Let $\QOp\in\dg^*\Op^c$. We say that a symmetric collection $\MOp\in\dg^*\Sigma\Seq_{>1}^c$
is a bicomodule over $\QOp$
when we have left and right coproduct operations
\begin{align}
\label{DeformationComplexes:Biderivations:Bicomodules:LeftCoproducts}
& \circ_i^*: \MOp(k+l-1)\rightarrow\MOp(k)\otimes\QOp(l),\\
\label{DeformationComplexes:Biderivations:Bicomodules:RightCoproducts}
& \circ_i^*: \MOp(k+l-1)\rightarrow\QOp(k)\otimes\MOp(l),
\end{align}
defined for all $k,l>1$, $i = 1,\dots,k$, and which satisfy an obvious extension
of the usual equivariance, counit, and coassociatity relations
of the coproducts of cooperads.
In what follows, we also consider the two-sided coproducts
\begin{equation}\label{DeformationComplexes:Biderivations:Bicomodules:TwoSidedCoproducts}
\circ_i^*: \MOp(k+l-1)\rightarrow\MOp(k)\otimes\QOp(l)\oplus\QOp(k)\otimes\MOp(l),
\end{equation}
whose components are defined by the above one-sided operations.
This notion of a bicomodule over a cooperad is just dual to the notion of an operadic infinitesimal bimodule,
such as considered in~\cite{AroneTurchin,MerkulovVallette},
in~\cite{DwyerHess} (under the name `linear bimodule')
and in~\cite[\S III.2.1]{FresseBook} (under the name `abelian bimodule').

If we assume $\QOp\in\ComOp^c/\dg^*\Lambda\Op^c$ and $\MOp\in\dg^*\Lambda\Seq_{>1}^c$, then we still require
that our left and right coproduct operations fulfill an obvious analogue, for bicomodules,
of the extended equivariance relations of the coproducts
of coaugmented $\Lambda$-cooperads
with respect to the action of corestriction operators (see also~\cite[\S III.2.1.1]{FresseBook} for the formulation
of the dual relations in the context of abelian bimodules over augmented $\Lambda$-operads).
We just assume that our bicomodules are equipped with null coaugmentations $\epsilon_*: 0\rightarrow\MOp(r)$
in the degenerate case of the equivariance relations
which involve the application of coaugmentation maps
on our objects (see~\cite[\S II.11.1.1(d)]{FresseBook} for the expression of these equivariance relations).
We omit to specify these extended equivariance requirements in general.
We may just say that the object $\MOp$ forms a bicomodule over $\QOp$ in the category of $\Lambda$-collections
when the context makes such a precision necessary.
\end{defn}

\begin{defn}[Coderivations]\label{DeformationComplexes:Biderivations:Coderivations}
In the context of the previous paragraph~\S\ref{DeformationComplexes:Biderivations:Bicomodules},
we say that a homomorphism $\theta: \MOp\rightarrow\QOp$ is a coderivation
when $\theta$ makes the following diagrams commute:
\begin{equation}
\xymatrix{ \MOp(k+l-1)\ar[rr]^-{\theta}\ar[d]_{\circ_i^*} && \QOp(k+l-1)\ar[d]^{\circ_i^*} \\
\MOp(k)\otimes\QOp(l)\oplus\QOp(k)\otimes\MOp(l)\ar[rr]^-{\theta\otimes\id+\id\otimes\theta} && \QOp(k)\otimes\QOp(l) }
\end{equation}
for all $k,l>1$, and $i = 1,\dots,k$.

In what follows, we mostly deal with the case where $\QOp$ is a coaugmented $\Lambda$-cooperad and $\MOp$ is a bicomodule over $\QOp$
in the category of $\Lambda$-collections.
We generally require, in this context, that $\theta: \MOp\rightarrow\QOp$ preserves the corestriction operators attached to our objects,
and hence, is defined by a homomorphism of $\Lambda$-collections.
We adopt the notation
\begin{equation}
\CoDer_{\ComOp^c/\dg^*\Lambda\Op^c}(\MOp,\QOp)\subset\Hom_{\dg\Lambda\Seq^c}(\MOp,\QOp)
\end{equation}
for this module of coderivations. We readily check that the differential of the dg-module of homomorphisms preserves coderivations.
Our module of coderivations inherits a natural dg-module structure therefore.
\end{defn}

We have the following observation:

\begin{prop}\label{DeformationComplexes:Biderivations:CofreeCoderivations}
If we have $\QOp = \FreeOp^c(\NOp)$ for some $\NOp\in\overline{\ComOp}{}^c/\dg^*\Lambda\Seq_{>1}^c$,
then we have
an isomorphism:
\begin{equation*}
\CoDer_{\ComOp^c/\dg^*\Lambda\Op^c}(\MOp,\QOp)\simeq\Hom_{\dg\Lambda\Seq^c}(\MOp,\NOp),
\end{equation*}
for any bicomodule $\MOp$ over the coaugmented $\Lambda$-cooperad $\QOp$.
\end{prop}

\begin{proof}
This isomorphism maps a coderivation $\theta: \MOp\rightarrow\FreeOp^c(\NOp)$
to its composite with the canonical projection $\pi: \overline{\FreeOp}^c(\NOp)\rightarrow\NOp$.
The converse mapping is defined by dualizing the construction~\cite[Theorem III.2.1.7]{FresseBook}
for derivations on free operads.
\end{proof}

\begin{defn}[Modules over Hopf collections]\label{DeformationComplexes:Biderivations:Modules}
Let now $\ROp\in\dg^*\Hopf\Sigma\Seq_{>1}^c$. We say that a symmetric collection $\NOp\in\dg^*\Sigma\Seq_{>1}^c$
is a symmetric bimodule over $\ROp$ (or just a module over $\ROp$ for short)
when we have symmetric left and right product operations
\begin{equation}
\xymatrix@R=1em{ \ROp(r)\otimes\NOp(r)\ar[dr]^{\lambda}\ar[dd]_{\simeq} & \\
& \NOp(r) \\
\NOp(r)\otimes\ROp(r)\ar[ur]_{\rho} & },
\end{equation}
defined for all $r>1$, and which satisfy obvious equivariance relations with respect to the action of permutations,
as well as a natural analogue, for symmetric collections, of the usual unit and associativity relations for modules over commutative algebras.
We equivalently require that these left and right product operations provide each object $\NOp(r)\in\dg^*\Mod$, $r>1$,
with the structure of a (symmetric bi)module over the commutative algebra $\ROp(r)\in\dg^*\ComCat_+$.

In the case where we have a Hopf $\Lambda$-collection $\ROp\in\dg^*\Hopf\Lambda\Seq_{>1}^c$
and a coaugmented $\Lambda$-collection $\NOp\in\ComOp^c/\dg^*\Lambda\Seq_{>1}^c$,
we also require that our products preserve the action of corestriction operators on our objects.
We again omit to specify this extra requirement in general. We may just say that the object $\NOp$ forms a module over $\ROp$
in the category of $\Lambda$-collections when the context makes such a precision necessary.
\end{defn}

\begin{defn}[Derivations]\label{DeformationComplexes:Biderivations:Derivations}
In the context of~\S\ref{DeformationComplexes:Biderivations:Modules}, we say that a homomorphism $\theta: \ROp\rightarrow\NOp$
is a derivation when $\theta$ makes the following diagrams
commute:
\begin{equation}
\xymatrix{ \ROp(r)\otimes\ROp(r)\ar[rr]^-{\theta\otimes\id + \id\otimes\theta}\ar[d]_{\mu} &&
\NOp(r)\otimes\ROp(r)\oplus\ROp(r)\otimes\NOp(r)\ar[d]^{\rho+\lambda} \\
\ROp(r)\ar[rr]^-{\theta} && \NOp(r) }
\end{equation}
for all $r>1$. This requirement is equivalent to the assumption that the components of our homomorphism $\theta: \ROp(r)\rightarrow\NOp(r)$
are derivations in the classical sense when we regard each object $\ROp(r)$
as a plain unitary commutative algebra and each object $\NOp(r)$
as a module over $\ROp(r)$.

In what follows, we still mostly deal with the case where $\ROp$ is a Hopf $\Lambda$-collection and $\NOp$ is a module over $\ROp$
in the category of $\Lambda$-collections.
In this context, we generally require that $\theta: \ROp\rightarrow\NOp$ preserves the corestriction operators attached to our object
and hence, is defined by a homomorphism of $\Lambda$-collections.
We adopt the notation
\begin{equation}
\Der_{\dg\Hopf\Lambda\Seq^c}(\ROp,\NOp)\subset\Hom_{\dg\Lambda\Seq^c}(\ROp,\NOp)
\end{equation}
for this module of derivations. We still readily check (as in the coderivation case) that the differential of the dg-module
of homomorphisms preserves derivations. Our module of derivations therefore inherits a natural dg-module structure.
\end{defn}

We now consider the case where our Hopf collection $\ROp$ is identified with a relative symmetric algebra $\ROp = \overline{\ComOp}{}^c/\Sym(\MOp)$
for some object $\MOp\in\overline{\ComOp}{}^c/\dg^*\Lambda\Seq_{>1}^c$.
We assume for simplicity that $\MOp$ is equipped with an augmentation $\eta_*: \MOp\rightarrow\overline{\ComOp}{}^c$
that splits the canonical coaugmentation morphism $\epsilon_*: \overline{\ComOp}{}^c\rightarrow\MOp$
attached to our object.
We equivalently have a decomposition $\MOp = \overline{\ComOp}{}^c\oplus\IOp\MOp$,
where we set $\IOp\MOp = \ker(\eta_*: \MOp\rightarrow\overline{\ComOp}{}^c)$.
We then get the identity $\ROp(r) = \Sym(\IOp\MOp(r))$, for each arity $r>0$,
when we pass to our relative symmetric algebra $\ROp$.
We may also set $\IOp\MOp(r) = \coker(\epsilon_*: \kk\rightarrow\MOp(r))$
and forget about the augmentation $\eta_*: \MOp\rightarrow\ComOp^c$,
but the existence of this structure simplifies
our constructions.
We have the following observation:

\begin{prop}\label{DeformationComplexes:Biderivations:SymmetricAlgebraDerivations}
If we have $\ROp = \overline{\ComOp}{}^c/\Sym(\MOp)$ for some coaugmented $\Lambda$-collection $\MOp\in\overline{\ComOp}{}^c/\dg^*\Lambda\Seq_{>1}^c$
such that $\MOp = \overline{\ComOp}{}^c\oplus\IOp\MOp$,
then we have
an isomorphism:
\begin{equation*}
\Der_{\dg\Hopf\Lambda\Seq^c}(\ROp,\NOp)\simeq\Hom_{\dg\Lambda\Seq^c}(\IOp\MOp,\NOp)
\end{equation*}
for any module $\NOp$ over the Hopf $\Lambda$-collection $\ROp$.
\end{prop}

\begin{proof}
This isomorphism maps a derivation $\theta: \ROp\rightarrow\NOp$
to its composite with the canonical morphism $\IOp\MOp\subset\MOp\xrightarrow{\iota}\overline{\ComOp}{}^c/\Sym(\MOp)$.

The converse mapping associates a homomorphism $f: \IOp\MOp\rightarrow\NOp$
to the derivation $\theta = \theta_f$
defined by the usual formula $\theta_f(\xi_1\cdot\ldots\cdot\xi_m) = \sum_{i=1}^m \pm \xi_1\cdots f(\xi_i)\cdots\xi_m$,
for any symmetric algebra monomials $\xi_1\cdot\ldots\cdot\xi_m\in\Sym(\IOp\MOp(r))$,
and for each arity $r>1$.
In this expression, we consider, for any $i = 1,\dots,m$, the action of the factors $\xi_j\in\IOp\MOp(r)$, $j\not=i$, on $f(\xi_i)\in\NOp(r)$
through the morphism $\iota: \IOp\MOp(r)\rightarrow\Sym(\IOp\MOp(r))$
and the action of the algebra $\ROp(r) = \Sym(\IOp\MOp(r))$
on $\NOp(r)$.
\end{proof}

\begin{defn}[Biderivations]\label{DeformationComplexes:Biderivations:Biderivations}
We now assume that we have a morphism of Hopf cooperads
\begin{equation}\label{DeformationComplexes:Biderivations:Biderivations:Morphism}
\phi: \ROp\rightarrow\QOp.
\end{equation}
We are also going to assume that $\ROp$ is equipped with an augmentation $\eta_*: \ROp\rightarrow\ComOp^c$ (for simplicity yet)
and we set $\IOp\ROp = \ker(\eta_*: \overline{\ROp}\rightarrow\overline{\ComOp}{}^c)$
for the kernel of this augmentation morphism (as in the case of the $n$-Poisson cooperad $\ROp = \PoisOp_n^c$).
We accordingly have the relation $\overline{\ROp} = \overline{\ComOp}{}^c\oplus\IOp\ROp$
in the category of symmetric collections.
We also say that $\IOp\ROp$ represents the (algebraic) augmentation ideal of the Hopf cooperad $\ROp$.

Let us observe that the augmentation kernel $\IOp\COp = \ker(\eta_*: \overline{\COp}\rightarrow\overline{\ComOp}{}^c)$ of any cooperad $\COp$
equipped with an augmentation over the cooperad of commutative coalgebras $\eta_*: \COp\rightarrow\ComOp^c$
forms a bicomodule over this cooperad $\COp$.
Dually, a Hopf symmetric collection $\AOp$ forms a module over itself.
In our situation, we can still provide the object $\IOp\ROp$
with the structure of a bicomodule
over the cooperad $\QOp$
by restriction through our morphism $\phi$.

We can symmetrically provide the object $\overline{\QOp}$ with the structure of a module over the Hopf symmetric collection $\overline{\ROp}$.
We then say that a homomorphism $\theta: \IOp\ROp\rightarrow\overline{\QOp}$ is a biderivation when it is both a coderivation
with respect to this coaugmented bicomodule structure on $\IOp\ROp$ and a derivation
with respect to this module structure on $\overline{\QOp}$.

In what follows, we also mostly deal with the case where our objects are Hopf $\Lambda$-cooperads
and $\phi$ is a morphism of Hopf $\Lambda$-cooperads.
We generally assume, in this context, that our biderivations preserve the action of corestriction operators.
We adopt the notation
\begin{equation}
\BiDer_{\dg^*\Hopf\Lambda\Op^c}(\ROp,\QOp)\subset\Hom_{\dg\Lambda\Seq^c}(\IOp\ROp,\overline{\QOp})
\end{equation}
for this module of biderivations. We easily see, once again, that the differential of the dg-module of homomorphisms
preserves biderivations. Our module of biderivations accordingly inherits a natural dg-module structure.
\end{defn}

We now consider the case where the Hopf $\Lambda$-cooperad $\ROp$ is identified with the relative symmetric algebra $\ROp = \ComOp^c/\Sym(\COp)$
associated to a coaugmented $\Lambda$-cooperad $\COp\in\ComOp^c/\dg^*\Lambda\Op^c$
and $\QOp$ is identified with a cofree cooperad $\QOp = \FreeOp^c(\NOp)$
on some Hopf $\Lambda$-collection $\NOp\in\dg^*\Hopf\Lambda\Seq_{>1}^c$.
We also assume for simplicity (as usual) that $\COp$ is equipped with an augmentation $\eta_*: \COp\rightarrow\ComOp^c$ as an object
of the category of coaugmented $\Lambda$-cooperads $\ComOp^c/\dg^*\Lambda\Op^c$.
We still set $\IOp\COp = \ker(\eta_*: \overline{\COp}\rightarrow\overline{\ComOp}{}^c)$
so that we have the splitting formula $\overline{\COp} = \overline{\ComOp}{}^c\oplus\IOp\COp$
in the category of $\Lambda$-collections.
We have a morphism of Hopf $\Lambda$-cooperads induced by $\eta_*: \COp\rightarrow\ComOp^c$
on the relative symmetric algebra $\ROp = \ComOp^c/\Sym(\COp)$,
and this morphism defines the augmentation of our Hopf $\Lambda$-cooperad $\eta_*: \ROp\rightarrow\ComOp^c$
in~\S\ref{DeformationComplexes:Biderivations:Biderivations}.
We have the following proposition:

\begin{prop}\label{DeformationComplexes:Biderivations:ModuleReduction}
If we have $\ROp = \ComOp^c/\Sym(\COp)$ in the definitions of~\S\ref{DeformationComplexes:Biderivations:Biderivations},
where $\COp\in\ComOp^c/\dg^*\Lambda\Op^c$ is a coaugmented $\Lambda$-cooperad
equipped with an augmentation over the commutative cooperad $\eta_*: \COp\rightarrow\ComOp^c$,
and if we have $\QOp = \FreeOp^c(\NOp)$ for some Hopf $\Lambda$-collection $\NOp\in\dg^*\Hopf\Lambda\Seq_{>1}^c$,
then we have a commutative square of isomorphisms:
\begin{equation*}
\xymatrix{ \BiDer_{\dg^*\Hopf\Lambda\Op^c}(\ROp,\QOp)\ar[r]^-{\simeq}\ar[d]_{\simeq} &
\Der_{\dg\Hopf\Lambda\Seq^c}(\overline{\ROp},\NOp)\ar[d]_-{\simeq} \\
\CoDer_{\ComOp^c/\dg^*\Lambda\Op^c}(\IOp\COp,\QOp)\ar[r]^-{\simeq}\ar[r]^-{\simeq} &
\Hom_{\dg\Lambda\Seq^c}(\IOp\COp,\NOp) }.
\end{equation*}
\end{prop}

\begin{proof}
To get this result, we basically check that the bijections of Proposition~\ref{DeformationComplexes:Biderivations:CofreeCoderivations}
and Proposition~\ref{DeformationComplexes:Biderivations:SymmetricAlgebraDerivations}
induce a one-to-one correspondence between biderivations and coderivations in the cofree cooperad case,
and a one-to-one correspondence between biderivations and coderivations
in the symmetric algebra case. This verification is straightforward.
\end{proof}

\subsection{The bicosimplicial biderivation complex}\label{DeformationComplexes:BicosimplicialComplex}
We now study the homotopy of the bicosimplicial function space
\begin{equation*}
X^{\bullet\,\bullet} = \Map_{\dg^*\Hopf\Lambda\Op^c}(\Res^{com}_{\bullet}(\PoisOp_n^c),\Res_{op}^{\bullet}(\KOp))
\end{equation*}
which we consider in our obstruction method. We first check that we have a bijection between the cohomotopy class set $\pi^0\pi_0(\Diag X)$
and the set of morphisms of Hopf $\Lambda$-cooperads in graded modules $\chi: \PoisOp_n^c\rightarrow\DGH^*(\KOp)$.
We prove afterwards that the cohomotopy class sets $\pi^s\pi_t(\Diag X)$
can be determined by using a bicosimplicial complex
of biderivations
associated to any morphism $\chi: \PoisOp_n^c\rightarrow\DGH_*(\KOp)$.

We consider any object of the category of Hopf $\Lambda$-cooperads in cochain graded modules $\PiOp\in\gr^*\Hopf\Lambda\Op^c$
equipped with an augmentation
over the commutative cooperad $\eta_*: \PiOp\rightarrow\ComOp^c$
for the moment.
We will take $\PiOp = \PoisOp_n^c$ later on, when we tackle the applications of our constructions
to our initial obstruction problem.
We just require that the augmentation ideal $\IOp\PiOp$ of our Hopf $\Lambda$-cooperad $\PiOp$
is free as a $\Lambda$-collection, in the sense that we have an identity $\IOp\PiOp = \Lambda\otimes_{\Sigma}\SOp\PiOp$,
for some generating symmetric collection $\SOp\PiOp\subset\IOp\PiOp$,
as in case of the $n$-Poisson cooperad $\PiOp = \PoisOp_n^c$ (see~\S\S\ref{Background:PoissonFreeStructure}-\ref{Background:PoissonAugmentationIdeal}).
We need such an assumption in order to guarantee the validity of our constructions.
We similarly assume that $\KOp$ is any Hopf $\Lambda$-cooperad in the category of dg-modules for the moment (regardless of our obstruction problem).
We do not need to make any extra assumption for this second Hopf $\Lambda$-cooperad $\KOp$.

Note that we have $\PiOp = \DGH^*(\PiOp)$, since we assume that $\PiOp$ is defined within the category of cochain graded modules.
In fact, we can extend the constructions of this subsection to a more general setting,
where $\PiOp$ is a Hopf $\Lambda$-cooperads in cochain graded dg-modules,
but this identity $\PiOp = \DGH^*(\PiOp)$ will simplify our layout.
We may also set $\IOp\PiOp(r) = \coker(\epsilon_*: \kk\rightarrow\PiOp(r))$, where we consider the unit morphism of the algebras~$\PiOp(r)$,
in order to give a sense to our construction without assuming the existence of an augmentation $\eta_*: \PiOp\rightarrow\ComOp^c$,
but the existence of this structure simplifies the analysis of our constructions
too (see the preliminary discussion before Proposition~\ref{DeformationComplexes:Biderivations:SymmetricAlgebraDerivations}).

We define the function spaces on the category~$\dg^*\Hopf\Lambda\Op^c$ by using general concepts of the theory of model categories.
Recall that these function spaces $\Map_{\dg^*\Hopf\Lambda\Op^c}(\ROp,\QOp)$,
where $\ROp$ is any cofibrant object in~$\dg^*\Hopf\Lambda\Op^c$
and $\QOp$ is a fibrant object,
basically depend either on the choice of a cosimplicial framing for $\ROp$,
or on the choice of a simplicial framing for $\QOp$,
and that different choices give weakly-equivalent objects
in the category of simplicial sets.
We make the definition of these function spaces explicit for the Hopf $\Lambda$-cooperads $\ROp_k = \Res^{com}_k(\PiOp)$, $k\in\NN$,
and $\QOp^l = \Res_{op}^l(\KOp)$, $l\in\NN$.
We use statements of the appendix sections, and the simplicial framing of~\S\ref{TripleCoresolution:SimplicialFraming}.
We revisit this construction first.

\begin{constr}[The explicit definition of function spaces of Hopf cooperads]\label{DeformationComplexes:BicosimplicialComplex:FunctionSpaces}
In~\S\ref{TripleCoresolution}, we explain that the Hopf $\Lambda$-cooperads $\QOp^l = \Res_{op}^l(\KOp)$, $l\in\NN$,
are, by construction, identified with cofree cooperads $\QOp^l = \FreeOp^c(\DGC_{op}^l(\KOp))$
on coaugmented Hopf $\Lambda$-collections $\DGC_{op}^l(\KOp)\in\overline{\ComOp}{}^c/\dg^*\Hopf\Lambda\Seq_{>1}^c$
such that:
\begin{equation}\label{DeformationComplexes:BicosimplicialComplex:FunctionSpaces:CofreeStructure}
\DGC_{op}^l(\KOp) = \underbrace{\overline{\FreeOp}^c\circ\dots\circ\overline{\FreeOp}^c}_{l}(\overline{\KOp}),
\end{equation}
for any $l\in\NN$, where we perform an $l$-fold composite of the coaugmentation coideal
of the cofree cooperad functor $\overline{\omega}\FreeOp^c = \overline{\FreeOp}^c$.
Then we check that the cofree cooperads
\begin{equation}\label{DeformationComplexes:BicosimplicialComplex:FunctionSpaces:SimplicialFraming}
\Res_{op}^l(\KOp)^{\Delta^{\bullet}} := \FreeOp^c(\DGC_{op}^l(\KOp)\otimes\DGOmega^*(\Delta^{\bullet})),
\end{equation}
where we take an arity-wise tensor product of this Hopf $\Lambda$-collection $\DGC_{op}^l(\KOp)\in\overline{\ComOp}{}^c/\dg^*\Hopf\Lambda\Seq_{>1}^c$
with the Sullivan dg-algebra $\DGOmega^*(\Delta^{\bullet})\in\dg^*\ComCat_+$
define a simplicial framing of each object $\QOp^l = \Res_{op}^l(\KOp)$, $l\in\NN$,
in the category $\dg^*\Hopf\Lambda\Op^c$.

In~\S\ref{TripleCoresolution}, we also observe that the operations of the cosimplicial structure on $\QOp^{\bullet} = \Res_{op}^{\bullet}(\KOp)$
extend to the cofree cooperads~(\ref{DeformationComplexes:BicosimplicialComplex:FunctionSpaces:SimplicialFraming}),
which accordingly form a simplicial framing of this object $\QOp^{\bullet} = \Res_{op}^{\bullet}(\KOp)$
in the category $\cosimp\dg^*\Hopf\Lambda\Op^c$. We will use this extra observation later on.
For the moment, simply record that we can define our function spaces
as simplicial morphism sets:
\begin{equation}\label{DeformationComplexes:BicosimplicialComplex:FunctionSpaces:Expression}
\Map_{\dg^*\Hopf\Lambda\Op^c}(\ROp_k,\QOp^l)
:= \Mor_{\dg^*\Hopf\Lambda\Op^c}\bigl(\ROp_k,\FreeOp^c(\DGC_{op}^l(\KOp)\otimes\DGOmega^*(\Delta^{\bullet}))\bigr),
\end{equation}
where we take our simplicial framing~(\ref{DeformationComplexes:BicosimplicialComplex:FunctionSpaces:SimplicialFraming})
of the objects $\QOp^l = \Res_{op}^l(\KOp)$ in the category $\dg^*\Hopf\Lambda\Op^c$,
for any $l\in\NN$.
This construction actually works for any source object, and not only for the Hopf $\Lambda$-cooperads $\ROp_k = \Res^{com}_k(\PiOp)$.
\end{constr}

\begin{constr}[The reduction to function spaces of $\Lambda$-collections]\label{DeformationComplexes:BicosimplicialComplex:Reductions}
In what follows, we apply the function space of the previous paragraph to the Hopf $\Lambda$-cooperads $\ROp_k = \Res^{com}_k(\PiOp)$, $k\in\NN$,
which form the components of the cotriple resolution $\Res^{com}_{\bullet}(\PiOp)$
of the Hopf $\Lambda$-cooperad $\PiOp$.
In~\S\ref{CotripleResolution}, we explain that these Hopf $\Lambda$-cooperads $\ROp_k = \Res_k^{com}(\PiOp)$, $k\in\NN$,
are, by construction, identified with relative symmetric algebras $\ROp_k = \ComOp^c/\Sym(\DGC^{com}_k(\PiOp))$
on coaugmented $\Lambda$-cooperads $\DGC^{com}_k(\PiOp)\in\ComOp^c/\dg^*\Lambda\Op^c$
such that:
\begin{equation}\label{DeformationComplexes:BicosimplicialComplex:Reductions:FreeStructure}
\DGC^{com}_k(\PiOp) = \underbrace{\ComOp^c/\Sym\circ\dots\circ\ComOp^c/\Sym}_{k}(\PiOp),
\end{equation}
for any $k\in\NN$, and where we take a $k$-fold composite of the functor $\ComOp^c/\Sym(-)$
on coaugmented $\Lambda$-cooperads.
Recall that we also use the notation $\overline{\DGC}^{com}_k(\PiOp)$
for the coaugmentation coideal of this coaugmented $\Lambda$-cooperad $\DGC^{com}_k(\PiOp)$.

The adjunction relations of relative symmetric algebras and of cofree cooperads
imply that the morphism sets
of~\S\ref{DeformationComplexes:BicosimplicialComplex:FunctionSpaces}(\ref{DeformationComplexes:BicosimplicialComplex:FunctionSpaces:Expression})
admit a bunch of reductions, abutting to a morphism set of coaugmented $\Lambda$-collections,
and which we deduce from the commutative square of isomorphisms
of Figure~\ref{Fig:FunctionSpaceReduction}.
\begin{figure}[t]
\begin{equation*}
\xymatrix@!C=3cm{
*+<4pt>{\Mor_{\dg^*\Hopf\Lambda\Op^c}\bigl(\ComOp^c/\Sym(\DGC^{com}_k(\PiOp)),\FreeOp^c(\DGC_{op}^l(\KOp)\otimes\DGOmega^*(\Delta^{\bullet}))\bigr)}
\ar[]!DC-<2cm,0cm>;[dd]!UC-<2cm,0cm>_{\simeq}\ar@/^1em/[]!R;[dr]!UC+<2cm,0cm>^{\simeq} & \\
& *+<4pt>{\Mor_{\ComOp^c/\dg^*\Lambda\Op^c}\bigl(\DGC^{com}_k(\PiOp),\FreeOp^c(\DGC_{op}^l(\KOp)\otimes\DGOmega^*(\Delta^{\bullet}))\bigr)}
\ar[]!DC+<2cm,0cm>;[dd]!UC+<2cm,0cm>^{\simeq} \\
*+<4pt>{\Mor_{\dg^*\Hopf\Lambda\Seq_{>1}^c}\bigl(\overline{\ComOp}{}^c/\Sym(\overline{\DGC}^{com}_k(\PiOp)),\DGC_{op}^l(\KOp)\otimes\DGOmega^*(\Delta^{\bullet})\bigr)}
\ar@/_1em/[]!DC-<2cm,0cm>;[dr]!L_{\simeq} & \\
& *+<4pt>{\Mor_{\overline{\ComOp}{}^c/\dg^*\Lambda\Seq_{>1}^c}\bigl(\overline{\DGC}^{com}_k(\PiOp),\DGC_{op}^l(\KOp)\otimes\DGOmega^*(\Delta^{\bullet})\bigr)}
\ar@{.>}[]!DC+<2cm,0cm>;[d]!UC+<2cm,0cm>^{\simeq}\\
& *+<4pt>{\Mor_{\dg^*\Lambda\Seq_{>1}^c}\bigl(\DGI\DGC^{com}_k(\PiOp),\DGC_{op}^l(\KOp)\otimes\DGOmega^*(\Delta^{\bullet})\bigr)}
}
\end{equation*}
\caption{}\label{Fig:FunctionSpaceReduction}
\end{figure}

Recall that we use the notation $\II\Sym(-)$ for the augmentation ideal of the symmetric algebra (where we drop the unit).
In~\S\ref{CotripleResolution:AugmentedCase}, we observe that the objects $\overline{\DGC}^{com}_k(\PiOp)$
are identified with direct sums $\overline{\DGC}^{com}_k(\PiOp) = \overline{\ComOp}{}^c\oplus\DGI\DGC^{com}_k(\PiOp)$,
where $\DGI\DGC^{com}_k(\PiOp)$ is the (plain) $\Lambda$-collection
such that:
\begin{equation}\label{DeformationComplexes:BicosimplicialComplex:Reductions:Splitting}
\DGI\DGC^{com}_k(\PiOp) = \underbrace{\II\Sym\circ\dots\circ\II\Sym}_{k}(\IOp\PiOp),
\end{equation}
for any $k\in\NN$. (Recall that $\PiOp$ is assumed to be equipped with an augmentation over the commutative cooperad for simplicity.)
From this relation $\overline{\DGC}^{com}_k(\PiOp) = \overline{\ComOp}{}^c\oplus\DGI\DGC^{com}_k(\PiOp)$,
we readily deduce that the morphism set on the lower right-hand side corner of the square
of Figure~\ref{Fig:FunctionSpaceReduction}
admits a further reduction,
into a morphism set of plain (un-coaugmented) $\Lambda$-collections,
which we materialize by the dotted isomorphism
of the figure.

The cosimplicial objects $\Res_{op}^l(\KOp)^{\Delta^{\bullet}} := \FreeOp^c(\DGC_{op}^l(\KOp)\otimes\DGOmega^*(\Delta^{\bullet}))$,
of which we recall the definition in~\S\ref{DeformationComplexes:BicosimplicialComplex:FunctionSpaces},
clearly define a simplicial framing of the objects~$\Res_{op}^l(\KOp)$
in the category of coaugmented $\Lambda$-cooperads too
since our forgetful functor from Hopf $\Lambda$-cooperads to coaugmented $\Lambda$-cooperads
preserves weak-equivalences and fibrations.
Let also $\DGC_{op}^l(\KOp)^{\Delta^{\bullet}} := \DGC_{op}^l(\KOp)\otimes\DGOmega^*(\Delta^{\bullet})$
be the Hopf $\Lambda$-collection that occurs in this simplicial framing construction.
We readily see that these objects $\DGC_{op}^l(\KOp)^{\Delta^{\bullet}} = \DGC_{op}^l(\KOp)\otimes\DGOmega^*(\Delta^{\bullet})$
still define a simplicial framing of the objects $\DGC_{op}^l(\KOp)$
in the category of Hopf $\Lambda$-collections and in the category of (coaugmented) $\Lambda$-collections
similarly. We then use that the weak-equivalences and fibrations of the category of Hopf $\Lambda$-collections
are created arity-wise in the category of unitary commutative cochain dg-algebras (actually in the category of cochain graded dg-modules),
and that the tensor product $-\otimes\DGOmega^*(\Delta^{\bullet})$
defines a simplicial framing functor on this base model category (see~\cite[Theorem II.7.1.5]{FresseBook}).
We use similar arguments in the context of (coaugmented) $\Lambda$-collections.

We deduce from these observations that all morphism sets in the reduction diagram of Figure~\ref{Fig:FunctionSpaceReduction}
represent function spaces of the corresponding model categories.
We get in particular that the abutment of our reduction process represents a function space
on the category of (plain) $\Lambda$-collections:
\begin{equation}\label{DeformationComplexes:BicosimplicialComplex:Reductions:FunctionModule}
\Map_{\dg^*\Lambda\Seq_{>1}^c}(\MOp,\NOp)
:= \Mor_{\dg^*\Lambda\Seq_{>1}^c}\bigl(\MOp,\NOp\otimes\DGOmega^*(\Delta^{\bullet})\bigr),
\end{equation}
which we apply to the objects $\MOp = \DGI\DGC^{com}_k(\PiOp)$, $\NOp = \DGC_{op}^l(\KOp)$.
\end{constr}

We use the reductions of the previous paragraph to establish the following preliminary statement:

\begin{prop}\label{DeformationComplexes:BicosimplicialComplex:ModuleHomotopy}
The spaces
$X^{k l} = \Map_{\dg^*\Hopf\Lambda\Op^c}(\Res^{op}_k(\PiOp),\Res_{op}^l(\KOp))$
are isomorphic to simplicial modules, defined by the function spaces of $\Lambda$-collections $Y^{k l} = \Map_{\dg^*\Lambda\Seq_{>1}^c}(\DGI\DGC^{com}_k(\PiOp),\DGC_{op}^l(\KOp))$
which we obtain in the abutment of the relations of Figure~\ref{Fig:FunctionSpaceReduction}, and which inherit an obvious simplicial module structure,
for all $(k,l)\in\NN^2$.
Moreover, when we pass to homotopy, we have an isomorphism of graded modules
\begin{equation*}
\pi_*(X^{k l})\simeq\tau_*\Hom_{\gr\Lambda\Seq^c}(\DGI\DGC^{com}_k(\PiOp),\DGC_{op}^l(\DGH^*(\KOp))),
\end{equation*}
for any choice of base point in this function space, and where, on the right-hand side, we consider the truncation $\tau_*$
in lower degree $*\geq 0$ of the enriched hom-bifunctor of $\Lambda$-collections
in graded modules $\Hom_{\gr\Lambda\Seq^c}(-,-)$.
\end{prop}

\begin{proof}
We aim to determine the homotopy of the simplicial module such that $Y^{k l} = \Map_{\dg^*\Lambda\Seq_{>1}^c}(\MOp,\NOp)$,
where we set $\MOp = \DGI\DGC^{com}_k(\PiOp)$ and $\NOp = \DGC_{op}^l(\KOp)$,
for any $(k,l)\in\NN^2$.
We use that the homotopy of a simplicial module (and, more generally, of a simplicial group)
is given by the homology of the normalized complex associated to our object.
We have the relations
\begin{align}
\label{DeformationComplexes:BicosimplicialComplex:ModuleHomotopy:Defn}
\DGN_*\Map_{\dg^*\Lambda\Seq_{>1}^c}(\MOp,\NOp)
& \xrightarrow{=}\DGN_*\Mor_{\dg^*\Lambda\Seq_{>1}^c}(\MOp,\NOp\otimes\DGOmega^*(\Delta^{\bullet}))\\
\label{DeformationComplexes:BicosimplicialComplex:ModuleHomotopy:Integration}
& \xrightarrow{\sim}\DGN_*\Mor_{\dg^*\Lambda\Seq_{>1}^c}(\MOp,\NOp\otimes\DGN^*(\Delta^{\bullet}))\\
\label{DeformationComplexes:BicosimplicialComplex:ModuleHomotopy:Adjunction}
& \xrightarrow{\simeq}\DGN_*\Mor_{\dg_*\Mod}(\DGN_*(\Delta^{\bullet}),\tau_*\Hom_{\dg\Lambda\Seq^c}(\MOp,\NOp))\\
\label{DeformationComplexes:BicosimplicialComplex:ModuleHomotopy:DoldKan}
& \xrightarrow{\simeq}\tau_*\Hom_{\dg\Lambda\Seq^c}(\MOp,\NOp),
\end{align}
where: the identity of Equation (\ref{DeformationComplexes:BicosimplicialComplex:ModuleHomotopy:Defn})
is just our definition of our function space
in~\S\ref{DeformationComplexes:BicosimplicialComplex:Reductions}(\ref{DeformationComplexes:BicosimplicialComplex:Reductions:FunctionModule});
the morphism of Equation (\ref{DeformationComplexes:BicosimplicialComplex:ModuleHomotopy:Integration})
is yielded by the integration map $\rho: \DGOmega^*(\Delta^{\bullet})\rightarrow\DGN^*(\Delta^{\bullet})$
on the dg-algebra of piecewise linear forms (see~\S\ref{TripleCoresolution:SullivanDGAlgebra});
the isomorphism of Equation (\ref{DeformationComplexes:BicosimplicialComplex:ModuleHomotopy:Adjunction})
follows from obvious duality and adjunction relations;
while the isomorphism of Equation (\ref{DeformationComplexes:BicosimplicialComplex:ModuleHomotopy:DoldKan})
follows from the Dold-Kan correspondence.
Just observe that the integration map $\rho: \DGOmega^*(\Delta^{\bullet})\rightarrow\DGN^*(\Delta^{\bullet})$
induces a weak-equivalence of simplicial framings
when we take the tensor product of this map
with our $\Lambda$-collection:
\begin{equation}\label{DeformationComplexes:BicosimplicialComplex:ModuleHomotopy:FramingMap}
\rho_*: \NOp\otimes\DGOmega^*(\Delta^{\bullet})\xrightarrow{\sim}\NOp\otimes\DGN^*(\Delta^{\bullet}),
\end{equation}
and the weak-equivalence in our relation (\ref{DeformationComplexes:BicosimplicialComplex:ModuleHomotopy:Integration})
simply follows from the assertion
that function spaces do not depend on the choice of a particular simplicial framing.

If we recap our definitions, then we easily get that our correspondence~(\ref{DeformationComplexes:BicosimplicialComplex:ModuleHomotopy:Defn}-\ref{DeformationComplexes:BicosimplicialComplex:ModuleHomotopy:DoldKan})
maps any morphism $f: \MOp\rightarrow\NOp\otimes\DGOmega^*(\Delta^m)$
in the category of $\Lambda$-collections
to the homomorphism $\rho_*(f)\in\Hom_{\dg\Lambda\Seq^c}(\MOp,\NOp)_m$
such that
\begin{equation}\label{DeformationComplexes:BicosimplicialComplex:ModuleHomotopy:IntegrationMap}
\rho_*(f)(\xi) = \int_{\Delta^m} f(\xi)
\end{equation}
for any $\xi\in\MOp(r)$, where we now consider the mapping $\int_{\Delta^m}: \NOp(r)\otimes\DGOmega^*(\Delta^m)\rightarrow\NOp(r)$
given by the integration of forms $\omega\in\DGOmega^*(\Delta^m)$
over the simplex $\Delta^m$ (we just assume that this mapping vanishes when $\deg^*(\omega)\not=m$).
The other way round, to a homomorphism $g: \MOp\rightarrow\NOp$
of (lower) degree $\deg(g) = m$,
we can associate the morphism of $\Lambda$-collections $g_{\sharp}: \MOp\rightarrow\NOp\otimes\DGOmega^*(\Delta^m)$
such that
\begin{equation}\label{DeformationComplexes:BicosimplicialComplex:ModuleHomotopy:ConverseMap}
g_{\sharp}(\xi) = g(\xi)\otimes(dx_1\cdot\ldots\cdot dx_m),
\end{equation}
for any $\xi\in\MOp(r)$. This morphism clearly represents a pre-image of the homomorphism $g\in\Hom_{\dg\Lambda\Seq^c}(\MOp,\NOp)$
in the normalized complex $\DGN_*\Map_{\dg^*\Lambda\Seq_{>1}^c}(\MOp,\NOp)$ (up to a $1/n!$ factor).

We now have a K\"unneth morphism
\begin{equation}\label{DeformationComplexes:BicosimplicialComplex:ModuleHomotopy:KuennethMorphism}
\DGH_*\Hom_{\dg\Lambda\Seq^c}(\MOp,\NOp)\xrightarrow{\simeq}\Hom_{\dg\Lambda\Seq^c}(\DGH_*(\MOp),\DGH_*(\NOp))
\end{equation}
which is an isomorphism since the $\Lambda$-collection $\MOp = \DGI\DGC^{com}_k(\PiOp)$
is freely generated by a symmetric collection (see~\S\ref{CotripleResolution:AugmentedCase}).
We moreover have the identity $\DGH_*(\MOp) = \MOp = \DGI\DGC^{com}_k(\PiOp)$
because $\PiOp$ is equipped with a trivial differential.
We also have the relation $\NOp = \DGC_{op}^l(\KOp)\Rightarrow\DGH_*(\NOp) = \DGC_{op}^l(\DGH^*(\KOp))$
by the K\"unneth formula.
Eventually, we obtain the isomorphism of the proposition
when we combine these relations
with the isomorphisms
of Figure~\ref{Fig:FunctionSpaceReduction}.
\end{proof}

In degree $*=0$ and in bicosimplicial dimension $(k,l) = (0,0)$, the relation of this proposition
implies that we have a bijection $\pi_0(X^{0 0})\simeq\Mor_{\gr\Lambda\Seq_{>1}^c}(\IOp\PiOp,\DGH^*(\overline{\KOp}))$,
which we can prolong to
\begin{equation*}
\pi_0(X^{0 0})\simeq\Mor_{\overline{\ComOp}{}^c/\gr\Lambda\Seq_{>1}^c}(\overline{\PiOp},\DGH^*(\overline{\KOp}))
\end{equation*}
by going back to the splitting formula $\overline{\PiOp} = \overline{\ComOp}{}^c\oplus\IOp\PiOp$.
We aim to determine which morphisms in this set correspond to degree $0$ cocycles in the cohomotopy of our bicosimplicial homotopy class set.
We still use that any morphism of $\Lambda$-collections $\overline{\chi}: \overline{\PiOp}\rightarrow\DGH^*(\overline{\KOp})$
has a unique extension $\chi: \PiOp\rightarrow\DGH^*(\KOp)$
which preserves the counit of the cooperads $\PiOp$
and $\DGH^*(\KOp)$ (though this map $\chi$ is not a morphism of cooperads in general).
We establish the following result:

\begin{thm}\label{DeformationComplexes:BicosimplicialComplex:DegreeZeroCocycles}
The correspondence of Proposition~\ref{DeformationComplexes:BicosimplicialComplex:ModuleHomotopy}
induces a bijection:
\begin{equation*}
\pi^0\pi_0(X)\xrightarrow{\simeq}\Mor_{\gr^*\Hopf\Lambda\Op^c}(\PiOp,\DGH^*(\KOp))
\end{equation*}
where we consider the set of cohomotopy cocycles $z\in\pi^0\pi_0(X)$
in the diagonal
of the bicosimplicial set $\pi_0(X) = \Mor_{\dg^*\Hopf\Lambda\Op^c}(\Res^{op}_{\bullet}(\PiOp),\Res_{op}^{\bullet}(\DGH^*(\KOp)))$
on the one hand, and the set of morphisms of Hopf $\Lambda$-cooperads in graded modules $\chi: \PiOp\rightarrow\DGH^*(\KOp)$
on the other hand.
\end{thm}

\begin{proof}
The set $\pi^0\pi_0(X)$ consists of the homotopy classes $z\in\pi_0(X)$
such that $d^0(z) = d^1(z)$ in $\pi_0(X^{1 1})$,
where we also consider the diagonal coface operators $d^0,d^1: \pi_0(X^{0 0})\rightarrow\pi_0(X^{1 1})$
on the bicosimplicial set $\pi_0(X^{\bullet\,\bullet})$.

The theorem follows from a straightforward inspection of the correspondence of Proposition~\ref{DeformationComplexes:BicosimplicialComplex:ModuleHomotopy}
and of the definition, in terms of monadic and comonadic adjunctions,
of the simplicial and cosimplicial structure
of our resolutions.
Note simply that the K\"unneth morphism
in the proof of Proposition~\ref{DeformationComplexes:BicosimplicialComplex:ModuleHomotopy}
reduces to the obvious functoriality mapping $\Mor_{\dg^*\Lambda\Mod}(\MOp,\NOp)\rightarrow\Mor_{\gr^*\Lambda\Mod}(\DGH^*(\MOp),\DGH^*(\NOp))$
in homotopical degree $*=0$.
\end{proof}

We assume, from now on, that we have a morphism of Hopf $\Lambda$-cooperads $\chi: \PiOp\rightarrow\DGH^*(\KOp)$.
We also set $\HOp = \DGH^*(\KOp)$ for short.

We fix a morphism $\phi: \Res^{com}_0(\PiOp)\rightarrow\Res_{op}^0(\KOp)$
whose class $[\phi]$
in the homotopy of the function space $\Map_{\dg^*\Hopf\Lambda\Op^c}(\Res^{com}_{\bullet}(\PiOp),\Res_{op}^{\bullet}(\KOp))$
represents the cohomotopy cocycle $z$
corresponding to $\chi$.
By definition of our correspondence in \S\ref{DeformationComplexes:BicosimplicialComplex:Reductions},
Proposition~\ref{DeformationComplexes:BicosimplicialComplex:ModuleHomotopy},
and Theorem~\ref{DeformationComplexes:BicosimplicialComplex:DegreeZeroCocycles},
this morphism $\phi$ is actually identified with a morphism of Hopf $\Lambda$-cooperads $\phi = \phi_f$
associated to a morphism of coaugmented $\Lambda$-collections $f: \overline{\PiOp}\rightarrow\overline{\KOp}$
that corresponds to $\chi$
in cohomology.
We form the composite morphism:
\begin{equation*}
\Res^{com}_k(\PiOp)\xrightarrow{d_1\cdots d_k}\Res^{com}_0(\PiOp)
\xrightarrow{\phi}\Res_{op}^0(\KOp)\xrightarrow{d^l\cdots d^1}\Res_{op}^l(\PiOp),
\end{equation*}
and we provide the space $X^{k l} = \Map_{\dg^*\Hopf\Lambda\Op^c}(\Res^{com}_k(\PiOp),\Res_{op}^l(\KOp))$
with this morphism $\phi^0 = (d^l\cdots d^1)\cdot\phi\cdot(d_1\cdots d_k)$
as base point, for each pair $(k,l)\in\NN^2$.
We aim to determine the cohomotopy groups $\pi^s\pi_t(X,\phi^0)$, where we consider the homotopy of the spaces $\Diag(X)^k = X^{k k}$
at our base point $\phi^0\in X^{k k}$, for all $k\in\NN$, and for any degree $t>0$.
We use that these cohomotopy groups are given by the cohomology of a conormalized complex $\DGN^*\pi_t(X,\phi^0)$
associated to the cosimplicial object $E^{\bullet} = \pi_t(\Diag(X^{\bullet\,\bullet}),\phi^0)$ (see~\cite{BousfieldKan}).

We consider modules of biderivations on the cotriple resolution $\ROp_{\bullet} = \Res^{com}_{\bullet}(\PiOp)$
of the object $\PiOp$ in the category Hopf $\Lambda$-cooperads in graded modules
and with values in the triple coresolution $\QOp^{\bullet} = \Res_{op}^{\bullet}(\HOp)$
of the cohomology cooperad $\HOp = \DGH^*(\KOp)$.
We just take the composite of our cohomology morphism $\chi$ with the augmentation of the cotriple resolution
and with the coaugmentation of the triple coresolution
\begin{equation*}
\Res^{com}_k(\PiOp)\xrightarrow{\epsilon}\PiOp\xrightarrow{\chi}\DGH^*(\KOp)\xrightarrow{\eta}\Res_{op}^l(\HOp),
\end{equation*}
in order to get a morphism of Hopf $\Lambda$-cooperad $\chi^0: \Res^{com}_k(\PiOp)\rightarrow\Res_{op}^l(\HOp)$
and to give a sense to this graded module of biderivations
\begin{equation*}
B^{k l} = B^{k l}(\PiOp,\HOp) = \BiDer_{\gr^*\Hopf\Lambda\Op^c}(\Res^{com}_k(\PiOp),\Res_{op}^l(\HOp)),
\end{equation*}
for each pair $(k,l)\in\NN^2$. These objects still form a bicosimplicial module
and we consider the conormalized complex of the associated
diagonal object.

Recall that the conormalized complex $\DGN^*(A)$ of a cosimplicial module (or group) $A = A^{\bullet}$
is generally defined by
$\DGN^k(A) = \bigcap_{j = 0}^{k-1}\ker(s^j: A^k\rightarrow A^{k-1})$
in each degree $k\in\NN$,
and has the alternate sum of coface operators $\partial = \sum_{i=0}^{k} (-1)^i d^i$
as differential.

We have the following statement:

\begin{thm}\label{DeformationComplexes:BicosimplicialComplex:MainResult}
Let $\HOp = \DGH^*(\KOp)$. We have an isomorphism of cochain complexes
\begin{equation*}
\DGN^*\pi_*(X,\phi^0)\xrightarrow{\simeq}\DGN^*\BiDer_{\dg^*\Hopf\Lambda\Op^c}(\Res^{com}_{\bullet}(\PiOp),\Res_{op}^{\bullet}(\HOp))
\end{equation*}
where we consider:
\begin{itemize}
\item
on the one hand, the conormalized complex of the (diagonal complex of the) bicosimplicial complex of homotopy groups
\begin{equation*}
\qquad\qquad T^{k l} = \pi_*(X^{k l},\phi^0) = \pi_*\bigl(\Map_{\dg^*\Hopf\Lambda\Op^c}(\Res^{com}_k(\PiOp),\Res_{op}^l(\KOp)),\phi^0\bigr)
\end{equation*}
taken at our base point $\phi^0$;
\item
on the other hand, the conormalized complex of the (diagonal complex of the) bicosimplicial complex of biderivations
\begin{equation*}
B^{k l} = B^{k l}(\PiOp,\HOp) = \BiDer_{\dg^*\Hopf\Lambda\Op^c}\bigl(\Res^{com}_k(\PiOp),\Res_{op}^l(\HOp)\bigr)
\end{equation*}
taken with respect to our cohomology morphism $\chi^0 = \eta\chi\epsilon$.
\end{itemize}
\end{thm}

\begin{proof}
Proposition~\ref{DeformationComplexes:Biderivations:ModuleReduction}
and Proposition~\ref{DeformationComplexes:BicosimplicialComplex:ModuleHomotopy}
give one-to-one correspondences:
\begin{equation*}
T^{k l}\simeq\Hom_{\gr\Lambda\Seq^c}\bigl(\DGI\DGC^{com}_k(\PiOp),\DGC_{op}^l(\DGH^*(\KOp))\bigr)\simeq B^{k l},
\end{equation*}
and our main purpose is to check that these bijections carry the degeneracy and coface operators of our bicosimplicial structure
on the complex of homotopy groups $T = T^{\bullet\,\bullet}$
to the degeneracy and coface operators of our bicosimplicial structure
on our complex of biderivations $B = B^{\bullet\,\bullet}$.
This is done by going back to the explicit definition of our bijections.
We can actually check this correspondence in the horizontal and vertical directions separately, by using one half of the reduction paths
of Figure~\ref{Fig:FunctionSpaceReduction} in each case. The case of the codegeneracies $s^j$ and of the coface operators $d^i$
such that $i>0$ is immediate because these operations are induced by morphisms
of the generating objects $\DGI\DGC^{com}_{\bullet}(\PiOp)$
and $\DGC_{op}^{\bullet}(\KOp)$.
We therefore focus on the case of the $0$-coface operators.

We also check that the equality between these coface operators
holds in the dg-module $\Hom_{\gr\Lambda\Seq^c}\bigl(\DGI\DGC^{com}_k(\PiOp),\DGC_{op}^l(\KOp)\bigr)$
before we perform the K\"unneth isomorphism.
We then consider a morphism of $\Lambda$-collections $g: \DGI\DGC^{com}_k(\PiOp)\rightarrow\DGC_{op}^l(\KOp)\otimes\DGOmega^*(\Delta^t)$
satisfying $g(-) = h(-)\otimes(dx_1\cdots dx_t)$
and which corresponds to an element of this dg-hom $h\in\Hom_{\gr\Lambda\Seq^c}\bigl(\DGI\DGC^{com}_k(\PiOp),\DGC_{op}^l(\KOp)\bigr)$
(as in the proof of Proposition~\ref{DeformationComplexes:BicosimplicialComplex:ModuleHomotopy}).

We first note that, when we base our homotopy groups at the morphism $\phi^0 = \phi_f^0$,
we have to carry such a morphism of $\Lambda$-collections $g: \DGI\DGC^{com}_k(\PiOp)\rightarrow\DGC_{op}^l(\KOp)\otimes\DGOmega^*(\Delta^t)$,
which satisfies the face relations $d_i(g) = 0$ for all $i\geq 0$,
to the morphism of Hopf $\Lambda$-cooperads $\phi_{f+g}: \Res^{com}_k(\PiOp)\rightarrow\Res_{op}^l(\KOp)^{\Delta^t}$
associated to the translated morphism $f+g: \DGI\DGC^{com}_k(\PiOp)\rightarrow\DGC_{op}^l(\KOp)\otimes\DGOmega^*(\Delta^{\bullet})$
in order to get a simplicial cycle that satisfies the face relations $d_i(\phi_{f+g}) = \phi_f^0$
for all $i\geq 0$
in the function space $\Map_{\dg^*\Hopf\Lambda\Op^c}(\Res^{com}_k(\PiOp),\Res_{op}^l(\KOp))$.
We have a similar observation when we carry out only one half of the correspondences
of Figure~\ref{Fig:FunctionSpaceReduction},
and we only deal with a morphism $\phi_{f+g}: \overline{\ComOp}{}^c/\Sym(\overline{\DGC}^{com}_k(\PiOp))\rightarrow\DGC_{op}^l(\KOp)\otimes\DGOmega^*(\Delta^{\bullet})$
in the category of Hopf $\Lambda$-collections (respectively, with a morphism $\phi_{f+g}: \DGC^{com}_k(\PiOp)\rightarrow\FreeOp^c(\DGC_{op}^l(\KOp)\otimes\DGOmega^*(\Delta^{\bullet}))$
in the category of coaugmented $\Lambda$-cooperads).

We easily retrieve the expression of the derivation
$\theta_g: \overline{\ComOp}{}^c/\Sym(\overline{\DGC}^{com}_k(\PiOp))\rightarrow\DGC_{op}^l(\KOp)\otimes\DGOmega^*(\Delta^{\bullet})$
associated to $g$ in the expansion of the morphism
\begin{equation*}
\DGI\DGC^{com}_{k+1}(\PiOp)\subset\overline{\ComOp}{}^c/\Sym(\overline{\DGC}^{com}_{k+1}(\PiOp))
\xrightarrow{d_0}\overline{\ComOp}{}^c/\Sym(\overline{\DGC}^{com}_k(\PiOp))
\xrightarrow{\phi_{f+g}}\DGC_{op}^l(\KOp)\otimes\DGOmega^*(\Delta^{\bullet})
\end{equation*}
after observing that the terms involving more than one $g$ factors vanish in the outcome of this process. This verification gives
the correspondence of our coface operators in the algebraic direction,
and we address the case of the operadic direction
similarly.
\end{proof}

\subsection{The deformation bicomplex of Hopf cooperads}\label{DeformationComplexes:DGComplex}
The purpose of this subsection is to prove that the conormalized complex of the bicosimplicial module of biderivations of the previous subsection
\begin{equation*}
B^{\bullet\,\bullet} = B^{\bullet\,\bullet}(\PiOp,\HOp) = \BiDer_{\dg^*\Hopf\Lambda\Op^c}(\Res^{com}_{\bullet}(\PiOp),\Res_{op}^{\bullet}(\HOp)),
\end{equation*}
where we set $\HOp = \DGH^*(\KOp)$, is weakly-equivalent to a deformation bicomplex of Hopf $\Lambda$-cooperads:
\begin{equation*}
D^{* *} = D^{* *}(\PiOp,\HOp) = \BiDef^{* *}_{\dg^*\Hopf\Lambda\Op^c}(\PiOp,\HOp).
\end{equation*}
This deformation bicomplex of Hopf $\Lambda$-cooperads is an operadic counterpart
of the classical Gerstenhaber-Schack complex
of bialgebras~\cite{GerstenhaberSchack} (see also~\cite{FoxMarkl}).
In short, we combine a cooperadic cobar construction, which governs deformations of cooperad morphisms,
with (an extension to Hopf $\Lambda$-cooperads of) the Harrison complex
of commutative algebras, which governs deformations
of commutative algebra morphism.

In the previous subsection, we assumed for simplicity that $\PiOp$ is a Hopf $\Lambda$-cooperad in the category of graded modules,
or equivalently, that $\PiOp$ is a Hopf $\Lambda$-cooperad in dg-modules equipped with a trivial differential.
This simplifying assumption is not necessary in this subsection, and our subsequent definitions make sense without change when the object $\PiOp\in\dg^*\Hopf\Lambda\Op^c$
is equipped with a non-trivial differential.
But, on the other hand, we still prefer to assume that $\PiOp$ is equipped with an augmentation over the commutative cooperad $\ComOp^c$ (in order to simplify our constructions),
and we keep the notation $\IOp\PiOp$ for the collection of augmentation ideals of the algebras $\PiOp(r)$, $r>1$.
Then we require that this object $\IOp\PiOp$ is free as a $\Lambda$-collection $\IOp\PiOp = \Lambda\otimes_{\Sigma}\SOp\PiOp$ (as in the previous subsection).
We use this assumption in our proof that our deformation bicomplex is weakly-equivalent to the conormalized complex of the bicosimplicial module
of biderivations of the previous subsection.

We also assume that $\HOp$ is a general Hopf $\Lambda$-cooperad in the category of dg-modules (regardless of our initial obstruction problem).
We just take $\HOp = \DGH^*(\KOp)$ when we tackle the applications of our constructions
to our initial obstruction problem.

In a first step, we revisit the definition of the cobar complex of cooperads and the definition of the Harrison complex.
In these constructions, we consider a suspension functor $\DGSigma: \dg\Mod\rightarrow\dg\Mod$
defined, on any dg-module $C\in\dg\Mod$, by the tensor product $\DGSigma C = \kk\ecell_1\otimes C$,
where $\ecell_1$ represents a homogeneous element of (lower) degree $\deg(\ecell_1) = 1$
equipped with a trivial differential $\delta(\ecell_1) = 0$.
We also deal with the inverse desuspension operation $\DGSigma^{-1}: \dg\Mod\rightarrow\dg\Mod$,
which is defined by a similar tensor product $\DGSigma^{-1} C = \kk\ecell^1\otimes C$,
but where we now assume $\deg^*(\ecell^1) = 1\Leftrightarrow\deg(\ecell^1) = -1$.

\begin{rem}[The cobar construction of cooperads]\label{DeformationComplexes:DGComplex:CobarConstruction}
Briefly recall that the operadic cobar construction $\BB_{op}^c(\COp)$
of a cooperad $\COp$
is defined by the expression (see~\cite{GetzlerJones}):
\begin{equation}\label{DeformationComplexes:DGComplex:CobarConstruction:QuasiFreeDefinition}
\BB_{op}^c(\COp) = (\FreeOp(\DGSigma^{-1}\overline{\COp}),\partial'),
\end{equation}
where we consider the free operad $\FreeOp(-)$ on the arity-wise desuspension $\DGSigma^{-1}$
of the coaugmentation coideal of our cooperad $\overline{\COp}$
together with a twisting differential $\partial': \FreeOp(\DGSigma^{-1}\overline{\COp})\rightarrow\FreeOp(\DGSigma^{-1}\overline{\COp})$
which we determine by the composition coproduct of $\COp$ (we also refer to~\cite{FressePartition}
for a detailed survey of this construction).

In what follows, we consider a cochain complex such that:
\begin{equation}\label{DeformationComplexes:DGComplex:CobarConstruction:Complex}
\DGB_{op}^*(\COp) = \DGSigma\overline{\BB}{}_{op}^c(\COp),
\end{equation}
where $\overline{\BB}{}_{op}^c(\COp)$ denotes the augmentation ideal of the cobar operad $\BB_{op}^c(\COp)$,
and $\DGSigma$ is our suspension functor on dg-modules, which we apply arity-wise
to this symmetric collection $\overline{\BB}{}_{op}^c(\COp)$.
This object $\DGB_{op}^*(\COp)$ naturally forms a cochain complex in the category of symmetric collections in dg-modules, with a cochain grading,
referred to by the superscript $*$ in our notation $\DGB_{op}^*(\COp)$,
which is determined by the decomposition of the (augmentation ideal of the) free operad $\overline{\FreeOp}(-) = \bigoplus_{s\geq 1}\FreeOp_s(-)$
into components of homogeneous weight $\FreeOp_s(-)\subset\FreeOp(-)$, $s\geq 1$.
We explicitly set $\DGB_{op}^l(\COp) = \DGSigma\FreeOp_{l+1}(\DGSigma^{-1}\overline{\COp})$ for each $l\geq 0$,
and each of these homogeneous components of our complex naturally forms an object
of the category of symmetric collections in dg-modules by construction,
while the twisting differential $\partial'$
is equivalent to a homomorphism $\partial': \DGB_{op}^*(\COp)\rightarrow\DGB_{op}^{*+1}(\COp)$
that raises the cochain grading of our object by one (and decreases the internal lower grading by one).
If necessary, then we specify the natural (internal) grading of the objects $\DGB_{op}^l(\COp)$
by an external subscript in our notation $\DGB_{op}^l(\COp) = \DGB_{op}^l(\COp)_*$.

We now assume that $\COp$ is a coaugmented $\Lambda$-cooperad. The object $\DGB_{op}^*(\COp)$ inherits natural corestriction operators
in this situation, and we also have a coaugmentation $\epsilon_*: \overline{\ComOp}{}^c\rightarrow\DGB_{op}^*(\COp)$
so that $\DGB_{op}^*(\COp)$ forms a coaugmented $\Lambda$-collection. We refer to~\cite[Proposition C.2.18]{FresseBook}
for an explicit definition of this extra structure (in the dual context of the bar construction
of an augmented $\Lambda$-operad).
\end{rem}

\begin{constr}[The deformation complex of cooperads with coefficients in a bicomodule]\label{DeformationComplexes:DGComplex:CooperadBicomodule}
We define the deformation complex $\CoDef_{\dg^*\Lambda\Op^c}^*(\MOp,\COp)$ of the cooperad $\COp$
with coefficients in a bicomodule $\MOp$
as the dg-module of homomorphisms:
\begin{equation}\label{DeformationComplexes:DGComplex:CooperadBicomodule:Construction}
\CoDef_{\dg^*\Lambda\Op^c}^*(\MOp,\COp) = (\Hom_{\dg\Lambda\Seq^c}(\MOp,\DGB_{op}^*(\COp)),\partial'')
\end{equation}
where we take the $\Lambda$-collection underlying $\MOp$ as source object, the reduced cobar complex $\DGB_{op}^*(\COp)$
as target object and an extra twisting differential $\partial''$,
which we determine by the coaction of $\COp$ on $\MOp$.

We proceed as follows.
We use that the coaction of the cooperad $\COp$
on the bicomodule $\MOp$
can be represented by tree-wise coproducts $\rho: \MOp(r)\rightarrow\FreeOp_{\gammatree}(\MOp,\COp)$,
where $\FreeOp_{\gammatree}(\MOp,\COp)$ denotes a tensor product
of the objects $\COp$ and $\MOp$
over a tree $\gammatree$ with two vertices $x = u,v$, and $r$ ingoing edges indexed by $i = 1,\dots,r$.
The structure of such a tree can be determined by giving a partition $\{1,\dots,r\} = \{i_1,\dots,\widehat{i_e},\dots,i_k\}\amalg\{j_1,\dots,j_l\}$,
where $\{i_1,\dots,i_e,\dots,i_k\}$ serves to index the ingoing edges of the lower vertex of the tree $u$,
while $\{j_1,\dots,j_l\}$ serves to index the ingoing edges of the upper vertex $v$. The index $i_e$ is a dummy variable
which we associate to the inner edge of the tree $v\rightarrow u$
between the vertices $u,v$. We then set $\FreeOp_{\gammatree}(\MOp,\COp) = \MOp(k)\otimes\COp(l)\oplus\COp(k)\otimes\MOp(l)$
for any such tree $\gammatree$, and we define the tree-wise coproduct $\rho: \MOp(r)\rightarrow\FreeOp_{\gammatree}(\MOp,\COp)$
by the two-sided coproduct operation of~\S\ref{DeformationComplexes:Biderivations:Bicomodules}.
We also consider the sum $\FreeOp_2(\MOp,\COp)' = \bigoplus_{\gammatree}\FreeOp_{\gammatree}(\MOp,\COp)$
running over isomorphism classes of trees with two vertices $\gammatree$
and the operation $\rho: \MOp(r)\rightarrow\FreeOp_2(\MOp,\COp)$
that collects these two-fold tree-wise coproducts. Recall that the homogeneous component of weight two of the free operad $\FreeOp_2(-)$
is defined by the same sum of two-fold tree-wise tensor products $\FreeOp_2(-) = \bigoplus_{\gammatree}\FreeOp_{\gammatree}(-)$.

We now define the twisting differential $\partial''(h)$ of a homomorphism $h: \MOp\rightarrow\DGB_{op}^*(\COp)$
by the composite:
\begin{equation}\label{DeformationComplexes:DGComplex:CooperadBicomodule:TwistingDifferential}
\MOp\xrightarrow{\rho}\FreeOp_2(\MOp,\COp)'\xrightarrow{\partial_{\iota}(h)}\DGSigma\FreeOp_2(\BB^c_{op}(\COp))
\xrightarrow{\lambda}\DGSigma\overline{\BB}^c_{op}(\COp) = \DGB_{op}^*(\COp),
\end{equation}
where:
\begin{itemize}
\item
we perform our coproduct operation $\rho: \MOp\rightarrow\FreeOp_2(\MOp,\COp)$ first;
\item
we consider the map $\partial_{\iota}(h): \FreeOp_2(\MOp,\COp)'\rightarrow\DGSigma\FreeOp_2(\BB^c_{op}(\COp))$
given by the application
of our homomorphism
$\MOp\xrightarrow{h}\DGB_{op}^*(\COp) = \DGSigma\BB^c_{op}(\COp)\xrightarrow{\simeq}\BB^c_{op}(\COp)$
on the $\MOp$ factors of our tree-wise coproducts,
together with the universal morphism of the free operad
$\overline{\COp}\xrightarrow{\ecell^1\otimes-}\DGSigma^{-1}\overline{\COp}\xrightarrow{\iota}\FreeOp(\DGSigma^{-1}\overline{\COp}) = \BB^c_{op}(\COp)$
on the other factor;
\item
and we take the universal morphism $\FreeOp_2(\BB^c_{op}(\COp))\subset\FreeOp(\BB^c_{op}(\COp))\xrightarrow{\lambda}\BB^c_{op}(\COp)$
determined by the composition structure of the operad $\BB^c_{op}(\COp)$
afterwards.
\end{itemize}
We easily check that this mapping $\partial'': h\mapsto\partial''(h)$
fulfills the equation of a twisting differential
in the dg-module $\Hom_{\dg\Lambda\Seq^c}(\MOp,\DGB_{op}^*(\COp))$.

The object $\CoDef_{\dg^*\Lambda\Op^c}^*(\MOp,\COp)$ actually forms a cochain complex of dg-mo\-du\-les,
with an internal grading given by the natural grading of our dg-module
of homomorphisms,
a cochain grading such that
\begin{equation}\label{DeformationComplexes:DGComplex:CooperadBicomodule:Components}
\CoDef_{\dg^*\Lambda\Op^c}^l(\MOp,\COp) = \Hom_{\dg\Lambda\Seq^c}(\MOp,\DGB_{op}^l(\COp))
\end{equation}
for any $l\in\NN$,
and a total twisting differential $\partial = \partial'+\partial''$
that raises this cochain grading by one.
In what follows, we also use an extra subscript $v$ in order to distinguish this differential $\partial_v = \partial'_v+\partial''_v$
from an algebraic twisting differential which occurs in our deformation bicomplex of Hopf cooperads.

Let us mention that the deformation complex defined in this paragraph represents a generalization (in the context of cooperads)
of the operadic deformation complexes studied by the second author in~\cite{WillwacherGraphs}.
To be specific, in comparison with this reference, we consider general bicomodules (and not only cooperads)
as coefficient objects, and we take care of the extra corestriction operators
associated to our objects.
\end{constr}

We have the following structure statement:

\begin{prop}\label{DeformationComplexes:DGComplex:CobarModule}
In the case of a Hopf $\Lambda$-cooperad $\COp = \HOp$, the cooperadic cobar complex $\DGB_{op}^*(\HOp) = \DGSigma\overline{\BB}^c_{op}(\HOp)$
is naturally a complex of modules over the Hopf $\Lambda$-collection $\overline{\HOp}$.
\end{prop}

\begin{proof}
We use the expansion of the free operad $\FreeOp(-) = \bigoplus_{\ttree}\FreeOp_{\ttree}(-)$,
where the sum runs over (isomorphism classes of) trees $\ttree$,
and the summands $\FreeOp_{\ttree}(-)$
are tree-wise tensor products
associated to each object of the category of trees $\ttree$.
We define our left product operations $\lambda: \HOp(r)\otimes\DGB_{op}^*(\HOp)(r)\rightarrow\DGB_{op}^*(\HOp)(r)$
on each term $\DGSigma\FreeOp_{\ttree}(\DGSigma^{-1}\overline{\HOp})$
of this expansion in the complex $\DGB_{op}^*(\HOp) = \DGSigma\FreeOp(\DGSigma^{-1}\overline{\HOp})$.

We proceed as follows.
We perform a tree-wise coproduct operation $\rho_{\ttree}: \HOp(r)\rightarrow\FreeOp_{\ttree}(\HOp)$,
determined by the composition structure of our cooperad $\HOp$,
and we consider the morphism
$\mu_*: \FreeOp_{\ttree}(\HOp)\otimes\FreeOp_{\ttree}(\DGSigma^{-1}\overline{\HOp})\rightarrow\FreeOp_{\ttree}(\DGSigma^{-1}\overline{\HOp})$
given by the multiplication operations
$\HOp(r_v)\otimes\DGSigma^{-1}\HOp(r_v)\simeq\DGSigma^{-1}\HOp(r_v)\otimes\HOp(r_v)\rightarrow\DGSigma^{-1}\HOp(r_v)$
on the factors of this tree-wise tensor product, where $v$ runs over the vertices of the tree $\ttree$.
We can use the same definition to determine right product operations $\rho: \DGB_{op}^*(\HOp)(r)\otimes\HOp(r)\rightarrow\DGB_{op}^*(\HOp)(r)$ (which obviously agree
with the left product operations when we apply a symmetry isomorphism on the source).
The distribution relation between the commutative algebra product and the composition coproducts
in a Hopf cooperad implies that these product operations on $\DGB_{op}^*(\HOp)$
intertwine the twisting differential of the cobar construction.

We easily check that our product operations preserve the corestriction operators too, as well as the coaugmentation morphisms,
and the proof of our statement is therefore complete.
\end{proof}

\begin{rem}[Harrison complexes]\label{DeformationComplexes:DGComplex:HarrisonConstruction}
The classical Harrison chain complex with trivial coefficients $\BB_{com}(A)$
of an augmented unitary commutative algebra $A$
can be defined by the expression:
\begin{equation}\label{DeformationComplexes:DGComplex:HarrisonConstruction:QuasiCofreeDefinition}
\BB_{com}(A) = (\LLie^c(\DGSigma IA),\partial'),
\end{equation}
where we consider the cofree Lie coalgebra $\LLie^c(-)$ on the suspension $\DGSigma$
of the augmentation ideal of our algebra $IA$
and a twisting differential $\partial': \LLie^c(\DGSigma IA)\rightarrow\LLie^c(\DGSigma IA)$
which we determine by the product of $A$. This interpretation of the Harrison complex arises
from the application of the Koszul duality of operads
to commutative algebras~\cite{GinzburgKapranov}.

In what follows, we consider a chain complex:
\begin{equation}\label{DeformationComplexes:DGComplex:HarrisonConstruction:ChainComplex}
\DGB^{com}_*(A) = \DGSigma^{-1}\BB_{com}(A)
\end{equation}
which we form by taking the desuspension $\DGSigma^{-1}$ of this quasi-cofree Lie coalgebra $\BB_{com}(A)$.
In what follows, we use the expression of the Harrison chain complex
to refer to this desuspended dg-module
rather than to our initial quasi-cofree Lie coalgebra~(\ref{DeformationComplexes:DGComplex:HarrisonConstruction:QuasiCofreeDefinition}).
This object $\DGB^{com}_*(A)$ naturally forms a chain complex of dg-modules, with a chain grading,
referred to by the subscript $*$ in our notation $\DGB^{com}_*(A)$,
which is determined by the decomposition of the cofree Lie coalgebra $\LLie^c(-) = \bigoplus_s\LLie^c_s(-)$
into components of homogeneous weight $\LLie^c_s(-)$, $s\geq 1$.
We explicitly set $\DGB^{com}_k(A) = \DGSigma^{-1}\LLie^c_{k+1}(\DGSigma IA)$ for each $k\geq 0$, and each of these homogeneous components
of our complex naturally forms an object of the category of dg-modules,
while the twisting differential $\partial'$
is equivalent to a homomorphism $\partial': \DGB^{com}_*(A)\rightarrow\DGB^{com}_{*-1}(A)$
that decreases (both) the chain grading (and the natural grading)
of our object by one.
If necessary, then we specify the internal (natural) grading of the objects $\DGB^{com}_k(\COp)$
by an external subscript in our notation $\DGB^{com}_k(\COp) = \DGB^{com}_k(\COp)_*$.

The Harrison cochain complex $\Def_{\dg^*\ComCat_+}^*(A,N)$ of a (plain) augmented unitary commutative algebra $A$
with coefficients in a module $N$
is defined by the dg-module of homomorphisms:
\begin{equation}\label{DeformationComplexes:DGComplex:HarrisonConstruction:CochainComplex}
\Def_{\dg^*\ComCat_+}^*(A,N) = (\Hom_{\dg\Mod}(\DGB^{com}_*(A),N),\partial'')
\end{equation}
where we consider the chain complex $\DGB^{com}_*(A)\in\dg\Mod$ as source object, the module $N$ as target object
and an extra twisting differential $\partial''$,
which we determine by the action of $A$ on $N$ (we refer to~\cite{Balavoine} for the explicit definition
of this twisting differential in our operadic approach
of the Harrison cohomology).
This object $\Def_{\dg^*\ComCat_+}^*(A,N)$ actually forms a cochain complex of dg-mo\-du\-les,
with an internal grading given by the natural grading of our dg-module
of homomorphisms,
a cochain grading such that
\begin{equation}\label{DeformationComplexes:DGComplex:HarrisonConstruction:CochainComplexComponents}
\Def_{\dg^*\ComCat_+}^k(A,N) = \Hom_{\dg\Mod}(\DGB^{com}_k(A),N),
\end{equation}
for any $k\in\NN$, and a total twisting differential $\partial = \partial'+\partial''$
that raises this cochain grading by one.
\end{rem}

\begin{constr}[The Harrison complex and the deformation complex of Hopf collections]\label{DeformationComplexes:DGComplex:CommutativeAlgebraModule}
We extend the construction of the previous section to Hopf $\Lambda$-collections $\AOp\in\dg^*\Hopf\Lambda\Seq_{>1}^c$
equipped with an augmentation $\eta_*: \AOp\rightarrow\overline{\ComOp}{}^c$
over the Hopf collection $\overline{\ComOp}{}^c(r) = \kk$, $r>1$,
that underlies the commutative cooperad $\ComOp^c$.
We then form the dg-modules:
\begin{equation}\label{DeformationComplexes:DGComplex:CommutativeAlgebraModule:HarrisonConstruction}
\DGB^{com}_*(\AOp)(r) = \DGB^{com}_*(\AOp(r)),
\end{equation}
for $r>1$, where we consider the Harrison chain complex of the collection of augmented unitary commutative algebras $\AOp(r)$
underlying $\AOp$.
These dg-modules clearly form a $\Lambda$-collection $\DGB^{com}_*(\AOp)\in\dg\Lambda\Seq_{>1}^c$ (by functoriality of the Harrison complex construction).
In the case where we have a module $\NOp$ over $\AOp$ (in the sense of the definition of~\S\ref{DeformationComplexes:Biderivations:Modules})
we also form the end
\begin{equation}\label{DeformationComplexes:DGComplex:CommutativeAlgebraModule:EndConstruction}
\Def_{\dg\Hopf\Lambda\Seq^c}^*(\AOp,\NOp) = \int_{\rset\in\Lambda}(\Hom_{\dg\Mod}(\DGB^{com}_*(\AOp(r)),\NOp(r)),\partial'')
\end{equation}
to get a Harrison cochain complex $\Def_{\dg\Hopf\Lambda\Seq^c}^*(\AOp,\NOp)\in\dg\Mod$
with coefficients in $\NOp$.
Thus, an element of this cochain complex $h\in\Def_{\dg\Hopf\Lambda\Seq^c}^*(\AOp,\NOp)$ consists of a collection
of homomorphisms $h: \DGB^{com}_*(\AOp(r))\rightarrow\NOp(r)$
that intertwine the action of the corestriction operators on our objects and define Harrison cochains in the classical sense for the commutative algebras $\AOp(r)$
and the modules of coefficients $\NOp(r)$.
We equivalently have:
\begin{equation}\label{DeformationComplexes:DGComplex:CommutativeAlgebraModule:HomConstruction}
\Def_{\dg\Hopf\Lambda\Seq^c}^*(\AOp,\NOp) = (\Hom_{\dg\Lambda\Seq^c}(\DGB^{com}_*(\AOp),\NOp),\partial''),
\end{equation}
where we consider the dg-module of $\Lambda$-collection homomorphisms $h: \DGB^{com}_*(\AOp)\rightarrow\NOp$.

The differential of a homomorphism $h$ in $\Def_{\dg\Hopf\Lambda\Seq^c}^*(\AOp,\NOp)$
is defined by an arity-wise application
of the Harrison differential.
In what follows, we also use an extra $h$ subscript in order to distinguish this differential $\partial_h = \partial'_h+\partial''_h$
from the operadic differential of our deformation bicomplex.
\end{constr}

We have the following structure statement:

\begin{prop}\label{DeformationComplexes:DGComplex:HarrisonComodule}
In the case of the coaugmentation coideal $\AOp = \overline{\PiOp}$
of an augmented Hopf $\Lambda$-cooperad $\PiOp$,
the Harrison complex with trivial coefficients $\DGB^{com}_*(\overline{\PiOp}) = \DGSigma(\LLie^c(\DGSigma\IOp\PiOp),\partial'_v)$
naturally forms a complex of bicomodules over the coaugmented $\Lambda$-cooperad
underlying~$\PiOp$.
\end{prop}

\begin{proof}
We use the Lie cooperad $\LieOp^c$ and the operadic expansion of the cofree Lie coalgebra
$\LLie^c(X) = \bigoplus_r(\LieOp^c(r)\otimes X^{\otimes r})_{\Sigma_r}$
in order to check this statement.
We explicitly define the left coproducts of our bicomodule structure
as the composites:
\begin{equation*}
\LLie^c_s(\DGSigma\IOp\PiOp(k+l-1))
\xrightarrow{\circ_i^*}\LLie^c_s(\PiOp(k)\otimes\DGSigma\IOp\PiOp(l))
\xrightarrow{\mu}\PiOp(k)\otimes\LLie^c_s(\DGSigma\IOp\PiOp(l)),
\end{equation*}
where we perform the coproduct $\circ_i^*: \IOp\PiOp(k+l-1)\rightarrow\PiOp(k)\otimes\IOp\PiOp(l)$, inherited from the cooperad $\PiOp$,
on each tensor factor of the cofree Lie coalgebra first,
and we gather and multiply the factors $\PiOp(k)$ together afterwards.
We proceed similarly with the right coproducts.
The distribution relation between the commutative algebra product and the composition coproducts
in a Hopf cooperad implies, again, that these coproduct operations on $\DGB^{com}_*(\overline{\PiOp})$
intertwine the twisting differential of the Harrison construction.
\end{proof}

\begin{constr}[The deformation complex of Hopf cooperads]\label{DeformationComplexes:DGComplex:HopfCooperad}
We now assume that $\PiOp$ is an augmented Hopf $\Lambda$-cooperad (as stated in the introduction
of this subsection), that $\HOp$ is another Hopf $\Lambda$-cooperad,
and that we have a morphism between these operads $\chi: \PiOp\rightarrow\HOp$.
The object $\DGB^{com}_*(\PiOp)$ in Proposition~\ref{DeformationComplexes:DGComplex:HarrisonComodule}
forms a complex of bicomodules over $\HOp$ by corestriction of structure through $\chi$,
while the object $\DGB_{op}^*(\HOp)$ inherits the structure of a complex of modules
over the Hopf $\Lambda$-collection $\overline{\PiOp}$.

We then set
\begin{equation}
D^{k l} = D^{k l}(\PiOp,\HOp) = \Hom_{\dg\Lambda\Seq^c}(\DGB^{com}_k(\PiOp),\DGB_{op}^l(\HOp)),
\end{equation}
for each bidegree $(k,l)\in\NN^2$.
This double sequence of dg-modules inherits a horizontal twisting differential $\partial_h: D^{k l}\rightarrow D^{k+1 l}$,
which we deduce from the identity of our object with the Harrison cochain complex $D^{* l} = \Def_{\dg^*\ComCat_+}(\PiOp,\NOp)$
for $\NOp = \DGB_{op}^l(\HOp)$,
and a vertical twisting differential $\partial_v: D^{k l}\rightarrow D^{k l+1}$
which we deduce from the identity of our object with the operadic deformation complex $D^{k *} = \CoDef_{\dg^*\ComCat_+}(\MOp,\HOp)$
for $\MOp = \DGB^{com}_k(\PiOp)$.
We readily check that these twisting differentials (anti)commute to each other (by using the distribution relation
between the commutative algebra product and the composition coproducts
in a Hopf cooperad again).
We then define the deformation bicomplex $D^{* *} = D^{* *}(\PiOp,\HOp)$
of the Hopf $\Lambda$-cooperads $(\PiOp,\HOp)$
as the bicomplex of dg-modules
\begin{equation}\label{DeformationComplexes:DGComplex:HopfCooperad:Expression}
D^{* *} = \BiDef^{* *}_{\dg^*\Hopf\Lambda\Op^c}(\PiOp,\HOp) = (\Hom_{\dg\Lambda\Seq^c}(\DGB^{com}_*(\PiOp),\DGB_{op}^*(\HOp)),\partial''_h+\partial''_v),
\end{equation}
which we form from this bigraded cochain complex of dg-modules, and where we take the sum of the differentials of the Harrison
and operadic deformation complexes as twisting differential.
\end{constr}

We now consider the bicosimplicial complex of biderivations
\begin{equation*}
B^{\bullet\,\bullet} = B^{\bullet\,\bullet}(\PiOp,\HOp) = \BiDer_{\dg^*\Hopf\Lambda\Op^c}(\Res^{com}_{\bullet}(\PiOp),\Res_{op}^{\bullet}(\HOp))
\end{equation*}
which we associate to the Hopf $\Lambda$-cooperads $(\PiOp,\HOp)$ in~\S\ref{DeformationComplexes:BicosimplicialComplex},
and where, to define our biderivation relations, we consider the morphisms $\chi^0: \Res^{com}_k(\PiOp)\rightarrow\Res_{op}^l(\HOp)$
formed by composing $\chi: \PiOp\rightarrow\HOp$
with the augmentation of the cotriple resolution $\epsilon: \Res^{com}_k(\PiOp)\rightarrow\PiOp$
and with the coaugmentation of the triple coresolution $\eta: \HOp\rightarrow\Res_{op}^l(\HOp)$.
In our obstruction problem, we also consider the conormalized complex of the diagonal cosimplicial complex
of this bicomplex of biderivations.
We can actually form a conormalized complex in each direction to get a cochain bicomplex whose total complex is,
according to the (cosimplicial version of the) Eilenberg-Zilber theorem,
weakly equivalent to this diagonal conormalized complex. We have the following result
which connects this bicomplex of biderivations to the deformation bicomplex
of the previous paragraph:

\begin{thm}\label{DeformationComplexes:DGComplex:MainResult}
We consider a morphism of Hopf $\Lambda$-cooperads $\chi: \PiOp\rightarrow\HOp$.

We still assume that $\PiOp$ is equipped with an augmentation over the commutative cooperad $\ComOp^c$,
and that the augmentation ideals $\IOp\PiOp(r)$ of the algebras $\PiOp(r)$, $r>1$,
form a free $\Lambda$-collection $\IOp\PiOp = \Lambda\otimes_{\Sigma}\SOp\PiOp$,
for some symmetric collection such that $\SOp\PiOp\subset\IOp\PiOp$ (as we explain in the introduction of this subsection).
We then have a weak-equivalence of bicomplexes of dg-modules
\begin{equation*}
\DGN^{* *}\BiDer_{\dg^*\Hopf\Lambda\Op^c}(\Res^{com}_{\bullet}(\PiOp),\Res_{op}^{\bullet}(\HOp))
\xrightarrow{\sim}\BiDef^{* *}_{\dg^*\Hopf\Lambda\Op^c}(\PiOp,\HOp),
\end{equation*}
between the conormalized bicomplex of the bicosimplicial complex
of biderivations $B^{\bullet\,\bullet} = \BiDer_{\dg^*\Hopf\Lambda\Op^c}(\Res^{com}_{\bullet}(\PiOp),\Res_{op}^{\bullet}(\HOp))$
and the deformation bicomplex of~\S\ref{DeformationComplexes:DGComplex:HopfCooperad}:
\begin{equation*}
D^{* *} = \BiDef^{* *}_{\dg^*\Hopf\Lambda\Op^c}(\PiOp,\HOp).
\end{equation*}
\end{thm}

\begin{proof}
Proposition~\ref{DeformationComplexes:Biderivations:ModuleReduction}
implies that the modules of biderivations
\begin{gather}
B^{k l} = \BiDer_{\dg^*\Hopf\Lambda\Op^c}(\Res^{com}_k(\PiOp),\Res_{op}^l(\HOp))
\intertext{are isomorphic to dg-modules of homomorphisms}
C^{k l} = \Hom_{\dg\Lambda\Seq^c}(\DGI\DGC^{com}_k(\PiOp),\DGC_{op}^l(\HOp)).
\end{gather}
Recall that the object $\DGI\DGC^{com}_{\bullet}(\PiOp)$ is preserved by the degeneracy operators of the cotriple resolution of our Hopf cooperad
and that $\DGC_{op}^{\bullet}(\HOp)$ is similarly preserved by the codegeneracy operators
of the triple coresolution.
We accordingly have a component-wise isomorphism when we pass to the conormalized bicomplex:
\begin{equation}\label{DeformationComplexes:DGComplex:MainResult:BiderivationComplexComponents}
\DGN^{k l}(B^{\bullet\,\bullet})\simeq\Hom_{\dg\Lambda\Seq^c}(\DGN_k\DGI\DGC^{com}_{\bullet}(\PiOp),\DGN^l\DGC_{op}^{\bullet}(\HOp)),
\end{equation}
where we use these internal degeneracy and codegeneracy operators to define the normalized complex
of the object $\DGI\DGC^{com}_{\bullet}(\PiOp)$
and the conormalized complex of the object $\DGC_{op}^{\bullet}(\HOp)$.
We eventually get that the conormalized bicomplex $\DGN^{* *}(B^{\bullet\,\bullet})$
is isomorphic to a bicomplex formed by the dg-modules~(\ref{DeformationComplexes:DGComplex:MainResult:BiderivationComplexComponents})
together with twisting differentials $\partial_h$ and $\partial_v$
which we transport from $\DGN^{* *}(B^{\bullet\,\bullet})$.
We can still decompose the horizontal twisting differential $\partial_h$ into a sum $\partial_h = \partial'_h+\partial''_h$,
where $\partial'_h$ is yielded by the alternate sum of the face operators of the cotriple resolution $d_i$ such that $i>0$,
whereas $\partial''_h$ is yielded by the $0$-face $d_0$.
We can actually identify $\partial'_h$ with a differential of the complex $\DGN_*\DGI\DGC^{com}_{\bullet}(\PiOp)$
since these face operators $d_i$ such that $i>0$
preserve $\DGI\DGC^{com}_{\bullet}(\PiOp)$
inside the cotriple resolution. We have a similar observation for the vertical twisting differential $\partial_v$
which we can decompose into $\partial_v = \partial'_v+\partial''_v$,
where $\partial'_v$ is identified with a differential of the complex $\DGN^*\DGC_{op}^{\bullet}(\HOp)$
which is determined by the action of the coface operators $d^i$ such that $i>0$
on the object $\DGC_{op}^{\bullet}(\HOp)$.
Hence, we eventually get:
\begin{equation}\label{DeformationComplexes:DGComplex:MainResult:BiderivationComplexReduction}
\DGN^{* *}(B^{\bullet\,\bullet})\simeq(\Hom_{\dg\Lambda\Seq^c}(\DGN_*\DGI\DGC^{com}_{\bullet}(\PiOp),\DGN^*\DGC_{op}^{\bullet}(\HOp)),\partial''_h+\partial''_v),
\end{equation}

We have an arity-wise identity:
\begin{equation}\label{DeformationComplexes:DGComplex:MainResult:CotripleComplexComponents}
\DGN_*\DGI\DGC^{com}_{\bullet}(\PiOp(r)) = \DGN_*(\underbrace{\II\Sym\circ\cdots\circ\II\Sym}_{\bullet}(\IOp\PiOp(r))),
\end{equation}
and we have, according to~\cite{FressePartition}, a chain of weak-equivalences that connects the Harrison complex with trivial coefficients
to this normalized chain complex, where we retain the component $\partial'_h$
of our twisting differential:
\begin{equation}\label{DeformationComplexes:DGComplex:MainResult:AritywiseCotripleComplexReduction}
\DGSigma(\LLie^c(\DGSigma\IOp\PiOp(r)),\partial')\xrightarrow{\sim}\cdot
\xrightarrow{\sim}\DGN_*(\underbrace{\II\Sym\circ\cdots\circ\II\Sym}_{\bullet}(\IOp\PiOp(r))).
\end{equation}
We refer to \emph{loc. cit.} for the explicit definition of this mapping. We easily check from this construction
that our map preserves corestriction operators as well as the coaction of the Hopf cooperad $\PiOp$
on our objects and hence, defines a weak-equivalence of chain complexes of bicomodules
over the $\Lambda$-cooperad $\PiOp$:
\begin{equation}\label{DeformationComplexes:DGComplex:MainResult:CotripleComplexReduction}
\DGB^{com}_*(\PiOp)\xrightarrow{\sim}\DGN_*\DGI\DGC^{com}_{\bullet}(\PiOp).
\end{equation}

We have an analogous weak-equivalence that connects the cooperadic cobar complex to the conormalized cochain complex
of the object $\DGC_{op}^{\bullet}(\HOp)$:
\begin{equation}\label{DeformationComplexes:DGComplex:MainResult:TripleComplexReduction}
\DGN^*\DGC_{op}^{\bullet}(\HOp)\xrightarrow{\sim}\DGB_{op}^*(\HOp).
\end{equation}
We refer to~\cite[Proposition C.2.16]{FresseBook} and to~\cite{Livernet} for (a dual version, in the context of operads, of) this construction.
We easily check from the explicit definition of these references that this map preserves the action
of the Hopf collection $\overline{\HOp}$
on our objects, and hence, defines a weak-equivalence of chain complexes
of modules over~$\overline{\HOp}$.

We plug these maps (\ref{DeformationComplexes:DGComplex:MainResult:CotripleComplexReduction}-\ref{DeformationComplexes:DGComplex:MainResult:TripleComplexReduction})
in our dg-modules of homomorphisms~(\ref{DeformationComplexes:DGComplex:MainResult:BiderivationComplexReduction})
to get a comparison map:
\begin{multline}\label{DeformationComplexes:DGComplex:MainResult:ComparisonMap}
(\Hom_{\dg\Lambda\Seq^c}(\DGN_*\DGI\DGC^{com}_{\bullet}(\PiOp),\DGN^*\DGC_{op}^{\bullet}(\HOp)),\partial''_h+\partial''_v)
\\
\rightarrow(\Hom_{\dg\Lambda\Seq^c}(\DGB^{com}_*(\PiOp),\DGB_{op}^*(\HOp)),\partial''_h+\partial''_v).
\end{multline}
We easily check that this map preserves the extra twisting differentials $\partial''_h$ and $\partial''_v$
of our object.

We observe in~\S\ref{CotripleResolution} that $\DGI\DGC^{com}_{\bullet}(\PiOp)$ is dimension-wise free as $\Lambda$-collection
when $\IOp\PiOp$ satisfies the assumption of the theorem.
We have a similar result for the Harrison complex $\DGB^{com}_*(\PiOp)$ (we then use the operadic expansion of the cofree Lie algebra
underlying $\DGB^{com}_*(\PiOp)$ and the result of Proposition~\ref{CotripleResolution:FreeCollectionTensors}).
These structure results imply that both $\DGN_*\DGI\DGC^{com}_{\bullet}(\PiOp)$ and $\DGB^{com}_*(\PiOp)$
form cofibrant objects of the category of $\Lambda$-collections
in dg-modules (with respect to the projective model structure),
and as a by-product, we get that our comparison maps (\ref{DeformationComplexes:DGComplex:MainResult:CotripleComplexReduction}-\ref{DeformationComplexes:DGComplex:MainResult:TripleComplexReduction})
induce a weak-equivalence on the hom-objects
of~(\ref{DeformationComplexes:DGComplex:MainResult:ComparisonMap}).
We can then use an obvious spectral sequence argument to conclude that our map in~(\ref{DeformationComplexes:DGComplex:MainResult:ComparisonMap})
induces a weak-equivalence on total complexes,
and hence,
to complete the proof of the theorem (compare with~\cite[Theorem III.3.1.4]{FresseBook}).
\end{proof}

\subsection{The application of the Koszul duality of operads}\label{DeformationComplexes:KoszulDuality}
We now examine the case where $\HOp$ is (isomorphic to) the $m$-Poisson cooperad $\HOp = \PoisOp_m^c$
in the deformation bicomplex $\BiDef^{* *}_{\dg^*\Hopf\Lambda\Op^c}(\PiOp,\HOp)$
of the previous section.
We have in this case a reduction of the operadic cobar construction $\BB_{op}^c(\PoisOp_m^c)$ given by the observation
that the $m$-Poisson operad (and the $m$-Poisson cooperad dually) is Koszul.
We more precisely have a weak-equivalence of operads
\begin{equation*}
\kappa: \BB_{op}^c(\PoisOp_m^c)\xrightarrow{\sim}\SuspOp^m\PoisOp_m,
\end{equation*}
where $\SuspOp^m$ refers to an $m$-fold suspension operation. We refer to~\cite{GetzlerJones} for the proof of this result,
to~\cite{GinzburgKapranov} for the general definition of the notion of a Koszul operad,
and to~\cite{LodayVallette} for a general reference on this subject.
Recall simply that the definition of the above weak-equivalence follows from the observation
that the cobar construction $\BB_{op}^c(\PoisOp_m^c)$
vanishes when the weight grading of the free operad
in the expression
of the cobar construction $\BB_{op}^c(\PoisOp_m^c) = (\FreeOp(\DGSigma\PoisOp_m^c),\partial)$
exceeds the arity. Then we just use that the top cobar differentials $\partial: \FreeOp_{r-2}(\DGSigma\PoisOp_m^c)(r)\rightarrow\FreeOp_{r-1}(\DGSigma\PoisOp_m^c)(r)$
determine, up to suspension, a presentation of the $m$-Poisson operad
by generators and relations. We equivalently get that the components of the operad $\SuspOp^m\PoisOp_m$
represent the top cohomology of the operad $\BB_{op}^c(\PoisOp_m^c)$
in each arity $r>0$. The claim is that $\BB_{op}^c(\PoisOp_m^c)$ has no cohomology outside these top components.

We refer to this object $\KK_{op}^c(\PoisOp_m^c) := \SuspOp^m\PoisOp_m$ as the Koszul dual operad
of the $m$-Poisson cooperad $\PoisOp_m^c$.
We now set $\DGK_{op}^*(\PoisOp_m^c) := \DGSigma\overline{\SuspOp}^m\overline{\PoisOp}_m$
to get a complex whose components represent the top cohomology of the complex $\DGB_{op}^*(\PoisOp_m^c)$
of~\S\ref{DeformationComplexes:DGComplex}.
We have the following observation:

\begin{prop}\label{DeformationComplexes:KoszulDuality:Module}
The object $\DGK_{op}^*(\PoisOp_m^c)$ inherits the structure of a module
over the Hopf collection $\overline{\HOp} = \overline{\PoisOp}{}_m^c$
so that the morphism $\DGB_{op}^*(\PoisOp_m^c)\xrightarrow{\sim}\DGK_{op}^*(\PoisOp_m^c)$
defines a morphism of modules over $\overline{\PoisOp}{}_m^c$.
\end{prop}

\begin{proof}
This proposition is an immediate consequence of the definition of the object $\DGK_{op}^*(\PoisOp_m^c)$ and of the observation
that the cobar differential $\partial: \FreeOp(\DGSigma\PoisOp_m^c)\rightarrow\FreeOp(\DGSigma\PoisOp_m^c)$
is preserved by the action of the commutative algebra $\PoisOp_m^c(r)$
on the module $\DGB_{op}^*(\PoisOp_m^c)(r) = \FreeOp(\DGSigma\PoisOp_m^c)(r)$
in each arity $r>1$.
\end{proof}

\begin{constr}[The application of the Koszul reduction to the deformation complex]\label{DeformationComplexes:KoszulDuality:KoszulReduction}
We plug the Koszul duality weak-equivalence $\kappa: \DGB_{op}^*(\PoisOp_m^c)\xrightarrow{\sim}\DGK_{op}^*(\PoisOp_m^c)$
in the deformation bicomplex of~\S\ref{DeformationComplexes:DGComplex:HopfCooperad}.
We accordingly set:
\begin{equation}
K^{k l} = K^{k l}(\PiOp,\PoisOp_m^c) = \Hom_{\dg\Lambda\Seq^c}(\DGB^{com}_k(\PiOp),\DGK_{op}^l(\PoisOp_m^c)),
\end{equation}
for each bidegree $(k,l)\in\NN^2$, where we consider the same dg-modules of homomorphisms as in the definition
of the deformation bicomplex~\S\ref{DeformationComplexes:DGComplex:HopfCooperad}(\ref{DeformationComplexes:DGComplex:HopfCooperad:Expression}),
but we now substitute the object $\DGK_{op}^l(\PoisOp_m^c)$
to the cobar complex $\DGB_{op}(\HOp) = \DGB_{op}^l(\PoisOp_m^c)$.

We can still provide this double sequence of dg-modules $K^{* *}$
with a horizontal twisting differential $\partial_h: K^{k l}\rightarrow K^{k+1 l}$
by identity of our object with the Harrison cochain complex $K^{* l} = \Def_{\dg^*\ComCat_+}(\PiOp,\NOp)$,
where we consider the object $\NOp = \DGK_{op}^l(\PoisOp_m^c)$
together with the module structure of Proposition~\ref{DeformationComplexes:KoszulDuality:Module}.
We also have a vertical twisting differential $\partial_v = \partial''_v: K^{k l}\rightarrow K^{k l+1}$,
defined by replacing the cobar operad $\BB^c_{op}(\HOp) = \BB^c_{op}(\PoisOp_m^c)$
in the construction of~\S\ref{DeformationComplexes:DGComplex:CooperadBicomodule}(\ref{DeformationComplexes:DGComplex:CooperadBicomodule:TwistingDifferential})
by the Koszul dual operad $\KK^c_{op}(\HOp) = \SuspOp^m\PoisOp_m$,
by using the prolongment of the morphism $\iota: \HOp\rightarrow\BB^c_{op}(\HOp)$
to this object $\KK^c_{op}(\HOp) = \SuspOp^m\PoisOp_m$
through the Koszul duality weak-equivalence $\BB^c_{op}(\PoisOp_m^c)\xrightarrow{\sim}\SuspOp^m\PoisOp_m$,
and by using the universal morphism $\FreeOp_2(\KK^c_{op}(\HOp))\subset\FreeOp(\KK^c_{op}(\HOp))\xrightarrow{\lambda}\KK^c_{op}(\HOp)$
attached to this operad $\KK^c_{op}(\HOp) = \SuspOp^m\PoisOp_m$.
We note that in this case the vertical twisting differential of our complex reduces to this term $\partial_v = \partial''_v$,
which we determine from the construction
of~\S\ref{DeformationComplexes:DGComplex:CooperadBicomodule}(\ref{DeformationComplexes:DGComplex:CooperadBicomodule:TwistingDifferential}),
because the object $\DGK_{op}^*(\PoisOp_m^c) = \DGSigma\overline{\SuspOp}^m\overline{\PoisOp}_m$
has no internal twisting differential (and actually no differential at all).

We again readily check that these twisting differentials (anti)commute to each other, and hence
our construction returns a bicomplex of dg-modules:
\begin{equation}\label{DeformationComplexes:KoszulDuality:KoszulReduction:Complex}
K^{* *} = K^{* *}(\PiOp,\PoisOp_m^c) = (\Hom_{\dg\Lambda\Seq^c}(\DGB^{com}_*(\PiOp),\DGK_{op}^*(\PoisOp_m^c)),\partial''_h+\partial''_v)
\end{equation}
associated to this Koszul construction $\DGK_{op}^*(\PoisOp_m^c) = \DGSigma\overline{\SuspOp}^m\overline{\PoisOp}{}_m$.
This complex $K^{* *} = K^{* *}(\PiOp,\PoisOp_m^c)$ is also an analogue (for Hopf cooperads)
of the deformation complex, denoted by $\Def(\hoe_n\rightarrow\POp)$,
which is studied by the second author in~\cite[\S 4]{WillwacherGraphs}.
\end{constr}

We now have the following statement:

\begin{thm}\label{DeformationComplexes:KoszulDuality:MainResult}
In the case $\HOp\simeq\PoisOp_m^c$, the deformation bicomplex of Theorem~\ref{DeformationComplexes:DGComplex:MainResult}
\begin{gather*}
D^{* *} = D^{* *}(\PiOp,\PoisOp_m^c) = \BiDef^{* *}_{\dg^*\Hopf\Lambda\Op^c}(\PiOp,\PoisOp_m^c)
\intertext{admits a further reduction}
D^{* *}(\PiOp,\PoisOp_m^c)\xrightarrow{\sim} K^{* *}(\PiOp,\PoisOp_m^c),
\intertext{where we consider the bicomplex of~\S\ref{DeformationComplexes:KoszulDuality:KoszulReduction}:}
K^{* *} = K^{* *}(\PiOp,\PoisOp_m^c) = (\Hom_{\dg\Lambda\Seq^c}(\DGB^{com}_*(\PiOp),\DGK_{op}^*(\PoisOp_m^c)),\partial''_h+\partial''_v).
\end{gather*}
\end{thm}

\begin{proof}
We immediately get that the Koszul duality weak-equivalence
\begin{equation}\label{DeformationComplexes:KoszulDuality:MainResult:KoszulDualityEquivalence}
\kappa: \DGB_{op}^*(\PoisOp_m^c)\xrightarrow{\sim}\DGK_{op}^*(\PoisOp_m^c)
\end{equation}
induces a comparison map
\begin{multline}\label{DeformationComplexes:KoszulDuality:MainResult:ComparisonMap}
\underbrace{(\Hom_{\dg\Lambda\Seq^c}(\DGB^{com}_*(\PiOp),\DGB_{op}^*(\PoisOp_m^c)),\partial''_h+\partial''_v)}_{= D^{* *}}
\\
\rightarrow\underbrace{(\Hom_{\dg\Lambda\Seq^c}(\DGB^{com}_*(\PiOp),\DGK_{op}^*(\PoisOp_m^c)),\partial''_h+\partial''_v)}_{= K^{* *}}
\end{multline}
which carries the twisting differentials of the deformation complex $D^{* *}$ in~\S\ref{DeformationComplexes:DGComplex:HopfCooperad}
to the twisting differentials of the deformation complex $K^{* *}$
of~\S\ref{DeformationComplexes:KoszulDuality:KoszulReduction}
because this map~(\ref{DeformationComplexes:KoszulDuality:MainResult:KoszulDualityEquivalence})
preserves the structures involved in the definition
of our twisting differentials.

We still use that the Harrison complex $\DGB^{com}_*(\PiOp)$ forms a cofibrant object of the category of $\Lambda$-collections
in dg-modules (see the proof of Theorem~\ref{DeformationComplexes:DGComplex:MainResult})
to check that our map (\ref{DeformationComplexes:KoszulDuality:MainResult:KoszulDualityEquivalence})
induce a weak-equivalence on the hom-objects
of~(\ref{DeformationComplexes:KoszulDuality:MainResult:ComparisonMap}).
We again use an obvious spectral sequence argument to conclude that the comparison map
in (\ref{DeformationComplexes:KoszulDuality:MainResult:ComparisonMap})
defines a weak-equivalence on total complexes,
and the result of the theorem follows.
\end{proof}

\section{The reduction to graph homology}\label{GraphHomology}
The goal of this section is to compute the homology of the Koszul deformation complex $K^{* *} = K^{* *}(\PiOp,\PoisOp_m^c)$
of Theorem~\ref{DeformationComplexes:KoszulDuality:MainResult}
in the case where $\PiOp$ is a Poisson cooperad~$\PiOp = \PoisOp_n^c$ (with possibly $m\not=n$).
We aim to prove the vanishing of this homology in order to apply the obstruction method of~\S\ref{Background:ObstructionProblem}.
We will more precisely show that the dg-modules $K^{* *}$ are weakly-equivalent to (variants of) the graph complexes
defined by Maxim Kontsevich (see~\cite{KontsevichSymplectic,KontsevichFormalityConj,KontsevichMotives}),
and we use simple degree counting to deduce our vanishing statement
from this relationship.

Recall that the Koszul deformation complex $K^{* *} = K^{* *}(\PiOp,\PoisOp_m^c)$
depends on the choice of a morphism of Hopf $\Lambda$-cooperads $\chi: \PiOp\rightarrow\PoisOp_m^c$
(like all deformation complexes which we defined in the previous section).
We examine the Koszul deformation complexes associated to particular morphisms $\chi: \PiOp\rightarrow\PoisOp_m^c$,
where we take a Poisson cooperad $\PiOp = \PoisOp_n^c$
as source object each time.

First, for $m=n$, we consider the case where $\chi = \id$ is the identity morphism
of the $n$-Poisson cooperad $\PiOp = \PoisOp_n^c$.
We adopt the short notation:
\begin{equation*}
L^{* *}_n = K^{* *}(\PoisOp_n^c,\PoisOp_n^c)
\end{equation*}
for the Koszul deformation complex associated to this morphism $\chi = \id: \PoisOp_n^c\rightarrow\PoisOp_n^c$,
for any $n\geq 2$.
We compute the homology of this deformation complex in order to prove
our intrinsic formality statement
for the little $n$-discs operad (Theorem~\ref{Result:HopfDGOperadIntrinsicFormality}).
We then assume $n\geq 3$.

We also consider the morphism of Hopf $\Lambda$-cooperads $\iota^*: \PoisOp_n^c\rightarrow\PoisOp_m^c$
whose dual $\iota_*: \PoisOp_m\rightarrow\PoisOp_n$
carries the commutative product operation of the $m$-Poisson operad $\mu\in\PoisOp_m(2)$
to the commutative product operation of the $n$-Poisson operad $\mu\in\PoisOp_n(2)$
and sends the Lie bracket operation $\lambda\in\PoisOp_m(2)$
to zero.
In the case $m<n$, we can identify this morphism $\chi = \iota^*$ with the morphism induced by the embedding
of little discs operads $\iota: \DOp_m\rightarrow\DOp_n$
in cohomology. But our definition makes also sense when $m=n$
and returns a morphism $\iota^*: \PoisOp_n^c\rightarrow\PoisOp_n^c$
which still differs from the identity morphism
in this case.
We adopt the notation
\begin{equation*}
K^{* *}_{m n} = K^{* *}(\PoisOp_n^c,\PoisOp_m^c)
\end{equation*}
for the Koszul deformation complex which we associate to the morphism $\chi = \iota^*: \PoisOp_n^c\rightarrow\PoisOp_m^c$.
This complex is defined for all pairs $m,n\geq 2$.
We actually use the complex $K^{* *}_{n n}$, with $m=n$, as an auxiliary device when we compute the homology of the Koszul deformation complex
of the identity morphism $L^{* *}_n$,
while we focus on the case $m<n$ (and, actually, on the case $n-m\geq 2$) when we prove our formality statement
for the morphisms $\iota: \DOp_m\rightarrow\DOp_n$
that link the little discs operads together (Theorem~\ref{Result:RelativeChainOperadFormality}).

We use an operad of graphs $\GraphOp_n$, weakly-equivalent to the $n$-Poisson operad~$\PoisOp_n$,
in order to obtain our graph complex models of the Koszul deformation complexes $L^{* *}_n$
and $K^{* *}_{m n}$.
We need auxiliary categories of operads and cooperads in order to carry out our constructions.
First of all, the operad of graphs $\GraphOp_n$ is dual to a Hopf $\Lambda$-cooperad $\GraphOp_n^c$
which does not fulfill the connectedness conditions $\KOp(1) = \kk$
of the previous sections.
We therefore introduce an extension of our category of Hopf $\Lambda$-cooperads
in which we can define this object $\GraphOp_n^c$.
We will also see that the operad of graphs $\GraphOp_n$ is defined within a base category of complete dg-modules
and we have to work with a completed tensor product in order to provide this object $\GraphOp_n$
with a Hopf structure.

We explain this background in the first subsection of this section. Then we revisit the definitions of the previous sections
in order to give a dual expression, in terms of operads in complete dg-modules,
of the Koszul deformation complexes
of~\S\ref{DeformationComplexes:KoszulDuality:KoszulReduction}.
We use this construction when we relate our Koszul deformation complex to the graph complex.
We briefly recall the definition of the operads of graphs and we review basic results regarding graph complexes
in the second subsection.
We tackle the applications of graph complexes to the Koszul deformation complex of the cooperads $\PoisOp_n^c$ afterwards.

\subsection{The definition of Koszul deformation complexes revisited}\label{GraphHomology:TwistedEndComplexes}
We use a general notion of Hopf $\Lambda$-operad which is dual (in the categorical sense) to the notion
of Hopf $\Lambda$-cooperad that we consider in the previous sections.
We just drop the connectedness condition of cooperads $\COp(1) = \kk$
when we deal with operads.
We only require that our objects $\POp$ vanish in arity zero $\POp(0) = 0$.
We refer to~\cite[\S I.2]{FresseBook} for a detailed definition of the structure of an augmented $\Lambda$-operad
in the general context of symmetric monoidal categories.
Recall simply that the underlying diagram structure of an augmented $\Lambda$-operad $\POp$
is defined by restriction operators $u^*: \POp(l)\rightarrow\POp(k)$,
which we associate to the injective maps $u: \{1<\dots<k\}\rightarrow\{1<\dots<l\}$,
and which are duals to the corestriction operators
of~\S\ref{Background:LambdaCooperads}.

In what follows, we also consider (contravariant) $\Lambda$-collections, underlying our $\Lambda$-operads,
which are dual to the (covariant) $\Lambda$-collections
consider in the previous sections.
We just use the name `$\Lambda$-collection' to refer to these objects.
We drop the adjective `contravariant' in general since the notion that we consider is usually clearly specified by the context.
Recall simply that, in the book~\cite{FresseBook}, the contravariant $\Lambda$-collections are just called `$\Lambda$-sequences,
while the covariant $\Lambda$-collections are called `covariant $\Lambda$-sequences' (see~\S\ref{Background:Collections}).
We also drop the connectedness condition of~\S\ref{Background:Collections} when we deal with (contravariant) $\Lambda$-collections underlying $\Lambda$-operads.
We just assume that our objects vanish in arity zero in general.

The operad of graphs, to which we apply our constructions, is naturally defined in a base category of complete dg-modules
and, in fact, we have to work in this category in order to provide
this operad with a Hopf structure.
We review the definition of this base category of complete dg-modules in the next paragraph.
We examine the definition of Hopf $\Lambda$-operads in complete dg-modules afterwards.
We then explain the definition of the (previously alluded to) generalization of the category cooperads,
where we have a possibly non-trivial term in arity one,
and which we get when we take the dual objects of these complete Hopf $\Lambda$-operads
in the category of dg-modules. We eventually tackle the applications of complete Hopf $\Lambda$-operads
to the Koszul deformation complexes.

\begin{rem}[The symmetric monoidal category of complete dg-modules]\label{GraphHomology:TwistedEndComplexes:CompleteModules}
The category of complete dg-modules, denoted by $\hat{\f}\dg\Mod$, explicitly consists of dg-modules $K\in\dg\Mod$
equipped with a filtration
$K = \DGF_0 K\supset\cdots\supset\DGF_s K\supset\cdots$
by dg-submodules $\DGF_s K\subset K$ such that $K = \lim_s K/\DGF_s K$.
We take the filtration preserving morphisms of dg-modules as morphisms in $\hat{\f}\dg\Mod$.
We explicitly assume that our morphisms $\phi: K\rightarrow L$
satisfy the relation $\phi(\DGF_s K)\subset\DGF_s L$
for every $s\geq 0$.
We refer to~\cite[\S II.13.0]{FresseBook} for a thorough study of this category of complete dg-modules.
We only briefly recall
the definition of a symmetric monoidal structure
on $\hat{\f}\dg\Mod$.

We explicitly equip $\hat{\f}\dg\Mod$ with the completed tensor product $\hat{\otimes}$,
defined by
$K\hat{\otimes} L = \lim_s K\otimes L/\DGF_s(K\otimes L)$,
for any $K,L\in\hat{\f}\dg\Mod$, and where we set $\DGF_s(K\otimes L) = \sum_{p+q=s}\DGF_p(K)\otimes\DGF_q(L)\subset K\otimes L$,
for $s\geq 0$.
The ground field $\kk$, which we identify with a complete dg-module such that $\DGF_1\kk = 0$,
forms a unit for the completed tensor product.
We also have associativity and symmetry isomorphisms for $\hat{\otimes}$
which are inherited from the base category of dg-modules (see again~\cite[\S II.13.0]{FresseBook}).
We moreover have a symmetric monoidal transformation $\eta: K\otimes L\rightarrow K\hat{\otimes} L$,
where we use the obvious forgetful functor $\omega: \hat{\f}\dg\Mod\rightarrow\dg\Mod$
to compare the plain tensor product of the objects $K,L\in\hat{\f}\dg\Mod$
in the category of dg-modules
with the completed tensor product.
\end{rem}

\begin{defn}[Complete Hopf $\Lambda$-operads]\label{GraphHomology:TwistedEndComplexes:CompleteHopfOperads}
We now define a complete Hopf $\Lambda$-operad $\POp$ as an augmented $\Lambda$-operad
in the category of counitary cocommutative coalgebras
in $\hat{\f}\dg\Mod$.

To be explicit, when we use this definition, we first assume that each term of our operad $\POp(r)$, $r>0$,
is a counitary cocommutative coalgebra
in the complete sense,
with a coproduct $\Delta: \POp(r)\rightarrow\POp(r)\hat{\otimes}\POp(r)$
that lands in the two-fold completed tensor product $\hat{\otimes}$
of the object $\POp(r)\in\hat{\f}\dg\Mod$.
Then we define the structure morphisms of our operad within the category of complete counitary cocommutative coalgebras.
The composition operations of our operad are therefore given by morphisms $\circ_i: \POp(k)\hat{\otimes}\POp(l)\rightarrow\POp(k+l-1)$
which we define on the completed tensor products of the objects $\POp(k),\POp(l)\in\hat{\f}\dg\Mod$,
for any $k,l>0$ and $i = 1,\dots,k$.
We may restrict these composition operations to the plain tensor product $\otimes$
to get the composition operations of a plain operad
in dg-modules $\circ_i: \POp(k)\otimes\POp(l)\rightarrow\POp(k+l-1)$
on the operad $\POp$. The coproduct of our coalgebra structure,
on the other hand, does not restrict
to the plain tensor product
in general.

In our subsequent constructions, we will assume for simplicity that our complete Hopf $\Lambda$-operads $\POp$
are equipped with a coaugmentation $\eta: \ComOp\rightarrow\POp$,
where we regard the commutative operad $\ComOp$
as a complete Hopf $\Lambda$-operad such that $\DGF_0\ComOp(r) = \ComOp(r) = \kk$ and $\DGF_1\ComOp(r) = 0$
for each arity $r>0$.  We then say that $\POp$ forms a coaugmented complete Hopf $\Lambda$-operad.
The existence of this coaugmentation implies that we have a splitting $\POp(r) = \kk\oplus\IOp\POp(r)$,
where $\IOp\POp(r)$ is the kernel of the counit
of the coalgebra $\POp(r)$.
In this situation, we can make the extra assumption that we have the relation $\IOp\POp(r) = \DGF_1\POp(r)$,
for each $r>0$. We use this identity to subsequently simplify the expression
of our deformation complexes.
We therefore adopt the convention that this connectedness requirement is fulfilled
when we deal with a coaugmented complete Hopf $\Lambda$-operad
in what follows.

Now we can naturally identify the $n$-Poisson operad $\PoisOp_n$ with the complete Hopf $\Lambda$-operad equipped with the filtration
such that $\DGF_0\PoisOp_n(r) = \PoisOp_n(r)$, $\DGF_1\PoisOp_n(r) = \IOp\PoisOp_n(r)$
and $\DGF_s\PoisOp_n(r) = 0$ for $s\geq 2$, and for any arity $r>0$.
In this definition, we just take the simplest filtration which fits our requirements
and which makes $\PoisOp_n$
a coaugmented complete Hopf $\Lambda$-operad
in our sense.
\end{defn}

\begin{defn}[Filtered Hopf $\Lambda$-cooperads]\label{GraphHomology:TwistedEndComplexes:FilteredHopfCooperads}
The generalized Hopf $\Lambda$-cooperads which we consider in this section are collections $\PiOp(r) = \{\PiOp(r),r>0\}$
equipped with the same algebraic structure as the Hopf $\Lambda$-cooperads
of~\S\ref{Background:HopfLambdaCooperads},
but where we allow an arbitrary term in arity one.

We assume, besides, that the cochain dg-algebras $\PiOp(r)\in\dg^*\ComCat_+$ which form the components of these cooperads
are endowed with an increasing filtration $0 = \DGF^{-1}\PiOp(r)\subset\cdots\subset\DGF^s\PiOp(r)\subset\cdots\subset\colim_s\DGF^s\PiOp(r) = \PiOp(r)$
such that $\DGF^0\PiOp(r)$ contains the algebra unit $1\in\PiOp(r)$
and we have $\DGF^p\PiOp(r)\cdot\DGF^q\PiOp(r)\subset\DGF^{p+q}\PiOp(r)$, for any $p,q\geq 0$.
We then require that the corestriction operators $u^*: \PiOp(l)\rightarrow\PiOp(k)$ which define the $\Lambda$-diagram structure of our cooperad
preserve this filtration. We explicitly assume that we have the relation $u^*(\DGF^s\PiOp(l))\subset\DGF^s\PiOp(k)$ for every $s\geq 0$.
We similarly assume that the composition coproducts of our cooperad $\circ_i^*: \PiOp(k+l-1)\rightarrow\PiOp(k)\otimes\PiOp(l)$
satisfy $\circ_i^*(\DGF^s\PiOp(k+l-1))\subset\sum_{p+q=s}\DGF^p\PiOp(k)\otimes\DGF^q\PiOp(l)$
for all $s\geq 0$.
We adopt the notation $\dg^*\Hopf\Lambda\Op_f^c$ for the category formed by these objects
together with the filtration preserving morphism
of Hopf $\Lambda$-cooperads
as morphisms.
We also say that $\dg^*\Hopf\Lambda\Op_f^c$ is the category of filtered Hopf $\Lambda$-cooperads.

In our subsequent constructions, we will also assume that our filtered Hopf $\Lambda$-cooperads $\PiOp$
are endowed with an augmentation $\eta_*: \PiOp\rightarrow\ComOp^c$,
where we regard the commutative cooperad $\ComOp^c$ as an object of the category of filtered Hopf $\Lambda$-cooperads
equipped with a trivial filtration (for which we have $\DGF^0\ComOp^c(r) = \kk$
for all $r>0$).
We then say that $\PiOp$ forms an augmented filtered Hopf $\Lambda$-cooperads.
We may also make the assumption that this augmentation restricts to an isomorphism on the zeroth layer
of our filtration $\eta_*: \DGF^0\PiOp(r)\xrightarrow{\simeq}\kk$.
We equivalently require that $\DGF^0\PiOp(r)$ is identified with the module spanned by the unit of the commutative algebra~$\PiOp(r)$.
We adopt the convention that this connectedness requirement is fulfilled
when we deal with an augmented filtered Hopf $\Lambda$-operad
in what follows.
\end{defn}

We have the following statement:

\begin{prop}\label{GraphHomology:TwistedEndComplexes:DualHopfOperad}
The collection $\PiOp^{\vee} = \{\PiOp(r)^{\vee},r>0\}$ formed by the dual dg-modules $\PiOp^{\vee}(r) = \PiOp(r)^{\vee}$
of the components of an (augmented) filtered Hopf $\Lambda$-cooperad $\PiOp\in\dg^*\Hopf\Lambda\Op_f^c$
inherits the structure of a (coaugmented) complete Hopf $\Lambda$-operad
as soon as the subquotients $\DGE^0_s(-) = \DGF^s(-)/\DGF^{s-1}(-)$ of the filtration of the dg-modules $\PiOp(r)$
form modules of finite rank over the ground field degree-wise.
\end{prop}

\begin{proof}
We explicitly equip the dg-modules $\PiOp(r)^{\vee}$, $r>0$,
with the filtration such that
$\DGF_s\PiOp(r)^{\vee} = \ker\bigl(\PiOp(r)^{\vee}\rightarrow\DGF^{s-1}\PiOp(r)^{\vee}\bigr)$,
for any $s\geq 0$.
We then have the relation $\PiOp(r)^{\vee}/\DGF_s\PiOp(r)^{\vee} = (\DGF^{s-1}\PiOp(r))^{\vee}$ for each $s\geq 0$
and we readily get that $\PiOp(r)^{\vee} = \lim_s\DGF^{s-1}\PiOp(r)^{\vee}\Rightarrow\PiOp(r)^{\vee} = \lim_s\PiOp(r)^{\vee}/\DGF_s\PiOp(r)^{\vee}$.

We use the finiteness assumption of the proposition
to get an isomorphism
$(\PiOp(r)\otimes\PiOp(r))^{\vee}\xleftarrow{\simeq}\PiOp(r)^{\vee}\hat{\otimes}\PiOp(r)^{\vee}$
for every $r>0$. In turn, we compose this isomorphism with the morphism induced by the product of the algebra $\PiOp(r)$
in order to provide each $\PiOp(r)^{\vee}$ with the structure of a counitary cocommutative dg-coalgebra
in complete dg-modules.
We also use the morphism $(\PiOp(k)\otimes\PiOp(l))^{\vee}\xleftarrow{\simeq}\PiOp(k)^{\vee}\hat{\otimes}\PiOp(l)^{\vee}$,
defined for every $k,l>0$, in order to provide our collection $\PiOp^{\vee}$
with the composition structure of a Hopf cooperad
in complete dg-modules,
and we use the functoriality of the duality operation to get restriction operators on $\PiOp^{\vee}$.
We therefore have a full structure of complete Hopf $\Lambda$-operad on this object~$\PiOp^{\vee}$.

The augmentation $\eta_*: \PiOp\rightarrow\ComOp^c$ which we attach to our filtered Hopf $\Lambda$-cooperad $\PiOp$
also gives a coaugmentation $\eta: \ComOp\rightarrow\PiOp^{\vee}$
by duality. We moreover have $\DGF_1\PiOp(r)^{\vee} = \IOp\PiOp(r)$, for any arity $r>0$,
as soon as the filtration of the augmented filtered Hopf $\Lambda$-cooperad $\PiOp$
fulfills our connectedness requirement $\DGF^0\PiOp(r) = \kk$ (see~\S\ref{GraphHomology:TwistedEndComplexes:FilteredHopfCooperads}).
We eventually get that $\PiOp^{\vee}$ forms a coaugmented complete Hopf $\Lambda$-operad
in our sense (see~\S\ref{GraphHomology:TwistedEndComplexes:CompleteHopfOperads}).
\end{proof}

\begin{remark}
In Proposition~\ref{GraphHomology:TwistedEndComplexes:DualHopfOperad},
we can recover the underlying filtered Hopf $\Lambda$-cooperad $\PiOp$
of the complete Hopf $\Lambda$-operad $\POp = \PiOp^{\vee}$
by taking a continuous dual of our objects.
Indeed, if our finiteness assumptions are satisfied, then we have $\DGF^{s-1}\PiOp(r) = (\POp(r)/\DGF_s\POp(r))^{\vee}$,
for each $s\geq 0$, where we consider the natural filtration of our object $\POp(r)$
and the dual of the quotient dg-modules $\POp(r)/\DGF_s\POp(r)$.
Then we can use the identity $\PiOp(r) = \colim_s\DGF^{s-1}\PiOp(r)$
to retrieve the full dg-module $\PiOp(r)$ from these dg-modules,
for each arity $r>0$.

The duality relation $\PoisOp_n(r) = \PoisOp_n^c(r)^{\vee}$, $r>0$,
for the $n$-Poisson operad $\PoisOp_n$
can be regarded as a special case of this duality relation between filtered Hopf $\Lambda$-cooperads
and complete Hopf $\Lambda$-operads. We then set $\DGF^s\PoisOp_n^c(r) = \PoisOp_n^c(r)$ for $s\geq 2$
in order to retrieve the filtration considered in~\S\ref{GraphHomology:TwistedEndComplexes:CompleteHopfOperads}
on the $n$-Poisson operad $\PoisOp_n$.
In this case, we just use that each dg-module $\PoisOp_n(r)$, $r>0$, forms a module of finite rank over the ground field degree-wise
to get our duality relation $\PoisOp_n = (\PoisOp_n^c)^{\vee}$.
We mainly use filtrations and completions in order to extend the duality between operads and cooperads
to objects which do not satisfy the local finiteness property of the $n$-Poisson operad.
\end{remark}

We now explain the definition of the dual objects of the Harrison chain complexes of~\S\ref{DeformationComplexes:DGComplex:HarrisonConstruction}.

\begin{constr}[The Harrison cochain complex of complete Hopf operads]\label{GraphHomology:TwistedEndComplexes:HarrisonConstruction}
In a first step, we consider a coaugmented counitary cocommutative coalgebra
in complete dg-modules $C$ (underlying a coaugmented complete Hopf $\Lambda$-operad).
We have $C = \kk\oplus IC$, where $IC$ denotes the coaugmentation coideal of $C$.
We still assume that we have the relation $IC = \DGF_1 C$ (as in our definition of a coaugmented complete Hopf $\Lambda$-operad).
We associate to this object $C$ a Harrison cochain complex with trivial coefficients $\hat{\DGB}{}_{com}^*(C)$
which we define by:
\begin{equation}
\hat{\DGB}{}_{com}^k(C) = \DGSigma\hat{\LLie}_{k+1}(\DGSigma^{-1} IC),
\end{equation}
for any degree $k>0$, where we use the notation $\hat{\LLie}_r(-)$, for any $r>0$, to denote a component of homogeneous weight $r>0$
of a complete version of the free Lie algebra $\hat{\LLie}(-)$.
We explicitly have $\hat{\LLie}_r(\DGSigma^{-1} IC) = (\LieOp(r)\otimes(\DGSigma^{-1} IC)^{\hat{\otimes} r})_{\Sigma_r}$,
where we still use the notation $\LieOp$
for the operad of Lie algebras.
When we form this expression, we identify the components of this operad $\LieOp$
with complete modules such that $\DGF_0\LieOp(r) = \LieOp(r)$
and $\DGF_1\LieOp(r) = 0$.
We then have $\LieOp(r)\otimes K^{\hat{\otimes} r} = \LieOp(r)\hat{\otimes} K^{\hat{\otimes} r}$,
because $\LieOp(r)$ forms a module of finite rank over the ground field.
We equip $\hat{\DGB}{}_{com}^*(C)$ with a differential $\partial': \hat{\DGB}{}_{com}^*(C)\rightarrow\hat{\DGB}{}_{com}^{*+1}(C)$,
which we determine by the homomorphism
$\DGSigma^{-1} IC\xrightarrow{\Delta_*}\DGSigma^{-2}(IC\hat{\otimes}IC)_{\Sigma_2}
\simeq(\LieOp(2)\otimes(\DGSigma^{-1} IC)^{\hat{\otimes} 2})_{\Sigma_2}$
yielded by the coproduct of our coalgebra $\Delta: C\rightarrow C\otimes C$
on the dg-module $\DGSigma^{-1} IC\subset\hat{\LLie}(\DGSigma^{-1} IC)$.
We just assume that $\partial'$ defines a derivation with respect to Lie brackets in order to extend this map to the free complete Lie algebra
in $\hat{\DGB}{}_{com}^*(C) = \DGSigma\hat{\LLie}_{*+1}(\DGSigma^{-1} IC)$.

We now consider the case of a coaugmented complete Hopf $\Lambda$-operad $\POp$
in the sense
of our definition of~\S\ref{GraphHomology:TwistedEndComplexes:CompleteHopfOperads}.
The components of this complete Hopf $\Lambda$-operad $\POp(r)$ form coaugmented complete counitary cocommutative coalgebras
since the morphism $\eta: \ComOp\rightarrow\POp$
defines a natural coaugmentation of the coalgebra $\POp(r)$
arity-wise.
We can therefore apply our Harrison cochain complex construction arity-wise to this complete Hopf $\Lambda$-operad $\POp$.
We then get a (contravariant) $\Lambda$-collection such that $\hat{\DGB}{}_{com}^*(\POp)(r) = \hat{\DGB}{}_{com}^*(\POp(r))$ for any $r>0$.

We can also dualize the construction of Proposition~\ref{DeformationComplexes:DGComplex:CobarModule}
to get left and right composition products
\begin{align}
\label{GraphHomology:TwistedEndComplexes:HarrisonConstruction:LeftAction}
& \circ_i: \POp(k)\hat{\otimes}\hat{\DGB}{}_{com}^*(\POp)(l)\rightarrow\hat{\DGB}{}_{com}^*(\POp)(k+l-1),\\
\label{GraphHomology:TwistedEndComplexes:HarrisonConstruction:RightAction}
& \circ_i: \hat{\DGB}{}_{com}^*(\POp)(k)\hat{\otimes}\POp(l)\rightarrow\hat{\DGB}{}_{com}^*(\POp)(k+l-1),
\end{align}
defined for all $k,l>1$, $i = 1,\dots,k$, and which provide $\hat{\DGB}{}_{com}^*(\POp)$
with the structure of a bimodule (in the complete sense) over the augmented $\Lambda$-operad $\POp$.
(We then consider the categorical dual and a complete version of the notions introduced in~\S\ref{DeformationComplexes:Biderivations:Bicomodules}).
We basically take a diagonal action of the operad $\POp$
on the tensors that span $\hat{\DGB}{}_{com}^*(\POp)$
in order to get these structure operations.
\end{constr}

We have the following duality statement, where we consider a straightforward generalization, for coaugmented filtered Hopf $\Lambda$-cooperads,
of the Harrison chain complex of~\S\ref{DeformationComplexes:DGComplex:HarrisonConstruction}:

\begin{prop}\label{GraphHomology:TwistedEndComplexes:DualHarrisonComplex}
We assume that the complete Hopf $\Lambda$-operad $\POp$ arises as the dual $\POp = \PiOp^{\vee}$
of an augmented filtered Hopf $\Lambda$-cooperad $\PiOp$
as in Proposition~\ref{GraphHomology:TwistedEndComplexes:DualHopfOperad}.
We then have the duality relation
\begin{equation*}
\hat{\DGB}{}_{com}^*(\PiOp^{\vee}) = \DGB^{com}_*(\PiOp)^{\vee}
\end{equation*}
between the Harrison cochain complex of this complete Hopf $\Lambda$-operad $\POp = \PiOp^{\vee}$
and the Harrison chain complex of the Hopf $\Lambda$-cooperad $\PiOp$.
\end{prop}


\begin{proof}
We have an identity $\LLie^c_{k+1}(\DGSigma\IOp\PiOp(r))^{\vee} = \hat{\LLie}_{k+1}(\DGSigma^{-1}\IOp\PiOp(r)^{\vee})$
for the components of our complex (we use mostly the same tensor product isomorphisms
as in the proof of Proposition~\ref{GraphHomology:TwistedEndComplexes:DualHopfOperad})
and we readily check that our differential
in~\S\ref{GraphHomology:TwistedEndComplexes:HarrisonConstruction}
represents the dual homomorphism
of the differential of the Harrison chain complex.
\end{proof}

We aim to replace the Harrison chain complex~$\DGB^{com}_*(\PiOp)$ in the Koszul deformation complex of~\S\ref{DeformationComplexes:KoszulDuality}
by the dual construction of this subsection~$\hat{\DGB}{}_{com}^*(\POp)$.
We give an explicit definition of a complex which involves this cochain complex~$\hat{\DGB}{}_{com}^*(\POp)$
and an extended version of the Koszul construction
of~\S\ref{DeformationComplexes:KoszulDuality:KoszulReduction}
first.
We explain the correspondence between this complex and our Koszul deformation complex afterwards.

\begin{constr}[Twisted end complexes]\label{GraphHomology:TwistedEndComplexes:TwistedEndComplexConstruction}
We assume that $\POp$ is a coaugmented complete Hopf $\Lambda$-operad in the sense
of~\S\ref{GraphHomology:TwistedEndComplexes:CompleteHopfOperads}.
We also assume that $\POp$ is equipped with a coaugmentation $\gamma_*: \PoisOp_m\rightarrow\POp$
where we regard the $m$-Poisson operad $\PoisOp_m$
as a coaugmented complete Hopf $\Lambda$-operad
equipped with a trivial filtration (see~\S\ref{GraphHomology:TwistedEndComplexes:CompleteHopfOperads}).

Recall that, in the Koszul construction $\DGK_{op}^*(\PoisOp_m^c)(r) = \DGSigma\overline{\SuspOp}^m\overline{\PoisOp}_m$,
the cochain grading $l$ and the arity $r$
are linked by the relation $l = r-2$.
If our operad $\POp$ has an arbitrary term (not necessarily reduced to the unit) in arity $r=1$,
then we need to consider an extended version of this operadic Koszul construction,
where we keep the unit term $\SuspOp^m\PoisOp_m(1) = \kk 1$
of the Koszul dual operad
of the $m$-Poisson cooperad $\KK^c_{op}(\PoisOp_m^c) = \SuspOp^m\PoisOp_m$.
We explicitly set $\tilde{\DGK}{}_{op}^*(\PoisOp_m^c) = \DGSigma\SuspOp^m\PoisOp_m$,
and we keep the rule $l = r-2$ to determine the grading of this cochain complex. (We accordingly have $l=-1$
for the extra term of arity $r=1$.)
We provide this object with a trivial $\Lambda$-collection structure.

We then consider the Harrison cochain complex $\hat{\DGB}{}_{com}^*(\POp)$ associated to $\POp$
and we form the double sequence of dg-modules:
\begin{equation}\label{GraphHomology:TwistedEndComplexes:TwistedEndComplexConstruction:Components}
E^{k l} = E^{k l}(\POp,\PoisOp_m^c) = \int_{\rset\in\Lambda}\hat{\DGB}{}_{com}^k(\POp(r))\otimes\tilde{\DGK}{}_{op}^l(\PoisOp_m^c)(r),
\end{equation}
for $k\geq 0$ and $l\geq -1$.
This double sequence inherits a first horizontal twisting differential $\partial'_h: E^{k l}\rightarrow E^{k+1 l}$,
which is yielded by the internal twisting differential of the Harrison complex~$\hat{\DGB}{}_{com}^*(\POp)$.
We use the morphism $\gamma_*: \PoisOp_m\rightarrow\POp$, the Lie structure of the free complete Lie algebra~$\hat{\LLie}(-)$
in the Harrison complex~$\hat{\DGB}{}_{com}^k(\POp)$,
and the module structure of the Koszul construction $\tilde{\DGK}{}_{op}^l(\PoisOp_m^c)$
in order to define an extra horizontal twisting homomorphism $\partial''_h: E^{k l}\rightarrow E^{k+1 l}$
which we add to the twisting differential of the Harrison complex~$\partial'_h: E^{k l}\rightarrow E^{k+1 l}$.

We proceed as follows.
We consider, for any $r>0$, the tensor $\varpi(r) = \sum_{\pi}\gamma_*(\pi)\otimes\pi^{\vee}$,
where the sum runs over the Poisson monomials such that $\pi\in\IOp\PoisOp_m(r)$
and $\pi^{\vee}\in\PoisOp_m(r)^{\vee}$
denotes the dual of this basis in the augmentation ideal $\IOp\PoisOp_m(r)^{\vee}$
of the algebra $\PoisOp_m^c(r) = \PoisOp_m(r)^{\vee}$
underlying the Poisson cooperad $\PoisOp_m^c$.
We define our twisting differential term-wise on the end of our definition.
For a tensor $\alpha\otimes\xi\in\hat{\DGB}{}_{com}^k(\POp(r))\otimes\tilde{\DGK}{}_{op}^l(\PoisOp_m^c)(r)$,
with $\alpha\in\hat{\DGB}{}_{com}^k(\POp(r)) = \DGSigma\hat{\LLie}_{k+1}(\DGSigma^{-1}\IOp\POp(r))$
and $\xi\in\tilde{\DGK}{}_{op}^l(\PoisOp_m^c)(r)$,
we explicitly set:
\begin{equation}\label{GraphHomology:TwistedEndComplexes:TwistedEndComplexConstruction:HorizontalTwistingDifferential}
\partial''_h(\alpha\otimes\xi) := \sum_{\pi}\pm[\gamma_*(\pi),\alpha]\otimes(\pi^{\vee}\cdot\xi),
\end{equation}
where we form the Lie bracket $[\gamma_*(\pi),\alpha]\in\DGSigma\hat{\LLie}_{k+1}(\DGSigma^{-1}\IOp\POp(r))$,
and we use the action of the commutative algebra $\PoisOp_m(r)^{\vee}$
on the dg-module $\tilde{\DGK}{}_{op}^l(\PoisOp_m^c)(r)$.
We easily check that this construction is compatible with the action of the category $\Lambda$
on each factor of our tensor product,
and hence, does yield a homomorphism on our end.

We then use the notation $\partial_h = \partial'_h+\partial''_h$
for the total horizontal twisting differential
which we associate to our object $E^{* *}$.

We also provide our double sequence with a vertical twisting differential $\partial_v = \partial''_v: E^{k l}\rightarrow E^{k l+1}$
which we determine from the underlying operad structure
of the Koszul construction $\tilde{\DGK}{}_{op}^*(\PoisOp_m^c) = \DGSigma\SuspOp^m\PoisOp_m$
and from the bimodule structure of the Harrison cochain complex $\hat{\DGB}{}_{com}^*(\POp)$.
We then consider the arity two component of our tensor $\varpi(2) = \gamma_*(\mu)\otimes\mu^{\vee} + \gamma_*(\lambda)\otimes\lambda^{\vee}$,
with one term associated to the commutative product operation in the $m$-Poisson operad $\mu\in\PoisOp_m(2)$,
and one term associated to the Lie bracket operation $\lambda\in\PoisOp_m(2)$.
We take the image of the factors $\mu^{\vee},\lambda^{\vee}\in\PoisOp_m^{\vee}(2)$
under the Koszul duality weak-equivalence $\kappa: \BB_{op}^c(\PoisOp_m^{\vee})\xrightarrow{\sim}\SuspOp^m\PoisOp_m$.
We explicitly have $\kappa(\mu^{\vee}) = \lambda$, $\kappa(\lambda^{\vee}) = \mu$, and we now set:
\begin{multline}\label{GraphHomology:TwistedEndComplexes:TwistedEndComplexConstruction:VerticalTwistingDifferential}
\partial''_v(\alpha\otimes\xi) :=
\sum_{i=1,2}\bigl[\pm(\gamma_*(\mu)\circ_i\alpha)\otimes(\lambda\circ_i\xi)
+ \pm(\gamma_*(\lambda)\circ_i\alpha)\otimes(\mu\circ_i\xi)\bigr]
\\
+ \sum_{i=1,\dots,r}\bigl[\pm(\alpha\circ_i\gamma_*(\mu))\otimes(\xi\circ_i\lambda)
+ \pm(\alpha\circ_i\gamma_*(\lambda))\otimes(\xi\circ_i\mu)\bigr],
\end{multline}
for any tensor $\alpha\otimes\xi\in\hat{\DGB}{}_{com}^k(\POp(r))\otimes\tilde{\DGK}{}_{op}^l(\POp)(r)$
as above.
We easily check, again, that this construction is compatible with the action of the category $\Lambda$
on each factor of our tensor product, and hence, does yield a homomorphism
on our end.

We readily check that these maps fulfill the defining relations of a differential
and the commutation relation $\partial_h\partial_v + \partial_v\partial_h = 0$.
We accordingly get that our double sequence~(\ref{GraphHomology:TwistedEndComplexes:TwistedEndComplexConstruction:Components})
forms a double complex $E^{* *} = E^{* *}(\POp,\PoisOp_m^c)$
naturally associated to~$\POp$.
\end{constr}

We also consider a natural generalization of the Koszul deformation complex of~\S\ref{DeformationComplexes:KoszulDuality:KoszulReduction}
to coaugmented filtered Hopf $\Lambda$-cooperads.
We just have to take the extended Koszul construction of~\S\ref{GraphHomology:TwistedEndComplexes:TwistedEndComplexConstruction}
(instead of the truncated object $\DGK_{op}^*(\PoisOp_n)$)
when we work in this setting.
We accordingly have:
\begin{equation*}
K^{* *} = K^{* *}(\PiOp,\PoisOp_m^c) = (\Hom_{\dg\Lambda\Seq^c}(\DGB^{com}_*(\PiOp),\tilde{\DGK}{}_{op}^*(\PoisOp_m^c)),\partial''_h+\partial''_v)
\end{equation*}
for a straightforward generalization of the twisting differentials of~\S\ref{DeformationComplexes:KoszulDuality:KoszulReduction}
(check the general definition of these twisting differentials in~\S\ref{DeformationComplexes:DGComplex:CooperadBicomodule}
and~\S\ref{DeformationComplexes:DGComplex:CommutativeAlgebraModule}).
We may simply note that this complex reduces to the Koszul deformation complex of~\S\ref{DeformationComplexes:KoszulDuality:KoszulReduction}
when $\PiOp$ is a plain augmented Hopf $\Lambda$-cooperad, because we have $\IOp\PiOp(1) = 0\Rightarrow\DGB^{com}_*(\PiOp(1))=0$
in this case.

We now have the following correspondence between these (extended) Koszul deformation complexes
and the twisted end complexes of the previous paragraphs:

\begin{thm}\label{GraphHomology:TwistedEndComplexes:MainResult}
We assume that the complete Hopf $\Lambda$-operad $\POp$ arises as the dual $\POp = \PiOp^{\vee}$
of an augmented filtered Hopf $\Lambda$-cooperad $\PiOp$ (as in Proposition~\ref{GraphHomology:TwistedEndComplexes:DualHopfOperad}).
We also assume that $\POp$ is equipped with an coaugmentation over the Poisson $m$-operad $\gamma_*: \PoisOp_m\rightarrow\POp$
which comes from an augmentation $\gamma: \PiOp\rightarrow\PoisOp_m^c$
associated to this filtered Hopf $\Lambda$-cooperad $\PiOp$.

We have, in this context, an isomorphism of bicomplexes of dg-modules between the twisted end complex
of~\S\ref{GraphHomology:TwistedEndComplexes:TwistedEndComplexConstruction}
and the Koszul deformation complex of~\S\ref{DeformationComplexes:KoszulDuality:KoszulReduction}:
\begin{equation*}
E^{* *}(\PiOp^{\vee},\PoisOp_m^c)
\xrightarrow{\simeq}
K^{* *}(\PiOp,\PoisOp_m^c).
\end{equation*}
\end{thm}

\begin{proof}
For any pair $(k,l)$, we have an isomorphism of ends:
\begin{align*}
\int_{\rset\in\Lambda}\hat{\DGB}{}_{com}^k(\PiOp(r)^{\vee})\otimes\tilde{\DGK}{}_{op}^l(\PoisOp_m^c)(r)
\xrightarrow{\simeq}\int_{\rset\in\Lambda}\DGB^{com}_k(\PiOp(r))^{\vee}\otimes\tilde{\DGK}{}_{op}^l(\PoisOp_m^c)(r)
\\
\xrightarrow{\simeq}\int_{\rset\in\Lambda}\Hom_{\dg\Mod}(\DGB^{com}_k(\PiOp(r)),\tilde{\DGK}{}_{op}^l(\PoisOp_m^c)(r)),
\end{align*}
which gives the relation $E^{k l}(\PiOp^{\vee},\PoisOp_m^c)\xrightarrow{\simeq}K^{k l}(\PiOp,\PoisOp_m^c)$,
and which we deduce from the duality relation $\hat{\DGB}{}_{com}^*(\PiOp(r)^{\vee}) = \DGB^{com}_*(\PiOp)^{\vee}$
of Proposition~\ref{GraphHomology:TwistedEndComplexes:DualHarrisonComplex}
together with the observation that the Koszul construction $\tilde{\DGK}{}_{op}^*(\PoisOp_m^c)$
consists of modules of finite rank in each arity $r>0$ (check the definition of~\S\ref{DeformationComplexes:KoszulDuality}).
We easily check that this isomorphism carries the twisting differentials of the bicomplex $E^{* *} = E^{* *}(\PiOp^{\vee},\PoisOp_m^c)$
to the twisting differentials of the Koszul deformation bicomplex $K^{* *} = K^{* *}(\PiOp,\PoisOp_m^c)$.
We therefore get the conclusion of the theorem.
\end{proof}

In~\S\S\ref{DeformationComplexes:DGComplex}-\ref{DeformationComplexes:KoszulDuality},
we use that the Harrison chain complex $\DGB^{com}_*(\PiOp)$ has a free structure as a $\Lambda$-collection
in order to prove the validity of our reduction processes. We briefly recall the definition of dual cofree structures
in the category of (contravariant) $\Lambda$-collections
in the next paragraph.

\begin{rem}[Cofree $\Lambda$-structures]\label{GraphHomology:TwistedEndComplexes:CofreeLambdaStructures}
We first consider the category formed by collections $\MOp = \{\MOp(r),r>0\}$
whose terms are dg-modules $\MOp(r)\in\dg\Mod$
equipped with an action of the symmetric groups $\Sigma_r$,
for $r>0$.
We can identify this category of symmetric collections with the category of covariant diagrams $\MOp$
over the category $\Sigma = \coprod_r\Sigma_r$
which have a trivial term $\MOp(0) = 0$
in arity $r=0$.

In parallel, we consider the category, underlying our category of augmented $\Lambda$-operads,
whose objects $\MOp$ are contravariant diagrams over the category $\Lambda$
with a trivial term $\MOp(0) = 0$
in arity $r=0$ (as in the case of symmetric collections).
We say that an object of this category $\MOp$ is cofreely cogenerated by a symmetric collection $\SOp\MOp$
when we have the end formula:
\begin{equation}\label{GraphHomology:TwistedEndComplexes:CofreeLambdaStructures:EndExpression}
\MOp(r) = \int_{\kset\in\Sigma} \SOp\MOp(k)^{\Mor_{\Lambda}(\kset,\rset)},
\end{equation}
for any arity $r>0$, where we use the notation $\SOp\MOp(k)^S$ for the product of copies of the object $\SOp\MOp(k)$
over the set $S = \Mor_{\Lambda}(\kset,\rset)$
and we assume that $\MOp$ inherits the action of the category $\Lambda$ by left translation
over these morphism sets $S = \Mor_{\Lambda}(\kset,\rset)$.
The collection $\SOp\MOp$ is identified with a quotient object of $\MOp$ (in the category of symmetric collections)
when we have such a decomposition.

This concept applies to the $n$-Poisson operad $\PoisOp_n$, for any $n\geq 2$ (and actually, for any $n\geq 1$).
To be more precise, we observed in~\S\ref{Background:PoissonAugmentationIdeal}
that the augmentation ideals $\IOp\PoisOp_n^c(r)$
of the commutative algebras $\PoisOp_n^c(r)$
underlying the $n$-Poisson cooperad $\PoisOp_n^c$
form a free $\Lambda$-collection
over a symmetric sequence such that $\SOp\PoisOp_n^c\subset\IOp\PoisOp_n^c$.
In the case of the $n$-Poisson operad, we dually get that the coaugmentation coideals $\IOp\PoisOp_n(r)$
of the cocommutative coalgebras $\PoisOp_n(r)$
form a cofree (contravariant) $\Lambda$-collection in the sense of the previous definition.
The cogenerating symmetric sequence $\SOp\PoisOp_n$ underlying $\IOp\PoisOp_n$
is just the dual of the symmetric sequence considered in~\S\ref{Background:PoissonFreeStructure}
and consists of monomials $\pi(x_1,\dots,x_r) = \pi_1(x_{1 j_1},\dots,x_{j_{1 n_1}})\cdot\ldots\cdot\pi_s(x_{s j_1},\dots,x_{s j_{n_s}})$
whose factors $\pi_i(x_{i j_1},\dots,x_{i j_{n_i}})$, $i = 1,\dots,s$,
are Lie monomials of weight $n_i>1$.
\end{rem}

\begin{remark}\label{GraphHomology:TwistedEndComplexes:DualCofreeLambdaStructures}
We may extend the duality correspondence between the cofree structure of the $\Lambda$-collection $\IOp\PoisOp_n^c$
underlying the $n$-Poisson cooperad $\PoisOp_n^c$
and the free structure of the $\Lambda$-collection $\IOp\PoisOp_n$
underlying the $n$-Poisson operad $\PoisOp_n$
in the setting of coaugmented complete Hopf $\Lambda$-operads.

We then consider a coaugmented Hopf $\Lambda$-operad $\POp$ dual to an augmented filtered Hopf $\Lambda$-cooperad $\PiOp$
in the sense of the construction of Proposition~\ref{GraphHomology:TwistedEndComplexes:DualHopfOperad}.
We also consider an obvious generalization, for general covariant $\Lambda$-collections,
of the coend construction of~\S\ref{Background:AlgebraicAdjunctions}.
We easily see that the collection of coaugmentation coideals $\IOp\POp(r) = \IOp\PiOp(r)^{\vee}$
underlying the coalgebras $\POp(r) = \PiOp(r)^{\vee}$
has the structure of a cofree $\Lambda$-collection
if and only if the collection of augmentation ideals $\IOp\PiOp(r)$ underlying the algebras $\PiOp(r)$
forms a free $\Lambda$-collection,
in the sense that we have the relation $\IOp\PiOp = \Lambda\otimes_{\Sigma}\SOp\PiOp$
for some symmetric collection such that $\SOp\PiOp\subset\IOp\PiOp$.
We explicitly have the duality relations:
\begin{equation*}
\IOp\PiOp(r) = (\Lambda\otimes_{\Sigma}\SOp\PiOp)(r) = \int^{\kset\in\Sigma}\Mor_{\Lambda}(\kset,\rset)\otimes\SOp\PiOp(k)
\Leftrightarrow\IOp\POp(r) = \int_{\kset\in\Sigma} \SOp\POp(k)^{\Mor_{\Lambda}(\kset,\rset)},
\end{equation*}
for the symmetric collection in complete dg-modules such that $\SOp\POp(r) = \SOp\PiOp(r)^{\vee}$.
\end{remark}

We use such free structures in the proof of the following homotopy invariance statement:

\begin{prop}\label{GraphHomology:TwistedEndComplexes:TwistedEndComplexHomotopyInvariance}
Let $\phi: \AOp\rightarrow\BOp$ be a morphism of augmented filtered Hopf $\Lambda$-cooperads.
We assume that $\AOp$ and $\BOp$ are both equipped with an augmentation
over the $m$-Poisson cooperad $\PoisOp_m^c$,
and that $\phi$ preserves this augmentation.

We then consider the generalized Koszul deformation complex associated to these augmented filtered Hopf $\Lambda$-cooperads.
If $\phi$ is a weak-equivalence, then the morphism induced by $\phi$ on these complexes
defines a weak-equivalence as well:
\begin{equation}
\phi^*: K^{* *}(\BOp,\PoisOp_m^c)\xrightarrow{\sim}K^{* *}(\AOp,\PoisOp_m^c)
\end{equation}
provided that the $\Lambda$-collections $\IOp\AOp$ and $\IOp\BOp$ underlying our augmented cooperads $\AOp$ and $\BOp$
admit a free structure in the sense recalled in~\S\ref{GraphHomology:TwistedEndComplexes:DualCofreeLambdaStructures}
(but we do not require that our morphism $\phi$ preserves the generating symmetric collections of our objects).
\end{prop}

\begin{proof}
We observed in the proof of Theorem~\ref{DeformationComplexes:DGComplex:MainResult} (see also Theorem~\ref{DeformationComplexes:KoszulDuality:MainResult})
that the components of the Harrison chain complex $\DGB^{com}_k(\PiOp)$
associated to an augmented Hopf $\Lambda$-cooperad $\PiOp$
form a free $\Lambda$-collection when the $\Lambda$-collection $\IOp\PiOp$
underlying our Hopf $\Lambda$-cooperad $\PiOp$
does so.
We have a straightforward generalization of this structure result in the setting of filtered Hopf $\Lambda$-cooperads.
We moreover get in this situation that $\DGB^{com}_k(\PiOp)$ forms a cofibrant
object in the projective model category of $\Lambda$-collections
in dg-modules.

In the context of our proposition, we first use that the cofree Lie coalgebra functor $\LLie^c(-)$
preserves weak-equiva\-len\-ces (when we work over a field of characteristic zero)
in order to check that our weak-equivalence $\phi: \AOp\xrightarrow{\sim}\BOp$
induces a weak-equivalence
on each component of the Harrison chain complex $\phi_*: \DGB^{com}_k(\AOp(r))\xrightarrow{\sim}\DGB^{com}_k(\BOp(r))$,
and for any arity $r>0$.
Then we use our structure results on this complex and standard model category arguments
in order to check that this morphism induces a weak-equivalence
on the hom-objects which define the components
of the Koszul deformation complex
$\phi^*: \Hom_{\dg\Lambda\Seq^c}(\DGB^{com}_k(\BOp),\tilde{\DGK}{}_{op}^l(\PoisOp_m^c))
\xrightarrow{\sim}\Hom_{\dg\Lambda\Seq^c}(\DGB^{com}_k(\AOp),\tilde{\DGK}{}_{op}^l(\PoisOp_m^c))$,
for all $k\geq 0$ and $l\geq -1$, and we use a spectral sequence argument
to get the conclusion of our theorem.
\end{proof}

\begin{remark}\label{GraphHomology:TwistedEndComplexes:HarrisonConstructionCofreeLambdaStructure}
We use in this proof that the components of the Harrison chain complex $\DGB^{com}_k(\PiOp(r))$ of an augmented Hopf $\Lambda$-operad $\PiOp$
admit a free structure as a $\Lambda$-collection when the $\Lambda$-collection $\IOp\PiOp$ underlying $\PiOp$
does so.
We also use that this observation admits a straightforward generalization
in the setting of augmented filtered Hopf $\Lambda$-cooperads.
We can use the duality correspondence of Proposition~\ref{GraphHomology:TwistedEndComplexes:DualHarrisonComplex}
to get an equivalent structure result for the Harrison cochain complex $\hat{\DGB}{}_{com}^*(\POp)$
associated to the dual complete Hopf $\Lambda$-operad $\POp = \PiOp^{\vee}$
of our cooperad $\PiOp$.
We then have the equivalence:
\begin{multline*}
\DGB^{com}_k(\PiOp) = \int^{\rset\in\Sigma}\Mor_{\Lambda}(\rset,-)\otimes\DGS\DGB^{com}_k(\PiOp)(r)
\\
\Leftrightarrow\hat{\DGB}{}_{com}^k(\POp) = \int_{\rset\in\Sigma}\DGS\hat{\DGB}{}_{com}^k(\POp)(r)^{\Mor_{\Lambda}(\rset,-)},
\end{multline*}
for any degree $k\in\NN$, where we consider the symmetric collection $\SOp\hat{\DGB}{}_{com}^k(\POp)$
such that $\DGS\hat{\DGB}{}_{com}^k(\POp)(r) = \DGS\DGB^{com}_k(\PiOp)(r)^{\vee}$,
for any $r>0$.

We use this formula in the next subsection in order to simplify the expression
of the twisted end complexes which arise from the duality relation
of Theorem~\ref{GraphHomology:TwistedEndComplexes:MainResult}.
We mainly use that, in degree zero, we have the relation $\DGS\hat{\DGB}{}_{com}^0(\POp) = \SOp\POp$
which arises from the effective definition of the symmetric collection $\DGS\DGB^{com}_k(\PiOp)$
underlying $\DGB^{com}_k(\PiOp)$ (see the proof of Theorem~\ref{DeformationComplexes:DGComplex:MainResult}).
\end{remark}

\subsection{Graph complexes and graph operads}\label{GraphHomology:GraphComplexes}
The purpose of this subsection is to recall the definition of the graph complexes which we use to compute
the homology of the Koszul deformation complexes $K^{* *}_{m n} = K^{* *}(\PoisOp_n^c,\PoisOp_m^c)$
and $L^{* *}_n = K^{* *}(\PoisOp_n^c,\PoisOp_n^c)$.
We review the definition of a first graph complex, denoted by $\GCOp_n$, which is related to the complex $L^{* *}_n$
associated to the identity morphism
of the $n$-Poisson cooperad $\id: \PoisOp_n^c\rightarrow\PoisOp_n^c$.
We deal with a variant of this graph complex, which we call the hairy graph complexes and denote by $\HGCOp_{m n}$,
in order to compute the homology of the Koszul deformation complexes $K^{* *}_{m n}$
of the morphisms $\iota^*: \PoisOp_n^c\rightarrow\PoisOp_m^c$.
We also recall the definition of the operads of graphs $\GraphOp_n$, weakly-equivalent to the Poisson operads $\PoisOp_n$,
which we use to get these graphical reductions of our Koszul deformation complexes.

We only give a short outline of the definition of these objects in this subsection.
We refer the reader to~\cite[\S 3]{WillwacherGraphs},
besides Kontsevich's initial papers~\cite{KontsevichFormalityConj,KontsevichMotives},
for further details.
We start with the definition of a basic graph operad $\GraOp_n$
from which we derive the construction
of our other objects.
The operad of graphs $\GraphOp_n$ actually forms a complete Hopf $\Lambda$-operad in the sense
of~\S\ref{GraphHomology:TwistedEndComplexes:CompleteHopfOperads},
and is dual to a filtered Hopf $\Lambda$-cooperad
in the sense of~\S\ref{GraphHomology:TwistedEndComplexes:FilteredHopfCooperads}.
In what follows, we generally do not make explicit the filtered Hopf $\Lambda$-cooperads which we associate to our objects
(we can use the observations of~\S\ref{GraphHomology:TwistedEndComplexes:DualHopfOperad} to retrieve them).

\begin{rem}[The plain operad of graphs]\label{GraphHomology:GraphComplexes:PlainGraphOperad}
We first denote by $\gra_{r k}$ the set of directed graphs with vertex set $\rset = \{1,\dots,r\}$ and edge set $\kset = \{1,\dots,k\}$.
We use that the group $\Sigma_r\times\Sigma_k\ltimes\Sigma_2^k$ acts on this set of graphs by permuting vertices,
edge labels and changing the edge directions.
We then consider an operad $\GraOp_n$, $n\geq 2$, such that:
\begin{equation}\label{GraphHomology:GraphComplexes:PlainGraphOperad:Expansion}
\GraOp_n(r) = \prod_{k\geq 0}(\kk\langle\gra_{r k}\rangle\otimes\kk\ecell_{n-1}^{\otimes k})_{\Sigma_k\ltimes\Sigma_2^k},
\end{equation}
for any arity $r>0$, where $\kk\langle\gra_{r k}\rangle$ is the free module spanned by the set $\gra_{r k}$
while $\ecell_{n-1}$ denotes a homogeneous element of lower degree $*={n-1}$
and we consider the ($k$-fold tensor product of the) graded module
of rank one
spanned by this element.
The action of $\Sigma_k$ on $\kk\langle\gra_{r k}\rangle\otimes\kk\ecell_{n-1}^{\otimes k}$ is twisted by a sign when $n$ is even,
while the action of $\Sigma_2^k$ carries a sign
when $n$ is odd.
The following picture shows a typical basis element of $\GraOp_n(5)$:
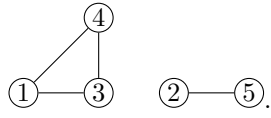
\begin{equation*}
\begin{tikzpicture}{scale=.5}
\node[ext] (v) at (0,0) {1};
\node[ext] (w) at (1,1) {4};
\node[ext] (u) at (1,0) {3};
\node[ext] (x) at (2,0) {2};
\node[ext] (y) at (3,0) {5};
\draw (v) edge (w) edge (u) (u) edge (w) (x) edge (y);
\end{tikzpicture}.
\end{equation*}
The operadic composite $\alpha\circ_i\beta\in\GraOp_n(k+l-1)$ of such basis elements $\alpha\in\GraOp_n(k)$, $\beta\in\GraOp_n(l)$,
is defined, for any $i = 1,\dots,k$,
by plugging the graph $\beta\in\GraOp_n(l)$ in the $i$th vertex of $\alpha\in\GraOp_n(k)$
and reconnecting the adjacent edges of this vertex in the graph $\alpha$
to vertices of the graph $\beta$ in all possible ways (we just sum over all reconnections).
We refer to~\cite[\S 3, pp. 679-680]{WillwacherGraphs} for further details on this process.

We can also provide $\GraOp_n$ with the structure of an augmented $\Lambda$-operad.
The augmentation $\epsilon: \GraOp_n(r)\rightarrow\kk$ carries the fully disconnected graph (where we only have isolated vertices) to $1$
and is zero otherwise.
The restriction operator $u^*: \GraOp_n(l)\rightarrow\GraOp_n(k)$ associated to any injective map $u: \{1<\dots<k\}\rightarrow\{1<\dots<l\}$
is obtained by removing the vertices labeled by indices $j\not=\{u(1),\dots,u(k)\}$
in a graph, and by the obvious renumbering operation $j\mapsto u^{-1}(j)$
on the remaining vertices.
We just assume that this map $u^*: \alpha\mapsto u^*(\alpha)$ vanishes when our removal operation
involves vertices with a non-empty set of incident edges
inside $\alpha$.

The signs and degrees of the graph operad are chosen so that, for any $n\geq 2$,
we have an operad morphism $\gamma_*: \PoisOp_n\rightarrow\GraOp_n$,
which carries the product element $\mu(x_1,x_2) = x_1 x_2\in\PoisOp_n(2)$
to the discrete graph
in $\GraOp_n(2)$
and the Lie bracket element $\lambda(x_1,x_2) = [x_1,x_2]\in\PoisOp_n(2)$
to the one edge graph:
\begin{equation}\label{GraphHomology:GraphComplexes:PlainGraphOperad:Mapping}
\gamma_*(\mu) := \vcenter{\xymatrix@!0@C=1.5em@R=1em@M=0pt{ *+<6pt>[o][F]{1} & *+<6pt>[o][F]{2} }},
\qquad\gamma_*(\lambda) := \vcenter{\xymatrix@!0@C=1.5em@R=1em@M=0pt{ *+<6pt>[o][F]{1}\ar@{-}[r] & *+<6pt>[o][F]{2} }}.
\end{equation}
\end{rem}

\begin{rem}[Graph complexes]\label{GraphHomology:GraphComplexes:GraphComplexConstruction}
The full graph complex, denoted by $\FGCOp_n$, is a graded module defined by twisted invariants of the operad $\GraOp_n$.
We more explicitly have:
\begin{equation}\label{GraphHomology:GraphComplexes:GraphComplexConstruction:Expression}
\FGCOp_n := \prod_{l\geq 2}\DGSigma^n(\GraOp_n(l)\otimes\kk\ecell_{-n}^{\otimes l})^{\Sigma_l},
\end{equation}
where we still use the notation $\kk\ecell_{-n}$ for the graded module of rank one spanned by an element $\ecell_{-n}$ of (lower) degree $-n$,
while $\DGSigma$ denotes the suspension functor on the category of dg-modules (see the introduction of~\S\ref{DeformationComplexes:DGComplex}).
We can represent the elements of this module $\FGCOp_n$ as formal series of graphs with unidentifiable (non-numbered) vertices.
We will indicate this in drawings by filling the vertices black.

We have a natural Lie algebra structure on $\FGCOp_n$, which we define by taking the commutator of pre-Lie composition operations
in the operad of graphs (see for instance~\cite{DolgushevTwisting}).
We may check that the graph
\begin{equation}\label{GraphHomology:GraphComplexes:GraphComplexConstruction:Picture}
m=
\begin{tikzpicture}[scale=.5]
\node[int](u) at (0,0) {};
\node[int](v) at (1,0) {};
\draw (u) edge (v);
\end{tikzpicture}
\end{equation}
defines a Maurer-Cartan element in this graded Lie algebra $\FGCOp_n$. The differential of the complex $\FGCOp_n$
is given by the Lie bracket with this Maurer-Cartan element $\delta = [m,-]$.
Basically, the differential of a graph in the complex $\FGCOp_n$ can be obtained by splitting any vertex of this graph
into two vertices connected by an edge and by reconnecting the adjacent edges of this vertex
in the original graph in all possible ways. We just sum over these reconnections
and over the set of vertices
of our graph
when we perform this boundary operation (see~\cite[Remark 3.3]{WillwacherGraphs} for further details on this process).

In what follows, we mainly deal with sub-complexes of the full graph complex $\GCOp_n\subset\GCOp_n^2\subset\FGCOp_n$,
where $\GCOp_n$ consists of connected graphs whose vertices are at least trivalent,
and $\GCOp_n^2$ consists of connected graphs whose vertices
are at least bivalent.
\end{rem}

The relations between the homology of the full graph complex $\FGCOp_n$
and the homology of these sub-complexes of connected graphs
is studied in~\cite[\S 3, pp. 682-685]{WillwacherGraphs}.
The most significant outcome of this study, for our purpose, is the result of the following proposition:

\begin{prop}[{see~\cite[Proposition 3.4]{WillwacherGraphs}}]\label{GraphHomology:GraphComplexes:GraphComplexHomology}
We have
\begin{equation*}
\DGH_*(\GCOp_n^2) = \DGH_*(\GCOp_n)\oplus\bigoplus_{l\equiv 2n+1(\mymod 4)}\kk\gamma_l
\end{equation*}
where the class $\gamma_l$, of (lower) degree $n-l$, is represented by the $l$-loop graph
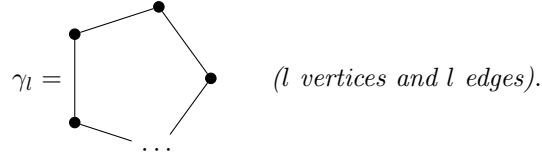
\begin{equation*}
\gamma_l = \begin{tikzpicture}[baseline=-.65ex]
\node[int] (v1) at (0:1) {};
\node[int] (v2) at (72:1) {};
\node[int] (v3) at (144:1) {};
\node[int] (v4) at (216:1) {};
\node (v5) at (-72:1) {$\cdots$};
\draw (v1) edge (v2) edge (v5) (v3) edge (v2) edge (v4) (v4) edge (v5);
\end{tikzpicture}
\qquad\text{($l$ vertices and $l$ edges)}.
\end{equation*}
Moreover, the module $\DGH_*(\GCOp_n)$ vanishes in degree $*<n$ when $n\geq 3$.
\end{prop}

\begin{proof}
We refer to the cited reference for the first assertion of this proposition.
We only check the vanishing statement.
Let $l$ be the number of vertices and $k$ be the number of edges of a graph $\alpha$.
The trivalence assumption implies $2 k\geq 3 l$.
The degree of our graph $\alpha$ in the complex $\GCOp_n$
therefore satisfies the relation
\begin{equation*}
\deg(\alpha) = k(n-1) - (l-1)n\geq\frac{3}{2}l(n-1) - (l-1)n = \frac{l}{2}(n-3) + n\geq n
\end{equation*}
as soon as we assume $n\geq 3$. Hence, the complex $\GCOp_n$ is concentrated in degrees $*\geq n$.
The conclusion of the proposition follows.
\end{proof}



\begin{rem}[Graphical models of the $n$-Poisson operad]\label{GraphHomology:GraphComplexes:GraphOperads}
We use an operadic twisting process (see \cite{DolgushevTwisting}) to construct operads $\GraphOp_n$, $\GraphOp_n^2$, and $\FGraphOp_n$
from the operad of graphs $\GraOp_n$.
We briefly survey the main features of these objects in this paragraph (we refer to~\cite[\S 3.2]{WillwacherGraphs}
for further details on the definition of these operads and cooperads).

The modules $\FGraphOp_n(r)$ interpolate between the components of the operad $\GraOp_n$
and the full graph complex $\FGCOp_n$ (where we also forget about the front suspension).
We formally have:
\begin{equation}\label{GraphHomology:GraphComplexes:GraphOperads:Expression}
\FGraphOp_n(r) := \prod_{l\geq 0}(\GraOp_n(r+l)\otimes\kk\ecell_{-n}^{\otimes l})^{\Sigma_l},
\end{equation}
for any arity $r>0$, where we consider the action of the group $\Sigma_l$
on the vertices $r+1,\dots,r+l$
of graphs $\alpha\in\GraOp_n(r+l)$.
We can identify the elements of this module with formal sums of graphs with $r$ numbered (``external'') vertices,
together with an arbitrary number of un-identifiable (``internal'') vertices
which we color in black as in the following picture:
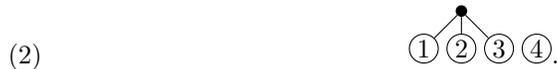
\begin{equation}\label{GraphHomology:GraphComplexes:GraphOperads:Picture}
\begin{tikzpicture}[scale=.5]
\node[ext] (u) at (0,0) {1};
\node[ext] (v) at (1,0) {2};
\node[ext] (w) at (2,0) {3};
\node[ext] (x) at (3,0) {4};
\node[int] (y) at (1,1) {};
\draw (y) edge (u) edge (v) edge (w);
\end{tikzpicture}.
\end{equation}
We now define $\GraphOp_n(r)$ (respectively, $\GraphOp_n^2(r)$) as the submodule of this object $\FGraphOp_n(r)$
that consists of graphs whose connected components
contain at least one external vertex,
and where each internal vertex is at least trivalent (respectively, bivalent).
The differential of these dg-modules $\GraphOp_n(r)$, $\GraphOp_n^2(r)$ and $\FGraphOp_n(r)$
is given by a natural vertex splitting process.
We schematically have:
\begin{equation}\label{GraphHomology:GraphComplexes:GraphOperads:Differential}
\delta\;\vcenter{\xymatrix@!0@C=1em@R=1em@M=0pt{ \ar@{-}[drr] & \ar@{-}[dr] & *{\cdots} & \ar@{-}[dl] & \ar@{-}[dll] \\
&& *{\bullet} && \\
\ar@{-}[urr] & \ar@{-}[ur] & *{\cdots} & \ar@{-}[ul] & \ar@{-}[ull] }}
\;=\;\vcenter{\xymatrix@!0@C=1em@R=1em@M=0pt{ \ar@{-}[drr] & \ar@{-}[dr] & *{\cdots} & \ar@{-}[dl] & \ar@{-}[dll] \\
&& *{\bullet}\ar@{-}[d] && \\
&& *{\bullet} && \\
\ar@{-}[urr] & \ar@{-}[ur] & *{\cdots} & \ar@{-}[ul] & \ar@{-}[ull] }}
\qquad\text{and}
\qquad\delta\;\vcenter{\xymatrix@!0@C=1em@R=1.3em@M=0pt{ \ar@{-}[drr] & \ar@{-}[dr] & *{\cdots} & \ar@{-}[dl] & \ar@{-}[dll] \\
&& *+<6pt>[o][F]{i} && }}
\;=\;\vcenter{\xymatrix@!0@C=1em@R=1.3em@M=0pt{ & \ar@{-}[dr] & *{\cdots} & \ar@{-}[dl] & \\
\ar@{-}[drr] & *{\cdots} & *{\bullet}\ar@{-}[d] & *{\cdots} & \ar@{-}[dll] \\
&& *+<6pt>[o][F]{i} && }}.
\end{equation}

We provide the collections $\GraphOp_n$, $\GraphOp_n^2$, and $\FGraphOp_n$ with an obvious extension of the composition operations
of the plain operad of graphs $\GraOp_n$.
We also use an obvious extension of the restriction operators of this operad $\GraOp_n$
to provide our objects with the structure
of an augmented $\Lambda$-operad.

We equip each dg-module $\FGraphOp_n(r)$ with the (descending) filtration by the number of edges.
We readily have $\FGraphOp_n(r) = \lim_s\FGraphOp_n(r)/\DGF_s\FGraphOp_n(r)$,
so that our object forms a complete dg-module in the sense of~\S\ref{GraphHomology:TwistedEndComplexes:CompleteModules}.
We can moreover provide $\FGraphOp_n(r)$ with the structure of a coaugmented commutative dg-coalgebra
in complete dg-modules.
We accordingly get that $\FGraphOp_n$ forms a complete Hopf $\Lambda$-operad.
We have a similar result in the case of the operads $\GraphOp_n$ and $\GraphOp_n^2$.
We explicitly define the coproduct of this coalgebra $\FGraphOp_n(r)$
by splitting graphs into disjoint unions of subgraphs
of internal vertices, and keeping a copy of the external vertices
in each component of this decomposition,
as in the following example:
\begin{equation}\label{GraphHomology:GraphComplexes:GraphOperads:Coproduct}
\begin{aligned}[t] \Delta\bigl(\;\begin{tikzpicture}[scale=.5, baseline=-.65ex]
\node[ext] (u) at (0,0) {1};
\node[ext] (v) at (1.25,0) {2};
\node[ext] (w) at (2.5,0) {3};
\node[ext] (x) at (3.75,0) {4};
\node[int] (y) at (1.25,1) {};
\draw (y) edge (u) edge (v) edge (w) (w) edge (x);
\end{tikzpicture}\;\bigr)
\;=\; & \begin{tikzpicture}[scale=.5, baseline=-.65ex]
\node[ext] (u) at (0,0) {1};
\node[ext] (v) at (1.25,0) {2};
\node[ext] (w) at (2.5,0) {3};
\node[ext] (x) at (3.75,0) {4};
\node[int] (y) at (1.25,1) {};
\draw (y) edge (u) edge (v) edge (w) (w) edge (x);
\end{tikzpicture}
\;\hat{\otimes}\;\begin{tikzpicture}[scale=.5, baseline=-.65ex]
\node[ext] (u) at (0,0) {1};
\node[ext] (v) at (1.25,0) {2};
\node[ext] (w) at (2.5,0) {3};
\node[ext] (x) at (3.75,0) {4};
\end{tikzpicture}
\\
\;\pm\; & \begin{tikzpicture}[scale=.5, baseline=-.65ex]
\node[ext] (u) at (0,0) {1};
\node[ext] (v) at (1.25,0) {2};
\node[ext] (w) at (2.5,0) {3};
\node[ext] (x) at (3.75,0) {4};
\node[int] (y) at (1.25,1) {};
\draw (y) edge (u) edge (v) edge (w);
\end{tikzpicture}
\;\hat{\otimes}\;\begin{tikzpicture}[scale=.5, baseline=-.65ex]
\node[ext] (u) at (0,0) {1};
\node[ext] (v) at (1.25,0) {2};
\node[ext] (w) at (2.5,0) {3};
\node[ext] (x) at (3.75,0) {4};
\draw (w) edge (x);
\end{tikzpicture}
\\
\;\pm\; & \begin{tikzpicture}[scale=.5, baseline=-.65ex]
\node[ext] (u) at (0,0) {1};
\node[ext] (v) at (1.25,0) {2};
\node[ext] (w) at (2.5,0) {3};
\node[ext] (x) at (3.75,0) {4};
\draw (w) edge (x);
\end{tikzpicture}
\;\hat{\otimes}\;\begin{tikzpicture}[scale=.5, baseline=-.65ex]
\node[ext] (u) at (0,0) {1};
\node[ext] (v) at (1.25,0) {2};
\node[ext] (w) at (2.5,0) {3};
\node[ext] (x) at (3.75,0) {4};
\node[int] (y) at (1.25,1) {};
\draw (y) edge (u) edge (v) edge (w);
\end{tikzpicture}
\\
\;+\; & \begin{tikzpicture}[scale=.5, baseline=-.65ex]
\node[ext] (u) at (0,0) {1};
\node[ext] (v) at (1.25,0) {2};
\node[ext] (w) at (2.5,0) {3};
\node[ext] (x) at (3.75,0) {4};
\end{tikzpicture}
\;\hat{\otimes}\;\begin{tikzpicture}[scale=.5, baseline=-.65ex]
\node[ext] (u) at (0,0) {1};
\node[ext] (v) at (1.25,0) {2};
\node[ext] (w) at (2.5,0) {3};
\node[ext] (x) at (3.75,0) {4};
\node[int] (y) at (1.25,1) {};
\draw (y) edge (u) edge (v) edge (w) (w) edge (x);
\end{tikzpicture}.
\end{aligned}
\end{equation}

The dg-modules $\GraphOp_n(r)$ actually form cofree cocommutative coalgebras in complete dg-modules (when we forget about differentials).
The cogenerating module of this cofree cocommutative coalgebra is defined by the suspension of a complex of internally connected graphs $\ICGOp_n(r)$
which consists of graphs that are connected after deleting all external vertices (see~\cite[\S 3.2, p. 689]{WillwacherGraphs}).
Let us observe that the subquotients of our filtration on $\GraphOp_n(r)$
form modules of finite rank over the ground field degree-wise, and this observation
implies that we can associate $\GraphOp_n$
to a filtered Hopf $\Lambda$-cooperad  $\GraphOp_n^c$
such that $\GraphOp_n(r) = \GraphOp_n^c(r)^{\vee}$
for any $r>0$, where we use the duality functor on dg-modules (see Proposition~\ref{GraphHomology:TwistedEndComplexes:DualHopfOperad}).
Then we have an identity of graded commutative algebras $\GraphOp_n^c(r)_{\flat} = \Sym(\DGSigma^{-1}\ICGOp_n^c(r)_{\flat})$,
for each $r>0$,
where we use the subscript $\flat$ to mark the forgetting of differentials, and we again use the notation $\ICGOp_n^c(r)$
for a dg-module such that $\ICGOp_n(r) = \ICGOp_n^c(r)^{\vee}$.

We have a natural morphism of complete Hopf $\Lambda$-operads $\gamma_*: \PoisOp_n\rightarrow\GraphOp_n$, given by an obvious extension
of the mapping
of~\S\ref{GraphHomology:GraphComplexes:PlainGraphOperad}(\ref{GraphHomology:GraphComplexes:PlainGraphOperad:Mapping}).
We then regard the $n$-Poisson operad $\PoisOp_n$ as a complete Hopf $\Lambda$-operad with the filtration such that $\DGF_s\PoisOp_n(r)\subset\PoisOp_n(r)$
is the module spanned by the products of Lie monomials $\pi(x_1,\dots,x_r) = \pi_1(x_{1 j_1},\dots,x_{j_{1 n_1}})\cdot\ldots\cdot\pi_l(x_{l j_1},\dots,x_{l j_{n_l}})$
which satisfy $n_1+\dots+n_l-l\leq s$ in the representation of~\S\ref{Background:PoissonFreeStructure}.
This morphism $\gamma_*: \PoisOp_n\rightarrow\GraphOp_n$ is a weak-equivalence by a theorem of Maxim Kontsevich \cite{KontsevichMotives} (see also \cite{LambrechtsVolic})
and so is the natural inclusion $\GraphOp_n\subset\GraphOp_n^2$
(while we have to remove graphs with connected components without external vertices in order to retrieve an object weakly-equivalent
to the $n$-Poisson operad from the full operad of graphs $\FGraphOp_n$).

We can also consider a restriction of our morphism $\gamma_*: \PoisOp_n\rightarrow\GraphOp_n$
to the commutative operad $\ComOp$
which we identify with the suboperad of $\PoisOp_n$
spanned by the commutative product monomial $\mu(x_1,\dots,x_r) = x_1\cdot\ldots\cdot x_r$
in each arity $r>0$. The components of this operad morphism define coaugmentations $\eta: \kk\rightarrow\GraphOp_n(r)$
of the coalgebras $\GraphOp_n(r)$.
We moreover have:
\begin{equation}\label{GraphHomology:GraphComplexes:GraphOperads:CoaugmentationCoideal}
\GraphOp_n(r) = \ComOp(r)\oplus\IOp\GraphOp_n(r),
\end{equation}
for any $r>0$, where we identify $\ComOp(r)$ with the image of this map, and $\IOp\GraphOp_n(r)$ denotes the submodule of $\GraphOp_n(r)$
spanned by graphs with a non-empty set of edges. This module $\IOp\GraphOp_n(r)$
represents the cokernel of our coaugmentation
on the coalgebra $\GraphOp_n(r)$. We have similar observations for the operads $\GraphOp_n^2$
and $\FGraphOp_n$.
\end{rem}

We now have the following statement:

\begin{prop}\label{GraphHomology:GraphComplexes:GraphOperadCofreeStructure}
The (contravariant) $\Lambda$-collection $\IOp\GraphOp_n$ has a cofree structure:
\begin{equation}
\IOp\GraphOp_n = \int_{\kset\in\Sigma} \SOp\GraphOp_n(k)^{\Mor_{\Lambda}(\kset,-)},
\end{equation}
where $\SOp\GraphOp_n(k)$ is the summand of the module $\IOp\GraphOp_n(k)$
spanned by the graphs in which all external vertices
have at least one incident edge.
We have a similar statement for our other variants of the operad of graphs $\GraphOp_n^2$, $\FGraphOp_n$,
and for the collection of internally connected graphs $\ICGOp_n$.
\end{prop}

\begin{proof}
This proposition follows from an immediate visual inspection
of the action of restriction operators
on graphs.
\end{proof}

\begin{rem}[The hairy graph complex]\label{GraphHomology:GraphComplexes:HairyGraphComplex}
Let $\GraphOp_n'\subset\GraphOp_n$ be the symmetric collection consisting of connected graphs all of whose external vertices have valence one.
Let $m\leq n$.
The hairy graph complex $\HGCOp_{m n}$ is the dg-module of twisted invariants:
\begin{equation}\label{GraphHomology:GraphComplexes:HairyGraphComplex:Expression}
\HGCOp_{m n} = \prod_{r\geq 1} \DGSigma^m(\GraphOp_n'(r)\otimes\kk\ecell_{-m}^{\otimes r})^{\Sigma_r},
\end{equation}
together with the differential inherited from $\GraphOp_n'$.
The elements of this module can be identified with formal sums of graphs $\alpha\in\GraphOp_n'(r)$
whose external vertices have one adjacent edge each (the ``hairs'' of the graph),
and are made undistinguishable. In what follows, we just draw the hairs in the picture of such a graph,
as we can omit to represent the univalent external vertex attach to each hair.

Note that the internal vertices of the graphs $\alpha\in\GraphOp_n'(r)$
are at least trivalent by convention on our definition
of the operad $\GraphOp_n$.
In what follows, we also consider a version of the hairy graph complex $\HGCOp_{m n}^2$,
which is defined by starting with the operad $\GraphOp_n^2$ instead of $\GraphOp_n$,
and where bivalent internal vertices are allowed.
We have an embedding $\HGCOp_{m n}\subset\HGCOp_{m n}^2$, yielded by the obvious operad inclusion $\GraphOp_n\subset\GraphOp_n^2$.

We also have a Lie dg-algebra structure on $\HGCOp_{m n}$ (respectively, $\HGCOp_{m n}^2$) yielded, as in the case of the full graph complex,
by the commutator of pre-Lie composition operations in the operad $\GraphOp_n$ (respectively, $\GraphOp_n^2$).
\end{rem}

We establish the following proposition:

\begin{prop}\label{GraphHomology:GraphComplexes:HairyGraphComplexHomology}
The inclusion $\HGCOp_{m n}\rightarrow\HGCOp_{m n}^2$ is a weak-equivalence.
Furthermore, the module $\DGH_*(\HGCOp_{m n}) = \DGH_*(\HGCOp_{m n}^2)$
vanishes in degree $*<1$ when $n-m\geq 2$.
\end{prop}

\begin{proof}
The first assertion of the proposition follows from the same arguments as the verification that the morphism $\GraphOp_n\rightarrow\GraphOp_n^2$
defines a weak-equivalence of operads. In short, we consider the filtration of the complex $\HGCOp_{m n}^2$
by the number of internal vertices of valence $\geq 3$.
Then we can check that the spectral sequence determined by this filtration degenerates to the complex $\HGCOp_{m n}$
from the first page on, with all terms on the first row, and the conclusion follows.

We check that any graph $\alpha\in\HGCOp_{m n}$ satisfies $\deg(\alpha)\geq 1$
to establish our second assertion.
The claim is immediate when our graph has no internal vertex since $\alpha$ necessarily consists of two external vertices
connected by an edge in this case, and we then have $\deg(\alpha) = m - 2 m + (n-1) = n-m-1$.
In the case of a hairy graph with a non-empty set of internal vertices, we pick an external vertex $u$,
and we consider the hair $e$ that connects this vertex $u$
to the rest of our graph
through an internal vertex $v$.
We remove this external vertex $u$, the edge $e$, and we split $v$ into a bunch of external vertices $v_1,\dots,v_s$
which we attach to the other incident edges of that vertex $v$
in the graph $\alpha$ (we accordingly have $s\geq 2$ since $v$ is supposed to be at least trivalent).
The connected components of the cell complex which we obtain by this removal operation
are equivalent to graphs $\alpha'_i\in\HGCOp_{m n}$, $i=1,\dots,k$ ($k\leq s$),
with less internal vertices than the graph $\alpha$,
and such that $\deg(\alpha) = \deg(\alpha'_1)+\dots+\deg(\alpha'_k)+(s-k)m-1$.
We have either $k=s\geq 2$ or $s>k\Rightarrow (s-k)m\geq 1$, and in both cases $\deg(\alpha'_1),\dots,\deg(\alpha'_k)\geq 1\Rightarrow\deg(\alpha)\geq 1$.
We can therefore proceed by induction on the number of internal vertices
of our graphs to get our conclusion.
\end{proof}

\begin{constr}[The comparison map from the graph complex to the hairy graph complex]\label{GraphHomology:GraphComplexes:ComparisonMap}
We now consider the morphism of dg-modules $\nu_*: \DGSigma^{-1}\GCOp_n^2\rightarrow\HGCOp_{m n}^2$
which carries a graph $\alpha\in\GCOp_n$ to the hairy graph $\nu_*(\alpha)\in\HGCOp_{m n}^2$
which we obtain by adding one external vertex in all possible ways
to $\alpha$ (we sum over all choices):
\begin{equation}\label{GraphHomology:GraphComplexes:ComparisonMap:Expression}
\nu_*(\alpha) := \sum\vcenter{\xymatrix@!0@C=2em@R=2em@M=0pt{ *+<6pt>{\alpha}\ar@{-}[d] \\ *+<6pt>[o][F]{1} }}.
\end{equation}
In the case $m=n$, the graph
\begin{equation}\label{GraphHomology:GraphComplexes:ComparisonMap:TwistingElement}
\lambda = \vcenter{\xymatrix@!0@C=2em@R=1em@M=0pt{ *+<6pt>[o][F]{1}\ar@{-}[r] & *+<6pt>[o][F]{2} }}
\in\HGCOp_{n n}^2
\end{equation}
satisfies the relations $\delta(\lambda) = 0$ and $[\lambda,\lambda] = 0$
in the Lie dg-algebra $\HGCOp_{n n}^2$.
We moreover have $\deg(\lambda) = -1$.
We therefore have a well defined twisted Lie dg-algebra $(\HGCOp_{n n}^2,\partial_{\lambda})$,
which is defined by adding the twisting derivation $\partial_{\lambda} = [\lambda,-]$,
associated to this Maurer-Cartan element $\lambda$,
to the internal differential of our dg-module $\HGCOp_{n n}^2$.
We readily see that this extra twisting homomorphism vanishes on the image of our mapping,
which therefore defines a morphism of dg-modules $\nu_*: \DGSigma^{-1}\GCOp_n^2\xrightarrow{\sim}(\HGCOp_{n n}^2,\partial_{\lambda})$.
\end{constr}

The Lie algebra $\GCOp_n^2$ (and the Lie algebra $\GCOp_n$ similarly) has a central extension
$\widetilde{\GCOp}{}_n^2 = \kk\ltimes\GCOp_n^2$,
which we define by considering the first Betti number $b_1: \alpha\mapsto b_1(\alpha)$
as a grading on the graph complex.
The natural action of the Lie algebra $\GCOp_n^2$ on $\HGCOp_{m n}^2$
also extends to $\kk\ltimes\GCOp_n^2$.
We do not explicitly use this extended Lie algebra structure in what follows.
We only use a morphism, reflecting this action, but which we can also define directly. To be more explicit, we are
going to use the following statement:

\begin{prop}\label{GraphHomology:GraphComplexes:ComparisonMapHomology}
We have a weak-equivalence of dg-modules
\begin{equation*}
\tilde{\nu}_*: \DGSigma^{-1}(\kk\ltimes\GCOp_n^2)\xrightarrow{\sim}(\HGCOp_{n n}^2,\partial_{\lambda})
\end{equation*}
which carries the factor $\kk$ of the extended Lie algebra $\kk\ltimes\GCOp_n^2$
to the module spanned by our Maurer-Cartan element $\lambda$ in $\HGCOp_{n n}$,
and which is given by the mapping of the previous paragraph~\S\ref{GraphHomology:GraphComplexes:ComparisonMap}
on the factor $\GCOp_n^2$.
\end{prop}

\begin{proof}
This result is contained in some form in \cite[\S 5]{WillwacherGraphs}, but we will give a shorter and self-contained proof.
Note that the operation $\partial_{\lambda} = [\lambda,-]$
is combinatorially described by adding one hair, in all possible ways,
and $\lambda$ represents the unique non-vanishing hairy graph without internal vertices.

We moreover have $(\HGCOp_{n n}^2,\partial_{\lambda}) = \kk\lambda\oplus(\HGCOp_{n n}',\partial_{\lambda})$,
where $\HGCOp_{n n}'\subset\HGCOp_{n n}^2$ consists of graphs with at least one internal vertex,
and our morphism in~\S\ref{GraphHomology:GraphComplexes:ComparisonMap}
carries $\DGSigma^{-1}\GCOp_n^2$
into this summand $\HGCOp_{n n}'$
of the hairy graph complex $\HGCOp_{n n}^2$.
We just check that our morphism induces
a weak-equivalence
between $\DGSigma^{-1}\GCOp_n^2$
and $(\HGCOp_{n n}',\partial_{\lambda})$.
We then consider the variant $(\HGCOp_{n n}'',\partial_{\lambda})$
of the complex $(\HGCOp_{n n}^2,\partial_{\lambda})$
which we form by taking graphs with at least one internal vertex,
but possibly without hairs.
We still assume that the vertices of the graphs of the complex $\HGCOp_{n n}''$
have at least two incident edges.
The twisted complex $(\HGCOp_{n n}'',\partial_{\lambda})$
is identified with the mapping cone
of our morphism
$\nu_*: \DGSigma^{-1}\GCOp_n^2\rightarrow(\HGCOp_{n n}',\partial_{\lambda})$
and our goal is to prove that this mapping cone is acyclic.

We consider a further variant $\FHGCOp_{n n}''$
of the complex $\HGCOp_{n n}''$
by allowing graphs with univalent internal vertices.
We easily check that the twisting homomorphism $\partial_{\lambda} = [\lambda,-]$
extends to this module $\FHGCOp_{n n}''$.
We can also check (along the lines of \cite[Proof of Proposition 3.4, pp. 683-684]{WillwacherGraphs})
that the inclusion $\HGCOp_{n n}''\hookrightarrow\FHGCOp_{n n}''$
is a weak-equivalence.
We therefore aim to prove that the dg-module $(\FHGCOp_{n n}'',\partial_{\lambda})$ is acyclic.
We provide $\FHGCOp_{n n}''$ with the complete descending filtration
\begin{equation*}
\FHGCOp_{n n}'' = \DGF_1\FHGCOp_{n n}''\supset\cdots\supset\DGF_l\FHGCOp_{n n}''\supset\cdots
\end{equation*}
such that $\DGF_l\FHGCOp_{n n}''$
consists of graphs with at least $l$ internal vertices.
We check that the graded object $(\DGE^0(\FHGCOp_{n n}''),\partial_{\lambda})$ which we associate to this filtration
forms an acyclic complex in order to get our result.
We readily see that the internal differential of the object~$\FHGCOp_{n n}''$
vanishes on the modules $\DGE^0_l(\FHGCOp_{n n}'')$
whose differential therefore reduces to the twisting homomorphism $\partial_{\lambda} = [\lambda,-]$.

Let the core of a hairy graph $\alpha$ be the graph obtained by removing all hairs from $\alpha$.
The complex $(\DGE^0_l(\FHGCOp_{n n}''),\partial_{\lambda})$
splits into a direct product of subcomplexes,
one for each isomorphism class of core,
and the graded module underlying each of these sub-complexes
is identified with a graded module of coinvariants
$\left((\kk\oplus\kk\ecell_1)^{\otimes l})\right)_G$,
where $l$ is the number of vertices in the core, which we number from $1$ to $l$, the letter $G$ denotes the automorphism group
of the core (which acts by permutation on the vertices),
and the $j$th factor of our tensor product $\kk\oplus\kk\ecell_1$
controls the presence or non-presence of a hair at the $j$th vertex
of the core. (We attach a hair to the $j$th vertex when this factor is $\ecell_1$.)
The differential acts on each factor $\kk\oplus\kk\ecell_1$ by sending $\kk$ isomorphically to $\kk\ecell_1$.
Hence the complex is clearly acyclic (since the coinvariants of an acyclic complex under the action of a finite group
is again an acyclic complex in characteristic zero).

To illustrate this procedure, we easily see that the subcomplex of~$(\DGE^0_l(\FHGCOp_{n n}''),\partial_{\lambda})$
associated to the core
$\alpha = \begin{tikzpicture}[baseline=-.65ex,scale=.5]
\node[int] (v1) at (0,-.5) {};
\node[int] (v2) at (0,.5) {};
\node[int] (v3) at (1,.5) {};
\node[int] (v4) at (1,-.5) {};
\draw (v1) edge (v2) edge (v3);
\draw (v4) edge (v2) edge (v3);
\draw (v2) edge (v3);
\end{tikzpicture}$,
with $n$ even, has the form:
\begin{equation*}
\kk\begin{tikzpicture}[baseline=-.65ex,scale=.5]
\node[int] (v1) at (0,-.5) {};
\node[int] (v2) at (0,.5) {};
\node[int] (v3) at (1,.5) {};
\node[int] (v4) at (1,-.5) {};
\draw (v1) edge (v2) edge (v3);
\draw (v4) edge (v2) edge (v3);
\draw (v2) edge (v3);
\end{tikzpicture}\xrightarrow{\delta}\kk\begin{tikzpicture}[baseline=-.65ex,scale=.5]
\node[int] (v1) at (0,-.5) {};
\node[int] (v2) at (0,.5) {};
\node[int] (v3) at (1,.5) {};
\node[int] (v4) at (1,-.5) {};
\draw (v1) edge (v2) edge (v3);
\draw (v4) edge (v2) edge (v3);
\draw (v2) edge (v3);
\draw (v3) edge +(60:.67);
\end{tikzpicture}
\oplus\kk\begin{tikzpicture}[baseline=-.65ex,scale=.5]
\node[int] (v1) at (0,-.5) {};
\node[int] (v2) at (0,.5) {};
\node[int] (v3) at (1,.5) {};
\node[int] (v4) at (1,-.5) {};
\draw (v1) edge (v2) edge (v3);
\draw (v4) edge (v2) edge (v3);
\draw (v2) edge (v3);
\draw (v4) edge +(-60:.67);
\end{tikzpicture}
\xrightarrow{\delta}\kk\begin{tikzpicture}[baseline=-.65ex,scale=.5]
\node[int] (v1) at (0,-.5) {};
\node[int] (v2) at (0,.5) {};
\node[int] (v3) at (1,.5) {};
\node[int] (v4) at (1,-.5) {};
\draw (v1) edge (v2) edge (v3);
\draw (v4) edge (v2) edge (v3);
\draw (v2) edge (v3);
\draw (v3) edge +(60:.67);
\draw (v4) edge +(-60:.67);
\end{tikzpicture},
\end{equation*}
with a differential given by the addition of a hair. The terms of the differential associated to the blow-up of a vertex
in our graphs automatically increase the degree of our filtration (by construction),
and hence, vanish in our spectral sequence.
In this example, we assume that $n$ is even, otherwise our graph would vanish for sign and symmetry reasons.
To form our correspondence with the complex $((\kk\oplus\kk\ecell_1)^{\otimes l})_G$,
we number the vertices of our core $\alpha$ from $1$ to $l = 4$,
as in the following picture:
\begin{equation*}
\alpha' = \begin{tikzpicture}[baseline=-.65ex,scale=.5]
\node[int] (v1) at (0,-.5) {};
\node[int] (v2) at (0,.5) {};
\node[int] (v3) at (1,.5) {};
\node[int] (v4) at (1,-.5) {};
\node (x1) at (-.5,-.5) {1};
\node (x2) at (-.5,.5) {2};
\node (x3) at (1.5,.5) {3};
\node (x4) at (1.5,-.5) {4};
\draw (v1) edge (v2) edge (v3);
\draw (v4) edge (v2) edge (v3);
\draw (v2) edge (v3);
\end{tikzpicture}
\end{equation*}
The action of the automorphism group of our graph is identified with the permutation action of the group $G = \langle(1\ 4),(2\ 3)\rangle\subset\Sigma_4$ (with no sign when $n$ is even)
on the numbering of the vertices.
To the tensor $\ecell_1\otimes\ecell_1\otimes 1\otimes 1$ (for instance),
we associate the graph
\begin{equation*}
\ecell_1\otimes\ecell_1\otimes 1\otimes 1\mapsto\begin{tikzpicture}[baseline=-.65ex,scale=.5]
\node[int] (v1) at (0,-.5) {};
\node[int] (v2) at (0,.5) {};
\node[int] (v3) at (1,.5) {};
\node[int] (v4) at (1,-.5) {};
\node (x1) at (-.5,-.5) {1};
\node (x2) at (-.5,.5) {2};
\node (x3) at (1.5,.5) {3};
\node (x4) at (1.5,-.5) {4};
\draw (v1) edge (v2) edge (v3);
\draw (v4) edge (v2) edge (v3);
\draw (v2) edge (v3);
\draw (v1) edge +(-120:.67);
\draw (v2) edge +(120:.67);
\end{tikzpicture}
\end{equation*}
which corresponds to the graph with two hairs depicted of our complex. To be precise, in order to get this correspondence,
we apply an automorphism of the core that corresponds to the action of the permutation $s = (1\ 4)(2\ 3)$
on our vertex numbering.
\end{proof}

\begin{remark}\label{GraphHomology:GraphComplexes:TwistedComplexReduction}
Proposition~\ref{GraphHomology:GraphComplexes:HairyGraphComplexHomology} (together with a standard spectral sequence argument)
can be used to establish
that the morphism of twisted dg-modules
$(\HGCOp_{n n},\partial_{\lambda})\rightarrow(\HGCOp_{n n}^2,\partial_{\lambda})$
is a weak-equivalence as well.
In particular, the homology class of the cycle $\tilde{\nu}_*(\gamma_{2 l+1})\in\HGCOp_{n n}^2$,
where we take the image of the loop graph $\gamma_{2 l+1}\in\GCOp_n^2$
in $\HGCOp_{n n}^2$ (see Proposition~\ref{GraphHomology:GraphComplexes:GraphComplexHomology}),
can be represented by a hairy graph with trivalent vertices in $\HGCOp_{n n}$.
To be explicit, we can easily see that we have a relation of the form
\begin{equation*}
[\tilde{\nu}_*(\gamma_{2 l+1})]\equiv\left[\begin{tikzpicture}[baseline=-.65ex, scale=.5]
\node[int] (v1) at (0:1) {};
\node[int] (v2) at (72:1) {};
\node[int] (v3) at (144:1) {};
\node[int] (v4) at (216:1) {};
\node (v5) at (-72:.75) {$\cdots$};
\draw (v1) edge (v2) edge (v5) (v3) edge (v2) edge (v4) (v4) edge (v5);
\draw (v1) edge +(0:.66);
\draw (v2) edge +(72:.66);
\draw (v3) edge +(144:.66);
\draw (v4) edge +(226:.66);
\end{tikzpicture}\right]
\end{equation*}
in $\DGH_*(\HGCOp_{n n}^2,\partial_{\lambda})$, where we consider a ``hedgehog'' graph with $l+1$-hairs, and we use the notation $[z]$
for the homology class of a cycle in the complex $(\HGCOp_{n n}^2,\partial_{\lambda})$.
\end{remark}

\begin{constr}[Involutions]\label{GraphHomology:GraphComplexes:Involutions}
Recall that the operad $\PoisOp_n$ is equipped with an involution $J_*: \PoisOp_n\rightarrow\PoisOp_n$
such that $J(\lambda) = -\lambda$
for the Lie bracket operation $\lambda\in\PoisOp_n(2)$.
This involution reflects the action of a hyperplane reflection on the operad of little $n$-discs.
We can extend this involution to the graph operad $\GraphOp_n^2$ so that our morphism $\gamma_*: \PoisOp_n\rightarrow\GraphOp_n^2$
preserves the action of involutions.
We explicitly set:
\begin{equation}\label{GraphHomology:GraphComplexes:Involutions:GraphOperadCase}
J_*(\alpha) = (-1)^{k+l}\alpha,
\end{equation}
for any graph $\alpha\in\GraphOp_n^2$ with $k$ edges and $l$ internal vertices.
We immediately see that this map does define a morphism of Hopf $\Lambda$-operads $J_*: \GraphOp_n^2\rightarrow\GraphOp_n^2$.
Furthermore, this involution $J_*: \GraphOp_n^2\rightarrow\GraphOp_n^2$
trivially admits a restriction
to the reduced ($\geq 3$-valent) graph operad $\GraphOp_n$.

In what follows, we also deal with an involution $I_*: \GCOp_n^2\rightarrow\GCOp_n^2$ on the graph complex $\GCOp_n^2$.
We define this morphism by the explicit formula:
\begin{equation}\label{GraphHomology:GraphComplexes:GraphComplexCase}
I_*(\alpha) = (-1)^{k+l+1}\alpha,
\end{equation}
for any graph $\alpha\in\GCOp_n^2$ with $k$ edges and $l$ vertices.
We immediately see that this map preserves the differential in $\GCOp_n^2$.

We still have an involution $I_*: \HGCOp_{m n}^2\rightarrow\HGCOp_{m n}^2$
on the hairy graph complex $\HGCOp_{m n}^2$
which we define by the formula
\begin{equation}\label{GraphHomology:GraphComplexes:HairyGraphComplexCase}
I_*(\alpha) = (-1)^{k+l+r-1}\alpha,
\end{equation}
for any graph $\alpha\in\HGCOp_{m n}^2$ with $k$ edges, $l$ internal vertices, and $r$ external vertices.
We readily check that this involution preserves the differential of our complex $\HGCOp_{m n}^2$ (again)
and commutes with the extra twisting homomorphism $\partial_{\lambda} = [\lambda,-]$
in the case $m=n$.
The morphism $\nu_*: \GCOp_n^2\rightarrow\HGCOp_{m n}^2$
preserves the action of involutions.
Furthermore, we immediately see that both $I_*: \GCOp_n^2\rightarrow\GCOp_n^2$ and $I_*: \HGCOp_{m n}^2\rightarrow\HGCOp_{m n}^2$
admit a restriction to the reduced ($\geq 3$-valent) version of our graph complexes $\GCOp_n$
and $\HGCOp_{m n}$.
We will see (in the next subsection) that these involutions correspond to a conjugate action
of the involution of the $n$-Poisson cooperad
on the Koszul deformation complex
of~\S\ref{DeformationComplexes:KoszulDuality:KoszulReduction}.

Let us observe that the loop classes $\gamma_l\in\GCOp_n^2$ are odd with respect to the action of the involution $I_*(\gamma_l) = - \gamma_l$.
We use this observation in the next subsection in order to prove the vanishing of the obstructions
to the existence of a formality weak-equivalence $\phi: \KOp\xrightarrow{\sim}\PoisOp_n^c$
when $n\equiv 0(\mymod 4)$ (see~\S\ref{Background:ObstructionProblem}).
\end{constr}

\subsection{The homology of the deformation complexes}\label{GraphHomology:DeformationComplexHomology}
We go back to the study of the deformation complexes
$K^{* *}_{m n} = K^{* *}(\PoisOp_n^c,\PoisOp_m^c)$
and
$L^{* *}_n = K^{* *}(\PoisOp_n^c,\PoisOp_n^c)$.
We prove that the homology of these complexes reduce to the homology of graph complexes.
We then use the connectedness of the graph complexes in order to establish the vanishing statements which we need in our obstruction problem.

We address the case of the complex $K^{* *}_{m n}$ first.
We rely on the following statement:

\begin{prop}\label{GraphHomology:DeformationComplexHomology:RelativeObstructions:GraphOperadReduction}
We have a weak-equivalence
\begin{equation*}
K^{* *}_{m n}\simeq E^{* *}(\PoisOp_n,\PoisOp_m^c)
\xrightarrow{\sim} E^{* *}(\GraphOp_n,\PoisOp_m^c),
\end{equation*}
where on the right hand side we consider the twisted end complex
\begin{equation*}
E^{* *}_{m n} = E^{* *}(\GraphOp_n,\PoisOp_m^c)
\end{equation*}
associated to the graph operad $\POp = \GraphOp_n$ together with the morphism $\iota_*: \PoisOp_m\rightarrow\GraphOp_n$
such that $\iota_*(\mu) = \vcenter{\xymatrix@!0@C=1.5em@R=1em@M=0pt{ *+<6pt>[o][F]{1} & *+<6pt>[o][F]{2} }}$
and $\iota_*(\lambda) = 0$.
\end{prop}

\begin{proof}
This proposition follows from the result of Proposition~\ref{GraphHomology:TwistedEndComplexes:TwistedEndComplexHomotopyInvariance}
(together with the duality result of Theorem~\ref{GraphHomology:TwistedEndComplexes:MainResult})
since we observed in~\S\ref{GraphHomology:TwistedEndComplexes:CofreeLambdaStructures}
and Proposition~\ref{GraphHomology:GraphComplexes:GraphOperadCofreeStructure}
that the operads $\POp = \PoisOp_n,\GraphOp_n$ fulfill a cofreeness property equivalent to the freeness requirement
of this statement for the dual cooperads
of these operads.
\end{proof}

We examine the definition of the twisted end complex of this proposition $E^{* *}_{m n} = E^{* *}(\GraphOp_n,\PoisOp_m^c)$ in the next paragraph.
We mainly review the general definition of~\S\ref{GraphHomology:TwistedEndComplexes:TwistedEndComplexConstruction},
and we prove that this twisted end complex can be related to the hairy graph complex
of the previous subsection.

\begin{constr}[The comparison map with the hairy graph complex]\label{GraphHomology:DeformationComplexHomology:HairyGraphComplexMap}
By definition (see~\S\ref{GraphHomology:TwistedEndComplexes:TwistedEndComplexConstruction}),
the total differential of the complex $E^{* *}_{m n} = E^{* *}(\GraphOp_n,\PoisOp_m^c)$
consists of the following pieces:
\begin{itemize}
\item
a component-wise differential $\delta: E^{k l}_{m n}\rightarrow E^{k l}_{m n}$
yielded by the internal differential of the operad of graphs;
\item
a horizontal twisting differential $\partial_h = \partial'_h+\partial''_h: E^{k l}_{m n}\rightarrow E^{k+1 l}_{m n}$,
where $\partial'_h$ is yielded by the twisting differential of the Harrison complexes
with trivial coefficients $\hat{\DGB}{}_{com}^*(\GraphOp_n(r))$
whereas the map $\partial''_h$ involves the action of the algebras $\GraphOp_n(r)$
on the Koszul complexes $\tilde{\DGK}{}_{op}^l(\PoisOp_m^c)(r)$
through $\PoisOp_m^c(r)$;
\item
and a vertical twisting differential $\partial_v = \partial''_v: E^{k l}_{m n}\rightarrow E^{k l+1}_{m n}$
determined by the bimodule structure of the Harrison complex $\hat{\DGB}{}_{com}^*(\GraphOp_n)$
over the Poisson operad $\PoisOp_m^c$.
\end{itemize}
Recall that we determine our twisting homomorphisms term-wise on the end of our twisted complex
(see~\S\ref{GraphHomology:TwistedEndComplexes:TwistedEndComplexConstruction}).
Let $\alpha\otimes\xi\in\hat{\DGB}{}_{com}^k(\GraphOp_n(r))\otimes\tilde{\DGK}{}_{op}^l(\PoisOp_m^c)(r)$
be any tensor in a term of this end.
The explicit expression of the twisting homomorphism $\partial''_h$
in~\S\ref{GraphHomology:TwistedEndComplexes:TwistedEndComplexConstruction}(\ref{GraphHomology:TwistedEndComplexes:TwistedEndComplexConstruction:HorizontalTwistingDifferential})
implies that we have the relation:
\begin{align}
\label{GraphHomology:DeformationComplexHomology:HairyGraphComplexMap:HorizontalTwistingDifferential}
\partial''_h(\alpha\otimes\xi) & = 0,
\intertext{because our map $\iota_*: \PoisOp_m\rightarrow\GraphOp_n$ vanishes on the Poisson monomials
such that $\pi\in\IOp\PoisOp_m(r)$.
For the operadic twisting differential $\partial''_v$,
we get the formula:}
\label{GraphHomology:DeformationComplexHomology:HairyGraphComplexMap:VerticalTwistingDifferential}
\partial''_v(\alpha\otimes\xi)
& = \sum_{i=1,2}\pm(\vcenter{\xymatrix@!0@C=1.5em@R=1em@M=0pt{ *+<6pt>[o][F]{1} & *+<6pt>[o][F]{2} }}\circ_i\alpha)\otimes(\lambda\circ_i\xi)
+ \sum_{i=1,\dots,r}\pm(\alpha\circ_i\vcenter{\xymatrix@!0@C=1.5em@R=1em@M=0pt{ *+<6pt>[o][F]{1} & *+<6pt>[o][F]{2} }})
\otimes(\xi\circ_i\lambda).
\end{align}

We also have the degree-wise end change formula:
\begin{multline}
\label{GraphHomology:DeformationComplexHomology:HairyGraphComplexMap:Components}
E^{k l}_{m n} = \int_{\rset\in\Lambda}\hat{\DGB}{}_{com}^k(\GraphOp_n)(r)\otimes\tilde{\DGK}{}_{op}^l(\PoisOp_m^c)(r)
\\
\simeq\int_{\rset\in\Sigma}\DGS\hat{\DGB}{}_{com}^k(\GraphOp_n)(r)\otimes\tilde{\DGK}{}_{op}^l(\PoisOp_m^c)(r),
\end{multline}
where we consider the cogenerating symmetric collection $\DGS\hat{\DGB}{}_{com}^k(\GraphOp_n)$
of the $\Lambda$-collection $\hat{\DGB}{}_{com}^k(\GraphOp_n)$ (see~\S\ref{GraphHomology:TwistedEndComplexes:HarrisonConstructionCofreeLambdaStructure}).
The end over the category $\Sigma$ which we obtain in this formula is obviously equivalent
to the cartesian product of the modules
of invariants $(\DGS\hat{\DGB}{}_{com}^k(\GraphOp_n)(r)\otimes\tilde{\DGK}{}_{op}^l(\PoisOp_m^c)(r))^{\Sigma_r}$,
where we consider the diagonal action of the symmetric groups $\Sigma_r$
on our tensors.

When $k=0$, we have $\DGS\hat{\DGB}{}_{com}^0(\GraphOp_n)(r) = \SOp\GraphOp_n(r)$,
where we consider the underlying cogenerating symmetric collection $\SOp\GraphOp_n$
of the coaugmentation coideal
of the graph operad $\IOp\GraphOp_n$ (see~\S\ref{GraphHomology:TwistedEndComplexes:HarrisonConstructionCofreeLambdaStructure}).
We can now identify a hairy graph with $r$ hairs $\alpha\in\HGCOp_{m n}$ with an element of the module of invariants $\SOp\GraphOp_n(r)^{\Sigma_r}$.
We then consider the mapping $\psi: \DGSigma\HGCOp_{m n}\rightarrow E^{* *}_{m n}$
such that:
\begin{equation}
\label{GraphHomology:DeformationComplexHomology:HairyGraphComplexMap:Mapping}
\psi(\alpha) = \alpha\otimes\mu_r\in\int_{\rset\in\Sigma}\SOp\GraphOp_n(r)\otimes\tilde{\DGK}{}_{op}^{r-2}(\PoisOp_m^c)(r) = E^{0 r-2},
\end{equation}
for any such $\alpha\in\HGCOp_{m n}$,
where we use the identity
\begin{equation}
\tilde{\DGK}{}_{op}^{r-2}(\PoisOp_m^c)(r) = \DGSigma\SuspOp^m\PoisOp_m(r)
\end{equation}
and $\mu_r$ represents the $r$-fold commutative product operation
in the $m$-Poisson operad $\PoisOp_m$ (the unit element for $r=1$).

We easily check that $\psi$ satisfies the relation $\psi(\delta\alpha) = \delta\psi(\alpha)$, where we consider the internal differential
of the graph complex on the one hand,
and the component of differential of the dg-module $E^{* *}_{m n}$
yielded by the internal differential of the operad of graphs $\GraphOp_n$
on the other hand.
We moreover have $\partial'_h\psi(\alpha) = 0$ when we consider the differential of the Harrison complex $\hat{\DGB}{}_{com}^*(\GraphOp_n)$
because the connectedness assumption in the definition
of the hairy graph complex implies that $\alpha$
corresponds to an indecomposable element
of the graph operad.
We still have $\partial''_h\psi(\alpha) = 0$ since this twisting homomorphism $\partial''_h$ entirely vanishes in our bicomplex $E^{* *}_{m n}$.
We also easily check that the terms of the operadic twisting differential
in~(\ref{GraphHomology:DeformationComplexHomology:HairyGraphComplexMap:VerticalTwistingDifferential})
cancel each other for the element $\psi(\alpha) = \alpha\otimes\mu_r$ (we use that the external vertices
have valence one in the hairy graph complex to check this claim).
We conclude from these relations that our mapping~(\ref{GraphHomology:DeformationComplexHomology:HairyGraphComplexMap:Mapping})
defines a morphism of dg-modules $\psi: \DGSigma\HGCOp_{m n}\rightarrow E^{* *}_{m n}$.
\end{constr}

We now have the following theorem:

\begin{thm}\label{GraphHomology:DeformationComplexHomology:HairyGraphComplexMapHomology}
The morphism of~\S\ref{GraphHomology:DeformationComplexHomology:HairyGraphComplexMap}
defines a weak-equivalence of dg-modules
$\psi: \DGSigma\HGCOp_{m n}\xrightarrow{\sim} E^{* *}_{m n}$,
for any pair $m,n\geq 2$.
\end{thm}

\begin{proof}
We filter the graph operad $\GraphOp_n$ by the number of edges in graphs.
We equip the twisted end complex $E^{* *}_{m n}$
with the complete descending filtration
$E^{* *}_{m n} = \DGF_0 E^{* *}_{m n}\supset\cdots\supset\DGF_s E^{* *}_{m n}\supset\cdots$
yielded by this filtration of the graph operad $\GraphOp_n$
in our twisted end construction.
We consider a similar filtration, by the number of edges, on the hairy graph complex $\HGCOp_{m n}$.
We just check that the morphism of the theorem induces a weak-equivalence
on the graded complexes determined by this filtration
to get our result, and we deduce this claim from the following arguments.

We first see that the differential of the hairy graph complex vanishes in the graded complex $\DGE^0(\HGCOp_{m n})$
associated to $\HGCOp_{m n}$
since this differential creates one edge
in all cases.
We similarly get that the internal differential of the graph operad vanishes in $\DGE^0(E^{* *}_{m n})$,
and only the pieces $\partial_h = \partial'_h$ and $\partial_v = \partial''_v$
of the differential remain non-trivial in $\DGE^0(E^{* *}_{m n})$.

We use a filtration by the grading of the Koszul construction $\tilde{\DGK}{}_{op}^*(\PoisOp_m^c)$
to compute the homology of the complex $(\DGE^0(E^{* *}_{m n}),\partial'_h+\partial''_v)$.
We then get a spectral sequence $\DGD^1(E^{* *}_{m n})\Rightarrow\DGE^1(E^{* *}_{m n})$
with $\DGD^1 = \DGH_*(\DGE^0(E^{* *}_{m n}),\partial'_h)$
as $E^1$-page,
and whose $d^1$ differential is yielded by the operadic twisting differential $\partial''_v$.
We now have a weak-equivalence
\begin{equation}\label{GraphHomology:DeformationComplexHomology:HairyGraphComplexMapHomology:HarrisonHomology}
\DGB^{com}_*(\GraphOp_n^c(r))\xrightarrow{\sim}\ICGOp_n^c(r)
\Leftrightarrow\ICGOp_n(r)\xrightarrow{\sim}\hat{\DGB}{}_{com}^*(\GraphOp_n(r)),
\end{equation}
for any $r>0$, where we consider the dual filtered Hopf $\Lambda$-cooperad $\GraphOp_n^c$
of the operad of graphs $\GraphOp_n$ and the dual $\ICGOp_n^c(r)$ of the dg-modules
of internally connected graphs $\ICGOp_n(r)$. We just use that $\GraphOp_n^c(r)$
forms a symmetric algebra on $\ICGOp_n^c(r)$
when we forget about the differential (see~\S\ref{GraphHomology:GraphComplexes:GraphOperads})
to get this statement.

We readily get that this weak-equivalence induces
a weak-equivalence
on our end
\begin{equation}\label{GraphHomology:DeformationComplexHomology:HairyGraphComplexMapHomology:EndMapping}
\int_{\rset\in\Lambda}\ICGOp_n(r)\otimes\tilde{\DGK}{}_{op}^*(\PoisOp_m^c)(r)
\xrightarrow{\sim}\int_{\rset\in\Lambda}\hat{\DGB}{}_{com}^*(\GraphOp_n(r))\otimes\tilde{\DGK}{}_{op}^*(\PoisOp_m^c)(r),
\end{equation}
because both $\ICGOp_n$ and $\hat{\DGB}{}_{com}^*(\GraphOp)$ admit a cofree structure
as $\Lambda$-collection, and this structure result implies
that $\ICGOp_n$ and $\hat{\DGB}{}_{com}^*(\GraphOp)$
form cofibrant $\Lambda$-collections
in dg-modules (with respect to the Reedy model structure
of~\cite[Proposition III.2.3.4]{FresseBook}).
We therefore have the relations:
\begin{multline}\label{GraphHomology:DeformationComplexHomology:HairyGraphComplexMapHomology:SpectralSequence}
\DGD^1 = \DGH_*(\DGE^0(E^{* *}_{m n}),\partial'_h)\simeq\int_{\rset\in\Lambda}\ICGOp_n(r)\otimes\tilde{\DGK}{}_{op}^*(\PoisOp_m^c)(r)
\\
\simeq\int_{\rset\in\Sigma}\SOp\ICGOp_n(r)\otimes\tilde{\DGK}{}_{op}^*(\PoisOp_m^c)(r).
\end{multline}
This graded module $\DGD^1$ together with the differential $d^1 = \partial''_v$
is identified with a summand
of the (non-Hopf) operadic deformation complex
of \cite[Lemma 4.4]{WillwacherGraphs}.
The result of this reference implies that the homology of this complex is identified with the module
that contains only terms of the form $\alpha\otimes\mu_k$
and where all external vertices of the graph $\alpha$
are univalent. But this is exactly the image of $\HGCOp_{m n}$ in $E^{* *}_{m n}$
and hence we are done.
\end{proof}

This theorem, together with the results of Proposition~\ref{GraphHomology:GraphComplexes:HairyGraphComplexHomology}
and Proposition~\ref{GraphHomology:DeformationComplexHomology:RelativeObstructions:GraphOperadReduction},
immediately implies:

\begin{prop}\label{GraphHomology:DeformationComplexHomology:RelativeObstructions}
We have the vanishing relation
\begin{equation*}
\DGH_0(K^{* *}_{m n})\simeq\DGH_0(E^{* *}_{m n})\simeq\DGH_{-1}(\HGCOp_{m n}) = 0.
\end{equation*}
as soon as $n-m\geq 2$ and $m\geq 2$.\qed
\end{prop}

We now address the case of the complex $L^{* *}_n$.
We then have the following statement:

\begin{prop}\label{GraphHomology:DeformationComplexHomology:IdentityObstructions:GraphOperadReduction}
We have a weak-equivalence
\begin{equation*}
L^{* *}_n\simeq E^{* *}(\PoisOp_n,\PoisOp_n^c)
\xrightarrow{\sim} E^{* *}(\GraphOp_n,\PoisOp_n^c),
\end{equation*}
where on the right-hand side we consider the twisted end complex
\begin{equation*}
F^{* *}_n = E^{* *}(\GraphOp_n,\PoisOp_n^c)
\end{equation*}
associated to the graph operad $\POp = \GraphOp_n$ together with the morphism $\gamma_*: \PoisOp_n\rightarrow\GraphOp_n$
such that $\gamma_*(\mu) = \vcenter{\xymatrix@!0@C=1.5em@R=1em@M=0pt{ *+<6pt>[o][F]{1} & *+<6pt>[o][F]{2} }}$
and $\gamma_*(\lambda) = \vcenter{\xymatrix@!0@C=1.5em@R=1em@M=0pt{ *+<6pt>[o][F]{1}\ar@{-}[r] & *+<6pt>[o][F]{2} }}$.
\end{prop}

\begin{proof}
We deduce this proposition follows from the result of Proposition~\ref{GraphHomology:TwistedEndComplexes:TwistedEndComplexHomotopyInvariance},
as in the case of Proposition~\ref{GraphHomology:DeformationComplexHomology:RelativeObstructions:GraphOperadReduction},
by using that both operads $\POp = \PoisOp_n,\GraphOp_n$
fulfill the cofreeness requirement
of our statement.
\end{proof}

We examine the definition of the twisted end complex that occurs in this proposition $F^{* *}_n = E^{* *}(\GraphOp_n,\PoisOp_n^c)$.
We aim to compare this complex with the twisted hairy graph complex of Proposition~\ref{GraphHomology:GraphComplexes:ComparisonMapHomology}.

\begin{constr}[The comparison map with the twisted hairy graph complex]\label{GraphHomology:DeformationComplexHomology:TwistedHairyGraphComplexMap}
The bicomplex $F^{* *}_n$ is defined by the same double sequence
of modules~\S\ref{GraphHomology:DeformationComplexHomology:HairyGraphComplexMap}(\ref{GraphHomology:DeformationComplexHomology:HairyGraphComplexMap:Components})
as the complex $E^{* *}_{n n}$
which we consider in~\S\ref{GraphHomology:DeformationComplexHomology:HairyGraphComplexMap}.
We just provide the double collection $F^{k l}_n = E^{k l}_{n n}$
with different horizontal and vertical differentials
which we associate to the map of Proposition~\ref{GraphHomology:DeformationComplexHomology:IdentityObstructions:GraphOperadReduction}
in order to get this new bicomplex $F^{* *}_n$.
In fact, this map $\gamma_*: \PoisOp_n\rightarrow\GraphOp_n$ is given by the addition of new terms $\gamma_*(\pi)\in\GraphOp_n(r)$,
associated to Poisson monomials $\pi\in\IOp\PoisOp_n(r)$, $r>0$,
to the map $\iota_*: \PoisOp_n\rightarrow\GraphOp_n$
which we use to determine the horizontal and vertical differentials of the bicomplex $E^{* *}_{n n}$.
We can accordingly determine the horizontal (respectively, vertical) differential of the bicomplex $F^{* *}_n$
by adding an extra twisting homomorphisms,
corresponding to the extra terms of our mapping $\gamma_*(\pi)\in\GraphOp_n(r)$,
to the horizontal (respectively, vertical) differential of our first bicomplex $E^{* *}_{n n}$.

In what follows, we keep the notation $\partial'_h: F^{k l}_n\rightarrow F^{k+1 l}_n$
for the piece of the horizontal differential
inherited from the Harrison cochain complex with trivial coefficients $\hat{\DGB}{}_{com}^*(\GraphOp_n)$
and the notation $\partial''_h: F^{k l}_n\rightarrow F^{k+1 l}_n$
for the horizontal twisting differential which the bicomplex $F^{* *}_n$
inherits from $E^{* *}_{n n}$.
We adopt the notation $\partial'''_h: F^{k l}_n\rightarrow F^{k+1 l}_n$
for the new terms of our horizontal twisting differential. We adopt similar conventions for the vertical differentials,
which we therefore decompose as $\partial_v = \partial''_v + \partial'''_v$,
where $\partial''_v: F^{k l}_n\rightarrow F^{k l+1}_n$
is the part which the bicomplex $F^{* *}_n$
inherits from $E^{* *}_{n n}$,
while $\partial'''_v: F^{k l}_n\rightarrow F^{k l+1}_n$
comes from the extra terms of our map $\gamma_*: \PoisOp_n\rightarrow\GraphOp_n$.

We determine these maps term-wise on the end~(\ref{GraphHomology:TwistedEndComplexes:TwistedEndComplexConstruction:Components})
of~\S\ref{GraphHomology:TwistedEndComplexes:TwistedEndComplexConstruction}
as in the case of the bicomplexes $E^{* *}_{m n}$
of~\S\ref{GraphHomology:DeformationComplexHomology:HairyGraphComplexMap}.
We again consider a tensor $\alpha\otimes\xi\in\hat{\DGB}{}_{com}^k(\GraphOp_n(r))\otimes\tilde{\DGK}{}_{op}^l(\PoisOp_n^c)$
which represents an element in a term of this end.
We refer to~\S\ref{GraphHomology:DeformationComplexHomology:HairyGraphComplexMap}(\ref{GraphHomology:DeformationComplexHomology:HairyGraphComplexMap:VerticalTwistingDifferential})
for the expression of the twisting differential $\partial''_v$
on this tensor. Recall that the piece $\partial''_h$ of the horizontal differential
entirely vanish.
By~\S\ref{GraphHomology:TwistedEndComplexes:TwistedEndComplexConstruction}(\ref{GraphHomology:TwistedEndComplexes:TwistedEndComplexConstruction:HorizontalTwistingDifferential}),
we have on the other hand:
\begin{align}
\label{GraphHomology:DeformationComplexHomology:TwistedHairyGraphComplexMap:HorizontalTwistingDifferential}
\partial'''_h(\alpha\otimes\xi) & = \sum_{\pi}\pm[\gamma_*(\pi),\alpha]\otimes(\pi^{\vee}\cdot\xi),
\intertext{where we now consider the non-trivial elements $\gamma_*(\pi)\in\IOp\GraphOp_n(r)$, $r>0$,
associated to the Poisson monomials such that $\pi\in\IOp\PoisOp_n(r)$.
By~\S\ref{GraphHomology:TwistedEndComplexes:TwistedEndComplexConstruction}(\ref{GraphHomology:TwistedEndComplexes:TwistedEndComplexConstruction:VerticalTwistingDifferential}),
we similarly get:}
\label{GraphHomology:DeformationComplexHomology:TwistedHairyGraphComplexMap:VerticalTwistingDifferential}
\partial'''_v(\alpha\otimes\xi)
& = \sum_{i=1,2}\pm(\vcenter{\xymatrix@!0@C=1.5em@R=1em@M=0pt{ *+<6pt>[o][F]{1}\ar@{-}[r] & *+<6pt>[o][F]{2} }}\circ_i\alpha)\otimes(\mu\circ_i\xi)
+ \sum_{i=1,\dots,r}\pm(\alpha\circ_i\vcenter{\xymatrix@!0@C=1.5em@R=1em@M=0pt{ *+<6pt>[o][F]{1}\ar@{-}[r] & *+<6pt>[o][F]{2} }})
\otimes(\xi\circ_i\mu),
\end{align}
for the extra piece of the operadic twisting differential, where we consider the operadic composites
with the extra term $\gamma_*(\lambda) = \vcenter{\xymatrix@!0@C=1.5em@R=1em@M=0pt{ *+<6pt>[o][F]{1}\ar@{-}[r] & *+<6pt>[o][F]{2} }}$
of our mapping $\gamma_*: \PoisOp_n\rightarrow\GraphOp_n$
in arity $2$.

We now consider the same mapping $\psi(\alpha) = \alpha\otimes\mu_r$ as in~\S\ref{GraphHomology:DeformationComplexHomology:HairyGraphComplexMap}
(where we now assume $m=n$)
for a hairy graph $\alpha\in\HGCOp_{n n}$
with $r$ external vertices,
and where $\mu_r$ denotes the $r$-fold product operation
in $\tilde{\DGK}{}_{op}^{r-2}(\PoisOp_n^c)(r) = \DGSigma\SuspOp^n\PoisOp_n(r)$.
We already checked that this mapping preserves the internal differential of graphs $\delta\psi(\alpha) = \psi(\delta\alpha)$
and that we have the relation $\partial'_h\psi(\alpha) = 0$
as well as $\partial''_v\psi(\alpha) = 0$
in our twisted end complex (see~\S\ref{GraphHomology:DeformationComplexHomology:HairyGraphComplexMap}).
We also have $\partial'''_h(\alpha\otimes\mu_r) = 0$, because $\pi^{\vee}\in\PoisOp_n(r)^{\vee}$
acts trivially on $\mu_r\in\tilde{\DGK}{}_{op}^{r-2}(\PoisOp_n^c)(r)$
when the monomial $\pi(x_1,\dots,x_r) = \pi_1(x_{1 j_1},\dots,x_{j_{1 n_1}})\cdot\ldots\cdot\pi_s(x_{s j_1},\dots,x_{s j_{n_s}})\in\PoisOp_n(r)$
contains a non-trivial Lie factor, and hence, belongs to $\IOp\PoisOp_n(r)$.
Indeed, the object $\tilde{\DGK}{}_{op}^*(\PoisOp_n^c)(r)$
forms, by construction, a quotient
of the operadic cobar construction $\DGB_{op}^*(\PoisOp_n^c)(r)$
as a module over the commutative algebra $\PoisOp_n^c(r)$.
The element $\mu_r\in\tilde{\DGK}{}_{op}^{r-2}(\PoisOp_n^c)(r)$
is represented by a tree-wise tensor of dual Lie bracket operations $\lambda^{\vee}\in\PoisOp_n^c(2)$
in $\DGB_{op}^*(\PoisOp_n^c)(r)$. We have $\lambda^{\vee}\cdot\lambda^{\vee} = 0$
in the commutative algebra $\PoisOp_n^c(2)$,
and we deduce from this relation that $\pi^{\vee}\in\PoisOp_n(r)^{\vee}$
operates trivially on $\mu_r\in\tilde{\DGK}{}_{op}^{r-2}(\PoisOp_n^c)(r)$
as soon as any coproduct of the element $\pi^{\vee}$
over a binary tree
contains a factor $\lambda^{\vee}$.

We readily see, on the other hand, that the extra pieces of our operadic twisting differential $\partial'''_v\psi(\alpha) = \partial'''_v(\alpha\otimes\mu_r)$
correspond to the twisting operation
$\partial_{\lambda}(\alpha) = [\vcenter{\xymatrix@!0@C=1.5em@R=1em@M=0pt{ *+<6pt>[o][F]{1}\ar@{-}[r] & *+<6pt>[o][F]{2} }},\alpha]$
in the hairy graph complex $\HGCOp_{n n}$,
where we consider the internal Lie bracket
of the hairy graph complex (see~\S\ref{GraphHomology:GraphComplexes:HairyGraphComplex}).
We therefore conclude that our map $\psi(\alpha) = \alpha\otimes\mu_r$
defines a morphism
of dg-modules $\psi: \DGSigma(\HGCOp_{n n},\partial_{\lambda})\rightarrow F^{* *}_n$
when we add this extra twisting homomorphism
$\partial_{\lambda} = [\vcenter{\xymatrix@!0@C=1.5em@R=1em@M=0pt{ *+<6pt>[o][F]{1}\ar@{-}[r] & *+<6pt>[o][F]{2} }},-]$
to the internal differential of the hairy graph complex $\HGCOp_{n n}$.
\end{constr}

We then have the following main theorem:

\begin{thm}\label{GraphHomology:DeformationComplexHomology:TwistedHairyGraphComplexMapHomology}
The morphism of~\S\ref{GraphHomology:DeformationComplexHomology:TwistedHairyGraphComplexMap} defines a weak-equivalence
of dg-modules
$\psi: \DGSigma(\HGCOp_{n n},\partial_{\lambda})\xrightarrow{\sim} F^{* *}_n$,
for any $n\geq 2$.
\end{thm}

\begin{proof}
We equip the dg-module $F^{* *}_n$ with the same filtration as the bicomplex $E^{* *}_n$
in the proof of Theorem~\ref{GraphHomology:DeformationComplexHomology:HairyGraphComplexMapHomology}.
We immediately get that the extra terms of the differentials of the bicomplex $F^{* *}_n$
vanish when we pass to the graded object associated to this filtration,
because these extra terms involve
the creation of edges
in graphs.
We accordingly have the relation $\DGE^1(F^{* *}_n) = \DGE^1(E^{* *}_{n n})$ on the $E^1$-page of the spectral
sequence associated to our filtration.

We also equip the dg-module $(\HGCOp_{n n},\partial_{\lambda})$
with the filtration by the number of edges
in graphs.
We get that the twisting homomorphism $\partial_{\lambda}$ of this dg-module vanishes
when we pass to the graded object associated to this filtration
(as in the case of our twisted end complex $F^{* *}_n$),
because this twisting homomorphism produces
a new edge in graphs
again.
We therefore have the relation $\DGE^1(\HGCOp_{n n},\partial_{\lambda}) = \DGE^1(\HGCOp_{n n})$,
where we drop the extra twisting homomorphism $\partial_{\lambda}$
from the complex $\HGCOp_{n n}$
on the right-hand side.
We then retrieve the spectral sequence of the proof of Theorem~\ref{GraphHomology:DeformationComplexHomology:HairyGraphComplexMapHomology}.
We can therefore rely on the arguments
of this statement
in order to establish that our morphism $\psi: (\HGCOp_{n n},\partial_{\lambda})\rightarrow F^{* *}_n$
induces an isomorphism
when we pass to the $E^1$-page of the spectral sequence
associated to our filtration. The conclusion follows.
\end{proof}

We can now use a reduction to the graph complex $\GCOp_n^2$ in order to determine the homology of the deformation complex $L^{* *}_n$.
Recall simply that we need to consider the version of the hairy graph complex where vertices of valence two are allowed $\HGCOp_{n n}^2$
rather than the reduced complex $\HGCOp_{n n}$
in order to get a correspondence with the graph complex $\GCOp_n^2$ (see Proposition~\ref{GraphHomology:GraphComplexes:ComparisonMapHomology}).
Thus, we need to replace $(\HGCOp_{n n},\partial_{\lambda})$ by the complex $(\HGCOp_{n n}^2,\partial_{\lambda})$
in the result of Theorem~\ref{GraphHomology:DeformationComplexHomology:TwistedHairyGraphComplexMapHomology}
in order to use the correspondence with the graph complex $\GCOp_n^2$.
This is not a problem since we observed that the twisted complexes $(\HGCOp_{n n},\partial_{\lambda})$ and $(\HGCOp_{n n}^2,\partial_{\lambda})$
are weakly-equivalent (see Remark~\ref{GraphHomology:GraphComplexes:TwistedComplexReduction}).

We may equivalently see that our previous constructions remain entirely valid when we use the graph operad with bivalent vertices allowed $\GraphOp_n^2$
and the corresponding hairy graph complex $\HGCOp_{m n}^2$ instead of the reduced version of the graph operad $\GraphOp_n$
and of the hairy graph complex $\HGCOp_{m n}$.
Then we abut to the same conclusion as previously.
Namely, we can take the twisted complex $(\HGCOp_{n n}^2,\partial_{\lambda})$
instead of $(\HGCOp_{n n},\partial_{\lambda})$
in the claim of Theorem~\ref{GraphHomology:DeformationComplexHomology:TwistedHairyGraphComplexMapHomology}.
Eventually, we use the result of Proposition~\ref{GraphHomology:GraphComplexes:ComparisonMapHomology}, as we just explained,
and the observations of Proposition~\ref{GraphHomology:GraphComplexes:GraphComplexHomology},
to obtain the following statement:

\begin{prop}\label{GraphHomology:DeformationComplexHomology:IdentityObstructions}
We have the formulas:
\begin{gather*}
\DGH_0(L^{* *}_n)\simeq\DGH_0(F^{* *}_n)\simeq\kk\oplus\begin{cases}\kk\gamma_n, & \text{if $n\equiv 3(\mymod 4)$}, \\
0, & \text{otherwise}, \end{cases}
\\
\text{and}\quad\DGH_{-1}(L^{* *}_n)\simeq\DGH_{-1}(F^{* *}_n)\simeq\begin{cases}\kk\gamma_{n+1}, & \text{if $n\equiv 0(\mymod 4)$}, \\
0, & \text{otherwise}, \end{cases}
\end{gather*}
as soon as $n\geq 3$.\qed
\end{prop}

The result of this proposition is not enough for our purpose. We need a full vanishing of the homology in degree $-1$
in order to apply our obstruction method.
We therefore consider the action of the involutions.
We easily check that the involution $J_*$ acting on the $n$-Poisson cooperad $\PoisOp_n^c$
and on the associated cotriple resolution $\Res_{op}^{\bullet}(\PoisOp_n^c)$
can be transported to the operadic cobar construction $\DGB_{op}^c(\PoisOp_n^c)$ (by functoriality)
and descends to the Koszul construction $\tilde{\DGK}{}_{op}^c(\PoisOp_n^c)$
as well.

In this dg-module~$\tilde{\DGK}{}_{op}^*(\PoisOp_n^c)$, we have the relation $J_*(\mu_r) = (-1)^{r-1}\mu_r$,
because the element $\mu_r\in\tilde{\DGK}{}_{op}^{r-1}(\PoisOp_n^c)(r)$,
which we consider in the definition of our mapping $\psi(\alpha) = \alpha\otimes\mu_r$
(see~\S\ref{GraphHomology:DeformationComplexHomology:HairyGraphComplexMap} and~\S\ref{GraphHomology:DeformationComplexHomology:TwistedHairyGraphComplexMap})
represents the image of an $r-1$-fold tree-wise composite of the dual of the Lie bracket operations $\lambda^{\vee}\in\PoisOp_n^c$
in the cobar construction, and we have the formula $J_*(\lambda^{\vee}) = -\lambda^{\vee}$
in the $n$-Poisson cooperad~$\PoisOp_n^c$.

By functoriality, we can also transport the action of the involution on the Poisson cooperad $\PoisOp_n^c$
to the Harrison complex $\DGB^{com}_*(\PoisOp_n^c)$,
and to the dual Harrison cochain complex $\hat{\DGB}{}_{com}^*(\PoisOp_n)$.
We can moreover extend this involution to the Harrison complex of the graph operad $\DGB^{com}_*(\GraphOp_n)$
since we observed in~\S\ref{GraphHomology:GraphComplexes:Involutions}
that the involution of the $n$-Poisson operad $\PoisOp_n$
extends to this operad $\GraphOp_n$.
We then provide the twisted end complex $F^{* *}_n$ with the conjugate action of these involutions
on the Harrison complex of the graph operad $\DGB^{com}_*(\GraphOp_n)$
and on the operadic Koszul construction~$\tilde{\DGK}{}_{op}^*(\PoisOp_n^c)$.
We readily check that the involution of the hairy graph complex,
such as defined in~\S\ref{GraphHomology:GraphComplexes:Involutions},
reflects this involution on the twisted end complex $F^{* *}_n$
through the mapping of~\S\ref{GraphHomology:DeformationComplexHomology:TwistedHairyGraphComplexMap}.
From the observations of~\S\ref{GraphHomology:GraphComplexes:Involutions},
we therefore deduce the following result:

\begin{prop}\label{GraphHomology:DeformationComplexHomology:EquivariantObstructions}
We have $\DGH_0(L^{* *}_n)^{J_*}\simeq\kk$ and $\DGH_{-1}(L^{* *}_n)^{J_*}\simeq 0$,
for all $n\geq 3$.\qed
\end{prop}

\section{Recap and proofs of the main theorems}\label{MainResultProofs}

We use the vanishing statements of the previous section, namely Proposition~\ref{GraphHomology:DeformationComplexHomology:RelativeObstructions},
Proposition~\ref{GraphHomology:DeformationComplexHomology:IdentityObstructions}
and Proposition~\ref{GraphHomology:DeformationComplexHomology:EquivariantObstructions},
to establish the main theorems
of the introduction.
We start with the algebraic forms of our statements (Theorem~\ref{Result:HopfDGOperadIntrinsicFormality}-\ref{Result:RelativeChainOperadFormality})
since we deduce our topological statements (Theorem~\ref{Result:SimplicialSetIntrinsicFormality}-\ref{Result:RelativeTopologicalFormality})
from these results and the rational homotopy theory of operads~\cite[\S\S II.8-12]{FresseBook}.

\begin{proof}[Proof of Theorem~\ref{Result:HopfDGOperadIntrinsicFormality}]
We apply the Bousfield obstruction theory to construct a morphism of Hopf $\Lambda$-cooperads
from the resolution $\ROp = |\Res^{com}_{\bullet}(\PoisOp_n^c)|$
of the $n$-Poisson cooperad $\PoisOp_n^c$
towards the coresolution $\QOp = \Tot\Res_{op}^{\bullet}(\KOp))$
of a given Hopf cooperad $\KOp$,
as we explain in~\S\ref{Background:BousfieldObstructionTheory}
and in~\S\S\ref{Background:ObstructionProblem}.

By Theorem~\ref{DeformationComplexes:BicosimplicialComplex:MainResult},
Theorem~\ref{DeformationComplexes:DGComplex:MainResult},
Theorem~\ref{DeformationComplexes:KoszulDuality:MainResult},
Theorem~\ref{GraphHomology:TwistedEndComplexes:MainResult},
and Proposition~\ref{GraphHomology:DeformationComplexHomology:IdentityObstructions:GraphOperadReduction},
the obstruction to the existence of such a morphism lies
in the component of degree $-1$
of the isomorphic homology modules:
\begin{multline*}
\DGH_*(B^{\bullet\bullet}(\PoisOp_n^c,\PoisOp_n^c))
\simeq\DGH_*(D^{* *}(\PoisOp_n^c,\PoisOp_n^c))
\simeq\DGH_*(K^{* *}(\PoisOp_n^c,\PoisOp_n^c))
\\
\simeq\DGH_*(E^{* *}(\PoisOp_n,\PoisOp_n^c))
\simeq\DGH_*(E^{* *}(\GraphOp_n,\PoisOp_n^c))
= \DGH_*(F^{* *}_n)
\end{multline*}
(with the notation adopted in these statements).
By Proposition~\ref{GraphHomology:DeformationComplexHomology:IdentityObstructions},
this homology vanishes
when $n\not\equiv 0(\mymod 4)$.
We can therefore conclude that our morphism exists
in this case.

If the Hopf $\Lambda$-cooperad $\KOp$ is equipped with an involution $J: \KOp\xrightarrow{\simeq}\KOp$
reflecting the involution of the $n$-Poisson operad
in homology, then we can consider the $J_*$-invariant
submodule of our obstruction
complex, as we explain in~\S\ref{Background:BousfieldObstructionTheory},
to check the existence of a $J$-equivariant
morphism by our obstruction method.
We see that the involution action which we consider in Proposition~\ref{GraphHomology:DeformationComplexHomology:EquivariantObstructions}
corresponds to this $J_*$-equivariant
structure
on the obstruction complex.
We therefore conclude, from the vanishing result of Proposition~\ref{GraphHomology:DeformationComplexHomology:EquivariantObstructions},
that our $J$-equivariant morphism exists for all $n\geq 3$.

We use a similar analysis to establish the homotopy uniqueness of our formality weak-equivalence.
The Bousfield obstruction theory and the vanishing result
of Proposition~\ref{GraphHomology:DeformationComplexHomology:IdentityObstructions}
basically imply that $[\ROp,\QOp] = \pi_0\Map_{\dg^*\Hopf\Lambda\Op^c}(\ROp,\QOp)$
is identified with a subquotient of the (underlying set of the) module $\DGH_0(L_n^{**})^{J_*} = \kk$
in the case $n\not\equiv 3(\mymod 4)$.
But the composition with the morphism of the $n$-Poisson operad
which we determine by re-scaling the dual of the Lie bracket
operation $\lambda^{\vee}$
exhausts all these factors in our obstruction spectral sequence,
and this re-scaling operation can be detected
in cohomology. We therefore have a unique representative of our map (up to homotopy)
when we fix the isomorphism $\chi: \DGH^*(\KOp)\xrightarrow{\simeq}\PoisOp_n^c$
induced by our map in cohomology. We argue similarly in the $J$-equivariant
setting. We then use the vanishing result of Proposition~\ref{GraphHomology:DeformationComplexHomology:EquivariantObstructions}
to get our conclusion.
\end{proof}

\begin{proof}[Proof of Theorem~\ref{Result:ChainOperadFormality}]
We merely apply Theorem~\ref{Result:HopfDGOperadIntrinsicFormality} in the special case $\KOp = \DGOmega^*_{\sharp}(\EOp_n)$,
where $\EOp_n$ is a (cofibrant) $E_n$-operad in simplicial sets
such that $\EOp_n(1) = \pt$
and which is equipped with an involution (mimicking the action of a hyperplane reflection on the little discs operad).
We can take for instance (a functorial cofibrant resolution of) the singular complex of the Fulton-MacPherson
operad for $\EOp_n$.
We then dualize and use that $\DGOmega^*_{\sharp}(\EOp_n)^{\vee}$
is weakly equivalent to $\DGC_*(\EOp_n,\QQ)$
as an operad in dg-modules
to get our result.
\end{proof}

\begin{proof}[Proof of Theorem~\ref{Result:RelativeChainOperadFormality}]
The proof proceeds along the same lines as that of Theorem~\ref{Result:HopfDGOperadIntrinsicFormality}
and Theorem~\ref{Result:ChainOperadFormality},
except that we use the relative version
of the Bousfield obstruction theory outlined in~\S\ref{Background:RelativeObstructionProblem},
and the vanishing result of Proposition~\ref{GraphHomology:DeformationComplexHomology:RelativeObstructions}.
\end{proof}

\begin{proof}[Proof of Theorem~\ref{Result:SimplicialSetIntrinsicFormality}]
We derive this result from Theorem~\ref{Result:HopfDGOperadIntrinsicFormality}.
We assume that $\POp$ is an operad in simplicial sets satisfying $\DGH_*(\POp)\simeq\PoisOp_n$
as stated in our theorem.
We pick a cofibrant replacement $\ROp$ of this operad (by using the model structure
of $\Lambda$-operads in simplicial sets~\cite[\S II.8.4]{FresseBook})
and we apply Theorem~\ref{Result:HopfDGOperadIntrinsicFormality}
to (the operadic enhancement of) the Sullivan model
of this operad $\KOp = \DGOmega^*_{\sharp}(\ROp)$.
We get a chain of weak equivalences
$\DGOmega^*_{\sharp}(\ROp)\xrightarrow{\sim}\cdot\xleftarrow{\sim}\PoisOp_n^c$,
to which we apply the derived Sullivan realization functor $\DGL\DGG_{\bullet}$.
We then obtain a chain of weak equivalences of $\Lambda$-operads in simplicial sets:
\begin{equation*}
\DGL\DGG_{\bullet}(\DGOmega^*_{\sharp}(\ROp))\xrightarrow{\sim}\cdot\xleftarrow{\sim}\DGL\DGG_{\bullet}(\PoisOp_n^c).
\end{equation*}
We just use that $\ROp\sphat := \DGL\DGG_{\bullet}(\DGOmega^*_{\sharp}(\ROp))$ represents a rationalization of the operad $\ROp$
when each space $\ROp(r)$ is $\QQ$-good in the sense of Bousfield-Kan (see~\cite[\S II.12.2.3]{FresseBook}).
\end{proof}

\begin{proof}[Proof of Theorem~\ref{Result:TopologicalFormality}]
We merely apply Theorem~\ref{Result:SimplicialSetIntrinsicFormality} to the case of a model of $E_n$-operads $\POp = \ROp_n$
in the category of $\Lambda$-operads in simplicial sets.
\end{proof}

\begin{proof}[Proof of Theorem~\ref{Result:RelativeTopologicalFormality}]
We again use the relative version of the Bousfield obstruction theory to establish the formality of the morphisms
$\iota^*: \DGOmega^*_{\sharp}(\EOp_n)\rightarrow\DGOmega^*_{\sharp}(\EOp_n)$
that model the embeddings of little discs operads
$\iota: \DOp_m\hookrightarrow\DOp_n$,
and we apply the functor $\DGL\DGG_{\bullet}$
to get our result.
\end{proof}

\begin{proof}[Proof of Theorem~\ref{Result:InitialRelativeTopologicalFormality}-\ref{Result:InitialRelativeChainOperadFormality}]
We can readily adapt our constructions and the proof of Theorem~\ref{Result:RelativeTopologicalFormality}-\ref{Result:RelativeChainOperadFormality}
in order to establish the results of Theorem~\ref{Result:InitialRelativeTopologicalFormality}
and Theorem~\ref{Result:InitialRelativeChainOperadFormality},
where we deal the particular case of the operad of little intervals $\DOp_1$
as source of our operad morphisms $\iota: \DOp_1\hookrightarrow\DOp_n$.
We then have to replace the $m$-Poisson operad $\PoisOp_m$
by the associative operad $\AsOp$
in all our constructions
since we have $\DGH_*(\DOp_1) = \AsOp$
in this case.
We also consider the dual cooperad of this operad $\AsOp^c$.
We essentially use that the associative cooperad $\AsOp^c$ is Koszul (with the operadic suspension
of the associative operad $\KK^c_{op}(\AsOp^c) = \SuspOp\AsOp$
as Koszul dual operad) to perform the Koszul reduction step
of~\S\ref{DeformationComplexes:KoszulDuality}.
We readily check that the rest of our arguments work same when we replace the Koszul construction
of the $m$-Poisson cooperad $\tilde{\DGK}{}_{op}^*(\PoisOp_m^c)$
by the Koszul construction $\tilde{\DGK}{}_{op}^*(\AsOp^c) = \DGSigma\overline{\SuspOp}\overline{\AsOp}$
associated to this cooperad $\AsOp^c$.
\end{proof}

\begin{appendix}

\renewcommand{\thesubsubsection}{\thesection.\arabic{subsubsection}}
\numberwithin{subsubsection}{section}

\section{The algebraic cotriple resolution}\label{CotripleResolution}
In this appendix, we explain the definition of the cotriple resolution for Hopf $\Lambda$-cooperads with full details,
and we check that the application of the geometric realization functor to these simplicial resolutions
returns cofibrant resolutions
in the category of Hopf $\Lambda$-cooperads.

\begin{constr}[The cotriple resolution of Hopf $\Lambda$-cooperads]\label{CotripleResolution:Construction}
In~\S\ref{Background:CotripleResolution}, we briefly explain that we use the adjunction
\begin{equation}\label{CotripleResolution:Construction:Adjunction}
\ComOp^c/\Sym(-): \ComOp^c/\dg^*\Lambda\Op^c\rightleftarrows\dg^*\Hopf\Lambda\Op^c :\omega
\end{equation}
between the category of coaugmented $\Lambda$-cooperads in cochain graded dg-modules $\ComOp^c/\dg^*\Lambda\Op^c$
and the category of Hopf $\Lambda$-cooperads $\dg^*\Hopf\Lambda\Op^c$
in order to define the cotriple resolution $\Res^{com}_{\bullet}(\AOp)\in\simp\dg^*\Hopf\Lambda\Op^c$
of any object $\AOp$
in the category of Hopf $\Lambda$-cooperads $\dg^*\Hopf\Lambda\Op^c$.

This simplicial object $\Res^{com}_{\bullet}(\AOp)$ is explicitly defined by the expression:
\begin{equation}\label{CotripleResolution:Construction:Expression}
\Res^{com}_n(\AOp) = \underbrace{(\ComOp^c/\Sym)\circ\cdots\circ(\ComOp^c/\Sym)}_{n+1}(\AOp),
\end{equation}
for each dimension $n\in\NN$, where we consider an $n+1$-fold composite of the functors of our adjunction relation.
We just omit to mark forgetful functors in order to simplify our formula.
We number the factors of this composite by $0,\dots,n$, from left to right.
We provide $\Res^{com}_{\bullet}(\AOp)$ with the face morphisms
$d_i: \Res^{com}_n(\AOp)\rightarrow\Res^{com}_{n-1}(\AOp)$
given, for $i = 0,\dots,n$,
by the application of the augmentation morphism of our adjunction $\epsilon: (\ComOp^c/\Sym) = (\ComOp^c/\Sym)\circ\omega\rightarrow\Id$
to the $i$th factor of our composite functor~(\ref{CotripleResolution:Construction:Expression}),
while the degeneracy morphisms
$s_j: \Res^{com}_n(\AOp)\rightarrow\Res^{com}_{n+1}(\AOp)$,
are given, for $j = 0,\dots,n$, by the insertion of the adjunction unit $\iota: \Id\rightarrow\omega\circ(\ComOp^c/\Sym)$
between the $j$ and $j+1$st factors of this composite.

The object~(\ref{CotripleResolution:Construction:Expression}) is a (relative) symmetric algebra
by construction $\Res^{com}_n(\AOp) = \ComOp^c/\Sym(\DGC^{com}_n(\AOp))$,
for a generating cooperad such that:
\begin{equation}\label{CotripleResolution:Construction:GeneratingCollection}
\DGC^{com}_n(\AOp) = \underbrace{(\ComOp^c/\Sym)\circ\cdots\circ(\ComOp^c/\Sym)}_n(\AOp),
\end{equation}
for any $n\in\NN$. In~\S\ref{DeformationComplexes:BicosimplicialComplex}, we also use the notation $\overline{\DGC}^{com}_n(\AOp)$
for the coaugmentation coideal of these cooperads $\DGC^{com}_n(\AOp)$,
where we drop the term of arity one $\DGC^{com}_n(\AOp)(1) = \kk$.
We immediately see that the face operators $d_i$ such that $i>0$
are identified with morphisms of symmetric algebras $d_i: \ComOp^c/\Sym(\DGC^{com}_n(\AOp))\rightarrow\ComOp^c/\Sym(\DGC^{com}_{n-1}(\AOp))$
which we associate to face morphisms of these generating cooperads $d_i: \DGC^{com}_n(\AOp)\rightarrow\DGC^{com}_{n-1}(\AOp)$,
and we have a similar observation for the degeneracy operators $s_j: \Res^{com}_n(\AOp)\rightarrow\Res^{com}_{n+1}(\AOp)$, for all $j$.
But the $0$-face $d_0: \Res^{com}_n(\AOp)\rightarrow\Res^{com}_{n-1}(\AOp)$, on the other hand,
is yielded by a morphism $d_0: \ComOp^c/\Sym(\DGC^{com}_n(\AOp))\rightarrow\ComOp^c/\Sym(\DGC^{com}_{n-1}(\AOp))$
which does not preserve our generating objects.

The simplicial object $\Res^{com}_{\bullet}(\AOp)$ is also equipped with an augmentation
\begin{equation}\label{CotripleResolution:Construction:Augmentation}
\Res^{com}_0(\AOp) = (\ComOp^c/\Sym)(\AOp)\xrightarrow{\epsilon}\AOp
\end{equation}
which we define by the augmentation morphism of our adjunction relation~(\ref{CotripleResolution:Construction:Adjunction}).
\end{constr}

We need more insights into the internal structure of the cotriple resolution of Hopf $\Lambda$-cooperads.
We are mainly going to observe that the components of this Hopf $\Lambda$-cooperad
are identified with the cotriple resolution of plain unitary commutative algebras
in cochain graded dg-modules.
We notably use this relationship when we define the equivalence between the bicosimplicial deformation complex of a Hopf $\Lambda$-cooperad
and our differential graded deformation complex in~\S\ref{DeformationComplexes:DGComplex}.

\begin{constr}[The components of the cotriple resolution]\label{CotripleResolution:ComponentwiseConstruction}
Recall that we use the notation $\dg^*\ComCat_+$ for the category of unitary commutative algebras in $\dg^*\Mod$
(the category of unitary commutative cochain dg-algebras).
We consider the category of cochain graded dg-modules equipped with a coaugmentation over the ground field $\kk/\dg^*\Mod$
and the relative symmetric algebra functor $\kk/\Sym(-): \kk/\dg^*\Mod\rightarrow\dg^*\ComCat_+$
such that $\kk/\Sym(M) = \kk\otimes_{\Sym(\kk)}\Sym(M)$
for any object $M\in\kk/\dg^*\Mod$.

We now have an adjunction $\kk/\Sym(-): \kk/\dg^*\Mod\rightleftarrows\dg^*\ComCat_+ :\omega$
between the category of coaugmented cochain graded dg-modules $\kk/\dg^*\Mod$
and the category of unitary commutative cochain dg-algebras $\dg^*\ComCat_+$
which we use to define the cotriple resolution $\Res^{com}_{\bullet}(A)$
of any plain unitary commutative cochain dg-algebra $A\in\dg^*\ComCat_+$.
We explicitly set:
\begin{equation}\label{CotripleResolution:ComponentwiseConstruction:Expression}
\Res^{com}_n(A) = \underbrace{(\kk/\Sym)\circ\cdots\circ(\kk/\Sym)}_{n+1}(A),
\end{equation}
for each dimension $n\in\NN$, where we consider the $n+1$-fold composite of our functor $\kk/\Sym(-)$ on $\dg^*\ComCat_+$.
We define the face and degeneracy operators of this simplicial object by the same operations
as in~\S\ref{CotripleResolution:Construction}.
We still have an identity $\Res^{com}_n(A) = \kk/\Sym(\DGC^{com}_n(A))$
by construction,
for a generating dg-module such that:
\begin{equation}\label{CotripleResolution:ComponentwiseConstruction:GeneratingComplex}
\DGC^{com}_n(A) = \underbrace{(\kk/\Sym)\circ\cdots\circ(\kk/\Sym)}_n(A),
\end{equation}
for any $n\in\NN$.
We readily see, again, that the face operators $d_i$ such that $i>0$ preserve these generating objects,
as well as the degeneracy operators $s_j$ for all $j$ (but not the $0$-face $d_0$).
We also have an augmentation $\Res^{com}_0(A) = \kk/\Sym(A)\xrightarrow{\epsilon} A$
yielded by the augmentation morphism of the adjunction $\kk/\Sym(-): \kk/\dg^*\Mod\rightleftarrows\dg^*\ComCat_+ :\omega$.

We already explained, in~\S\ref{Background:AlgebraicAdjunctions}, that we have an obvious identity
$\ComOp^c/\Sym(\COp)(r) = \kk/\Sym(\COp(r))$,
for any coaugmented $\Lambda$-cooperad $\COp\in\ComOp^c/\dg^*\Lambda\Op^c$, where we consider the relative symmetric algebra
associated to the coaugmented object $\COp(r)\in\kk/\dg^*\Mod$
on the right-hand side.
We obtain, as a consequence, that we have the identity:
\begin{equation}
\Res^{com}_{\bullet}(\AOp)(r) = \Res^{com}_{\bullet}(\AOp(r)),
\end{equation}
for any Hopf $\Lambda$-cooperad $\AOp\in\dg^*\Hopf\Lambda\Op^c$, where, on the right-hand side, we consider the cotriple resolution
of the object $\AOp(r)\in\dg^*\ComCat_+$
underlying $\AOp$, for each $r>0$.
\end{constr}

\begin{constr}[The case of augmented Hopf $\Lambda$-cooperads]\label{CotripleResolution:AugmentedCase}
We can simplify the expression of the cotriple resolution in the case of Hopf cooperads $\AOp$
which are endowed with an augmentation
over the commutative cooperad $\eta_*: \AOp\rightarrow\ComOp^c$
(like the Poisson cooperad $\AOp = \PoisOp_n^c$).
We assume, in our setting, that this augmentation is a morphism of Hopf $\Lambda$-cooperads.
We then get that the augmentation ideals of the algebras $\AOp(r)$, which we define by $\IOp\AOp(r) = \ker(\eta_*: \AOp(r)\rightarrow\ComOp^c(r))$,
define an object of the category of $\Lambda$-collections.
We moreover have the relation $\overline{\AOp} = \overline{\ComOp}{}^c\oplus\IOp\AOp$
in this category. (We equivalently have the identity $\AOp(r) = \ComOp^c(r)\oplus\IOp\AOp(r)$, for $r>1$.)
We immediately see that, in this situation, the simplicial Hopf $\Lambda$-cooperad $\Res^{com}_{\bullet}(\AOp)$
inherits an augmentation over the commutative cooperad $\eta_*: \Res^{com}_{\bullet}(\AOp)\rightarrow\ComOp^c$
which is defined, in dimension $0$,
by the composite of the augmentation morphism of~\S\ref{CotripleResolution:Construction}(\ref{CotripleResolution:Construction:Augmentation})
with the augmentation morphism of our object $\eta_*: \AOp\rightarrow\ComOp^c$.
We can still determine the morphisms of Hopf $\Lambda$-cooperads $\eta_*: \Res^{com}_n(\AOp)\rightarrow\ComOp^c$
by morphisms of coaugmented $\Lambda$-cooperads $\eta_*: \DGC^{com}_n(\AOp)\rightarrow\ComOp^c$
which we define on the generating objects of~\S\ref{CotripleResolution:Construction}(\ref{CotripleResolution:Construction:GeneratingCollection}).
We adopt the notation $\DGI\DGC^{com}_n(\AOp)$ for the kernel of this augmentation morphism
so that we have the splitting formula $\overline{\DGC}^{com}_n(\AOp) = \overline{\ComOp}{}^c\oplus\DGI\DGC^{com}_n(\AOp)$
in any dimension $n\in\NN$.

Let now $\II\Sym(-)$ denote the natural augmentation ideal of the symmetric algebra. We consider the obvious (arity-wise) extension
of this functor $\II\Sym(-)$ to the category of $\Lambda$-collections.
We actually have the identity:
\begin{equation}\label{CotripleResolution:AugmentedCase:Expression}
\DGI\DGC^{com}_n(\AOp) = \underbrace{\II\Sym\circ\cdots\circ\II\Sym}_n(\IOp\AOp),
\end{equation}
for any $n\in\NN$, and we get $\overline{\Res}{}^{com}_{\bullet}(\AOp) = \Sym(\DGI\DGC^{com}_n(\AOp))$,
for the coaugmentation coideal  $\overline{\Res}{}^{com}_{\bullet}(\AOp)$
of the cotriple resolution $\Res^{com}_{\bullet}(\AOp)$.

We immediately see that the objects $\DGI\DGC^{com}_n(\AOp)$, $n\in\NN$, are, like the cooperads $\DGC^{com}_n(\AOp)$, $n\in\NN$,
preserved by the face operators $d_i$ such that $i>0$
and by the degeneracies $s_j$ as well, for all $j$ (but not by the $0$-face as usual).
In the context of augmented objects, we can nonetheless set $d_0 = 0$ to provide the collections $\DGI\DGC^{com}_n(\AOp)$, $n\in\NN$,
with a full simplicial structure though this trivial face operator $d_0 = 0$
does not correspond to the $0$-face of the cotriple resolution.
\end{constr}

In~\S\ref{DeformationComplexes:BicosimplicialComplex}, we also use that if the object $\IOp\AOp$
has a free structure $\IOp\AOp = \Lambda\otimes_{\Sigma}\SOp\AOp$,
for some symmetric collection $\SOp\AOp\subset\IOp\AOp$,
then we still have a dimension-wise identity $\DGI\DGC^{com}_n(\AOp) = \Lambda\otimes_{\Sigma}\DGS\DGC^{com}_n(\AOp)$, for any $n\in\NN$,
for some sub-object $\DGS\DGC^{com}_n(\AOp)$ of $\DGI\DGC^{com}_n(\AOp)$
in the category of symmetric collections.
This observation follows from the expression of $\DGI\DGC^{com}_{\bullet}(\AOp)$
in terms of (iterated symmetric) tensors on $\IOp\AOp$,
and from the following general proposition:

\begin{prop}\label{CotripleResolution:FreeCollectionTensors}
Let $\MOp_i$, $i = 1,\dots,n$, be an $n$-tuple of $\Lambda$-collections in the category of (graded) modules.
If each $\MOp_i$, $i = 1,\dots,n$, has a free structure $\MOp_i = \Lambda\otimes_{\Sigma}\SOp\MOp_i$
for some symmetric collection $\SOp\MOp_i\subset\MOp_i$,
then so does the arity-wise tensor product
of these $\Lambda$-collections:
$(\MOp_1\otimes\dots\otimes\MOp_n)(r) = \MOp_1(r)\otimes\dots\otimes\MOp_n(r)$.
This result also holds for the $\Lambda$-collection
$\Sym_n(E,\MOp)(r) = (E\otimes\MOp(r)^{\otimes n})_{\Sigma_n}$
which we obtain by applying a functor of symmetric tensors $\Sym_n(E,X) = (E(n)\otimes X^{\otimes n})_{\Sigma_n}$
to a single $\Lambda$-collection $\MOp = \MOp_1 = \dots = \MOp_n$ arity-wise,
for any (graded) module of coefficients $E = E(n)$
endowed with an action of the symmetric group $\Sigma_n$.
\end{prop}

\begin{proof}
We consider the category $\Lambda^+\subset\Lambda$ with the same objects as the category $\Lambda$,
and the subset of order preserving injective maps
as morphisms (see~\cite[\S I.2.2.2]{FresseBook}).
The coend relation $\MOp = \Lambda\otimes_{\Sigma}\SOp\MOp$ for a $\Lambda$-collection equipped with a free structure $\MOp$
is then equivalent to:
\begin{equation}\label{CotripleResolution:FreeCollectionTensors:Decomposition}
\MOp(r) = \bigoplus_{\substack{u\in\Mor_{\Lambda^+}(\kset,\rset)\\
\kset\in\Lambda}}\SOp\MOp(k)_u,
\end{equation}
where $\SOp\MOp(\kset)_u\subset\MOp(r)$ denotes a formal copy of the object $\SOp\MOp(\kset)$
associated to any morphism $u\in\Mor_{\Lambda^+}(\kset,\rset)$
of our subcategory $\Lambda^+$ (see also~\cite[\S II.11.2.3]{FresseBook}).
In fact, the summands $\SOp\MOp(\kset)_u$ of this decomposition~(\ref{CotripleResolution:FreeCollectionTensors:Decomposition})
represent the image of the object $\SOp\MOp(k)\subset\MOp(k)$
under the corestriction operators $u_*: \MOp(k)\rightarrow\MOp(r)$
in the module $\MOp(r)$.
For a tensor product $\MOp = \MOp_1\otimes\dots\otimes\MOp_n$,
we have the distribution formula:
\begin{equation}\label{CotripleResolution:FreeCollectionTensors:DistributionFormula}
(\MOp_1\otimes\dots\otimes\MOp_n)(r) = \bigoplus_{\substack{u_i\in\Mor_{\Lambda^+}(\kset_i,\rset),\kset_i\in\Lambda\\
i = 1,\dots,n}}\SOp\MOp(k_1)_{u_1}\otimes\dots\otimes\SOp\MOp(k_n)_{u_n},
\end{equation}
and we consider the symmetric collection~$\SOp(\MOp_1\otimes\dots\otimes\MOp_n)$
such that
\begin{equation}\label{CotripleResolution:FreeCollectionTensors:GeneratingCollection}
\SOp(\MOp_1\otimes\dots\otimes\MOp_n)(r) = \bigoplus_{\substack{v_i\in\Mor_{\Lambda^+}(\kset_i,\rset)\\
v_1(\kset_1)\cup\dots\cup v_n(\kset_n) = \rset}}\SOp\MOp(k_1)_{v_1}\otimes\dots\otimes\SOp\MOp(k_n)_{v_n},
\end{equation}
for any $r\in\NN$, where the sum runs over the $n$-tuples of maps $v_i\in\Mor_{\Lambda^+}(\kset_i,\rset)$, $\kset_i\in\Lambda$,
whose images cover $\rset = \{1<\dots<r\}$.

For any $s\in\Sigma_r$, the composites $u_i = s v_i$ admit a decomposition $u_i = s v_i = w_i t_i$, where $w_i\in\Mor_{\Lambda^+}(\kset_i,\rset)$
and $t_i\in\Sigma_{k_i}$. This observation implies that our object~(\ref{CotripleResolution:FreeCollectionTensors:GeneratingCollection})
is preserved by the action of the symmetric group $\Sigma_r$
on $(\MOp_1\otimes\dots\otimes\MOp_n)(r) = \MOp_1(r)\otimes\dots\otimes\MOp_n(r)$,
for any $r\in\NN$.
Furthermore, for any $n$-tuple of maps $u_i\in\Mor_{\Lambda^+}(\kset_i,\rset)$, $i = \dots,n$, we have a unique map $u\in\Mor_{\Lambda^+}(\lset,\rset)$
satisfying $u_i = u v_i$, for $i = 1,\dots,n$, and so that the images of the maps $v_i\in\Mor_{\Lambda^+}(\kset_i,\lset)$, $i = \dots,n$,
cover the set~$\lset$.
We accordingly have:
\begin{equation}\label{CotripleResolution:FreeCollectionTensors:TensorFreeStructure}
(\MOp_1\otimes\dots\otimes\MOp_n)(r) = \bigoplus_{\substack{u\in\Mor_{\Lambda^+}(\lset,\rset)\\
\lset\in\Lambda}} u_*\SOp(\MOp_1\otimes\dots\otimes\MOp_n)(l),
\end{equation}
for any $r\in\NN$, and the first assertion of the proposition follows.

This collection in~(\ref{CotripleResolution:FreeCollectionTensors:GeneratingCollection}) is obviously preserved by tensor permutations,
and when we assume $\MOp = \MOp_1 = \dots = \MOp_n$,
we immediately get:
\begin{multline}\label{CotripleResolution:FreeCollectionTensors:SymmetricTensorFreeStructure}
\MOp(r)^{\otimes n} = \int^{\kset\in\Sigma}\Mor_{\Lambda}(\kset,\rset)\otimes\SOp(\MOp^{\otimes n})(k)
\\
\Rightarrow(E(n)\otimes\MOp(r)^{\otimes n})_{\Sigma_n}
= \int^{\kset\in\Sigma}\Mor_{\Lambda}(\kset,\rset)\otimes(E(n)\otimes\SOp(\MOp^{\otimes n})(k))_{\Sigma_n}
\end{multline}
by interchange of colimits, for any functor of symmetric tensors $\Sym_n(E,X) = (E(n)\otimes X^{\otimes n})_{\Sigma_n}$.
Hence, we are also done with the second assertion of the proposition.
\end{proof}

We have the following statement, where we use the model structures of~\S\ref{Background:CooperadModelCategories}:

\begin{prop}\label{CotripleResolution:ReedyCofibrant}
If $\AOp\in\dg^*\Hopf\Lambda\Op^c$ is cofibrant as an object of the category of coaugmented $\Lambda$-cooperads $\ComOp^c/\dg^*\Lambda\Op^c$,
then $\Res^{com}_{\bullet}(\AOp)$ forms a Reedy cofibrant simplicial object
in the category of Hopf $\Lambda$-cooperads $\dg^*\Hopf\Lambda\Op^c$.
\end{prop}

\begin{proof}[Proof (sketch)]
We have $\DGL_n\Res^{com}_{\bullet}(\AOp) = \ComOp^c/\Sym(\DGL_n\DGC^{com}_{\bullet}(\AOp))$,
where $\DGL_n$ refers to the $n$th latching functor, and the verification
of the proposition reduces to proving that the latching map
induces a cofibration of coaugmented $\Lambda$-cooperads
at the level of our generating object.

In this paper, we only really use the case where $\AOp$ is equipped with an augmentation over the commutative cooperad $\ComOp^c$
such that $\IOp\AOp$
is free as a $\Lambda$-collection, in the sense that we have an identity $\IOp\AOp = \Lambda\otimes_{\Sigma}\SOp\AOp$
for some symmetric collection $\SOp\AOp\subset\IOp\AOp$ (since we only consider the case $\AOp = \PoisOp_n^c$
in our applications of this proposition).
We then use that $\DGC^{com}_{\bullet}(\AOp)$ consists of tensor products of the $\Lambda$-collection $\IOp\AOp$
and that a tensor product of free $\Lambda$-collections is still free
as a $\Lambda$-collection (by the previous proposition).
In this case, we easily check that the latching map $\lambda: \DGL_n\DGC^{com}_{\bullet}(\AOp)\rightarrow\DGC^{com}_n(\AOp)$
is given by a twisted direct sum in the sense of~\cite[\S 11.3.6]{FresseBook}.
By Proposition 11.3.7 of \emph{loc. cit.}
this is enough to ensure that this map defines a cofibration
in the category of coaugmented $\Lambda$-cooperads.
Thus we easily get the claim of this proposition.

In the case where $\AOp$ is a general cofibrant object of the category of coaugmented $\Lambda$-cooperads,
we may use that the forgetful functor from coaugmented $\Lambda$-cooperads
to coaugmented $\Lambda$-collections create cofibrations
to extend the above argument line
and to conclude that $\Res^{com}_{\bullet}(\AOp)$ forms a Reedy cofibrant simplicial object
again.
\end{proof}

In~\S\ref{Background:CotripleResolution}, we take the geometric realization of the simplicial object $\Res^{com}_{\bullet}(\AOp)$
in the category $\dg^*\Hopf\Lambda\Op^c$ in order to get our cofibrant resolutions of Hopf $\Lambda$-cooperads.
We review the general definition of these geometric realization functors before tackling the applications to cotriple resolutions.


\begin{constr}[The geometric realization]\label{CotripleResolution:GeometricRealization}
We first assume that $\ROp_{\bullet}$ is a simplicial object in the category of Hopf $\Lambda$-cooperads.
To perform the geometric realization construction,
we first have to pick a cosimplicial framing
of this object
\begin{equation}\label{CotripleResolution:GeometricRealization:CosimplicialFraming}
\ROp_{\bullet}\otimes\Delta^{\bullet}\in\cosimp\simp\dg^*\Hopf\Lambda\Op^c
\end{equation}
in the Reedy model category of simplicial objects in $\dg^*\Hopf\Lambda\Op^c$.
Then we formally set:
\begin{equation}\label{CotripleResolution:GeometricRealization:CoendFormula}
|\ROp_{\bullet}| = \int^{\underline{n}\in\Delta}\ROp_n\otimes\Delta^n,
\end{equation}
where we form our coend in the category $\dg^*\Hopf\Lambda\Op^c$.

Recall simply that cosimplicial framings exist by general model category arguments (see for instance~\cite[\S 16.6]{Hirschhorn},
\cite[\S 5.2]{Hovey}, or~\cite[\S II.3.3.1]{FresseBook}),
and that the geometric realization of a Reedy cofibrant simplicial object
does not depend on the choice of a particular cosimplicial framing (see for instance~\cite[Theorem II.3.3.6]{FresseBook}).
The general theory of model categories moreover implies that $|\ROp_{\bullet}|$
forms a cofibrant object in $\dg^*\Hopf\Lambda\Op^c$
as soon as $\ROp_{\bullet}$
is Reedy cofibrant.
\end{constr}

The geometric realization of the cotriple resolution of a Hopf $\Lambda$-cooperad $\ROp_{\bullet} = \Res^{com}_{\bullet}(\AOp)$
is equipped with an augmentation $\epsilon: |\ROp_{\bullet}|\rightarrow\AOp$.
This morphism is defined, on our coend~\S\ref{CotripleResolution:GeometricRealization}(\ref{CotripleResolution:GeometricRealization:CoendFormula}),
by the augmentation morphism
of the cotriple resolution $\epsilon: \ROp_0\rightarrow\AOp$.
(Recall that, by definition of a cosimplicial framing, we have an identity $\ROp_0 = \ROp_0\otimes\Delta^0$ in dimension $n=0$.)
We have the following statement:

\begin{lemm}\label{CotripleResolution:ResolutionClaim}
If $\AOp\in\dg^*\Hopf\Lambda\Op^c$ is cofibrant as an object of the category of coaugmented $\Lambda$-cooperads $\ComOp^c/\dg^*\Lambda\Op^c$,
then the augmentation morphism of the cotriple resolution $\epsilon: \Res^{com}_0(\AOp)\rightarrow\AOp$
induces a weak-equivalence
in the category of Hopf $\Lambda$-cooperads $\dg^*\Hopf\Lambda\Op^c$
when we pass to geometric realizations:
\begin{equation*}
\epsilon: |\Res^{com}_{\bullet}(\AOp)|\xrightarrow{\sim}\AOp.
\end{equation*}
\end{lemm}

\begin{proof}
We already briefly mentioned that the components of the Hopf $\Lambda$-cooperad $|\Res^{com}_{\bullet}(\AOp)|$
are identified with the geometric realization of the objects $\Res^{com}_{\bullet}(\AOp(r))$
in the category of unitary commutative cochain dg-algebras.
We therefore have a weak-equivalence
$|\Res^{com}_{\bullet}(\AOp)|(r) = |\Res^{com}_{\bullet}(\AOp(r))|\xrightarrow{\sim}\AOp(r)$,
in each arity $r>0$, because this is so for the cotriple resolution of any object
in the category of plain unitary commutative cochain dg-algebras (see~\cite{CotripleResolutions}).
The conclusion of our lemma follows.
\end{proof}

In~\S\ref{CotripleResolution:GeometricRealization}, we mentioned that the geometric realization of a Reedy cofibrant simplicial object
defines a cofibrant object in the ambient category. Therefore, from Proposition~\ref{CotripleResolution:ReedyCofibrant}
and Lemma~\ref{CotripleResolution:ResolutionClaim}
together, we get the following statement:

\begin{thm}
Let $\AOp\in\dg^*\Hopf\Lambda\Op^c$.
If $\AOp$ is cofibrant as an object of the category of coaugmented $\Lambda$-cooperads $\ComOp^c/\dg^*\Lambda\Op^c$,
then the geometric realization of the cotriple resolution $\ROp_{\bullet} = \Res^{com}_{\bullet}(\AOp)\in\simp\dg^*\Hopf\Lambda\Op^c$
defines a cofibrant resolution $\ROp = |\Res^{com}_{\bullet}(\AOp)|$
of the object $\AOp$
in the category of Hopf $\Lambda$-cooperads in cochain graded dg-modules $\dg^*\Hopf\Lambda\Op^c$.
\qed
\end{thm}

The weak-equivalence attached to the resolution of this theorem $\ROp = |\Res^{com}_{\bullet}(\AOp)|$
is precisely defined as the morphism of Hopf $\Lambda$-cooperads
$\epsilon: |\Res^{com}_{\bullet}(\AOp)|\xrightarrow{\sim}\AOp$
which we deduce from the augmentation $\epsilon: \Res^{com}_0(\AOp)\rightarrow\AOp$
of the simplicial object $\Res^{com}_{\bullet}(\AOp)\in\simp\dg^*\Hopf\Lambda\Op^c$
(as in Lemma~\ref{CotripleResolution:ResolutionClaim}).

\section{The cooperadic triple coresolution}\label{TripleCoresolution}
In this appendix, we review the definition of the triple coresolution for Hopf $\Lambda$-cooperads,
and we check that the totalization of these cosimplicial coresolutions returns fibrant resolutions
in the category of Hopf $\Lambda$-cooperads.

\begin{constr}[The triple coresolution of Hopf $\Lambda$-cooperads]\label{TripleCoresolution:Construction}
In~\S\ref{Background:TripleCoresolution}, we briefly explain that we apply the standard triple coresolution construction
to the adjunction
\begin{equation}\label{TripleCoresolution:Construction:Adjunction}
\overline{\omega}: \dg^*\Hopf\Lambda\Op^c\rightleftarrows\dg^*\Hopf\Lambda\Seq_{>1}^c :\FreeOp^c
\end{equation}
in order to get a cosimplicial coresolution $\Res_{op}^{\bullet}(\KOp)$ of any object $\KOp$
in the category $\dg^*\Hopf\Lambda\Op^c$.

Let $\overline{\FreeOp}^c = \overline{\omega}\FreeOp^c$ denote the composite of the cofree cooperad functor $\FreeOp^c$
with the coaugmentation coideal functor on Hopf $\Lambda$-cooperads $\overline{\omega}$.
The structure morphisms of our adjunction~(\ref{TripleCoresolution:Construction:Adjunction})
determine a comonadic product $\nu: \overline{\FreeOp}^c\rightarrow\overline{\FreeOp}^c\circ\overline{\FreeOp}^c$
together with a comonadic counit $\epsilon: \overline{\FreeOp}^c\rightarrow\Id$
on this functor $\overline{\FreeOp}^c = \overline{\omega}\FreeOp^c$.
We moreover have a morphism $\nu: \FreeOp^c\rightarrow\FreeOp^c\circ\overline{\FreeOp}^c$
that gives a right coaction of this comonad $\overline{\FreeOp}^c$
on $\FreeOp^c$, and a morphism $\rho: \overline{\KOp}\rightarrow\overline{\FreeOp}^c(\overline{\KOp})$
that provides the object $\overline{\KOp}$
with the structure of a coalgebra over the comonad $\overline{\FreeOp}^c$.

We explicitly define the cosimplicial object $\Res_{op}^{\bullet}(\KOp)$
by the composite functor construction:
\begin{equation}\label{TripleCoresolution:Construction:Expression}
\Res_{op}^n(\KOp) = \FreeOp^c\circ\underbrace{\overline{\FreeOp}^c\circ\cdots\circ\overline{\FreeOp}^c}_n(\overline{\KOp}),
\end{equation}
for each dimension $n\in\NN$.

We number the factors of this composite by $0,\dots,n$ from left to right.
We provide $\Res_{op}^{\bullet}(\KOp)$
with the coface morphisms
$d^i: \Res_{op}^{n-1}(\KOp)\rightarrow\Res_{op}^{n}(\KOp)$
given by the structure morphism $\nu: \FreeOp^c\rightarrow\FreeOp^c\circ\overline{\FreeOp}^c$
of the coaction of the comonad $\overline{\FreeOp}^c$ on the functor $\FreeOp^c$
in the case $i = 0$,
by the application of the comonadic coproduct $\nu: \overline{\FreeOp}^c\rightarrow\overline{\FreeOp}^c\circ\overline{\FreeOp}^c$
to the $i$th factors of our composite
in the case $i = 1,\dots,n-1$,
and by the coalgebra structure morphism $\rho: \overline{\KOp}\rightarrow\overline{\FreeOp}^c(\overline{\KOp})$
in the case $i = n$.
We provide $\Res_{op}^{\bullet}(\KOp)$
with the codegeneracy morphisms $s^j: \Res_{op}^{n+1}(\KOp)\rightarrow\Res_{op}^{n}(\KOp)$
given by the application of the comonadic counit $\epsilon: \overline{\FreeOp}^c\rightarrow\Id$
to the $j+1$st factor of our composite
in all cases $j = 0,\dots,n$.

The object~(\ref{TripleCoresolution:Construction:Expression})
is a cofree cooperad by construction $\Res_{op}^n(\KOp) = \FreeOp^c(\DGC_{op}^n(\KOp))$,
for a cogenerating $\Lambda$-collection
such that:
\begin{equation}\label{TripleCoresolution:Construction:GeneratingCollection}
\DGC_{op}^n(\KOp) = \underbrace{\overline{\FreeOp}^c\circ\dots\circ\overline{\FreeOp}^c}_n(\overline{\KOp}),
\end{equation}
for any $n\in\NN$. We immediately see that the coface operators $d^i: \Res_{op}^{n-1}(\KOp)\rightarrow\Res_{op}^n(\KOp)$ such that $i>0$
are identified with morphisms of cofree cooperads $d^i: \FreeOp^c(\DGC_{op}^{n-1}(\KOp))\rightarrow\FreeOp^c(\DGC_{op}^{n}(\KOp))$
which we associate to coface morphisms of these cogenerating $\Lambda$-collections $d^i: \DGC_{op}^{n-1}(\KOp)\rightarrow\DGC_{op}^{n}(\KOp)$,
and we have a similar observation for the codegeneracy operators $s^j: \Res_{op}^{n+1}(\KOp)\rightarrow\Res_{op}^{n}(\KOp)$, for all $j$.
But the $0$-coface $d^0: \Res_{op}^{n-1}(\KOp)\rightarrow\Res_{op}^{n}(\KOp)$, on the other hand,
is yielded by a morphism $d^0: \FreeOp^c(\DGC_{op}^{n-1}(\KOp))\rightarrow\FreeOp^c(\DGC_{op}^{n}(\KOp))$
which does not preserve our cogenerating objects.
We can still set $d^0 = 0$ to provide these objects~(\ref{TripleCoresolution:Construction:GeneratingCollection})
with a full cosimplicial structure, but this $0$-coface does not correspond to the $0$-coface
of the triple coresolution.

The cosimplicial object $\Res_{op}^{\bullet}(\KOp)$ is also equipped with a canonical
coaugmentation
\begin{equation}\label{TripleCoresolution:Construction:Coaugmentation}
\KOp\xrightarrow{\eta}\FreeOp^c(\overline{\KOp}) = \Res_{op}^0(\KOp)
\end{equation}
which is yielded by the unit morphism
of our adjunction relation.
\end{constr}

\begin{constr}[The reduction of the triple coresolution to cooperads]\label{TripleCoresolution:CooperadReduction}
Recall that our cofree object functor on the category of Hopf $\Lambda$-cooperads
is defined
by lifting the plain cofree cooperad functor $\FreeOp^c: \dg^*\Sigma\Seq_{>1}^c\rightarrow\dg^*\Op^c$
to the category of Hopf $\Lambda$-cooperads. We can basically observe that the mapping $\FreeOp^c: \MOp\rightarrow\FreeOp^c(\MOp)$
carries coaugmented $\Lambda$-collections to coaugmented $\Lambda$-cooperads,
Hopf symmetric collections to Hopf cooperads,
and we put both observations together to get that $\FreeOp^c: \MOp\rightarrow\FreeOp^c(\MOp)$
carries Hopf $\Lambda$-collections to Hopf $\Lambda$-cooperads.

We are also going to consider objects, in these intermediate categories of cooperads,
which we obtain after forgetting some part of the structure
of our Hopf $\Lambda$-cooperads in~\S\ref{TripleCoresolution:Construction}.
We basically have a cofree object adjunction for plain cooperads
\begin{equation}\label{TripleCoresolution:CooperadReduction:Adjunction}
\overline{\omega}: \dg^*\Op^c\rightleftarrows\dg^*\Sigma\Seq_{>1}^c :\FreeOp^c
\end{equation}
which we can use to form a triple coresolution functor $\Res_{op}^{\bullet}: \COp\mapsto\Res_{op}^{\bullet}(\COp)$
on the category $\dg^*\Op^c$. (We use the same construction as in~\S\ref{TripleCoresolution:Construction}.)
We then readily check that this plain triple coresolution functor preserves the structure that define our extensions
of the category of cooperads. We accordingly have a square of coresolution functors,
depicted in Figure~\ref{TripleCoresolution:CooperadReduction:Diagram},
which extend our triple coresolution functor on plain cooperads.
\begin{figure}[t]
\begin{equation}
\xymatrix@R=4em@C=6em@!0{ & \cosimp\dg^*\Hopf\Lambda\Op^c\ar[rr]\ar'[d][dd] && \cosimp\dg^*\Hopf\Op^c\ar[dd] \\
\dg^*\Hopf\Lambda\Op^c\ar@<-1pt>[rr]\ar@{.>}[ur]^(0.4){\Res_{op}^{\bullet}}\ar[dd] && \dg^*\Hopf\Op^c\ar[dd]\ar@{.>}[ur] & \\
& \ComOp^c/\cosimp\dg^*\Lambda\Op^c\ar'[r][rr] && \cosimp\dg^*\Op^c \\
\ComOp^c/\dg^*\Lambda\Op^c\ar[rr]\ar@{.>}[ur] && \dg^*\Op^c\ar[ur]_-{\Res_{op}^{\bullet}} & }.
\end{equation}
\caption{}\label{TripleCoresolution:CooperadReduction:Diagram}
\end{figure}
The arrows on the forefront and background squares of this diagram are the obvious forgetful functors between our categories.
The coresolution functor of~\S\ref{TripleCoresolution:Construction}
just occurs on the top (left hand corner) of this chain.
\end{constr}

We have the following statement, where we use the model structures of~\S\ref{Background:CooperadModelCategories}:

\begin{prop}\label{TripleCoresolution:ReedyFibrant}
The triple coresolution $\Res_{op}^{\bullet}(\KOp)$ forms a Reedy fibrant cosimplicial object in the category of cooperads $\dg^*\Op^c$,
for any $\KOp\in\dg^*\Op^c$ (and without any further assumption on $\KOp$).

If we assume $\KOp\in\dg^*\Hopf\Op^c$ (respectively, $\KOp\in\ComOp^c/\dg^*\Lambda\Op^c$, $\KOp\in\dg^*\Hopf\Lambda\Op^c$),
then we get that $\Res_{op}^{\bullet}(\KOp)$ defines a Reedy fibrant cosimplicial object
in the category of Hopf cooperads (respectively, in the category of coaugmented $\Lambda$-cooperads, in the category of Hopf $\Lambda$-cooperads).
\end{prop}

\begin{proof}
The first assertion of this proposition implies the others because our forgetful functors preserve fibrations,
and one just has to dualize the arguments
of~\cite[Proposition II.8.5.5 and Proposition B.1.10]{FresseBook}
in order to check our claim in this case.
\end{proof}

In~\S\ref{Background:CotripleResolution}, we take the geometric realization of the simplicial object $\Res^{com}_{\bullet}(\AOp)$
in the category $\dg^*\Hopf\Lambda\Op^c$ in order to get our cofibrant resolutions
in the category of Hopf $\Lambda$-cooperads
in cochain graded dg-modules.

In~\S\ref{Background:TripleCoresolution}, we take the totalization of the cosimplicial object $\Res_{com}^{\bullet}(\KOp)$
in the category $\dg^*\Hopf\Lambda\Op^c$ in order to get our fibrant coresolutions
of Hopf $\Lambda$-cooperads.
We review the general definition of this totalization functor before tackling the applications to the triple coresolution.

\begin{constr}[The totalization]\label{TripleCoresolution:Totalization}
We assume that $\QOp^{\bullet}$ is any Reedy fibrant cosimplicial object of the category of Hopf $\Lambda$-cooperads for the moment.
To perform the totalization, we first have to pick a simplicial framing
of this object
\begin{equation}\label{TripleCoresolution:Totalization:Framing}
(\QOp^{\bullet})^{\Delta^{\bullet}}\in\simp\cosimp\dg^*\Hopf\Lambda\Op^c
\end{equation}
in the Reedy model category of cosimplicial objects in $\dg^*\Hopf\Lambda\Op^c$.
Then we formally set:
\begin{equation}\label{TripleCoresolution:Totalization:EndFormula}
\Tot(\QOp^{\bullet}) = \int_{\underline{n}\in\Delta}(\QOp^n)^{\Delta^n},
\end{equation}
where we form our end in the category $\dg^*\Hopf\Lambda\Op^c$.

Recall simply that simplicial framings exist by general model category arguments (see for instance~\cite[\S 16.6]{Hirschhorn},
\cite[\S 5.2]{Hovey}, or~\cite[\S II.3.3.3]{FresseBook}),
and that the totalization of a Reedy fibrant cosimplicial object
does not depend on the choice of a particular simplicial framing (see for instance~\cite[Theorem II.3.3.14]{FresseBook}).
The general theory of model categories also implies that $\Tot(\QOp^{\bullet})$
forms a fibrant object in $\dg^*\Hopf\Lambda\Op^c$
as soon as $\QOp^{\bullet}$
is Reedy fibrant.

We can forget about some part of the structure and consider the parallel totalization constructions
in the category of cooperads, in the category of Hopf cooperads,
and in the category of coaugmented $\Lambda$-cooperads.
Let us observe that the forgetful functors which connect all these cooperad categories preserve simplicial framings
since we checked in~\S\ref{Background:CooperadModelCategories}
that these forgetful functors preserve fibrations
and weak-equivalences. Our forgetful functors also preserve limits
since they fit in adjunction relations (see~\S\ref{Background:AlgebraicAdjunctions}).
These observations imply that the forgetful functors to cooperads, to Hopf cooperads,
and to coaugmented $\Lambda$-cooperads
preserve totalizations (up to the indeterminacy of the construction).
\end{constr}

We aim to apply the totalization construction to the triple coresolution $\QOp^{\bullet} = \Res_{op}^{\bullet}(\KOp)$
of our Hopf $\Lambda$-cooperad $\KOp\in\dg^*\Hopf\Lambda\Op^c$.
We need to specify a simplicial framing of this object $\QOp^{\bullet}$.
We devote the following paragraphs to this question.

\begin{rem}[The Sullivan dg-algebra]\label{TripleCoresolution:SullivanDGAlgebra}
To perform our constructions, we use the Sullivan dg-algebras of piecewise linear forms,
which we define by:
\begin{equation}\label{TripleCoresolution:SullivanDGAlgebra:Expression}
\DGOmega^*(\Delta^m) = \kk[x_1,\dots,x_m,dx_1,\dots,dx_m],
\end{equation}
for any $m\in\NN$, where $x_1,\dots,x_m$ are variables of degree $0$, while $dx_1,\dots,dx_m$ denote variables of (upper) degree $1$
which represent the differential of the elements $x_1,\dots,x_m$ inside $\DGOmega^*(\Delta^m)$.

These dg-algebras $\DGOmega^*(\Delta^m)$, $m\in\NN$, form a simplicial object in $\dg^*\ComCat_+$ (we refer to~\cite[\S II.7.1.1]{FresseBook}
for the explicit definition of the simplicial operators in our description
of these dg-algebras). Recall simply that $\DGOmega^*(\Delta^{\bullet})$
defines a simplicial framing of the object $A = \kk$
in the category $\dg^*\ComCat_+$ (see~\cite[Theorem II.7.1.5]{FresseBook}).

Let $\DGN_*(\Delta^m)$ denote the normalized complex of the simplicial set $\Delta^m$, for any $m\in\NN$.
Let $\DGN^*(\Delta^m)$ denote the dual complex of this dg-module $\DGN_*(\Delta^m)$.
In our constructions, we also consider the morphism of dg-modules $\rho: \DGOmega^*(\Delta^m)\rightarrow\DGN^*(\Delta^m)$,
determined by the pairing $\langle-,-\rangle: \DGOmega^*(\Delta^m)\otimes\DGN_*(\Delta^m)\rightarrow\kk$
such that:
\begin{equation}
\langle\omega,[\sigma]\rangle = \int_{\Delta^n}\sigma^*(\omega).
\end{equation}
for the class $[\sigma]$ of any $n$-simplex $\sigma\in(\Delta^m)_n$ in the normalized complex $\DGN_*(\Delta^m)$,
any form $\omega$ of degree $\deg^*(\omega) = n$ in the dg-algebra $\DGOmega^*(\Delta^m)$,
and where we consider the integral over the geometrical simplex $\Delta^m = \{0\leq x_1\leq\dots\leq x_m\leq 1\}$
of the pullback of this form $\omega\in\DGOmega^*(\Delta^m)$
through the morphism $\sigma_*: \Delta^n\rightarrow\Delta^m$
determined by our simplex $\sigma\in(\Delta^m)_n$.
\end{rem}

\begin{constr}[The simplicial framing of the triple coresolution]\label{TripleCoresolution:SimplicialFramingConstruction}
To the triple coresolution $\QOp^{\bullet} = \Res^{\bullet}(\KOp)$ of a Hopf $\Lambda$-cooperad $\KOp\in\dg^*\Hopf\Lambda\Op^c$,
we now associate the cofree cooperads:
\begin{equation}\label{TripleCoresolution:SimplicialFramingConstruction:Expression}
\Res_{op}^n(\KOp)^{\Delta^m}
= \FreeOp^c(\DGC_{op}^n(\KOp)\otimes\DGOmega^*(\Delta^m)),
\end{equation}
where, as cogenerating object, we take the arity-wise tensor products of the cogenerating $\Lambda$-collection
of the triple coresolution $\DGC_{op}^n(\KOp)$
with the cochain dg-algebra $\DGOmega^*(\Delta^m)\in\dg^*\ComCat_+$.
This object $\DGC_{op}^n(\KOp)\otimes\DGOmega^*(\Delta^m) = \{\DGC_{op}^n(\KOp)(r)\otimes\DGOmega^*(\Delta^m),r>1\}$
trivially inherits the structure of a Hopf $\Lambda$-collection when we assume that $\KOp$ is a Hopf $\Lambda$-cooperad,
and we accordingly get a Hopf $\Lambda$-cooperad when we perform our cofree object construction~(\ref{TripleCoresolution:SimplicialFramingConstruction:Expression}).

We have an identity $\Res^n(\KOp)^{\Delta^0} = \Res^n(\KOp)$, for any $n\in\NN$,
and these Hopf $\Lambda$-cooperads $\Res^n(\KOp)^{\Delta^m}$, $m\in\NN$,
clearly form a simplicial object in $\dg^*\Hopf\Lambda\Op^c$,
for any fixed $n\in\NN$.

Recall that the cofaces $d^i$ of the triple coresolution are induced by morphisms
on the cogenerating $\Lambda$-collections of our object $\DGC_{op}^n(\KOp)$
in the case $i>0$. We can just take the morphisms of cofree cooperads induced by the tensor product
of these morphisms $d^i: \DGC_{op}^{n-1}(\KOp)\rightarrow\DGC_{op}^n(\KOp)$
with the identity of the cochain dg-algebra $\DGOmega^*(\Delta^m)$
to provide our objects (\ref{TripleCoresolution:SimplicialFramingConstruction:Expression})
with coface morphisms $d^i: \Res_{op}^{n-1}(\KOp)^{\Delta^m}\rightarrow\Res_{op}^n(\KOp)^{\Delta^m}$
extending the coface morphisms of the triple coresolution
for $i>0$. We proceed similarly to define codegeneracy morphisms $s^j: \Res_{op}^{n+1}(\KOp)^{\Delta^m}\rightarrow\Res_{op}^n(\KOp)^{\Delta^m}$
on our object,
but we need another construction to get a $0$-coface $d^0: \Res_{op}^{n-1}(\KOp)^{\Delta^m}\rightarrow\Res_{op}^n(\KOp)^{\Delta^m}$
and to provide the objects $\Res_{op}^n(\KOp)^{\Delta^m}$
with a full cosimplicial structure.
To begin with, we can observe that we have a natural morphism of Hopf $\Lambda$-collections:
\begin{equation}\label{TripleCoresolution:SimplicialFramingConstruction:FormFactorization}
\overline{\FreeOp}^c(\DGC_{op}^{n-1}(\KOp)\otimes\DGOmega^*(\Delta^m))
\rightarrow\overline{\FreeOp}^c(\DGC_{op}^{n-1}(\KOp))\otimes\DGOmega^*(\Delta^m)
= \DGC_{op}^{n}(\KOp)\otimes\DGOmega^*(\Delta^m),
\end{equation}
for any $n\geq 1$. For this purpose, we use that the cofree cooperad $\overline{\FreeOp}^c(\DGC^{n-1}(\KOp)\otimes\DGOmega^*(\Delta^m))$
basically consists of tree-wise tensors which we form by arranging elements of our cogenerating $\Lambda$-collection
on the vertices of a tree (see for instance~\cite[\S C.1]{FresseBook}).
To form our map (\ref{TripleCoresolution:SimplicialFramingConstruction:FormFactorization}),
we basically gather and multiply the factors $\DGOmega^*(\Delta^m)$
occurring in these tree-wise tensors.
Then we just take the composite of the morphism of cofree cooperads induced by this map~(\ref{TripleCoresolution:SimplicialFramingConstruction:FormFactorization})
with the comonadic coproduct $\nu: \FreeOp^c\rightarrow\FreeOp^c\circ\overline{\FreeOp}^c$
to get our $0$-coface morphism $d^0: \Res_{op}^{n-1}(\KOp)^{\Delta^m}\rightarrow\Res_{op}^n(\KOp)^{\Delta^m}$.

We immediately check that these coface and codegeneracy morphisms fulfill the defining of relations of a cosimplicial object
and commute with the face and degeneracy morphisms of the simplicial structure
on $\DGOmega^*(\Delta^m)$. We accordingly get that the Hopf $\Lambda$-cooperads~(\ref{TripleCoresolution:SimplicialFramingConstruction:Expression})
form a cosimplicial-simplicial object in $\dg^*\Hopf\Lambda\Op^c$.
The identity $\Res^{\bullet}(\KOp)^{\Delta^0} = \Res^{\bullet}(\KOp)$
also holds in the category of cosimplicial objects in $\dg^*\Hopf\Lambda\Op^c$.
We can obviously apply the construction of this paragraph
when we just have a cooperad structure (respectively, a Hopf cooperad structure, a coaugmented $\Lambda$-cooperad structure) on $\KOp$.
In this context, our construction just returns a cosimplicial-simplicial object $\Res_{op}^{\bullet}(\KOp)^{\Delta^{\bullet}}$
of the category of cooperads (respectively, Hopf cooperads, respectively coaugmented $\Lambda$-cooperads)
in $\dg^*\Mod$.
\end{constr}

We now check that:

\begin{prop}\label{TripleCoresolution:SimplicialFraming}
Let $\KOp\in\dg^*\Op^c$. The objects $(\QOp^{\bullet})^{\Delta^m} = \Res_{op}^{\bullet}(\KOp)^{\Delta^m}$, $m\in\NN$,
defined in~\S\ref{TripleCoresolution:SimplicialFramingConstruction}
form a simplicial framing of the triple coresolution $\QOp^{\bullet} = \Res_{op}^{\bullet}(\KOp)$
in the Reedy model category of cosimplicial objects in $\dg^*\Op^c$.

If we assume $\KOp\in\dg^*\Hopf\Op^c$ (respectively, $\KOp\in\ComOp^c/\dg^*\Lambda\Op^c$, $\KOp\in\dg^*\Hopf\Lambda\Op^c$),
then we get that these objects $(\QOp^{\bullet})^{\Delta^m} = \Res_{op}^{\bullet}(\KOp)^{\Delta^m}$, $m\in\NN$,
define a simplicial framing of the triple coresolution $\QOp^{\bullet} = \Res_{op}^{\bullet}(\KOp)$
in the category of Hopf cooperads (respectively, in the category of coaugmented $\Lambda$-cooperads,
in the category of Hopf $\Lambda$-cooperads).
\end{prop}

\begin{proof}
The first assertion of this proposition implies the others since our forgetful functors
preserve fibrations
and weak-equivalences. The proof of this proposition parallels the proof of Proposition~\ref{TripleCoresolution:ReedyFibrant}
and follows
from the same analysis (adapt and dualize the decompositions of~\cite[Proposition II.8.5.5 and Proposition B.1.10]{FresseBook})
after observing that the tensor product with $\DGOmega^*(\Delta^{\bullet})$
gives a simplicial framing functor on cochain graded dg-modules (see~\cite[Theorem II.7.3.5]{FresseBook}).
\end{proof}

The totalization of the triple coresolution of a Hopf $\Lambda$-cooperad $\QOp^{\bullet} = \Res_{com}^{\bullet}(\KOp)$
is equipped with a coaugmentation $\eta: \KOp\rightarrow\Tot(\QOp^{\bullet})$
given on our end~\S\ref{TripleCoresolution:Totalization}(\ref{TripleCoresolution:Totalization:EndFormula})
by the coaugmentation morphism
of the triple coresolution $\epsilon: \KOp\rightarrow\QOp^0$.
(Recall that we have an identity $\QOp^0 = (\QOp^0)^{\Delta^0}$ by definition of a simplicial framing.)
We have a similar observation
when we just have a cooperad structure (respectively, a Hopf cooperad structure, respectively a coaugmented $\Lambda$-cooperad structure)
on $\KOp$.
We have the following statement:

\begin{lemm}\label{TripleCoresolution:CoresolutionClaim}
The coaugmentation of the totalization of the triple coresolution is a weak-equivalence $\eta: \KOp\xrightarrow{\sim}\Tot\Res_{op}^{\bullet}(\KOp)$
for any object of the category of cooperads $\KOp\in\dg^*\Op^c$ (and hence, for any object
of the categories of Hopf cooperads, coaugmented $\Lambda$-cooperads,
and Hopf $\Lambda$-cooperads).
\end{lemm}

\begin{proof}
This lemma follows from the general result of \cite[Theorem II.9.4.11]{FresseBook}.
To be explicit, if we set $\COp^{\bullet} = \Res_{op}^{\bullet}(\KOp)$,
then we have a chain of weak-equivalences:
\begin{equation*}
\Tot(\COp^{\bullet})\xrightarrow{\sim}\Tot(\DGB^c\DGB(\COp^{\bullet}))\xrightarrow{\sim}\DGB^c\DGB(\DGN^*(\COp^{\bullet}))\xleftarrow{\sim}\DGN^*(\COp^{\bullet}),
\end{equation*}
where the morphism on the left-hand side is induces by the natural coaugmentation $\COp^{\bullet})\xrightarrow{\sim}\DGB^c\DGB(\COp^{\bullet})$
(which is a weak-equivalence by the homotopy invariance properties of totalization functors for Reedy fibrant cosimplicial objects),
whereas the rest of our chain is the zigzag of weak-equivalences
defined in \emph{loc. cit.}.
The object $\DGN^*(\COp^{\bullet})$ on the right-hand side is given by an arity-wise application
of the conormalization functor $\DGN^*(-): \cosimp\dg^*\Mod\rightarrow\dg^*\Mod$
to the cosimplicial cooperad $\COp^{\bullet} = \Res_{op}^{\bullet}(\KOp)$ (see \cite[Proposition II.9.4]{FresseBook}).
The result of \cite[Proposition II.9.4.12]{FresseBook} also implies that the morphism $\eta: \KOp\xrightarrow{\sim}\Tot(\COp^{\bullet})$
induced by the coaugmentation of the cotriple resolution $\COp^{\bullet} = \Res_{op}^{\bullet}(\KOp)$
corresponds to the same morphism with values in the conormalized cochain complex $\DGN^*(\COp^{\bullet}) = \DGN^*(\Res_{op}^{\bullet}(\KOp))$.
To complete the proof of the lemma, we mainly use that the triple coresolution $\COp^{\bullet} = \Res_{op}^{\bullet}(\KOp)$
is equipped with contracting extra-codegeneracies
when we pass to the category of symmetric collections,
where we form our conormalized cochain complex $\DGN^*(\COp^{\bullet}) = \DGN^*(\Res_{op}^{\bullet}(\KOp))$
(see~\cite[\S B.1.3 and \S II.8.5.1]{FresseBook} for the verification of the dual statement concerning the cotriple of operads).
This observation implies that the morphism $\eta: \KOp\xrightarrow{\sim}\DGN^*(\COp^{\bullet})$
defines a weak-equivalence (use~\cite[Proposition II.5.4.6]{FresseBook} and a standard spectral sequence argument),
and the conclusion follows.
\end{proof}

In~\S\ref{TripleCoresolution:Totalization}, we mentioned that the totalization of a Reedy fibrant cosimplicial object
defines a fibrant object in the ambient category. Therefore, from Proposition~\ref{TripleCoresolution:ReedyFibrant}
and Lemma~\ref{TripleCoresolution:CoresolutionClaim}
together, we get the following statement:

\begin{thm}\label{TripleCoresolution:MainResult}
Let $\KOp\in\dg^*\Hopf\Lambda\Op^c$.
The totalization of the triple coresolution $\QOp^{\bullet} = \Res_{op}^{\bullet}(\KOp)\in\cosimp\dg^*\Hopf\Lambda\Op^c$
defines a fibrant coresolution $\QOp = \Tot\Res_{op}^{\bullet}(\KOp)$
of the object $\KOp$
in the category of Hopf $\Lambda$-cooperads in cochain graded dg-modules $\dg^*\Hopf\Lambda\Op^c$ (without any further assumption on $\KOp$).
\qed
\end{thm}

The weak-equivalence attached to the coresolution of this theorem
is precisely defined as the morphism of Hopf $\Lambda$-cooperads
$\eta: \KOp\xrightarrow{\sim}\Tot\Res_{op}^{\bullet}(\KOp)$
which we deduce from the coaugmentation $\eta: \KOp\rightarrow\Res_{op}^0(\KOp)$
of the cosimplicial object $\Res_{op}^{\bullet}(\KOp)\in\cosimp\dg^*\Hopf\Lambda\Op^c$
(as in Lemma~\ref{TripleCoresolution:CoresolutionClaim}).
We obviously have an analogous theorem in the context of plain cooperads (respectively, Hopf cooperads, coaugmented $\Lambda$-cooperads),
but we only use the Hopf $\Lambda$-cooperad case
of this statement.

\end{appendix}

\bibliographystyle{plain}
\bibliography{EnOperad-IntrinsicFormality}

\end{document}